\documentclass[CJK,10pt]{article}
\usepackage[T1]{fontenc}
\usepackage{amsfonts}
\usepackage{amssymb}
\usepackage{amsmath}
\usepackage[utf8]{inputenc}
\usepackage[english]{babel}
\usepackage{amsthm}
\usepackage{esint}
\usepackage{dsfont}
\usepackage{geometry}
\usepackage{fancyhdr}
\usepackage{graphicx}
\usepackage{xcolor}
\usepackage{comment}
\usepackage{stmaryrd}
\usepackage[toc,page]{appendix}
\usepackage[nottoc, notlof, notlot]{tocbibind}
\numberwithin{equation}{section}

\allowdisplaybreaks

\newtheorem{proposition}{Proposition}

\newtheorem{theorem}{Theorem}

\newtheorem{lemma}{Lemma}

\newtheorem{cor}{Corollary}
\newcommand{\dd}{\text{d}}
\newcommand{\cc}{\text{C}}
\newcommand{\bb}{\text{B}}
\newcommand{\LL}{\text{L}}
\newcommand{\HH}{\text{H}}

\theoremstyle{definition}
\newtheorem{rem}{Remark}

\usepackage[colorlinks,linkcolor=blue]{hyperref}

\title{\textbf{Optimal decay of the parabolic semigroup in stochastic homogenization for correlated coefficient fields}}
\author{Nicolas CLOZEAU\thanks{Sorbonne Université, CNRS, Université de Paris, Laboratoire Jacques-Louis Lions (LJLL), F-75005, France}}
\begin{document}
\maketitle

\begin{abstract}
We study the large scale behavior of elliptic systems with stationary random coefficient that have only slowly decaying correlations. To this aim we analyze the so-called corrector equation, a degenerate elliptic equation posed in the probability space. In this contribution, we use a parabolic approach and optimally quantify the time decay of the semigroup. For the theoretical point of view, we prove an optimal decay estimate of the gradient and flux of the corrector when spatially averaged over a scale $R\geq 1$. For the numerical point of view, our results provide convenient tools for the analysis of various numerical methods.
\end{abstract}
\begin{center}
\end{center}

\tableofcontents

\section{Introduction}

Elliptic systems with random coefficients were first considered in the $1970'$s by Kozlov in \cite{kozlov1979averaging} and by Papanicolaou and Varadhan in \cite{papanicolaou1979boundary} in the context of qualitative stochastic homogenization. They proved that if the law of the coefficient field is stationary and ergodic, then for all $f\in \LL^2(\mathbb{R}^d)^d$, the solution $(u_{\varepsilon})_{\varepsilon>0}$ of 
$$-\nabla\cdot a(\tfrac{\cdot}{\varepsilon})\nabla u_{\varepsilon}=\nabla\cdot f \text{ in $\mathbb{R}^d$},$$
converges, as $\varepsilon$ tends to $0$,  weakly in $\dot{\HH}^1(\mathbb{R}^d):=\{v\in \HH^1_{\text{loc}}(\mathbb{R}^d)|\nabla v\in \LL^2(\mathbb{R}^d)\}$ to the solution $u_{\text{hom}}$ of the homogenized problem 
$$-\nabla\cdot a_{\text{hom}}\nabla u_{\text{hom}}=\nabla\cdot f \text{ in $\mathbb{R}^d$}.$$
The homogenized coefficients $a_{\text{hom}}$ are characterized in the direction $e$ by the corrector $\phi_e$ defined as the unique (up to an additive constant) sub-linear distributional solution of 
\begin{equation}
-\nabla\cdot a(\nabla\phi_e+e)=0 \text{ in $\mathbb{R}^d$,}
\label{correctorequation}
\end{equation}
via the formula
\begin{equation}
a_{\text{hom}}e=\left\langle a(\nabla\phi_e+e)\right\rangle,
\label{homogematrixintro}
\end{equation}
where $\left\langle \cdot\right\rangle$ denotes the expectation. The corrector equation \eqref{correctorequation} is a key object for the homogenization theory of elliptic systems since its solution captures the spatial oscillations of $(u_{\varepsilon})_{\varepsilon>0}$ induced by the heterogeneity of the coefficient field $a$. This can be expressed in terms of a two-scale expansion
\begin{equation}
u^{\text{2sc}}_{\varepsilon}:=u_{\text{hom}}+\varepsilon \sum_{i=1}^d \phi_{e_i}(\tfrac{\cdot}{\varepsilon})\partial_{i}u_{\text{hom}},
\label{two-scaleintro}
\end{equation}
which allows to reconstruct the oscillation (at first order in $\varepsilon$) of $(\nabla u_{\varepsilon})_{\varepsilon>0}$ in the sense that $(\nabla u_{\varepsilon}-\nabla u^{\text{2sc}}_{\varepsilon})_{\varepsilon>0}$ tends to $0$ strongly in $\LL^2(\mathbb{R}^d)$ as $\varepsilon$ tends to $0$. This estimate is a drastic reduction of complexity since $u_{\text{hom}}$ is the solution of a constant-coefficient equation and $(\phi_{e_i})_{i\in\llbracket 1,d\rrbracket}$ does not depend on $f$.\\

The optimal quantification of those qualitative results is a much more recent and active research field. The first non-perturbative results were obtained in the case of discrete equations with independent and identically distributed coefficient, by Gloria and Otto in \cite{gloria2012optimal,gloria2011optimal} and by Gloria, Neukamm and Otto in \cite{gloria2014optimal,gloria2015quantification}, who studied the corrector equation \eqref{correctorequation} and proved optimal error estimates (in the approximation of the homogenized matrix $a_{\text{hom}}$ by the representative volume element method) as well as optimal variance estimates on the corrector and optimal estimate in $\varepsilon$ of the two-scale expansion \eqref{two-scaleintro}. Second, the continuum case has been studied by Gloria and Otto in \cite{gloria2014regularity,gloria2019quantitative,gloria2017quantitative} for more general model of coefficient fields which satisfy concentration of measure properties via \textbf{functional inequalities} including coefficients with fat tails. The continuum case has also been studied by Armstrong and Smart in \cite{armstrong2016lipschitz}, Armstrong, Mourrat and Kuusi in \cite{armstrong2016mesoscopic,armstrong2017additive,armstrong2019quantitative}, Gloria and Otto in \cite{gloria2015corrector}, in the case where the coefficient fields are randomly distributed according to a stationary ensemble of \textbf{a finite range of dependance}. In \cite{armstrong2016lipschitz}, the notion of large-scale regularity for the random operator $-\nabla\cdot a\nabla$ is introduced (this notion of regularity first started with the work of Avellaneda and Lin in \cite{avellanedaLp} for periodic coefficients). It constitutes by now a very powerful tool to the study of linear elliptic system and can be summarized by
saying that on large-scales (say, scale much larger than the correlation length), the heterogeneous operator $-\nabla\cdot a\nabla$ "inherits" (a suitable version of) the regularity theory for the homogenized operator $-\nabla\cdot a_{\text{hom}}\nabla$. The "large-scale" is characterized by a random minimal scale, for which an optimal moment bound is established, using a sensitivity calculus of Malliavin type in \cite{gloria2014regularity,gloria2019quantitative} or the assumption of a finite range of dependence in \cite{armstrong2016mesoscopic,armstrong2017additive,armstrong2019quantitative,armstrong2016quantitative,armstrong2016lipschitz,gloria2015corrector}.
\\

In the present work, we derive optimal estimates by following the ideas of the papers \cite{gloria2014optimal,gloria2015quantification,gloria2015corrector} and \cite[Sec.9]{armstrong2019quantitative} where the authors considered the semigroup associated with the corrector problem \eqref{correctorequation}, namely the solution $u_e$ of the parabolic system
\begin{equation}
\left\{
    \begin{array}{ll}
        \partial_{\tau}u_e-\nabla\cdot a\nabla u_e=0 & \text{ in $(0,+\infty)\times \mathbb{R}^d$},\\
        u_e(0)=\nabla\cdot a(\cdot)e, & 
    \end{array}
\right.
\label{introequationu}
\end{equation}
for a given unit vector $e\in\mathbb{R}^d$. The relationship between the semigroup $u_e$ and the corrector $\phi_e$ is the following formal integral formula
\begin{equation}
\phi_e=\int_{0}^{+\infty} u_e(t,\cdot)\dd t.
\label{introlinkuphi}
\end{equation}
Indeed, provided we have $\displaystyle\lim_{t\rightarrow +\infty}u_e(t,\cdot)=0$, integrating the equation \eqref{introequationu} in time yields
$$-\nabla\cdot ae-\nabla\cdot a\nabla \int_{0}^{+\infty} u_e(t,\cdot)\dd t=0,$$
and implies that $\int_0^{+\infty}u_e(t,\cdot)\dd t$ is a solution of the corrector equation \eqref{correctorequation} so that \eqref{introlinkuphi} follows by uniqueness. This theoretical relationship allow us to transfer optimal estimates on $u_e$ into optimal estimates on the corrector $\phi_e$. The semigroup $u_e$ is also of numerical interest and can be used as a very convenient tool for the study of numerical method for approximating the correctors and the homogenized matrix $a_{\text{hom}}$. As examples, we present three situations where optimal estimates of $u_e$ are used in the context of numerical methods. 
\begin{itemize}
\item[(i)]First, the study of the representative volume element method, where the method consists of replacing the corrector equation \eqref{correctorequation} by an equation posed in a large box $\text{Q}_L:=[-\frac{L}{2},\frac{L}{2})^d$, for $L\gg 1$,
$$-\nabla\cdot a_L(\nabla\phi^L_e+e)=0 \text{ in $\text{Q}_L$,}$$
with periodic boundary conditions, for a good choice of periodic realization $a_L$. We refer to \cite{gloria2015quantification} for an analysis via a semigroup approach. 
\item[(ii)]Second, the semigroup $u_e$ has been used more recently in \cite{abdulle2020analytical} for approximating $\phi_e$ via exponential regularization, that is we replace the corrector equation \eqref{correctorequation} by 
$$
\left\{
    \begin{array}{ll}
        u_e(T)-\nabla\cdot a(\nabla\phi_{e,T,R}+e)=0 & \text{ in $\text{Q}_R$}, \\
        \phi_{e,T,R}\equiv 0 & \text{ on $\partial \text{Q}_R$}, 
    \end{array}
\right.
$$
for $R\gg 1$ and $T\gg 1$. Optimal estimates on $u_e$ are used to control the bias (or the systematic error). 
\item[(iii)]Finally, we can derive the optimal convergence rate in the massive term approximation as in \cite{gloria2015quantification,gloria2015corrector}, and we propose a proof in Corollary \ref{approxcorrector} of the present contribution. Such optimal estimates have been recently used in \cite{lu2021optimal}. The authors proposed an algorithm for computing the solution of $-\nabla\cdot a\nabla u=\nabla\cdot g$, with a compactly supported dipole density $g$, knowing only the medium $a$ in a box $\text{Q}_{L}$. Despite those results are obtained for a stationary ensemble of a finite range of dependence, an extension in the case of correlated medium is left for a future investigation.
\end{itemize}

The first result in the vein of this article is the optimal decay estimate in time of $u_e$ and of its gradient in the case of discrete elliptic equation satisfying a spectral gap inequality proved by Gloria, Neukamm and Otto in \cite{gloria2015quantification}. In the case of finite range of dependance, similar results are obtained by Gloria and Otto in \cite{gloria2015corrector} and Armstrong, Kuusi and Mourrat in \cite{armstrong2019quantitative}. Their analysis strongly relies on the fast decay of correlations, and does not treat coefficients with fat tails. The aim of this contribution is to provide an extension of those results to more correlated coefficient fields such as Gaussian correlated coefficient fields with fat tails. Our quantitative analysis is based on functional inequalities: We assume that the ensemble $\left\langle\cdot\right\rangle$ satisfies a multi-scale logarithmic Sobolev inequality (as introduced in \cite{duerinckx2017weighted2,duerinckx2017weighted}), that is, there exist $\rho>0$ and an integrable weight $\pi$ such that for all random variables $F$ with finite second moment, the following inequality holds
\begin{equation}
\left\langle F^2(a)\log F^2(a)\right\rangle-\left\langle F^2(a)\right\rangle\left\langle \log F^2(a)\right\rangle\leq \frac{1}{\rho}\left\langle \int_{1}^{+\infty}\ell^{-d}\pi(\ell)\int_{\mathbb{R}^d}\vert\partial^{\text{fct}}_{x,\ell}F(a)\vert^2\dd x\, \dd \ell\right\rangle.
\label{LSI}
\end{equation}
In the inequality above, $\partial^{\text{fct}}_{x,\ell}F$ denotes the functional derivative of $F$ with respect to changes of the coefficient field $a$ localized in the ball of radius $\ell \geq 1$ centered at $x\in\mathbb{R}^d$, which corresponds to the $\LL^1(\bb_{\ell}(x))$ norm of the G\^ateau derivative of $F$ with respect to the coefficient field $a$. Loosely speaking, it measures how sensitively $F$ depends on the coefficient field. As an example of a class of coefficient field, the assumption \eqref{LSI} allow us to consider Gaussian type coefficient fields with slowly decaying correlations: In the case when $a=A(g)$, with $A$ Lipschitz, $g$ a vector-valued centered Gaussian whose covariance function $c : x\in\mathbb{R}^d\mapsto \left\langle g(x)\otimes g(0)\right\rangle$ satisfies $\vert c(x)\vert\leq \gamma(\vert x\vert)$ for some non-increasing Lipschitz function $\gamma: \mathbb{R}^+\rightarrow \mathbb{R}^+$, then \eqref{LSI} holds with the weight $\pi(\ell)=\vert\gamma'(\ell)\vert$ (see \cite[Theorem 3.1 (ii)]{duerinckx2017weighted}). This class includes for instance log-normal random coefficients, that is of the type
$$a(x)= \frac{b+e^{-\tilde{\kappa} (g(x)-m)}}{c+e^{-\kappa(g(x)-m)}}\text{Id},$$
where $b,c>0$ and $\tilde{\kappa},\kappa,m\in\mathbb{R}$. We point out that this contribution is not restricted to the Gaussian setting, and other type of coefficients can be considered as soon as such an estimate of the type \eqref{LSI} is satisfied (see for instance Section \ref{extensionsection} for possible extensions).
\\

Our approach is more in the vein of the series of work \cite{gloria2014regularity,gloria2015quantification,gloria2019quantitative,gloria2015corrector} than in \cite{armstrong2016mesoscopic,armstrong2017additive,armstrong2019quantitative,armstrong2016quantitative,armstrong2016lipschitz}. More precisely, as in \cite{gloria2015quantification}, we use a sensitivity calculus and functional inequalities, albeit in the much weaker form of multiscale logarithmic Sobolev inequality of \cite{duerinckx2017weighted2,duerinckx2017weighted} as in \cite{gloria2014regularity,gloria2019quantitative}. As in \cite{gloria2019quantitative,gloria2015corrector} however, our starting point is to prove fluctuation bounds on the time dependent flux $q(t,\cdot):=a(\cdot)\left(\int_{0}^t\nabla u_e(s,\cdot)\dd s+e\right)$ after averaging in scales $r\leq \sqrt{t}$. Yet, since functional inequalities cannot be easily iterated, one cannot rely on the same approach as for coefficients with a finite range of dependence. To this aim, as in \cite{gloria2019quantitative}, we rely on large-scale regularity, this time in the parabolic setting (in a few words, on large-scales, say scale much larger than the correlation length, the heterogeneous linear parabolic operator $\partial_{\tau}-\nabla\cdot a\nabla$ "inherits" a suitable version of the regularity theory for the homogenized linear parabolic operator $\partial_{\tau}-\nabla\cdot a_{\text{hom}}\nabla$), capitalizing on the bounds on the minimal radius proved in \cite{gloria2014regularity}. We obtain optimal decay estimates in time of the semigroup $u_e$, in terms of scaling, both for mildly and strongly correlated coefficient fields, with good stochastic integrability (stretched exponential moments). The decay of the semigroup gives an alternative proof of the bounds on the correctors recently obtained in \cite[Theorem 1]{gloria2019quantitative} and yields other results of interest in stochastic homogenization, extending the results of \cite{gloria2015quantification,gloria2015corrector} for this setting.\\

The paper is organized as follows: In Section \ref{assumptionssec} we introduce notations and make the assumptions on the coefficient field precise. In Section \ref{resultssec} we state our main result and its consequences. Section \ref{structureproof} describes the strategy of the proof and list the auxiliary results needed in the proof of the main theorem. The results are proved in Section \ref{proofsec}.
\section{Assumptions, notations and main results}
We decide in this paper to use scalar notations but the analysis remains true for systems. Also, we use the abbreviations "r.h.s" and "l.h.s" for right hand side and left hand side respectively.
\subsection{Assumptions and notations}\label{assumptionssec} 
\textbf{Assumptions on the coefficient field. }We fix the dimension $d\geq 2$ and we consider a coefficient field $a : \mathbb{R}^d\rightarrow \mathbb{R}^{d\times d}$ of the form,
\begin{equation}
a(x):=A(g(x)),
\label{coefa}
\end{equation}
for a given Gaussian field $g$ and a given Lipschitz map $A : \mathbb{R}^k\rightarrow \mathbb{R}^{d\times d}$, for some $k\geq 1$, which takes values in the set of uniformly elliptic and bounded matrices. More precisely we assume that there exists $0<\lambda\leq 1$ (fixed once for all) such that for all $s\in\mathbb{R}^k$ and $\xi\in\mathbb{R}^d$
\begin{equation}
\lambda\vert \xi\vert^2\leq \xi\cdot A(s)\xi\quad \text{and}\quad \vert A(s)\xi\vert\leq \vert\xi\vert,
\label{elliptic}
\end{equation}
and that $g : \mathbb{R}^d\rightarrow \mathbb{R}^k$ is a stationary Gaussian field on $\mathbb{R}^d$ of zero mean on some probability space $(\Omega, \mathcal{A},\left\langle\cdot \right\rangle)$, characterized by its covariance function $c : x\in\mathbb{R}^d\mapsto \left\langle g(0)\otimes g(x)\right\rangle$. The ensemble $\left\langle \cdot\right\rangle$ satisfies the standard stationarity and ergodicity assumptions, namely 
\begin{itemize}
\item[(i)]$\left\langle\cdot\right\rangle$ is invariant by the action of $(\mathbb{R}^d,+)$: for all $B\in \mathcal{A}$ and for all $z\in\mathbb{R}^d$, $\left\langle \mathds{1}_{B+z}\right\rangle=\left\langle \mathds{1}_B\right\rangle$ where $B+z:=\{g(\cdot+z)|g\in B\}$,
\item[(ii)]$\left\langle\cdot\right\rangle$ is ergodic for the action of $(\mathbb{R}^d,+)$: for all $B\in\mathcal{A}$ which satisfies $B+z=B$ for all $z\in\mathbb{R}^d$, then $\left\langle \mathds{1}_B\right\rangle\in \{0,1\}$.
\end{itemize}
We assume that there exists a smooth non-increasing function $\gamma$ such that for all $x\in\mathbb{R}^d$
\begin{equation}
\vert c(x)\vert\leq \gamma(\vert x\vert).
\label{correlation}
\end{equation}
 In this context, the following multiscale logarithmic Sobolev inequality holds for all square integrable functional $F$ of $a$ (see \cite[Theorem 3.1,(ii)]{duerinckx2017weighted}): there exists $\rho>0$ such that
\begin{equation}
\text{Ent}(F(a)):=\left\langle F^2(a)\log (F^2(a))\right\rangle-\left\langle F^2(a)\right\rangle\left\langle \log (F^2(a))\right\rangle\leq \frac{1}{\rho} \left\langle\int_{1}^{+\infty}\ell^{-d}\pi(\ell)\int_{\mathbb{R}^d}\vert\partial^{\text{fct}}_{x,\ell} F(a)\vert^2 \dd x\, \dd \ell\right\rangle,
\label{SGinegp}
\end{equation}
where for all $x\in\mathbb{R}^d$ and $\ell\in [1,\infty)$
\begin{equation}
\partial^{\text{fct}}_{x,\ell} F:=\sup\left\{\limsup_{h\rightarrow 0}\frac{F(a+h\delta a)-F(a)}{h},\,\sup_{\bb_{\ell}(x)}\vert \delta a\vert\leq 1,\, \delta a=0 \text{ outside $\bb_{\ell}(x)$}\right\},
\label{functioderidef}
\end{equation}
and where the weight $\pi$ satisfies
$$ \pi(\ell)=\vert\gamma'(\ell)\vert.$$
In this contribution, we specialize to algebraic decay and assume that there exists $\beta>0$ such that for all $\ell \in [1,\infty)$
\begin{equation}
\gamma(\ell)=(\ell+1)^{-\beta}.
\label{assumeMSPC}
\end{equation}
In the particular case where $\beta>d$, which implies that $c$ is integrable, $\left\langle\cdot\right\rangle$ satisfies the standard logarithmic Sobolev inequality (see \cite[Theorem 3.1,(i)]{duerinckx2017weighted}), namely 
$$\text{Ent}(F(a))\leq \frac{1}{\rho}\left\langle\int_{\mathbb{R}^d}\vert\partial_{x,1}^{\text{fct}}F(a)\vert^2\dd x\right\rangle.$$
Let us briefly comment on our quantitative assumption \eqref{SGinegp}. Functional inequalities allow to quantify the ergodicity assumption, in the sense that we obtain a rate of convergence in the Birkhoff ergodic theorem, which ensures that for all stationary random variables $F$ (meaning for all $x\in\mathbb{R}^d$, $F(a,\cdot+x)=F(a(\cdot+x),\cdot)$) with finite second moment, we have
\begin{equation}
\lim_{r\rightarrow +\infty}\fint_{\bb_r}F(a,x)\dd x=\langle F(a,0)\rangle\quad \text{almost surely}.
\label{birkhoff}
\end{equation}
In order to understand the effect of \eqref{SGinegp}, it is instructive to apply it to the spatial average $F_r(a)=\fint_{\bb_r} a(x)\dd x$ of the coefficient field itself. Indeed, for this particular choice, the functional derivative is given by $\partial^{\text{fct}}_{x,\ell}F(a)=\int_{\bb_{\ell}(x)}r^{-d}\mathds{1}_{\bb_r}(y)\dd y$ and by plugging the derivative into \eqref{SGinegp} we obtain the following rate of convergence for \eqref{birkhoff}:
\begin{align*}
\text{Ent}(F_r(a))\leq \frac{1}{\rho}\int_{1}^{+\infty}\ell^{-d}\pi(\ell)\int_{\mathbb{R}^d}\left(\int_{\bb_{\ell}(x)}r^{-d}\mathds{1}_{\bb_r}(y)\dd y\right)^2\dd x\, \dd \ell&\stackrel{\eqref{assumeMSPC}}{\lesssim}r^{-d}\int_{1}^r \ell^{d-\beta-1}\dd \ell+\int_{r}^{+\infty}\ell^{-1-\beta}\dd \ell\\
&\lesssim r^{-\beta}\mathds{1}_{\beta<d}+r^{-d}\log(r)\mathds{1}_{\beta=d}+r^{-d}\mathds{1}_{\beta>d}.\\ 
\end{align*}
Note that when $\beta>d$, we recover the central limit theorem scaling $r^{-\frac{d}{2}}$. The interest of multi-scale logarithmic Sobolev inequalities is that they entail fluctuation bounds for nonlinear functionals
of $a$, and therefore constitute a powerful tool for establishing quantitative estimates on the semigroup $u_e$ defined in \eqref{introequationu}. \newline
\newline
\textbf{General notations.}
For $\mathcal{U}\subset \mathbb{R}^d$ open and $p\in [1,+\infty]$, we denote by $\LL^p(\mathcal{U})$ the Lebesgue space on $\mathcal{U}$ with exponent $p$, that is, the set of measurable functions $f: \mathcal{U}\rightarrow \mathbb{R}^d$ satisfying 
$$\|f\|_{\LL^p(\mathcal{U})}:=\left(\int_{\mathcal{U}}\vert f(x)\vert^p\dd x\right)^{\frac{1}{p}}<+\infty,$$
and where for $p=+\infty$
$$\|f\|_{\LL^{\infty}(\mathcal{U})}:=\inf\{C>0|\vert f(x)\vert\leq C \text{ for almost all $x\in\mathcal{U}$}\}.$$
The vector space of functions on $\mathbb{R}^d$ which belongs to $\LL^p(\mathcal{U})$ whenever $\mathcal{U}$ is bounded is denoted by $\LL^p_{\text{loc}}(\mathbb{R}^d)$. If $\vert\mathcal U\vert<+\infty$ and $f\in \LL^1(\mathcal{U})$, then we write 
$$\fint_{\mathcal{U}} f(x)\dd x:=\frac{1}{\vert \mathcal{U}\vert}\int_{\mathcal{U}}f(x)\dd x.$$
For all $\mathcal{U}$, we denote by $\HH^1(\mathcal{U})$ the space of all measurable functions $f:\mathcal{U}\rightarrow \mathbb{R}^d$ in $\LL^2(\mathcal{U})$ such that $\nabla f$ is in $\LL^2(\mathcal{U})$. We also define $\HH^1_{\text{loc}}(\mathbb{R}^d)$ the space of functions which belongs to $\HH^1(\mathcal{U})$ whenever $\mathcal{U}$ is bounded.\newline
\newline
For all $p\in [1,\infty)$, we denote by $\LL^p_{\left\langle\cdot\right\rangle}(\Omega)$ the space of random variables $X: \Omega\rightarrow \mathbb{R}^d$ satisfying
$$\left\langle X^p\right\rangle^{\frac{1}{p}}<+\infty.$$
If $\mathcal{B}$ is a Banach space, then for all $p\in [1,\infty)$, we denote by $\LL^p(\mathbb{R}^d,\mathcal{B})$ (resp. $\LL^p_{\text{loc}}(\mathbb{R}^d,\mathcal{B})$) the space of measurable functions $f: \mathbb{R}^d\rightarrow \mathcal{B}$ such that $\|f(\cdot)\|_{\mathcal{B}}\in \LL^p(\mathbb{R}^d)$ (resp. $\|f(\cdot)\|_{\mathcal{B}}\in \LL^p_{\text{loc}}(\mathbb{R}^d)$).\newline
\newline
For all time interval $\text{I}:=[t_1,t_2)$ and open subset $\mathcal{U}\subset \mathbb{R}^d$, we define the function space
$$\text{H}^1_{\text{par}}(\text{I}\times \mathcal{U}):=\{u\in \LL^2(\text{I},\HH^1(\mathcal{U}))\vert \partial_{\tau} u\in \LL^2(\text{I},\HH^{-1}(\mathcal{U})\}.$$
We say that $u\in\HH^1_{\text{par}}(\text{I}\times\mathcal{U})$ is a weak solution of 
$$
\left\{
    \begin{array}{ll}
        \partial_{\tau} u-\nabla\cdot a\nabla u=\nabla\cdot f & \text{ in $\text{I}\times \mathcal{U}$}, \\
        u(t_1)=\nabla\cdot g, & 
    \end{array}
\right.
$$
for r.h.s $f\in \LL^2(\text{I}\times \mathcal{U})^d$ and initial data $g\in \LL^q(\mathcal{U})^d$ (for some $q\in [1,+\infty]$) if for all $\psi \in \text{C}^{\infty}_c(\text{I}\times \mathcal{U})$
\begin{align*}
-\int_{\mathcal{U}}\int_{t_1}^{t_2} u(t,x)\partial_{\tau}\psi(t,x)\dd t\, \dd x+\int_{\mathcal{U}}\int_{t_1}^{t_2}\nabla u(t,x)\cdot a(x)\nabla\psi(t,x)\dd t\, \dd x+\int_{\mathcal{U}}g\cdot \nabla\psi(t_1,x)\dd t\, \dd x=-\int_{\mathcal{U}}\int_{t_1}^{t_2} f(t,x)\nabla\psi(t,x)\dd t\, \dd x.
\end{align*}
\newline
\newline
For all $R\geq 1$, we define the exponential kernel $\eta_R$ by
$$\eta_R:=\frac{1}{R^d}e^{-\frac{\vert\cdot\vert}{R}},$$
and the Gaussian kernel $g_R$ by
$$g_R:=\frac{1}{R^d}e^{-\frac{\vert \cdot\vert^2}{R^2}}.$$
For all measurable functions $f$ and all $r>0$, we denote by $f_r$ the convolution with the Gaussian kernel $g_r$, namely
$$f_r:=f\star g_r=\int_{\mathbb{R}^d}f(y)g_r(\cdot-y)\dd y.$$
We say that a random field $X : \Omega\times \mathbb{R}^d\rightarrow \mathbb{R}$ is stationary if we have for all $x\in\mathbb{R}^d$
\begin{equation}
X(a,\cdot+x)=X(a(\cdot+x),\cdot)\text{ almost surely}.
\label{stationarityfield}
\end{equation}
For all $R>0$ and $(s,x)\in\mathbb{R}^{d+1}$, we write $\bb_R(x):=\{y\in\mathbb{R}^d|\vert x-y\vert<R\}$ for the ball of radius $R$ centered at $x$ and $\cc_R(s,x):=(s-R^2,s)\times \bb_R(x)$ for the parabolic cylinder centered at $(s,x)$ and of radius $R$ (for $(s,x)=(0,0)$, we do not write the dependance on $(s,x)$). We use the short-hand notation $\lesssim_{\alpha_1,...,\alpha_n}$ for $\leq C$ for a constant $C$ which depends only on the parameters $(\alpha_i)_{i\in \llbracket 1,n\rrbracket}$.\newline
\newline
We write for all $(a,b)\in\mathbb{R}^2$, $a\vee b=\max\{a,b\}$ and $a\wedge b=\min\{a,b\}$.\newline
\newline
\textbf{Homogenization theory. }We denote by $\phi_e(\cdot,a)\in \HH^{1}_{\text{loc}}(\mathbb{R}^d)$ the corrector, in the direction of a unit vector $e$ of $\mathbb{R}^d$, as the unique distributional solution in $\mathbb{R}^d$ of, for almost all realization of $a$
\begin{equation}
        -\nabla\cdot a(\nabla\phi_e+e)=0 \quad \text{with}\quad \displaystyle\limsup_{R\rightarrow +\infty}\frac{1}{R}\left(\fint_{\bb_R}\vert\phi_e(x)\vert^2\dd x\right)^{\frac{1}{2}}=0 \quad \text{and}\quad\int_{\bb_1}\phi_e(x)\dd x=0.
\label{correctorequation2}
\end{equation}
For the existence of correctors, we refer to \cite[Lemma 1]{gloria2014regularity}.\newline
\newline
For all $T\geq 1$, we denote by $\phi_{e,T}$ the massive corrector, defined as the Lax-Milgram solution in $\HH^1_{\text{uloc}}(\mathbb{R}^d):=\{\psi\in \HH^1_{\text{loc}}(\mathbb{R}^d)|\sup_{x\in\mathbb{R}^d}\int_{\bb_1(x)}\vert\psi\vert^2+\vert\nabla\psi\vert^2<+\infty\}$, to
\begin{equation}
\frac{1}{T}\phi_{e,T}-\nabla\cdot a(\nabla \phi_{e,T}+e)=0 \quad \text{in $\mathbb{R}^d$.}
\label{massivecorrector}
\end{equation}
For the existence and uniqueness of the massive correctors, we refer to \cite{gloria2017quantitative}. Likewise, we denote by $\phi^*_e$ and $\phi^*_{e,T}$ the solutions of \eqref{correctorequation2} and \eqref{massivecorrector} with $a$ replaced by $a^*$, the transposed field of $a$. We denote by $u_e\in \HH^{1,\text{par}}_{\text{uloc}}$ the semigroup associated with the corrector problem \eqref{correctorequation2}, defined as the weak solution of
\begin{equation}
\left\{
    \begin{array}{ll}
        \partial_{\tau}u_e-\nabla\cdot a\nabla u_e = 0 & \text{ in $ (0,+\infty)\times \mathbb{R}^d$}, \\
        u_e(0)=\nabla\cdot a(\cdot)e, &
    \end{array}
\right.
\label{equationu}
\end{equation}
with
$$\HH^{1,\text{par}}_{\text{uloc}}:=\left\{u\in \cc^{0}(\mathbb{R}^+_*,\HH^1_{\text{loc}}(\mathbb{R}^d))\bigg\vert \sup_{T>0}\sup_{R\geq \sqrt{T}}\fint_{\bb_R}\vert(T\nabla u(T,x),\sqrt{T}u(T,x))\vert^2\dd x+\left\vert\int_0^T(\nabla u(s,x),\frac{1}{\sqrt{T}}u(s,x))\dd s\right\vert^2\dd x<+\infty\right\}.$$
For existence and uniqueness of $u_e$, we refer to \cite[Lemma 1]{gloria2015corrector}.\newline
\newline
We also introduce the associated fluxes
\begin{equation}
q_e:= a(\nabla \phi_e+e),
\label{qelliptic}
\end{equation}
for all $T\geq 1$
\begin{equation}
q_{e,T}:= a(\nabla \phi_{e,T}+e),
\label{qellipticmassive}
\end{equation}
and for all $t\geq 0$
\begin{equation}
q_e(t,\cdot):=a(\cdot)\left(\int_{0}^t \nabla u_e(s,\cdot)\dd s+e\right),
\label{defq}
\end{equation}
as well as the associated time dependent corrector, for all $t\geq 0$
\begin{equation}
\phi_e(t,\cdot):=\int_{0}^t u_e(s,\cdot)\dd s.
\label{defphitime}
\end{equation}
We introduce the flux corrector $\sigma=(\sigma_{i,j,k})_{(i,j,k)\in \llbracket 1,d\rrbracket^3}$ as the unique distributional solution in $\mathbb{R}^d$ of, for almost all realization of $a$
\begin{equation}
\nabla\cdot \sigma_i= q_{e_i}\quad\text{and}\quad -\Delta\sigma_{i,j,k}=\partial_j (e_k\cdot q_{e_i})-\partial_k (e_j\cdot q_{e_i}) ,
\label{correctorequationsigma2}
\end{equation}
with
$$\displaystyle\limsup_{R\rightarrow +\infty}\frac{1}{R}\left(\fint_{\bb_R}\vert\sigma_i(x)\vert^2\dd x\right)^{\frac{1}{2}}=0 \quad\text{and}\quad \int_{\bb_1}\sigma_i(x)\dd x=0,$$
where $(\nabla\cdot \sigma_i)_j=\sum_{k=0}^d\partial_k\sigma_{i,j,k}$ and $\partial_i$ denotes the partial derivative with respect to the single coordinate $x_i$. For existence and uniqueness, we refer to \cite[Lemma 1]{gloria2014regularity}.\newline
\newline
Finally, for all $T\geq 1$, we denote by $\sigma_T=(\sigma_{T,i,j,k})_{(i,j,k)\in\llbracket 1,d\rrbracket^3}$ the massive flux corrector, defined as the Lax-Milgram solution in $\HH^1_{\text{uloc}}(\mathbb{R}^d)$ to
\begin{equation}
\frac{1}{T}\sigma_{T,i,j,k}-\Delta\sigma_{T,i,j,k}=\partial_j(e_k\cdot q_{e_i,T})-\partial_k(e_j\cdot q_{e_i,T}).
\label{correctorequationsigmamassive}
\end{equation}
For the existence and uniqueness of the massive flux corrector, we refer to \cite{gloria2017quantitative}.\newline
\newline
The quantities $u_e, \nabla \phi_e, \nabla\sigma, q_e, \phi_{e,T}$ and $q_{e,T}$ are stationary in the sense of \eqref{stationarityfield}, which implies that the distribution of their convolution with some smooth function $f$, under the stationary ensemble $\left\langle\cdot\right\rangle$, does not depend on the space variable. Thus, in the following, we do not distinguish between $F\star f(0)$ and $F\star f$ in our notation, for all stationary random fields $F$.
\subsection{Quantitative results}\label{resultssec}
Our first main result is split in two quantitative estimates on averages of the time dependent flux \eqref{defq}. First, we show that the fluctuations of $(q_e)_r(T)$ on scale $r\in [1,\sqrt{T}]$ decays as the central limit theorem scaling $r^{-\frac{d}{2}}$ times some growth in time which depends on the correlation (in particular, in the case $\beta>d$, we get exactly the central limit theorem scaling). Second, we show that the fluctuations of particular averages $q_e(r^2)\star f_r$, for all $r\geq 1$ and $f_r$ which behaves like $\int_{1}^{r^2}\nabla g_{\sqrt{s}}\dd s$, has some growth in $r$ depending on the correlation. The first result is a key estimate to obtain the optimal decay in time of the semigroup $u_e$, whereas the second is needed to get the optimal growth of the correctors stated in Corollary \ref{boundextendedcorrector}. We prove those estimates for stretched exponential moments. 
\begin{theorem}[Fluctuations of averages of the time dependent flux]\label{semigroup} Let $T\geq 1$ and $e$ be a unit vector of $\mathbb{R}^d$.
\begin{itemize}
\item For all $1\leq r\leq \sqrt{T}$, we have
\begin{equation}
\vert ((q_e)_r(T),\nabla(\phi_e)_r(T))-\left\langle ((q_e)_r(T),\nabla(\phi_e)_r(T))\right\rangle\vert\leq \mathcal{C}_{\star, d, \lambda, \beta}(r)r^{-\frac{d}{2}}\mu_\beta(T)(1+\log^2(\tfrac{\sqrt{T}}{r})),
\label{Sensitilem3}
\end{equation}
with 
\begin{equation}
\mu_{\beta}(T) := \left\{
    \begin{array}{ll}
        T^{\frac{d}{4}-\frac{\beta}{4}} & \text{ if $\beta<d$}, \\
        \log^{\frac{1}{2}}(T) &\text{ if $\beta=d$}, \\
				1     &\text{ if $\beta>d$}.
    \end{array}
\right.
\label{defmubeta}
\end{equation}
\item For all $r\geq 1$ and function $f_r\in \text{C}^{1}_b(\mathbb{R}^d)$ which satisfies for all $x\in\mathbb{R}^d$
\begin{equation}
\vert f_r(x)\vert\lesssim \vert x\vert\int_{1}^{r^2} s^{-1} g_{\sqrt{s}}(x)\dd s \quad\text{and}\quad \vert \nabla f_r(x)\vert \lesssim \vert x\vert^2\int_{1}^{r^2}s^{-2} g_{\sqrt{s}}(x)\dd s,
\label{assumesensiothertest}
\end{equation}
we have 
\begin{equation}
\vert (q_e(r^2),\nabla(\phi_e)(r^2))\star f_r-\left\langle (q_e(r^2),\nabla(\phi_e)(r^2))\star f_r\right\rangle\vert\leq \mathcal{C}_{\star,d,\lambda,\beta}(r)\chi_{d,\beta}(r),
\label{sensiothertest}
\end{equation}
with 
\begin{equation}
\chi_{d,\beta}(r) := \left\{
    \begin{array}{ll}
        (r+1)^{1-\frac{\beta}{2}} & \text{ for $\beta<2$ \text and $d> 2$}, \\
				(r+1)^{1-\frac{\beta}{2}}\log(r+2) & \text{ for $\beta\leq 2$ and
$d=2$},\\
        \log^{\frac{1}{2}}(r+2) &\text{ for $\beta=2$ and $d>2$ or $\beta>2$ and $d=2$}, \\
				1 & \text{ for $\beta>2$ and $d>2$}. 
    \end{array}
\right.
\label{defchibeta}
\end{equation}
The random variable $\mathcal{C}_{\star, d, \lambda, \beta}(r)$ depends on $d,\, \lambda,\, \beta$ and satisfies: for all $\alpha<\frac{1}{\frac{1}{2}+2\frac{d+1}{\beta\wedge d}}$ there exists some constant $C<\infty$ depending on $d$, $\lambda$, $\beta$ and $\alpha$ such that
\begin{equation}
\sup_{r\geq 0}\left\langle \exp(\tfrac{1}{C}\mathcal{C}^{\alpha}_{\star,d, \lambda}(r))\right\rangle\leq 2.
\label{momentboundlem3}
\end{equation}
\end{itemize}
\end{theorem}
Theorem \ref{semigroup} implies the following optimal decay in time of the semigroup $u_e$ (defined in \eqref{equationu}) and of its gradient. This result is in the spirit of \cite[Theorem 1]{gloria2015quantification} established in the discrete setting and extends \cite[Corollary 4]{gloria2015corrector} and \cite[Theorem 9.1]{armstrong2019quantitative} established in the case where the coefficients are randomly distributed according to a stationary ensemble of finite range of dependence to the Gaussian setting.
\begin{cor}[Decay of the semigroup]\label{decayu}There exists a constant $c<\infty$ depending on $\lambda$ and $d$ such that for all $T\geq 1$, $R\geq \sqrt{T}$ and unit vector $e\in\mathbb{R}^d$
\begin{equation}
\left(\int_{\mathbb{R}^d}\eta_{R}(\tfrac{y}{c})\vert (u_e(T,y),\sqrt{T}\nabla u_e(T,y))\vert^2\dd y\right)^{\frac{1}{2}}\leq \mathcal{C}_{\star, d, \lambda,\beta}(T)\eta_{\beta}(T),
\label{estisemigroup}
\end{equation}
with for all $T\geq 1$
\begin{equation}
 \eta_\beta(T)=\left\{
    \begin{array}{ll}
        T^{-\frac{1}{2}-\frac{\beta}{4}} &  \text{ if $\beta<d$}, \\
        \log^{\frac{1}{2}}(T)T^{-\frac{1}{2}-\frac{d}{4}} & \text{ if $\beta=d$},\\
				T^{-\frac{1}{2}-\frac{d}{4}} & \text{ if $\beta>d$},
    \end{array}
\right.
\label{defetabeta}
\end{equation}
and for some random variable $\mathcal{C}_{\star,d,\lambda,\beta}(T)$ which depends on $d,\lambda, \beta$ and satisfies: for all $\alpha< \frac{1}{\frac{1}{2}+
2\frac{d+1}{\beta\wedge d}}$ there exists some constant $C<\infty$ depending on $d,\,\lambda,\, \beta$ and $\alpha$ such that 
$$\displaystyle\sup_{T>0}\left\langle \exp(\tfrac{1}{C}\mathcal{C}^{\alpha}_{\star,d,\lambda,\beta}(T))\right\rangle\leq 2.$$
In particular for all $x\in\mathbb{R}^d$
\begin{equation}
\left\langle\vert \nabla u_e(T,x)\vert^2\right\rangle^{\frac{1}{2}}\lesssim_{d,\lambda}T^{-\frac{1}{2}}\eta_{\beta}(T).
\label{estip2}
\end{equation}
\end{cor}
\begin{rem}\label{laplaciancase}
We comment on the scalings in $T$ in the results of Theorem \ref{semigroup} and Corollary \ref{decayu}. 

\medskip

\noindent\textbf{Scalings in Corollary \ref{decayu}. }The time decay $\eta_{\beta}(T)$ of $u_e$, respectively $T^{-\frac{1}{2}}\eta_{\beta}(T)$ of $\nabla u_e$, is optimal and can be easily inferred in the case of small ellipticity contrast. Indeed, let us consider a coefficient field $a^{\delta}$ with small ellipticity contrast, namely
\begin{equation}\label{SmallContrastEqRemark}
a^{\delta}=\text{Id}+\delta a\quad \text{for} \quad \delta\ll 1.
\end{equation}
The first order approximation in the regime $\delta\downarrow 0$ of $u_e$ is given by $u_e=\delta \overline{u}+\text{o($\delta$)}$ where 
\begin{equation}
\left\{
    \begin{array}{ll}
        \partial_{\tau}\overline{u}-\Delta \overline{u} = 0 & \text{ in $ (0,+\infty)\times \mathbb{R}^d$}, \\
        \overline{u}(0)=\nabla\cdot a(\cdot)e.&
    \end{array}
\right.
\label{firstorderapproximation}
\end{equation}
Using the heat kernel $\Gamma : (T,x)\in\mathbb{R}^+\times \mathbb{R}^d\mapsto \frac{1}{(4\pi T)^{\frac{d}{2}}}e^{\frac{-\vert x\vert^2}{4T}}$, we have the explicit formula
\begin{equation}
\overline{u}(T,x)=\int_{\mathbb{R}^d}\nabla \Gamma(T,x-y)\cdot a(y)e\,\dd y.
\label{equationapproxu}
\end{equation}
It follows from \eqref{functioderidef} that for all $z\in\mathbb{R}^d$ and $\ell\in [1,\infty)$
\begin{equation}
\partial^{\text{fct}}_{z,\ell} \overline{u}(T,x)=\int_{\bb_{\ell}(z)}\vert \nabla \Gamma(T,x-y)\otimes e\vert \,\dd y.
\label{partialderiv}
\end{equation}
Hence, by the multiscale logarithmic Sobolev inequality in form of \eqref{SGinegp1}, we have the following control of the moments of $\overline{u}$: 
$$\left<\vert \overline{u}(T,x)\vert^p\right>^{\frac{1}{p}}\lesssim \sqrt{p}\left(\int_{1}^{+\infty}\ell^{-d}\pi(\ell)\int_{\mathbb{R}^d}\bigg(\int_{\bb_\ell(z)}\vert\nabla \Gamma(T,x-y)\vert \dd y\bigg)^2\dd z\,\dd \ell\right)^{\frac{1}{2}} \quad \text{for all}\quad p\geq 1.$$
We then obtain the decay by splitting the integral into two parts: 
\begin{itemize}
\item[(i)]For $\ell\leq \sqrt{T}$, we use the Jensen inequality and the identity $\int_{\mathbb{R}^d}=\int_{\mathbb{R}^d}\fint_{\bb_{\ell}(x)}\dd x$:
\begin{align*}
\left(\int_{1}^{\sqrt{T}}\ell^{-d}\pi(\ell)\int_{\mathbb{R}^d}\bigg(\int_{\bb_\ell(z)}\vert\nabla \Gamma(T,x-y)\vert \dd y\bigg)^2\dd z\,\dd \ell\right)^{\frac{1}{2}}&\lesssim_d \left(\int_{1}^{\sqrt{T}}\ell^{d}\pi(\ell)\,\dd \ell\right)^{\frac{1}{2}}\left(\int_{\mathbb{R}^d}\vert\nabla \Gamma(T,x-y)\vert^2 \dd y\right)^{\frac{1}{2}}\\
&\lesssim_{d,\beta}\eta_\beta(T).
\end{align*}
\item[(ii)]For $\ell\geq \sqrt{T}$, we use the Minkowski inequality in $\LL^{2}(\mathbb{R}^d)$ (exchanging the order of intagration in the $z$ and $y$ variables) combined with Fubini's theorem:
\begin{align*}
\left(\int_{\sqrt{T}}^{+\infty}\ell^{-d}\pi(\ell)\int_{\mathbb{R}^d}\bigg(\int_{\bb_\ell(z)}\vert\nabla \Gamma(T,x-y)\vert \dd y\bigg)^2\dd z\,\dd \ell\right)^{\frac{1}{2}}&\leq\int_{\mathbb{R}^d}\left(\int_{\sqrt{T}}^{+\infty}\ell^{-d}\pi(\ell)\int_{\mathbb{R}^d}\mathds{1}_{\bb_{\ell}(z)}(y)\vert \nabla \Gamma(T,x-y)\vert^2\dd z\, \dd\ell\right)^{\frac{1}{2}}\dd y\\
&=\vert \bb_1\vert\int_{\mathbb{R}^d}\vert \nabla \Gamma(T,x-y)\vert\dd y\left(\int_{\sqrt{T}}^{+\infty}\pi(\ell)\,\dd \ell\right)^{\frac{1}{2}}\\
&\lesssim_{d,\beta}\eta_\beta(T).
\end{align*}
\end{itemize}
A similar computation gives 
\begin{equation}
\left<\vert \nabla \overline{u}(T,x)\vert^p\right>^{\frac{1}{p}}\lesssim_{d,\beta} \sqrt{p}\,T^{-\frac{1}{2}}\eta_{\beta}(T).\label{boundsmallellipticitycontrastnablau}
\end{equation}
\textbf{Scalings in Theorem \ref{semigroup}. }First, the decay $r^{-\frac{d}{2}}\mu_d(T)$ in \eqref{Sensitilem3} is optimal and can be  inferred as well in the case of small ellipticity contrast. Indeed, we verify this by an explicit computation as previously, considering a coefficient field in the form of \eqref{SmallContrastEqRemark}. We fix $1\leq r\leq \sqrt{T}$. The first order approximation in the regime $\delta\downarrow 0$ of the time depend flux is given by $q_e=\delta \overline{q}+\text{o($\delta$)}$ with
\begin{equation}
\overline{q}(T,\cdot):=\int_{0}^T\nabla\overline{u}(s,\cdot)\dd s+e.
\label{equationfluxapprox}
\end{equation}
Using $\Gamma(T,\cdot)=(4\pi)^{-\frac{d}{2}}g_{\sqrt{4T}}$ and the semigroup property $g_{\sqrt{s}}\star g_r=g_{\sqrt{s+r^2}}$ as well as \eqref{equationapproxu} and \eqref{equationfluxapprox}, we have 
\begin{align}
\overline{q}_r(T,\cdot)&\stackrel{\eqref{equationapproxu}, \eqref{equationfluxapprox}}{=}(4\pi)^{-\frac{d}{2}}\int_{0}^{T}\nabla(\nabla g_{\sqrt{4s}}\star a(\cdot)e\star g_r)\dd s+e\star g_r\nonumber\\
&=(4\pi)^{-\frac{d}{2}}\int_{0}^T\nabla(\nabla g_{\sqrt{4s+r^2}}\star a(\cdot)e)\dd s+e\star g_r\nonumber\\
&\stackrel{\eqref{equationapproxu}}{=}\int_{0}^T\nabla\overline{u}(4s+r^2,\cdot)\dd s+e\star g_r.\label{smallcontrastregime1}
\end{align}
Consequently, using \eqref{boundsmallellipticitycontrastnablau} and $r\leq \sqrt{T}$, we get for all $p\geq 1$
\begin{align*}
\left\langle \vert\overline{q}_r(T)-\left\langle \overline{q}_r(T)\right\rangle\vert^p\right\rangle^{\frac{1}{p}}=\left\langle\left\vert\int_{0}^{T}\nabla\overline{u}(4s+r^2,0)-\left\langle\nabla\overline{u}(4s+r^2,0)\right\rangle\dd s\right\vert^p\right\rangle^{\frac{1}{p}}&\leq \int_{0}^{T}\left\langle\vert\nabla\overline{u}(4s+r^2,0)-\left\langle\nabla\overline{u}(4s+r^2,0)\right\rangle\vert^p\right\rangle^{\frac{1}{p}}\dd s\\
&\lesssim_{d,\beta}\sqrt{p}\int_{0}^{T} (s+r^2)^{-\frac{1}{2}}\eta_{\beta}(s+r^2)\dd s\\
&\lesssim_{d,\beta}\sqrt{p}\,r^{-\frac{d}{2}}\mu_{\beta}(T),
\end{align*}
where $\mu_{\beta}(T)$ is defined in \eqref{defmubeta}.\newline
\newline
Second, the scaling $\chi_{d,\beta}(r)$ in \eqref{sensiothertest} is optimal except for $\beta<2$ and $d=2$. Indeed, we verify this by an explicit computation as previously. We assume for simplicity the more particular form of the averaging function
$$f_r=\int_{1}^{r^2}\nabla g_{\sqrt{\tau}}\,\dd \tau,$$
which satisfies \eqref{assumesensiothertest}. In the regimes $\beta>2,\, d>2$ and $\beta<2,\,d\geq 2$, we argue as in \eqref{smallcontrastregime1} to get $\overline{q}(r^2)\star f_r-\left\langle \overline{q}(r^2)\star f_r\right\rangle=\int_{1}^{r^2}\int_{0}^{r^2}\nabla^2\overline{u}(s+\tau,\cdot)\dd s\, \dd \tau$, and so for all $p\geq 1$
$$\left\langle\vert \overline{q}(r^2)\star f_r-\left\langle \overline{q}(r^2)\star f_r\right\rangle\vert^p\right\rangle^{\frac{1}{p}}=\left\langle\left\vert\int_{1}^{r^2}\int_{0}^{r^2}\nabla^2\overline{u}(4s+\tau,0)\dd s\, \dd\tau\right\vert^p\right\rangle^{\frac{1}{p}}\leq \int_{1}^{r^2}\int_{0}^{r^2}\left\langle \nabla^2\overline{u}(4s+\tau,0)\vert^p\right\rangle^{\frac{1}{p}}\dd s\, \dd \tau,$$
and we then conclude using $\left\langle \vert\nabla^2\overline{u}(s+\tau,0)\vert^p\right\rangle^{\frac{1}{p}}\lesssim_{d,\beta}\sqrt{p}\,(s+\tau)^{-1}\eta_{\beta}(s+\tau)$ (which is obtained with similar arguments than the ones for \eqref{boundsmallellipticitycontrastnablau}). Note that, for $\beta<2$ and $d=2$, we obtain $r^{1-\frac{\beta}{2}}$ and thus the logarithmic contribution in \eqref{sensiothertest} in this case is not optimal. For the regimes $\beta>2,\, d=2$, $\beta=2,\, d>2$ and $\beta=d=2$, this is more subtle even in the small ellipticity contrast regime. We have to bound the fluctuations more carefully using the logarithmic Sobolev inequality, noticing that we have for all $(x,z)\in\mathbb{R}^d\times \mathbb{R}^d$
$$\partial^{\text{fct}}_{z,\ell}(\overline{q}(r^2)\star f_r-\left\langle \overline{q}(r^2)\star f_r\right\rangle)=\int_{\bb_{\ell}(z)}\left\vert \int_{1}^{r^2}\int_{0}^{r^2}\nabla^3\Gamma(s+\tau,y)\otimes e\,\dd y\right\vert.$$
For more details, we refer to the estimates of the first l.h.s term of \eqref{othertestesti2} and \eqref{proofth1othertest10}, which are exactly the ones needed since we may check that $y\in\mathbb{R}^d\mapsto \int_{1}^{r^2}\int_{0}^{r^2}\nabla^3\Gamma(s+\tau,y)\otimes e\,\dd y$ satisfies \eqref{derivothertestesti1primeprime}.

\medskip

Due to the computations done above, the logarithmic correction in \eqref{Sensitilem3} is not optimal. In fact, this correction is here for technical reasons and mostly a consequence of the logarithm contribution in \eqref{lem2}. However, in practice, it has no consequences in the proof of the optimal decay in time of $\nabla u_e$ and also in the proof to obtain the optimal growth of the correctors and its gradient (see Corollary \ref{boundphi} and \ref{boundextendedcorrector}) for which only the regime $T\sim r^2$ is needed.
\end{rem}
Theorem \ref{semigroup} and Corollary \ref{decayu} imply the following four results that are of interest in stochastic homogenization. The first one yields bounds on the gradient and flux of the extended corrector $(\phi_e,\sigma)$, as well as the massive correctors $(\phi_{e,T},\sigma_T)$, which gives an alternative proof of \cite[Theorem 1]{gloria2019quantitative}. Thanks to the decay \eqref{estisemigroup}, the idea of the proof is clear: Since $u_e(T,\cdot)\underset{T\uparrow\infty}{\rightarrow} 0$ and $\int_{0}^{+\infty} \nabla u_e(t,\cdot)\,\dd t$ is well defined in $\LL^2_{\text{loc}}(\mathbb{R}^d,\LL^2_{\left\langle\cdot\right\rangle}(\Omega))$, we have by integrating the equation \eqref{equationu} in time
$$-\nabla\cdot ae-\nabla\cdot a \int_{0}^{+\infty}\nabla u_e(t,\cdot)\dd t=0,$$
and we then recognize the corrector equation \eqref{correctorequation2}. By uniqueness, we then conclude that 
\begin{equation}
\nabla \phi_e=\int_{0}^{+\infty} \nabla u_e(t,\cdot)\,\dd t.
\label{expliequationcor}
\end{equation}
Formula \eqref{expliequationcor} combined with \eqref{estisemigroup} then allow us to prove bounds on the gradient of correctors as well as on the flux.
\begin{cor}[Bounds on the flux and the gradient of correctors]\label{boundphi} We have for all $r\geq 1$, $T\geq 1$ and unit vector $e\in\mathbb{R}^d$
\begin{equation}
\vert ((q_e)_r,(q_{e,T})_r)-\left\langle (q_e)_r,(q_{e,T})_r)\right\rangle\vert+\vert(\nabla (\phi_e)_r,\nabla\sigma_r)\vert+\vert(\nabla(\phi_{e,T})_r,\nabla(\sigma_T)_r)\vert\leq\mathcal{C}_{\star,d,\lambda,\beta}(r)\pi^{-\frac{1}{2}}_\star(r),
\label{correctorbound2}
\end{equation}
with some random variable $\mathcal{C}_{\star,d,\lambda,\beta}(r)$ which depends on $d,\,\lambda,\, \beta$ and satisfies: for all $\alpha<\frac{1}{\frac{1}{2}+\frac{\frac{5}{2}d+2}{\beta\wedge d}}$ there exists some constant $C<\infty$ depending on $d,\,\lambda,\, \beta$ and $\alpha$ such that
$$\displaystyle\sup_{r>0}\left\langle \exp(\tfrac{1}{C}\mathcal{C}^{\alpha}_{\star,d,\lambda,\beta}(r))\right\rangle\leq 2,$$ 
and 
$$
\pi_\star(r)=\left\{
    \begin{array}{ll}
        r^{\beta} & \text{ if $\beta<d$}, \\
        r^d\log^{-1}(r) & \text{ if $\beta=d$},\\
				r^d & \text{ if $\beta>d$} .
    \end{array}
\right.
$$
In particular, for $\beta>d$, the quantities decay as the central limit theorem scaling $r^{-\frac{d}{2}}$.
\end{cor}
Corollary \ref{boundphi} combined with Theorem \ref{semigroup} implies the following growth on the extended corrector $(\phi_e,\sigma)$. 
\begin{cor}\label{boundextendedcorrector}
We have for all unit vector $e\in\mathbb{R}^d$ and $x\in\mathbb{R}^d$
\begin{equation}
(\vert(\phi_e,\sigma)-(\phi_e,\sigma)_1(0)\vert^2)^{\frac{1}{2}}_1(x)\leq \mathcal{C}_{\star,d,\lambda,\beta}(x)\xi_{d,\beta}(\vert x\vert),
\label{growthcorrectormainresult}
\end{equation}
with 
\begin{equation}
\xi_{d,\beta}(\vert x\vert) := \left\{
    \begin{array}{ll}
       (\vert x\vert+1)^{1-\frac{\beta}{2}} &\text{ for $\beta<2$,}\\
        \log^{\frac{1}{2}}(\vert x\vert+2) &\text{ for $\beta=2$, $d> 2$ or $\beta>2$, $d=2$,}\\
				\log(\vert x\vert+2) &\text{ for $\beta=d=2$}, \\
		1  &\text{for $\beta>2$, $d>2$},
    \end{array}
\right.
\label{defxidbeta}
\end{equation}
and some random variable $ \mathcal{C}_{\star,d,\lambda,\beta}(x)$ which depends on d,$\lambda$,$ \beta$ and satisfies: for all $\alpha<\frac{1}{\frac{1}{2}+\frac{\frac{5}{2}d+2}{\beta\wedge d}}$ there exists some constant $C<\infty$ depending on $d,\,\lambda,\, \beta$ and $\alpha$ such that
$$\sup_{x\in\mathbb{R}^d}\left\langle \exp(\tfrac{1}{C}\mathcal{C}^{\alpha}_{\star,d,\lambda,\beta}(x))\right\rangle\leq 2.$$ 
\end{cor}
\begin{rem}
The choice of the convolution with the Gaussian in \eqref{correctorbound2} and \eqref{growthcorrectormainresult} is not crucial. Indeed, for all $r\geq 1$ and $f_r:=r^{-d}f(\frac{\cdot}{r})$ with $f\in \text{W}^{\alpha,1}(\mathbb{R}^d)$ (for some $\alpha>0$), we may deduce from \eqref{correctorbound2} and \eqref{growthcorrectormainresult} that
$$\left\vert \int_{\mathbb{R}^d}\psi(y) f_r(y)\dd y\right\vert\leq \mathcal{C}_{1,\star,d,\lambda,\beta}(r)\pi^{-\frac{1}{2}}_{\star}(r),$$
where $\psi$ can be replaced by one of the quantities which appears in \eqref{correctorbound2}, as well as for all $x\in\mathbb{R}^d$
$$(\vert(\phi_e,\sigma)-(\phi_e\star f,\sigma\star f)(0)\vert^2\star f)^{\frac{1}{2}}(x)\leq \mathcal{C}_{2,\star,d,\lambda,\beta}(x)\xi_{d,\beta}(\vert x\vert).$$
The two random variables $\mathcal{C}_{1,\star,d,\lambda,\beta}(r)$ and $\mathcal{C}_{2,\star,d,\lambda,\beta}(x)$ depend on $f$ and have the same stochastic integrability as in Corollary \ref{boundphi} and Corollary \ref{boundextendedcorrector}. For more details, we refer to \cite[Remark 4.28]{armstrong2019quantitative}.
\end{rem}
From Corollary \ref{boundextendedcorrector}, we obtain the following quantitative convergence of the two-scale expansion.
\begin{cor}[Convergence rate of the two-scale expansion] \label{2scaleexp}
Let $g\in \HH^1(\mathbb{R}^d)$ such that $\xi_{d,\beta}(\vert\cdot\vert)\nabla g\in \LL^{2}(\mathbb{R}^d)$, and for all $\varepsilon>0$ let $v_{\varepsilon}$ and $v_{\text{hom}}$ be the Lax-Milgram solutions, in $\dot{H}^1(\mathbb{R}^d):=\{v\in H^1_{\text{loc}}(\mathbb{R}^d)|\nabla v\in L^2(\mathbb{R}^d)\}/\mathbb{R}$, of 
\begin{equation*}
-\nabla\cdot a(\tfrac{\cdot}{\varepsilon})\nabla v_{\varepsilon}=\nabla\cdot g\quad \text{and}\quad -\nabla\cdot a_{\text{hom}}\nabla v_{\text{hom}}=\nabla\cdot g,
\label{twoscaleexpmainresult}
\end{equation*}
with $a_{\text{hom}}$ defined in \eqref{homogematrixintro}. Consider the two-scale expansion error 
$$z_{\varepsilon}:=v_{\varepsilon}-(v_{\text{hom},\varepsilon}+\varepsilon \sum_{i=1}^d \phi_{e_i}(\tfrac{\cdot}{\varepsilon})\partial_i v_{\text{hom},\varepsilon}),$$
where $v_{\text{hom},\varepsilon}$ is a simple moving average of $v_{\text{hom}}$ at scale $\varepsilon$, that is $v_{\text{hom},\varepsilon}=(v_{\text{hom}})_{\varepsilon}(0)$. Then
$$\left(\int_{\mathbb{R}^d}\vert\nabla z_{\varepsilon}(x)\vert^2\dd x\right)^{\frac{1}{2}}\leq \mathcal{C}_{\star,d,\lambda,\beta,g}(\varepsilon)\,\varepsilon \xi_{d,\beta}(\varepsilon^{-1})\left(\int_{\mathbb{R}^d}\xi^2_{d,\beta}(\vert x\vert)\vert\nabla g(x)\vert^2\dd x\right)^{\frac{1}{2}},$$
where $\xi_{d,\beta}$ is defined in \eqref{defxidbeta} and for some random variable $\mathcal{C}_{\star,d,\lambda,\beta,g}(\varepsilon)$ which depends on $d$, $\lambda$, $\beta$, $g$ and satisfies: for all $\alpha<\frac{1}{\frac{1}{2}+\frac{\frac{5}{2}d+2}{\beta\wedge d}}$, there exists some constant $C$ depending on $d$, $\lambda$, $\beta$, $g$ and $\alpha$ such that
$$\sup_{\varepsilon>0}\left\langle\exp(\tfrac{1}{C}\mathcal{C}^{\alpha}_{\star,d,\lambda,\beta,g}(\varepsilon))\right\rangle\leq 2.$$
\end{cor}
We emphasize that the results of Corollary \ref{boundextendedcorrector} and Corollary \ref{2scaleexp} are already contained in \cite{gloria2019quantitative}. The main differences are in the way of averaging and the stochastic integrability, slightly better in \cite{gloria2019quantitative} but still sub-optimal.
\begin{rem}
The need for local averages at scale $\varepsilon$ of $v_{\text{hom}}$ is due to the fact that the corrector estimate \eqref{defxidbeta} only holds for averages of $(\phi_e,\sigma)$ under minimal regularity assumption on $a$. However, from De Giorgi-Nash-Moser theory in the case of scalar equations and from the classical Schauder theory in the case of systems with Hölder continuous realization of the coefficient field $a$ (which can be ensured by additional assumptions on the covariance function $c$, see for instance \cite[Lemma 3.1]{josien2020annealed}), we may improve the estimate \eqref{defxidbeta} into a pointwise estimate. Therefore, in both cases, there is no need to consider local averages of $v_{\text{hom}}$ at scales $\varepsilon$.
\end{rem}
For a proof of Corollary \ref{2scaleexp} based on the results of Corollary \ref{boundextendedcorrector}, we refer the reader to \cite{gloria2019quantitative}. The second consequence of Corollary \ref{decayu} is a new optimal control of the sub-systematic error, extending the bound obtained in \cite[Lemma 8]{gloria2015quantification} in the case of discrete elliptic equations and the one in \cite[Theorem 3]{gloria2015corrector} for a finite range of dependence. This corollary is of numerical interest for approximating the homogenized matrix $a_{\text{hom}}$ defined in \eqref{homogematrixintro}.
\begin{cor}[Sub-systematic error]\label{approxcorrector}Let $(\phi_{e_i,T})_{i\in\llbracket 1,d\rrbracket}$ be defined in \eqref{massivecorrector}. For all $(i,n)\in\llbracket 1,d\rrbracket\times \mathbb{N}$, we define the Richardson extrapolation of $\phi_{e_i,T}$ with respect to $T$ by
$$
 \left\{
    \begin{array}{ll}
         \phi^{n+1}_{e_i,T}=\frac{1}{2^n-1}(2^n\phi^n_{e_i,2T}-\phi^n_{e_i,T})& \text{for all $n\geq 1$}, \\
        \phi^1_{e_i,T}=\phi_{e_i,T}, &
    \end{array}
\right.
$$
and likewise for $\phi^*_{e_i,T}$. We define the approximation $(\overline{a}^n_T)_{n\in\mathbb{N}}$ of the homogenized coefficients $a_{\text{hom}}$ by: for all $(i,j,n)\in\llbracket 1,d\rrbracket^2\times \mathbb{N}$
\begin{equation}
e_j\cdot \overline{a}^n_{T}e_i=\left\langle (\nabla \phi^{*,n}_{e_j,T}+e_j)\cdot a(\nabla \phi^{n}_{e_i,T}+e_i)\right\rangle.
\label{defabarnT}
\end{equation}
We have the following estimates of the sub-systematic errors: for all $d\geq 2$ and $n>\frac{\beta\vee d}{4}$
\begin{equation}
\left\langle \vert\nabla \phi^n_{e_i,T}-\nabla\phi_{e_i}\vert^2\right\rangle^{\frac{1}{2}}\lesssim T^{\frac{1}{2}}\eta_{\beta}(T),
\label{subsysteT}
\end{equation}
and 
\begin{equation}
\vert \overline{a}^n_T-a_{\text{hom}}\vert\lesssim T\eta^2_{\beta}(T),
\label{subsystehom}
\end{equation}
where $\eta_{\beta}$ is as in \eqref{defetabeta}.
\end{cor} 
Finally, Corollary \ref{approxcorrector} implies the following bound on the bottom of the spectrum of $-\nabla\cdot a\nabla$ projected on $\nabla\cdot a(0)e$ and extends \cite[Corollary 5]{gloria2015corrector}, \cite[Corollary 1]{gloria2015quantification} to correlated fields. Let us recall that stationarity allows us to define a differential calculus in probability through the correspondence for stationary fields: for all stationary fields $\psi: \Omega\times \mathbb{R}^d\rightarrow \mathbb{R}$ we define for all $i\in\llbracket 1,d\rrbracket$: 
$$D_i\psi(0)=\lim_{h\rightarrow 0}\frac{\psi(a(\cdot +h e_i),0)-\psi(a,0)}{h}=\lim_{h\rightarrow 0}\frac{\psi(a,he_i)-\psi(a,0)}{h}=\nabla_i \psi(a,0),$$
and we set $D\psi:=(D_i\psi(0))_{i\in\llbracket 1,d\rrbracket}$. We define the Hilbert space $\mathcal{H}^1:=\{\psi\in \LL_{\left\langle\cdot\right\rangle}^{2}(\Omega)|\left\langle \vert D\psi\vert^2\right\rangle<+\infty\}$. In the case when the coefficients $a$ are symmetric, the operator $L:=-D\cdot a(0)D$ defines a quadratic form on $\mathcal{H}^1$. We denote by $\mathcal{L}$ its Friedrichs extension on $\LL^2_{\langle\cdot\rangle}(\Omega)$. Since $\mathcal{L}$ is a self-adjoint non-negative operator, by the spectral theorem it admits a spectral resolution: for all $\Theta\in \LL^2_{\left\langle\cdot\right\rangle}(\Omega)$, there exists a unique measure $\nu_{\Theta}$ such that for all $g\in \LL^{\infty}(\mathbb{R}^+)$
\begin{equation}
\langle g(\mathcal{L})\Theta,\Theta\rangle=\int_{0}^{+\infty}g(\lambda)\dd \nu_{\Theta}(\lambda).
\label{spectralresolution}
\end{equation}
\begin{cor}[Spectral resolution]\label{spectralreso}Let assume that the map $A$ defined in \eqref{coefa} takes values in the set of symmetric matrices and assume that $\Theta:=D\cdot a(0)e$ is in $\LL^2_{\left\langle\cdot\right\rangle}(\Omega)$ for some unit vector $e\in\mathbb{R}^d$. We denote by $\nu_{\Theta}$ the spectral measure, defined in \eqref{spectralresolution}, of the operator $-D\cdot a(0)\cdot D$ associated to the vector $\Theta$. We have
$$\left\langle\int_{0}^{\mu} \dd\nu_{\Theta}(\zeta)\right\rangle\lesssim \eta^2_{\beta}(\mu^{-1}) \quad\text{for all}\quad 0<\mu\leq 1,$$
where $\eta_\beta$ is as in \eqref{defetabeta}.
\end{cor}
\subsection{Extension to other model of coefficient field}\label{extensionsection}
The approach we develop here is not limited to the Gaussian setting. For coefficient field $a$ for which the law satisfies multiscale functional inequalities with oscillation, similar result to the ones presented in this paper hold. More precisely, assume that there exists $\rho>0$ such that for all square integrable functional $F$ of $a$, we have 
\begin{equation}
\text{Ent}(F(a))\leq \frac{1}{\rho}\left\langle \int_{1}^{+\infty}\ell^d\pi(\ell)\int_{\mathbb{R}^d}\vert\partial^{\text{osc}}_{x,\ell} F(a)\vert^2\dd x\, \dd\ell\right\rangle,
\label{LSIoscillation}
\end{equation}
with, for some $C>0$ and $\beta>0$
\begin{equation}
\pi(\ell)=e^{-\frac{1}{C}\ell^{\beta}},
\label{LSiweightoscillation}
\end{equation}
and for all $(x,\ell)\in\mathbb{R}^d\times [1,\infty)$ 
\begin{equation}
\partial^{\text{osc}}_{x,\ell}F(a):=\sup\{F(a')-F(a'')|a'=a''=a \text{ on }\mathbb{R}^d\backslash \bb_{\ell}(x)\}.
\label{derivativeoscillation}
\end{equation}
Then, with the notations $\mu_{\beta}(T)=1$, $\eta_{\beta}(T)=T^{-\frac{1}{2}-\frac{d}{4}}$, $\pi_{\star}(r)=r^d$ and $\xi_{d,\beta}(\vert x\vert)=\log^{\frac{1}{2}}(\vert x\vert+2)$ if $d=2$ and $\xi_{d,\beta}(\vert x\vert)=1$ if $d\geq 3$, the results of section \ref{resultssec} hold with a random variable $\mathcal{C}_\star$ (possibly depending on $d$, $\lambda$, $\beta$, $g$, $x$, $r$ and $T$) with stretched exponential moments for some exponent $\alpha$ (depending on $d$ and $\beta$) uniform in $x$, $r$ and $T$ when it depends on this parameters.\newline
\newline
Multiscale logarithm Sobolev inequality of type \eqref{LSIoscillation} are satisfied, for instance, by random inclusions with random radii and random tessellations of Poisson points or the random parking measure. For more precise details, we refer to \cite{duerinckx2017weighted2,duerinckx2017weighted}. For completeness and to see the differences compared to the Gaussian setting, we provide in Appendix \ref{proofoscillation}, a proof of Theorem \ref{semigroup} under the assumption \eqref{LSIoscillation}, when $u_e$ is real valued and $a$ satisfy a regularity assumption. The proofs of the general case may be extended by following the arguments of Appendix \ref{proofoscillation} and Section \ref{proofsec}.
\section{Structure of the proof}\label{structureproof}
Let us now describe the strategy of the proof of Theorem \ref{semigroup}, together with a flow of auxiliary results. In the rest of the paper, for notational convenience, we do not write the dependence of all quantities on the unit vector $e$, fixed once for all. 
\subsection{Main steps and heuristic arguments}\label{stratsec}
\textbf{General strategy of the proof. }The proof uses two important quantities: for all $t\in \mathbb{R}^+_*$
\begin{equation}
\mathcal{Q}^r_1:=\sqrt{t}\left(\int_{\mathbb{R}^d}\eta_{\sqrt{2}r}(y)\vert\nabla u(t,y)\vert^2\dd y\right)^{\frac{1}{2}}\quad  \text{for any $r\geq\sqrt{t}$},
\label{strategyproof_quanti1}
\end{equation}
and 
\begin{equation}
\mathcal{Q}^{r}_2(y):= q_r(t,y)-\left\langle q_r(t,y)\right\rangle \quad \text{for any $r\leq \sqrt{t}$ and $y\in\mathbb{R}^d$,}
\label{strategyproof_quanti2}
\end{equation}
and their relationship. On the one hand, using the estimate \cite[Lemma 6]{gloria2015corrector}, we have a deterministic relationship between \eqref{strategyproof_quanti1} and averages in space and in $r$ of \eqref{strategyproof_quanti2}, recalled in Lemma \ref{cacciopo}. On the other hand, using sensitivity estimates (see Lemma \ref{functioderiv} and Proposition \ref{Sensitivitysubop}) and the multiscale logarithmic Sobolev inequality \eqref{SGinegp}, we can control moments of \eqref{strategyproof_quanti2} by moments of \eqref{strategyproof_quanti1}. The main difficulty is that the estimates are coupled in an intricate way, which does not allow to buckle easily. We overcome this difficulty by, first deriving nearly-optimal estimates in scaling in $r,\, t$ of moments of \eqref{strategyproof_quanti2} from a sub-optimal deterministic bound in $t$ of \eqref{strategyproof_quanti1}, which is itself based on deterministic energy estimates (see Lemma \ref{nergyestideter}). Second, from the nearly-optimal moment bounds of \eqref{strategyproof_quanti2}, we improve the decay in $t$ of the moments of \eqref{strategyproof_quanti1}, which allow us to deduce the optimal scaling in $r,\, t$ of \eqref{strategyproof_quanti2}, which leads to Theorem \ref{semigroup}. We then finally obtain from Theorem \ref{semigroup} and Lemma \ref{cacciopo} the optimal decay in time of the moments of \eqref{strategyproof_quanti1}, which leads to Corollary \ref{decayu}. In this contribution, our main effort is to derive the sensitivity estimates and the control of moments  of \eqref{strategyproof_quanti2}. In the following, we focus on the main ideas of the proof of \eqref{Sensitilem3}. For \eqref{sensiothertest}, the ideas are similar and a few words on the differences are given at the end of this section.\newline
\newline
\textbf{Sensitivity estimates. }The proof of the sensitivity estimates combined two different types of arguments.
\begin{enumerate}
\item \textbf{Deterministic arguments. }There are two main ingredients. The first ingredient is the classical $\LL^2$ theory of parabolic systems in form of localized energy type estimates, see Lemmas \ref{nergyestideter} and \ref{scott}. The second ingredient is the large-scale regularity theory for parabolic systems developed in \cite{bella2017liouville} that we recall  and extend in Appendix \ref{reggech}. This provides, in particular, a large-scale $\cc^{0,1}$ estimate: for all $x\in\mathbb{R}^d$, there exists a stationary random variable $r_*(x)\geq 1$ such that for all $t\in\mathbb{R}$ and weak solution $v$ of, for $R\geq r_*(x)$, $\partial_{\tau} v-\nabla\cdot a\nabla v=\nabla\cdot g \text{ in $\cc_R(t,x)$}$, we have 
\begin{equation}\label{stratmeanvalue}
\begin{aligned}
\fint_{\cc_{r_*(x)}(t,x)}\vert\nabla v(s,y)\vert^2\dd s\, \dd y\lesssim_{d,\lambda}&\fint_{\cc_{R}(t,x)}\vert\nabla v(s,y)\vert^2\dd s\, \dd y\\
&+\sup_{\rho\in [r_*(x),R]}\left(\frac{R}{\rho}\right)^{2\alpha}\fint_{\cc_{\rho}}\left\vert g(s,y)-\fint_{\cc_{\rho}(t,x)} g(s',z)\dd s'\, \dd z\right\vert^2\dd s\, \dd y .
\end{aligned}
\end{equation}
These properties can be used provided $r_*$ has good moment bounds, which have already been established in \cite{gloria2019quantitative} in our context.
\item \textbf{Stochastic arguments. }Moment bounds on $q_r(T)-\left\langle q_r(T)\right\rangle$ will be obtained from the multiscale logarithmic Sobolev inequality \eqref{SGinegp}, and more precisely in its version of Proposition \ref{SGp} allowing a control of high moments: for all $T\geq 1$, $r\leq \sqrt{T}$ and $p\geq 1$,
\begin{equation}
\left\langle \vert q_r(T)-\left\langle q_r(T)\right\rangle\vert^p\right\rangle^{\frac{1}{p}}\lesssim \sqrt{p}\left\langle \left(\int_{1}^{+\infty}\ell^{-d}\pi(\ell)\int_{\mathbb{R}^d}\vert\partial^{\text{fct}}_{x,\ell} q_r(T)\vert^2\dd x\, \dd\ell\right)^{\frac{p}{2}}\right\rangle^{\frac{1}{p}}.
\label{SGpstrategy}
\end{equation}
To use \eqref{SGpstrategy}, we have to estimate the sensitivity of $q_r(T)$ with respect to the coefficient field $a$, namely the quantity $\int_{\mathbb{R}^d}\vert\partial^{\text{fct}}_{x,\ell} q_r(T)\vert^2\dd x$ for any $\ell\geq 1$. The method used here is inspired by the series of articles \cite{fischer2017sublinear,gloria2014optimal,gloria2014regularity,gloria2019quantitative} which treats the case of elliptic systems and proceeds by duality. The results are summarized in Lemma \ref{functioderiv} for the computation of $\partial^{\text{fct}}_{x,\ell} q_r(T)$ and in Proposition \ref{Sensitivitysubop} for the sensitivity calculus and the control of moments.
\end{enumerate}
The localized energy estimates of the deterministic part are classical and rely only on $\LL^2$ theory for parabolic systems. The contribution of this paper is more on the stochastic part.  We now describe the main ideas and we perform heuristic computations leading to \eqref{Sensitilem3}. For simplicity, we do this in a simpler case with two additional assumptions: 
\begin{itemize}
\item[(i)]We assume that the solution $u$ of \eqref{equationu} is real-valued and $a$ is symmetric. In that case, we get the uniform bound of $\nabla u$: 
\begin{equation}
\|\nabla u(t,\cdot)\|_{\LL^{\infty}(\mathbb{R}^d)}\lesssim_{d,\lambda}t^{-1}\quad \text{for all $t>0$},
\label{strategyunifboundnablau}
\end{equation}
see \cite[Lemma 9.2]{armstrong2019quantitative} for a proof. The (sub-optimal) deterministic bound \eqref{strategyunifboundnablau} is our starting point to prove the sensitivity estimate. 
\item[(ii)]We assume that $\nabla\cdot ae\in \LL^{\infty}(\mathbb{R}^d)$ and 
\begin{equation}
\|\nabla\cdot ae\|_{\LL^{\infty}(\mathbb{R}^d)}\leq 1.
\label{strat_boundLinftyavoidsing}
\end{equation}
This assumption allow us to avoid the singularity at $t=0$ and to use the localized energy estimate in the form: for all $T\geq 1$, $R\geq \sqrt{T}$ and $z\in\mathbb{R}^d$
\begin{equation}
\int_{0}^T\fint_{\bb_R(z)}\vert\nabla u(t,x)\vert^2\dd x\, \dd t\lesssim_{d,\lambda} 1,
\label{strat_energyestiavoidsing}
\end{equation}
see for instance \cite[Lemma 2]{gloria2015corrector} for a proof.
\end{itemize}
We proceed in three steps.
\begin{itemize}
\item[I)] The first step identifies the functional derivative of each components $\partial^{\text{fct}}_{x,\ell}q_r(T)\cdot e_k$, defined in \eqref{functioderidef} for $T\geq 1$ and $r\leq \sqrt{T}$. Formally, we have for all $(x,\ell)\in\mathbb{R}^d\times [1,\infty)$, 
\begin{equation}\label{ReviewEquation1}
\partial^{\text{fct}}_{x,\ell}q_r(T)\cdot e_k=\int_{\bb_{\ell}(x)}\left\vert \frac{\partial }{\partial a(y)}q_r(T)\cdot e_k\right\vert\dd y.
\end{equation}
Using $\frac{\partial}{\partial a(y)} a=\delta_y$ and the chain rule, we compute
\begin{equation}
\frac{\partial }{\partial a(y)}q_r(T)\cdot e_k=g_r(y)e_k\cdot e+g_r(y)e_k\cdot \int_{0}^T\nabla u(t,y)\dd t+\int_{\mathbb{R}^d}g_r(z)e_k\cdot a(z)\left(\int_{0}^T\nabla \frac{\partial}{\partial a(y)}u(t,z)\dd t\right)\dd z,
\label{partialformalderivative}
\end{equation}
with from \eqref{equationu}
\begin{equation}
\left\{
    \begin{array}{ll}
        \partial_\tau\frac{\partial}{\partial a(y)}u-\nabla\cdot a\nabla\frac{\partial}{\partial a(y)}u=\nabla\cdot\delta_y\nabla u & \text{ in $(0,+\infty)\times \mathbb{R}^d$}, \\
        \frac{\partial}{\partial a(y)}u(0)=\nabla\cdot \delta_y\,e. & 
    \end{array}
\right.
\label{equationformalpartialu}
\end{equation}
The first two r.h.s terms of \eqref{partialformalderivative} are directly controlled, in $\LL^2(\mathbb{R}^d)$, via \eqref{strategyunifboundnablau} and \eqref{strat_energyestiavoidsing}, whereas the control of the last term is more technical. The idea is to rewrite this term by duality. Introducing the solution $v^T=(v^T_k)_{k\in\llbracket 1,d\rrbracket}$ of the corresponding dual problem of \eqref{equationu} with final time $T$, that is the backward parabolic system
\begin{equation}
\left\{
    \begin{array}{ll}
        \partial_{\tau}v_k^T+\nabla\cdot a^*\nabla v_k^T=\nabla\cdot ag_re_k & \text{ on $(-\infty,T)\times \mathbb{R}^d$}, \\
        v_k^T(T)= 0,& 
    \end{array}
\right.
\label{strategydualpb}
\end{equation}
we rewrite 
\begin{equation}
\int_{\mathbb{R}^d} g_r(z)e_k\cdot a(z)\left(\int_{0}^T\nabla\frac{\partial}{\partial a(y)}u(t,z)\dd t\right)\dd z=\nabla v^T_k(0,y)\otimes e+\int_{0}^T\nabla u(t,y)\otimes \nabla v^T_k(t,y)\dd t.
\label{dualecriture2}
\end{equation}
Consequently, recalling \eqref{ReviewEquation1}, the crucial terms of $\partial^{\text{fct}}_{x,\ell}q_r(T)\cdot e_k$ are
\begin{equation}
\mathcal{M}_k(T,x,\ell):=\int_{\bb_{\ell}(x)}\left\vert\nabla v^T_k(0,y)\otimes e\right\vert\dd y+\int_{\bb_{\ell}(x)}\left\vert\int_{0}^T\nabla u(t,y)\otimes \nabla v^T_k(t,y)\dd t\right\vert\dd y.
\label{maintermsensi}
\end{equation}
The rigorous computations are given in Lemma \ref{functioderiv}. 
\item[II)]In the second step, we deduce the control of moments of $q_r(T)-\left\langle q_r(T)\right\rangle$ from the formula of $\partial^{\text{fct}}_{x,\ell} q_r(T)$, for $T\geq 1$ and $1\leq r\leq \sqrt{T}$ combined with the estimate \eqref{SGpstrategy}. More precisely, at this stage, we are only able to reach a sub-optimal bound, with a $\log(T)$ correction in \eqref{Sensitilem3}. This additional contribution is due to the, purely deterministic, sub-optimal bound \eqref{strategyunifboundnablau}, as this will clearly appear in the computations below. In the following, we provide the idea of the control of the main term \eqref{maintermsensi}, for all $\ell\geq 1$ and $k=1$ and we write $\mathcal{M}$ and $v^T$ for $\mathcal{M}_1$ and $v^T_1$, respectively. As in Remark \ref{laplaciancase}, we have to distinguish between the two regimes $\ell< \sqrt{T}$ and $\ell\geq \sqrt{T}$.\newline
\newline
\textbf{Regime $\ell<\sqrt{T}$.} In this regime, we make use of $\LL^2$-type estimates. We start with the first r.h.s term of \eqref{maintermsensi}. We use the plain energy estimate: 
\begin{equation}
\int_{\mathbb{R}^d}\vert\nabla v^T(t,y)\vert^2\dd y\lesssim_{d,\lambda}\int_{\mathbb{R}^d}g^2_r(y)\dd y\lesssim r^{-d}\quad \text{for all $t\leq T$},
\label{stratproofesti1}
\end{equation}
applied for $t=0$ to get, using in addition Jensen's inequality and $\int_{\mathbb{R}^d}\int_{\bb_{\ell}(x)}\dd x\lesssim_d \ell^d\int_{\mathbb{R}^d}$,
\begin{equation}
\int_{\mathbb{R}^d}\left(\int_{\bb_{\ell}(x)}\vert\nabla v^T(0,y)\otimes e\vert\dd y\right)^2\dd x\lesssim_{d,\lambda} \ell^{2d}r^{-d}.
\label{stratproofesti2}
\end{equation}
We now turn to the estimate of the second r.h.s term of \eqref{maintermsensi}. We start by splitting the time integral into the contributions in $(0,1)$ and $[1,T]$. In $(0,1)$, we make use of Cauchy-Schwarz's inequality, the energy estimate \eqref{strat_energyestiavoidsing},
the estimate $\int_{\mathbb{R}^d}\int_{\bb_{\ell}(x)}\dd x\lesssim_d \ell^{d}\int_{\mathbb{R}^d}$ and the plain energy estimate \eqref{stratproofesti1} to obtain 
\begin{align}
\int_{\mathbb{R}^d}\left(\int_{\bb_{\ell}(x)}\left\vert\int_{0}^1\nabla u(t,y)\otimes \nabla v^T(t,y)\dd t\right\vert\dd y\right)^2\dd x&\leq \int_{\mathbb{R}^d}\int_{0}^1\int_{\bb_{\ell}(x)}\vert\nabla u(t,y)\vert^2\dd y\, \dd t\int_{0}^1\int_{\bb_{\ell}(x)}\vert\nabla v^T(t,y)\vert^2\dd y\, \dd t\, \dd x\nonumber\\
&\stackrel{\eqref{strat_energyestiavoidsing}}{\lesssim_{d,\lambda}}\ell^d\int_{\mathbb{R}^d} \int_{0}^1\int_{\bb_{\ell}(x)}\vert\nabla v^T(t,y)\vert^2\dd y\, \dd t\, \dd x\nonumber\\
&\stackrel{\eqref{stratproofesti1}}{\lesssim_{d,\lambda}}\ell^{2d}r^{-d}.\label{heuristic1}
\end{align}
In $[1,T]$, we make use of Jensen's inequality, the estimate $\int_{\mathbb{R}^d}\int_{\bb_{\ell}(x)}\dd x\lesssim_d \ell^{d}\int_{\mathbb{R}^d}$, the deterministic bound \eqref{strategyunifboundnablau} and Minkowski's inequality in $\LL^{2}(\mathbb{R}^d)$ as well as the plain energy estimate \eqref{stratproofesti1} to obtain, for all $\ell\geq 1$,
\begin{align}
\int_{\mathbb{R}^d}\left(\int_{\bb_{\ell}(x)}\left\vert\int_{1}^T\nabla u(t,y)\otimes \nabla v^T(t,y)\dd t\right\vert\dd y\right)^2\dd x&\stackrel{\eqref{strategyunifboundnablau}}{\lesssim_{d,\lambda}}\ell^{2d}\int_{\mathbb{R}^d}\left(\int_{1}^{T} t^{-1}\vert \nabla v^T(t,x)\vert\dd t\right)^2\dd x\nonumber\\
&\leq \ell^{2d}\left(\int_{1}^T t^{-1}\left(\int_{\mathbb{R}^d}\vert\nabla v^T(t,x)\vert^2\dd x\right)^{\frac{1}{2}}\dd t\right)^2\nonumber\\
&\lesssim_{d,\lambda} \ell^{2d}\log^2(T)r^{-d}.\label{heuristic2}
\end{align}
Therefore, the combination of \eqref{stratproofesti2}, \eqref{heuristic1} and \eqref{heuristic2} yields 
\begin{equation}
\int_{\mathbb{R}^d}\mathcal{M}^2(T,x,\ell)\dd x\lesssim_{d,\lambda} \ell^{2d}r^{-d}(1+\log^2(T)),\label{stratllesst}
\end{equation}
and gives the contribution in \eqref{SGpstrategy}
\begin{equation}
\int_{1}^{\sqrt{T}}\ell^{-d}\pi(\ell)\int_{\mathbb{R}^d}\mathcal{M}^2(T,x,\ell)\dd x\, \dd \ell\stackrel{\eqref{assumeMSPC},\eqref{stratllesst}}{\lesssim_{d,\lambda}}(1+\log^2(T))r^{-d}\int_{1}^{\sqrt{T}}\ell^{d-1-\beta}\dd\ell\lesssim_{\beta}(1+\log^2(T))r^{-d}\mu^2_{\beta}(T),\label{RefHeuristicLater:Eq3}
\end{equation}
where $\mu_{\beta}(T)$ is defined in \eqref{defmubeta}.\newline
\newline
\textbf{Regime $\ell\geq \sqrt{T}$.} Here, the bound \eqref{stratllesst} is of no use since $\ell\mapsto \ell^{-d}\pi(\ell)\int_{\mathbb{R}^d}\mathcal{M}^2(T,x,\ell)\dd x$ needs to be integrable at infinity. This is why we treat this regime a different way and we rather use $\LL^1$-type estimates. We start with the first r.h.s term of \eqref{maintermsensi}. This term is more subtle to control in this regime, even with the two additional assumptions \eqref{strategyunifboundnablau} and \eqref{strat_boundLinftyavoidsing}. We present here the argument in the homogeneous case $a=\text{Id}$. In that case, we may express $\nabla v^T$ in terms of the Duhamel formula: 
\begin{equation}\label{DuhamelHeuristic1}
\nabla v^T(t,x)=-\int_{0}^{T-t}\nabla^2\Gamma(s,\cdot)\star g_r(x)\dd s\quad \text{for all $(t,x)\in (-\infty,T)\times \mathbb{R}^d$},
\end{equation}
where we recall that $\Gamma$ denotes the heat kernel. Then, using Minkowski's inequality in $\LL^{2}(\mathbb{R}^d)$, we obtain
$$\int_{\mathbb{R}^d}\left(\int_{\bb_{\ell}(x)}\vert\nabla v^T(0,y)\otimes e\vert\dd y\right)^2\dd x\lesssim_d \ell^d\left(\int_{\mathbb{R}^d}\vert \nabla v^T(0,y)\vert\dd y\right)^2.$$
Finally, noticing that $\Gamma(s,\cdot)=(4\pi)^{-\frac{d}{2}} g_{\sqrt{4s}}$, using the semigroup property $g_{\sqrt{4s}}\star g_r=g_{\sqrt{4s+r^2}}$ and $\|g_{\sqrt{4s+r^2}}\|_{\LL^1(\mathbb{R}^d)}\lesssim_d (4s+r^2)^{-1}$, we get from \eqref{DuhamelHeuristic1}
\begin{align*}
\int_{\mathbb{R}^d}\vert\nabla v^T(0,y)\vert\dd y=(4\pi)^{-\frac{d}{2}}\int_{\mathbb{R}^d}\left\vert\int_{0}^T \nabla^2 g_{\sqrt{4s+r^2}}(y)\dd s\right\vert\dd y\lesssim_d\int_{0}^{T} (4s+r^2)^{-1}\dd s\lesssim_d \log(1+\tfrac{T}{r^2}),
\end{align*}
and we conclude that
\begin{equation}\label{UseRefHeuristicLater:Eq2}
\int_{\mathbb{R}^d}\left(\int_{\bb_{\ell}(x)}\vert\nabla v^T(0,y)\otimes e\vert\dd y\right)^2\dd x\lesssim_d \ell^d\log^2(1+\tfrac{T}{r^2}).
\end{equation}
In the heterogeneous case, we replace the use of the heat kernel by appealing to large-scale regularity, in form of estimate \eqref{stratmeanvalue}, to get a pointwise bound of local averages of $\nabla v^T$, see Lemma \ref{ctrlav2}, and we get, 
\begin{equation}
\int_{\mathbb{R}^d}\left(\int_{\bb_{\ell}(x)}\vert\nabla v^T(t,y)\vert\,\dd y\right)^2\dd x \leq \mathcal{C}\ell^d\log^2(1+\tfrac{\ell}{r})\quad \text{for all $t\in (-\infty,T]$},
\label{stratproofesti3}
\end{equation}
with $\mathcal{C}$ satisfying stretched exponential moments. We refer to the estimate of the first l.h.s term of \eqref{substep22} for more details. We now turn to the estimate of the second r.h.s term of \eqref{maintermsensi}. As before, we split the time integral into the two contributions in $(0,1)$ and $[1,T]$. In $(0,1)$, we make use of Minkowski's inequality in $\LL^2(\mathbb{R}^d)$ combined with the identity $\int_{\mathbb{R}^d}=\int_{\mathbb{R}^d}\fint_{\bb_{r}(z)}\dd z$, Cauchy-Schwarz's inequality, \eqref{strat_energyestiavoidsing} applied with $R=r$ and $T=1$ and \cite[Lemma 2]{gloria2015corrector} applied to the equation \eqref{strategydualpb} in form  of
\begin{equation}
\int_0^1\fint_{\bb_r(z)}\vert \nabla v^T(t,y)\vert^2\dd y\, \dd t\lesssim \int_{\mathbb{R}^d}  \eta_{r}(y-z)g^2_r(y)\dd y,
\label{strat_energyvT}
\end{equation}
to get
\begin{align}
&\int_{\mathbb{R}^d}\left(\int_{\bb_{\ell}(x)}\left\vert\int_{0}^1\nabla u(t,y)\otimes \nabla v^T(t,y)\dd t\right\vert\dd y\right)^2\dd x\nonumber\\
&\lesssim_d\ell^{d}\left(\int_{\mathbb{R}^d}\int_{0}^1\vert\nabla u(t,y)\vert\vert\nabla v^T(t,y)\vert\dd t\, \dd y\right)^2\nonumber\\
&\leq\ell^d\left(\int_{\mathbb{R}^d}\left(\int_{0}^1\fint_{\bb_r(z)}\vert\nabla u(t,y)\vert^2\dd y\, \dd t\right)^{\frac{1}{2}}\left(\int_{0}^1\fint_{\bb_r(z)}\vert\nabla v^T(t,y)\vert^2\dd y\, \dd t\right)^{\frac{1}{2}}\dd z\right)^2\nonumber\\
&\stackrel{\eqref{strat_energyestiavoidsing},\eqref{strat_energyvT}}{\lesssim}\ell^d\int_{\mathbb{R}^d}\left(\int_{\mathbb{R}^d}\eta_r(y-z)g^2_r(y)\dd y\right)^{\frac{1}{2}}\dd z\lesssim \ell^d.\label{stratnera0lgtr}
\end{align}
In $[1,T]$, we exchange the order of integration in the $x$ and $t$ variables, using Minkowski's inequality in $\LL^{2}(\mathbb{R}^d)$, which we combine with \eqref{strategyunifboundnablau} and \eqref{stratproofesti3} to get for all $\ell\geq 1$
\begin{align}
\int_{\mathbb{R}^d}\left(\int_{\bb_{\ell}(x)}\left\vert \int_{1}^T \nabla u(t,y)\otimes \nabla v^T(t,y)\dd t\right\vert\dd y\right)^2\dd x &\lesssim \left(\int_{1}^T\,t^{-1}\left(\int_{\mathbb{R}^d}\left(\int_{\bb_{\ell}(x)}\vert\nabla v^T(t,y)\vert\dd y\right)^2\dd x\right)^{\frac{1}{2}}\dd t\right)^2\nonumber\\
&\stackrel{\eqref{stratproofesti3}}{\lesssim}\mathcal{C}\ell^d\log^2(T)\log^2(1+\tfrac{\ell}{r}).\label{strataway0lgtr}
\end{align}
The combination of \eqref{stratproofesti3} applied with $t=0$, \eqref{stratnera0lgtr} and \eqref{strataway0lgtr} yields
\begin{equation}\label{UseRefHeuristicLater}
\int_{\mathbb{R}^d}\mathcal{M}^2(T,x,\ell)\dd x\lesssim_{d,\lambda}\ell^d(1+\mathcal{C}\log^2(T)\log^2(1+\tfrac{\ell}{r})),
\end{equation}
and gives the contribution in \eqref{SGpstrategy}, using that $r\leq \sqrt{T}$ in the last line
\begin{align}
\int_{\sqrt{T}}^{+\infty}\ell^{-d}\pi(\ell)\int_{\mathbb{R}^d}\mathcal{M}^2(T,x,\ell)\dd x\, \dd \ell&\lesssim\log^2(T)\int_{\sqrt{T}}^{+\infty}\ell^{-1-\beta}(1+\mathcal{C}\log^2(\tfrac{\ell}{r}))\dd \ell\nonumber\\
&\lesssim_{\beta}\log^2(T)T^{-\frac{\beta}{2}}(1+\mathcal{C}\log^2(\tfrac{\sqrt{T}}{r}))\nonumber\\
&\lesssim_{\beta}\log^2(T)r^{-d}\mu^2_{\beta}(T)(1+\mathcal{C}\log^2(\tfrac{\sqrt{T}}{r})).\label{RefHeuristicLater:Eq4}
\end{align}
Let us now talk about the main difficulties and changes which occur in the general case, that is when we do not assume \eqref{strategyunifboundnablau} and \eqref{strat_boundLinftyavoidsing}.
\begin{itemize}
\item[(i)] When the assumption \eqref{strat_boundLinftyavoidsing} is not satisfied, $u$ is now singular at $t=0$ and thus the second r.h.s term of \eqref{maintermsensi} is not well defined in the Lebesgue sense. In order to handle this singular part, we have to treat a different way the contribution in $(0,1)$ of the time integral of the third r.h.s term of \eqref{partialformalderivative}. This is done by using the localized energy estimates directly on the equation \eqref{equationformalpartialu}. As a consequence, we do not obtain an explicit formula for $\partial^{\text{fct}}_{x,\ell} q_r(T)$ but rather a bound, see Lemmas \ref{functioderiv} and \ref{functioderiv2}.
\item[(ii)] When $u$ is vector-valued, \eqref{strategyunifboundnablau} fails in general and has to be replaced by: 
\begin{equation}
\fint_{\bb_R}\vert\nabla u(t,y)\vert^2\dd y\lesssim t^{-2}\quad \text{for all $R\geq \sqrt{t}$}.
\label{stratboudnablaulargescale}
\end{equation}
This estimate is however not sufficient for our propose since we see in \eqref{maintermsensi} that we need to bound the average of $\nabla u$ over all balls $\bb_{\ell}$ of radius $\ell\geq 1$. We have to appeal to large-scale regularity theory in form of estimate \eqref{stratmeanvalue} to obtain the improvement
\begin{equation}
\fint_{\bb_{\ell}(x)}\vert\nabla u(t,y)\vert^2\dd y\lesssim \left(\left(\frac{r_*(x)}{\ell}\vee 1\right)^d\mathds{1}_{\ell<\sqrt{t}}+\mathds{1}_{\ell\geq \sqrt{t}}\right) t^{-2},
\label{stratimprovedecay}
\end{equation}
see Lemma \ref{ctrlav}. Equipped with \eqref{stratimprovedecay}, we may control the second r.h.s term of \eqref{maintermsensi} as in the scalar case. The only main change is that we cannot use the plain energy estimate \eqref{stratproofesti1} for the defining equation \eqref{strategydualpb} as we did in \eqref{heuristic2}. Instead, we prove a new lemma which states a pointwise bound (depending on the form of the r.h.s of \eqref{strategydualpb}) on $\fint_{\bb_{r_*(x)}}\vert\nabla v^T(t,y)\vert^2\dd y$ for all $x\in\mathbb{R}^d$ and $\sqrt{T-t}\geq r_*(x)$, see Lemma \ref{ctrlav2}. The sub-optimal estimate of moments of $q_r(T)-\left\langle q_r(T)\right\rangle$ is summarized in Proposition \ref{Sensitivitysubop}.
\end{itemize}
\item[III)] In the final step, we remove the $\log(T)$ contribution which appears in the previous step. To this aim, we need a little more decay in time of the averages of $\nabla u$ than the one obtained in \eqref{stratimprovedecay}, since the $\log(T)$ contribution clearly comes from this deterministic sub-optimal bound. The idea is to use the $\LL^2$-$\LL^1$ estimate of Lemma \ref{cacciopo} which essentially says, by stationarity, that for all $R\geq \sqrt{t}$, $x\in\mathbb{R}^d$ and $p\geq 1$
$$\left\langle \left(\fint_{\bb_R(x)}\vert\nabla u(t,y)\vert^2\dd y\right)^{\frac{p}{2}}\right\rangle^{\frac{1}{p}}\lesssim_{d,\lambda}t^{-1-\frac{d}{4}}\fint_{\frac{t}{4}}^{\frac{t}{2}}\fint_{0}^{\sqrt{s}}r^{\frac{d}{2}}\left\langle \vert q_r(s)-\left\langle q_r(s)\right\rangle\vert^p\right\rangle^{\frac{1}{p}}\dd r\, \dd s.$$
This gives, using the sub-optimal moment bounds of $q_r(T)-\left\langle q_r(T)\right\rangle$ of the previous step that, 
\begin{equation}
\fint_{\bb_R(x)}\vert\nabla u(t,y)\vert^2\dd y\lesssim \log^2(t)t^{-1}\eta^2_{\beta}(t)\mathcal{D}_\star(x)\quad \text{for all $R\geq \sqrt{t}$},
\label{interpol}
\end{equation}
with $\eta_{\beta}$ defined in \eqref{defetabeta} and where $\mathcal{D}_{\star}(x)$ is a random variable with stretched exponential moment. By interpolating \eqref{stratimprovedecay} and \eqref{interpol} we deduce for all $\varepsilon>0$ and $\ell\geq 1$
\begin{equation}
\fint_{\bb_{\ell}(x)}\vert\nabla u(t,y)\vert^2\dd y\lesssim \left(\left(\frac{r_*(x)}{\ell}\vee 1\right)^d\mathds{1}_{\ell<\sqrt{t}}+\mathds{1}_{\ell\geq \sqrt{t}}\right)\log^{2\varepsilon}(t)t^{-2+\varepsilon}\eta^{2\varepsilon}_{\beta}(t)\mathcal{D}^{2\varepsilon}_{\star}(x).
\label{proofstratesti78}
\end{equation}
This improvement allow us to prove the optimal estimates of Theorem \ref{semigroup}. The price to pay in this step is a small loss of stochastic integrability due to the random variable $\mathcal{D}^{2\varepsilon}_{\star}$. Note that, the exponent $\alpha$ that we get in \eqref{momentboundlem3}, is neither optimal for $\beta>d$ (since \cite{gloria2015corrector} indicates that we expect nearly-Gaussian moments), nor for $0<\beta\ll 1$ (since by \cite{fischer2017sublinear} we can obtain nearly Gaussian moments).
\item[IV)]The proof of \eqref{sensiothertest} follows the same ideas and is even easier since we do not need to begin with a sub-optimal estimate as we did for \eqref{Sensitilem3}. We use Lemma \ref{functioderiv2} for the estimate of $\partial^{\text{fct}}_{x,\ell} q(r^2)\star f_r$, the pointwise bound \eqref{dualboundd2} for local average of $\nabla v^{r^2}$ and the decay \eqref{proofstratesti78} of averages of $\nabla u$.
\end{itemize}
We finally mention that in the case of fast decay of correlations, that is for $\beta>d$, the proof is much simpler and only the regime $\ell<\sqrt{T}$ has to be considered.
We now state the lemmas needed in the proof of Theorem \ref{semigroup}. The first section lists the deterministic PDE ingredients, the second section the results derived from the large-scale regularity theory, and finally the third section the sub-optimal control of the fluctuations of the time dependent flux $q_r(\cdot,\cdot)$.
\subsection{Deterministic results}\label{lemsec3}
This section displays the deterministic PDE ingredients needed in the proof of Theorem \ref{semigroup} and Corollary \ref{decayu}. We start with two classical results from standard $\LL^2$ regularity theory of parabolic systems. The first one is the localized energy type estimates for parabolic systems.
\begin{lemma}[Localized energy estimates]\label{nergyestideter}
Let $v$ be the weak solution of the parabolic system
$$
\left\{
    \begin{array}{ll}
        \partial_{\tau}v-\nabla\cdot a\nabla v=\nabla\cdot f &\text{ in $(0,+\infty)\times \mathbb{R}^d$}, \\
        v(0)=\nabla\cdot q, & 
    \end{array}
\right.
$$
with $q\in \LL^2_{\text{loc}}(\mathbb{R}^d)$ and $f:\mathbb{R}^+\times\mathbb{R}^d\rightarrow \mathbb{R}^d$ such that for all $(x,\ell)\in\mathbb{R}^d\times [1,\infty)$
$$T\in \mathbb{R}^+\mapsto \int_{0}^T f(s,\cdot)\dd s \text{ is in $\LL^2_{\text{loc}}(\mathbb{R}^d)$} \quad \text{and}\quad\left(\int_{\bb_{\ell}(x)}\vert f(\cdot,y)\vert^2\dd y\right)^{\frac{1}{2}} \text{ is in $\LL^1_{\text{loc}}(\mathbb{R}^+)$}.$$ 
There exists a universal constant $c<\infty$ such that the three following estimates hold, recalling that $\eta_R:=R^{-d}e^{-\frac{\vert\cdot\vert}{R}}$;
\begin{itemize}
\item[(i)] Assume that $f\equiv 0$. We have for all $T>0$, $R\geq \sqrt{T}$ and $x\in\mathbb{R}^d$
\begin{equation}
\left(\int_{\mathbb{R}^d}\eta_{R}(\tfrac{y-x}{c})\left\vert\Big(T\nabla v(T,y),\,\int_{0}^T\nabla v(s,y)\dd s\Big)\right\vert^2\dd y\right)^{\frac{1}{2}}\lesssim_{d,\lambda} \left(\int_{\mathbb{R}^d}\eta_{R}(\tfrac{y-x}{c})\vert q(y)\vert^2\dd y\right)^{\frac{1}{2}}.
\label{LemE1}
\end{equation}
\item[(ii)] Assume that $q\equiv 0$. We have for all $R\geq 1$ and $x\in\mathbb{R}^d$
\begin{equation}\label{LemE2}
\begin{aligned}
\left(\int_{\mathbb{R}^d}\eta_{R}(\tfrac{y-x}{c})\left\vert\int_{0}^1 \nabla v(s,y)\dd s\right\vert^2\dd y\right)^{\frac{1}{2}}\lesssim_{d,\lambda} & \left(\int_{\mathbb{R}^d}\eta_{R}(\tfrac{y-x}{c})\left\vert \int_{0}^1 f(s,y)\dd s\right\vert^2\dd y\right)^{\frac{1}{2}}\\
&+\int_{0}^1\frac{1}{1-t}\int_{t}^1\left(\int_{\mathbb{R}^d}\eta_{R}(\tfrac{y-x}{c})\vert f(s,y)\vert^2\dd y\right)^{\frac{1}{2}}\dd s\,\dd t.
\end{aligned}
\end{equation}
\item[(iii)]Assume that $q$ and $f$ are supported in $B_{\ell}(x)$ for some $x\in\mathbb{R}^d$ and $\ell\in [1,\infty)$. Then we have
\begin{equation}\label{LemE3}
\begin{aligned}
&\left(\int_{\mathbb{R}^d\backslash B_{\ell}(x)}e^{\frac{\vert x-z\vert}{2c\ell}}\fint_{B_{\ell}(z)}\left\vert\int_{0}^1\nabla v(s,y)\dd s\right\vert^2\dd y\,\dd z\right)^{\frac{1}{2}}\\
&\lesssim_{d,\lambda} \left(\int_{B_{\ell}(x)}\vert q(y)\vert^2\dd y\right)^{\frac{1}{2}}+\left(\int_{B_{\ell}(x)}\left\vert \int_{0}^1 f(s,y)\dd s\right\vert^2\dd y\right)^{\frac{1}{2}}+\int_{0}^1\frac{1}{1-t}\int_{t}^1\left(\int_{B_{\ell}(x)}\vert f(s,y)\vert^2\dd y\right)^{\frac{1}{2}}\dd s\,\dd t.
\end{aligned}
\end{equation}
\end{itemize}
\end{lemma}
We then state a technical lemma needed in order to obtain pointwise estimates in time. For a proof, we refer to \cite[Lemma 8.2]{armstrong2019quantitative}.
\begin{lemma}\label{scott} Fix $r>0$, $(s,x)\in\mathbb{R}^{d+1}$ and $g\in L^{2}(\mathbb{R}^d)^d$. Assume that $v$ is a weak solution of 
$$\partial_{\tau}v-\nabla\cdot a\nabla v=\nabla\cdot g \quad \text{in $(s-4r^2,s)\times \bb_{r}(x)$,}$$
then we have
\begin{equation*}
\sup_{t\in (s-r^2,s)}\fint_{\bb_r(x)}\vert \nabla v(t,y)\vert^2\dd t\, \dd y \lesssim_{d,\lambda} \fint_{s-4r^2}^{s}\fint_{\bb_{2r}(x)}\vert \nabla v(s',y)\vert^2\dd s'\, \dd y+\fint_{\bb_{2r}(x)}\vert g(y)\vert^2\dd y.
\end{equation*}
The same holds for the operator $\partial_\tau+\nabla\cdot a\nabla$ on $(s,s+4r^2)\times B_{2r}(x)$ using the time reflexion $t\mapsto -t$.
\end{lemma}
We conclude this section by the relationship between spatial averages of $u(T,\cdot)$ and averages of $q_r-\left\langle q_r\right\rangle$ over scales $r\leq \sqrt{T}$. This lemma allow us to deduce Corollary \ref{decayu} from Theorem \ref{semigroup}. We refer the reader to \cite[Lemma 6]{gloria2015corrector} for the original proof of this result.
\begin{lemma}[$\LL^2$-$\LL^1$ estimate]\label{cacciopo}Let $u$ defined in \eqref{equationu}. There exists a universal constant $c<\infty$ such that for all $T>0$ and $R\geq \sqrt{T}$
\begin{align*}
\left(\int_{\mathbb{R}^d}\eta_{R}(\tfrac{y}{c})\vert (u(T,y),\sqrt{T}\nabla u(T,y))\vert^2\dd y\right)^{\frac{1}{2}}\lesssim_{d,\lambda} \frac{1}{\sqrt{T}}\fint_{\frac{T}{4}}^{\frac{T}{2}}\fint_{0}^{\sqrt{t}}\left(\frac{r}{\sqrt{t}}\right)^{\frac{d}{2}}\int_{\mathbb{R}^d}\eta_{\sqrt{2}R}(\tfrac{y}{c})\vert q_r(t,y)-\left\langle q_r(t,y)\right\rangle\vert\dd y\, \dd r\,\dd t,
\end{align*}
where $\eta_{R}:=R^{-d}e^{-\frac{\vert\cdot\vert}{R}}$ and $q(\cdot,\cdot)$ is defined in \eqref{defq}.
\end{lemma}
\subsection{Large-scale regularity results}\label{lemsec1}
We state in this section two estimates, needed in the proof of Theorem \ref{semigroup}, which are obtained from the large-scale regularity theory recalled in the appendix \ref{reggech}. We start with a lemma which gives a pointwise bound on a local average of the solution of the dual problem \eqref{strategydualpb}, depending on the behavior of the r.h.s. This constitutes the parabolic version of Lemmas $2$, $3$ and $4$ of \cite{gloria2019quantitative} established for elliptic systems.
\begin{lemma}[Pointwise estimates on the dual problem]\label{ctrlav2}Let $f_1\in C_b^1(\mathbb{R}^d)$, $\overline{e}$ be a unit vector of $\mathbb{R}^d$ and $v_r$ satisfy, in the weak sense, for some $r\geq 1$, the parabolic backward system\footnote{for the well-posedness, we refer to \cite{friedman2008partial}}
\begin{equation}
\left\{
    \begin{array}{ll}
        \partial_{\tau}v_r+\nabla\cdot a\nabla v_r= \nabla\cdot a f_r\overline{e}&\text{in $(-\infty,0)\times \mathbb{R}^d$}, \\
        v_r(0)=0 ,  & 
    \end{array}
\right.
\label{equationvlem2}
\end{equation}
with $f_r$ which satisfies one of the two following assumptions:
\begin{itemize}
\item 
$f_r:=r^{-d}f_1(\frac{\cdot}{r})$ such that for all $x\in\mathbb{R}^{d}$
\begin{equation}
\vert f_1(x)\vert\lesssim \frac{1}{(\vert x\vert+1)^d} \quad\text{and}\quad \vert \nabla f_1(x)\vert\lesssim \frac{1}{(\vert x\vert+1)^{d+1}}.
\label{assumerhsdual}
\end{equation}
\item For all $x\in\mathbb{R}^d$
\begin{equation}
\vert f_r(x)\vert\lesssim \frac{r}{(\vert x\vert+1)^{d}}\wedge \frac{1}{(\vert x\vert+1)^{d-1}} \quad\text{and}\quad\vert\nabla f_r(x)\vert\lesssim  \frac{r}{(\vert x\vert+1)^{d+1}}\wedge \frac{1}{(\vert x\vert+1)^{d}}.
\label{assumerhsdual2}
\end{equation}
\end{itemize}
We have, for all $x\in\mathbb{R}^d$
\begin{enumerate}
\item If \eqref{assumerhsdual} holds, then
\begin{equation}
\left(\fint_{\bb_{r_*(x)}(x)}\vert \nabla v_r(t,y)\vert^2\dd y\right)^{\frac{1}{2}}\lesssim \left(\frac{r_*(0)}{r}\vee 1\right)^{\frac{d}{2}}
\frac{\log(1+\frac{\vert x\vert}{r})}{(\vert x\vert+r)^{d}}\quad\text{for all $\sqrt{-t}\geq 2r_*(x)$}.
\label{lem2}
\end{equation}
\item If \eqref{assumerhsdual2} holds, then 
\begin{equation}
\left(\fint_{\bb_{r_*(x)}(x)}\vert \nabla v_r(t,y)\vert^2\dd y\right)^{\frac{1}{2}}\lesssim r^{\frac{d}{2}}_*(0)\left(\frac{r\log(1+\vert x\vert)}{(\vert x\vert+1)^d}\wedge \frac{1}{(\vert x\vert+1)^{d-1}}\right)\quad\text{for all $\sqrt{-t}\geq 2r_*(x)$}.
\label{dualboundd2}
\end{equation}
\end{enumerate}
\end{lemma}
Let us briefly comment on Lemma \ref{ctrlav2}.
\begin{enumerate}
\item The bound \eqref{lem2} is needed to replace the plain energy estimate for the solution $v^T$ of \eqref{strategydualpb}, in form of \eqref{stratproofesti1}, that we used in the heuristic argument to obtain \eqref{stratllesst}. In the homogeneous case, i.e $a=\text{Id}$, and in the case where $f_1=g_1$, the bound \eqref{lem2} takes the more natural form: 
\begin{equation}
\vert\nabla v_r(t,x)\vert\lesssim (\vert x\vert+r)^{-d}\quad \text{for all $(t,x)\in\mathbb{R}_-\times\mathbb{R}^{d}$}.
\label{morenaturalctrlav2}
\end{equation}
Indeed, \eqref{morenaturalctrlav2} is easy to see using the explicit formula involving the heat kernel $\Gamma$:
$$\nabla v_r(t,x)=-\int_{t}^{0}\nabla^2\Gamma(s-t,\cdot)\star g_r(x)\dd s\quad \text{for all $(t,x)\in \mathbb{R}_-\times\mathbb{R}^{d}$}.$$
Thus, using that for all $s\geq t$, $\Gamma(s-t,\cdot)=(4\pi)^{-\frac{d}{2}}g_{\sqrt{4(s-t)}}$, the semigroup property $g_{\sqrt{4(s-t)}}\star g_r=g_{\sqrt{r^2+4(s-t)}}$ and the estimate, for all $x\in\mathbb{R}^d$, $e^{-\frac{\vert x\vert^2}{r^2+s-t}}\lesssim (1+\frac{\vert x\vert^2}{r^2+s-t})^{-\frac{d}{2}-2}$, we have for all $(t,x)\in \mathbb{R}_-\times\mathbb{R}^d$
\begin{align*}
\vert \nabla v_r(t,x)\vert\lesssim  \int_{t}^{0}\vert\nabla^2 g_{\sqrt{r^2+4(s-t)}}(x)\vert\dd s &\lesssim \vert x\vert^2\int_{t}^{0}(r^2+s-t)^{-\frac{d}{2}-2}e^{-\frac{\vert x\vert^2}{r^2+s-t}}\dd s\\
&\lesssim \vert x\vert^2\int_{t}^{0}(\vert x \vert^2+s-t+r^2)^{-\frac{d}{2}-2}\dd s\\
&\lesssim (\vert x\vert+r)^{-d}.
\end{align*}
The same way, if we have the more precise structure $f_r=\int_{1}^{r^2}\nabla g_{\sqrt{s}}(\cdot)\dd s$ (which satisfies the assumption \eqref{assumerhsdual2}), the bound \eqref{dualboundd2} takes the more natural form: 
$$\vert\nabla v_r(t,x)\vert\lesssim \frac{r}{(\vert x\vert+1)^d}\wedge \frac{1}{(\vert x\vert+1)^{d-1}}\quad \text{for all $(t,x)\in\mathbb{R}_-\times\mathbb{R}^d$}.$$
Therefore, since the bounds \eqref{lem2} and \eqref{dualboundd2} are natural in the homogeneous case and we know from homogenization theory that on large-scales the heterogeneous parabolic operator $\partial_{\tau}-\nabla\cdot a\nabla$ inherits (in form on the $\cc^{0,1}$ estimate \eqref{meanvalue2}) the regularity theory of the homogenized operator $\partial_{\tau}-\nabla\cdot a_{\text{hom}}\nabla$, it is natural to expect that the two estimates \eqref{lem2} and \eqref{dualboundd2} hold in the heterogeneous case once we fix the scale (characterized by the minimal radius $r_*$). Note that the logarithmic contributions in \eqref{lem2} and \eqref{dualboundd2} are due to the fact that we have less structure on the r.h.s of \eqref{equationvlem2} than the two we took above. We also point out that the logarithmic contribution in \eqref{dualboundd2} may be removed (see for instance Lemmas 3 and 4 of \cite{gloria2019quantitative} for elliptic systems). However, since it is enough for its application in this article, we prefer to keep it this way and provide simple arguments for \eqref{dualboundd2} rather than going trough additional technical difficulties.
\item We may deduce the results of Lemmas 3 and 4 of \cite{gloria2019quantitative}  from \eqref{lem2} by sending $t\downarrow -\infty$. Indeed, one may prove, from the localized energy of Lemma \ref{nergyestideter} that $v_r(t,\cdot)\underset{t\downarrow -\infty}{\rightarrow} \tilde{v}_r$ in $\LL^2_{\text{loc}}(\mathbb{R}^d)$ with $\nabla\cdot a\nabla\tilde{v}_r=\nabla\cdot a f_r\overline{e}$ and then pass to the limit in \eqref{lem2}.
\end{enumerate}
The next lemma allows us to control spatial averages of $\nabla u$ at scale $R<\sqrt{T}$, and is a consequence of Corollary \ref{schauder}. Combined with the energy estimate \eqref{LemE1}, it implies in particular the estimate \eqref{stratimprovedecay} needed in the proof of Theorem \ref{semigroup}. 
\begin{lemma}[Control of averages]\label{ctrlav} Let $u$ be defined in \eqref{equationu}. Assume that there exists an increasing function $f$ and a decreasing function $g$ such that for all $T\geq 1$ and for all $x\in\mathbb{R}^d$ there exists a constant $C(x,T)<+\infty$ for which
\begin{equation}
\fint_{B_{\sqrt{T}}(x)}\vert \nabla u(T,y)\vert^2\dd y\leq C(x,T)f(T)g(T).
\label{assumelem5}
\end{equation}
Then we have for all $T\geq 1$, $x\in\mathbb{R}^d$ and $R<\sqrt{T}$
$$\fint_{B_R(x)}\vert \nabla u(T,y)\vert^2\dd y\lesssim \left(\frac{r_*(x)}{R}\vee 1\right)^d \tilde{C}(x,T)f(T)g(\tfrac{T}{2}),$$
with $\tilde{C}(x,T):=\max\left\{C(x,T),\displaystyle\fint_{\frac{T}{2}}^{T} C(x,s)\dd s\right\}$.
\end{lemma}
\subsection{Suboptimal control of the fluctuations of the time dependent flux}\label{lemsec2}
In this section, we state the suboptimal moment bounds of $q_r(T)-\left\langle q_r(T)\right\rangle$ with the auxiliary lemmas needed in the proof. We prove that it displays the central limit theorem scaling $r^{-\frac{d}{2}}$, a growth in $T$ which depends on the parameter $\beta$ defined in \eqref{assumeMSPC} and a $\log(T)$ correction (which makes it suboptimal and will be removed later). We first state the main result of this section.
\begin{proposition}[Sub-optimal fluctuation estimates]\label{Sensitivitysubop}Let $q(\cdot,\cdot)$ defined in \eqref{defq} and fix $\gamma<1$.  For all $T\geq 1$, $1\leq r\leq \sqrt{T}$ and $p\in [1,\infty)$
\begin{equation}
\left\langle \vert q_r(T)-\left\langle q_r(T)\right\rangle\vert^p\right\rangle^{\frac{1}{p}}\lesssim_{d,\lambda,\beta,\gamma} p^{\alpha_\gamma}r^{-\frac{d}{2}}\log(T)\log^2(\tfrac{\sqrt{T}}{r})\mu_{\beta}(T),
\label{subopsensiesti}
\end{equation}
with for any $\gamma>0$
$$
\mu_{\beta}(T) := \left\{
    \begin{array}{ll}
        T^{\frac{d}{4}-\frac{\beta}{4}} & \text{ if $\beta<d$}, \\
        \log^{\frac{1}{2}}(T) &\text{ if $\beta=d$}, \\
				1      &\text{ if $\beta>d$}.
    \end{array}
\right.\quad \text{and}\quad \alpha_\gamma:=\bigg(\frac{1}{2}+\frac{d+1}{\beta\wedge d}\Big(1+\gamma\mathds{1}_{\beta= d}\Big)\bigg).
$$
\end{proposition}
The proof of Proposition \ref{Sensitivitysubop} follows the strategy presented in Section \ref{stratsec} and falls by combining the following lemma (which states an estimate on the functional derivative of averages of the flux $q_r(T)$ for $T\geq 1$ and $r\leq \sqrt{T}$) with the logarithmic Sobolev inequality \eqref{SGinegp}. This makes rigorous the computation done in \eqref{partialformalderivative}.
\begin{lemma}[Functional derivative]\label{functioderiv}Let $q(\cdot,\cdot)$ be defined in \eqref{defq}. There exists a universal constant $c<\infty$ such that for all $T\geq 1$, $r> 0$, $x\in\mathbb{R}^d$ and $\ell\in [1,\infty)$, we have
\begin{equation}
\begin{aligned}
\vert\partial_{x,\ell}^{\text{fct}}q_r(T)\vert\lesssim_{d,\lambda}&\int_{\bb_{\ell}(x)}g_r(y)\dd y+\int_{\bb_{\ell}(x)} \left\vert\int_{0}^T\nabla u(s,y)\dd s\right\vert g_r(y)\,\dd y+\int_{B_{\ell}(x)}\vert \nabla v^T(1,y)\vert\left(1+\left\vert\int_{0}^{1}\nabla u(s,y)\dd s\right\vert\right)\dd y\\
&+\mathcal{F}_{r,\ell}(x)\mathds{1}_{\ell< \sqrt{T}}+\mathcal{G}_{r,\ell}(x)\mathds{1}_{\ell\geq\sqrt{T}}+\int_{\bb_{\ell}(x)} \int_{1}^T\vert\nabla u(t,y)\vert \vert\nabla v^T(t,y)\vert\dd t\dd y,\label{lem6}
\end{aligned}
\end{equation}
where $v^T=(v^T_k)_{k\in\llbracket 1,d\rrbracket}$ is a weak solution of the backward parabolic system
\begin{equation}
\left\{
    \begin{array}{ll}
        \partial_{\tau}v^T_k+\nabla\cdot a^*\nabla v^T_k=\nabla\cdot a g_r e_k & \text{ on $(-\infty,T)\times \mathbb{R}^d$}, \\
        v^T_k(T)= 0, & 
    \end{array}
\right.
\label{dualequationlem7}
\end{equation}
with
\begin{equation}
\mathcal{F}_{r,\ell}(x)=\ell^{\frac{d}{2}}\left(\int_{\mathbb{R}^d}e^{-\frac{\vert x-z\vert}{2c\ell}}\fint_{\bb_{\ell}(z)}\vert (g_r(y),\nabla v^T(1,y))\vert^2\dd y\,\dd z\right)^{\frac{1}{2}},
\label{Frllem6}
\end{equation}
and
\begin{equation}\label{Grllem6}
\begin{aligned}
\mathcal{G}_{r,\ell}(x)=&\mathcal{T}_{x,\ell}(\eta_r)(0)+\mathds{1}_{\ell\geq r_*(0)}\bigg(\int_{\mathbb{R}^d\backslash \bb_{4\ell}}\Big(\fint_{\bb_{\ell}(y)}\vert\nabla v^T(1,z)\vert^2\dd z\Big)^{\frac{1}{2}}\mathcal{T}_{x,\ell}(\eta_{\ell})(y)\dd y\\
&+\int_{\bb_{7\ell}}\Big(\fint_{\bb_{r_*(y)}(y)}\vert\nabla v^T(1,z)\vert^2\dd z\Big)^{\frac{1}{2}}\mathcal{T}_{x,\ell}(\eta_{r_*(y)})(y)\dd y\bigg)+\mathds{1}_{\ell<r_*(0)}r^{\frac{d}{2}}_*(0)\left(\int_{\mathbb{R}^d}e^{-\frac{\vert x-z\vert}{2c\ell}}\fint_{\bb_{\ell}(z)}\vert\nabla v^T(1,y)\vert^2\dd y\,\dd z\right)^{\frac{1}{2}}.
\end{aligned}
\end{equation}
as well as for all $y\in\mathbb{R}^d$ and $\rho>0$
\begin{equation}
\mathcal{T}_{x,\ell}(\eta_{\rho})(y)=\left(\int_{\bb_{\ell}(x)}\eta_{\rho}(z-y)\left(1+\left\vert\int_{0}^1\nabla u(t,z)\dd t\right\vert^2\right)\dd z\right)^{\frac{1}{2}}+\int_{0}^1\frac{1}{1-t}\int_{t}^1\left(\int_{\bb_{\ell}(x)}\eta_{\rho}(z-y)\vert\nabla u(s,z)\vert^2\dd z\right)^{\frac{1}{2}}\dd s\, \dd t.
\label{functionallem6}
\end{equation}
\end{lemma}
The estimate \eqref{lem6} has to be compared with the heuristic computations in \eqref{partialformalderivative} and \eqref{maintermsensi} done in Section \ref{stratsec}. Note that the splitting in \eqref{lem6} between $\ell<\sqrt{T}$ and $\ell\geq \sqrt{T}$ reflects the two different strategies done in Section \ref{stratsec} in this two regimes, using $\LL^2$-type estimates in the first case rather than $\LL^1$-type estimates in the second case.

\medskip

In prevision of the proof of \eqref{sensiothertest}, we estimate in the following lemma the functional derivative of averages of the flux $q(T)\star f_r$ for $T\geq 1$ and $r\leq \sqrt{T}$, where $f_r$ satisfies \eqref{assumesensiothertest}.

\begin{lemma}\label{functioderiv2}
Let $q(\cdot,\cdot)$ be defined in \eqref{defq} and for all $r>0$ we consider $f_r\in C^1_b(\mathbb{R}^d)$ satisfying, for all $y\in\mathbb{R}^d$
\begin{equation}
 f_r(y)=\int_{1}^{r^2}\tilde{f}_r(s,y)\dd s\quad \text{with}\quad\vert \tilde{f}_r(s,y)\vert \lesssim s^{-1}\vert y\vert g_{\sqrt{s}}(y).
\label{assumeothertest}
\end{equation}
There exists a universal constant $c<\infty$ such that for all $r\geq 2$, $x\in\mathbb{R}^d$ and $\ell\in [1,\infty)$, we have
\begin{align}
\vert\partial^{\text{fct}}_{x,\ell}q(r^2)\star f_r\vert\lesssim & \int_{\bb_{\ell}(x)}\vert f_r(y)\vert\dd y+\int_{\bb_{\ell}(x)} \left\vert\int_{0}^{r^2}\nabla u(s,y)\dd s\right\vert \vert f_r(y)\vert\,\dd y+\int_{\bb_{\ell}(x)}\vert\nabla v^{r^2}(1,y)\vert\left(1+\left\vert\int_{0}^1\nabla u(s,y)\dd s\right\vert\right) \dd y\nonumber\\
&+\mathcal{K}_{r,\ell}(x)+\mathcal{G}_{r,\ell}(x)
+\int_{\bb_{\ell}(x)} \int_{1}^{r^2}\vert\nabla u(t,y)\vert \vert\nabla v^{r^2}(t,y)\vert\dd t\dd y,
\label{functioderivothertest}
\end{align}
where $v^{r^2}=(v^{r^2}_k)_{k\in\llbracket 1,d\rrbracket}$ is the weak solution of the backward parabolic system
\begin{equation}
\left\{
    \begin{array}{ll}
        \partial_{\tau}v^{r^2}_k+\nabla\cdot a^*\nabla v^{r^2}_k=\nabla\cdot a f_r e_k & \text{ on $(-\infty,r^2)\times \mathbb{R}^d$}, \\
        v^{r^2}_k(r^2)= 0, & 
    \end{array}
\right.
\label{dualequationlem7othertest}
\end{equation}
with 
\begin{equation}\label{Kothertest}
\begin{aligned}
\mathcal{K}_{r,\ell}(x):=& \mathds{1}_{\ell\geq r}\int_{1}^{r^2} s^{-\frac{1}{2}}\mathcal{T}_{x,\ell}(\eta_{\sqrt{s}})(0)\dd s\\
&+\mathds{1}_{\ell <r}\left(
     \mathcal{T}_{x,\ell}(\eta_1)(0)+\sum_{n=0}^{\left\lceil \log_2(3\ell)\right\rceil}2^n\mathcal{T}_{x,\ell}(\eta_{2^{n+1}})(0)+\int_{\mathbb{R}^d\backslash \bb_{2\ell}}\left(\fint_{\bb_{\ell}(y)}\vert f_r(z)\vert^2\dd z\right)^{\frac{1}{2}}\mathcal{T}_{x,\ell}(\eta_{\ell})(y)\dd y\right).
\end{aligned}
\end{equation}
 and $\mathcal{G}_{r,\ell}$ as well as $\mathcal{T}_{x,\ell}$ are defined in \eqref{Grllem6} (for $T=r^2$) and \eqref{functionallem6} respectively.
\end{lemma}
We finally state the following bound on $q_r(\cdot,\cdot)$, for $r\leq 1$. It is only needed for technical reasons since, in view of the application of Lemma \ref{cacciopo}, $r$ is allowed to be arbitrary close to $0$. 

\begin{lemma}\label{near0estiq}Let $q(\cdot,\cdot)$ be defined in \eqref{defq}. For all $r\in (0,1)$ and $x\in\mathbb{R}^d$, there exists a random variable $\mathcal{C}_\star(r,x)$ such that for all $\sqrt{T}\geq \frac{1}{2}$ we have
\begin{equation}
\vert q_r(T,x)\vert \leq (1+r^{-\frac{d}{2}}\log(\tfrac{\sqrt{T}}{r}))\mathcal{C}_{\star}(r,x),
\label{near0est4}
\end{equation}
with for all $\gamma>0$, 
\begin{equation}\label{rajoutfinale1}
\displaystyle\sup_{(r,x)\in\mathbb{R}^+\times\mathbb{R}^d}\bigg\langle \exp\Big(\frac{1}{C}\mathcal{C}^{2\eta_\gamma}_{\star}(r,x)\Big)\bigg\rangle\leq 2\quad \text{and}\quad \eta_\gamma=\frac{\beta\wedge d}{d}\mathds{1}_{\beta\neq d}+\frac{1}{1+\gamma}\mathds{1}_{\beta=d},
\end{equation}
for some constant $C<\infty$ depending on $d$, $\lambda$ and $\gamma$.
\end{lemma}

\section{Proofs}\label{proofsec}
We give in the section the all proofs of the results stated in the sections \ref{lemsec3}, \ref{lemsec1} and \ref{lemsec2}. For notational convenience, we shall assume that the results of Lemmas \ref{nergyestideter} and \ref{cacciopo} hold for the universal constant $c=1$. In the general case, it suffices to change the kernels $g_r$ and $\eta_r$ from line to line (by allowing a constant in the exponential). We also drop the dependance on $d$, $\lambda$ and $\beta$ in the inequalities.
\subsection{Proof of the deterministic results}
\subsubsection{Proof of Lemma \ref{nergyestideter}: Localized energy estimates} 
We only provide the arguments for \eqref{LemE2} and \eqref{LemE3}, the proof of \eqref{LemE1} can be found in \cite[Lemma 1]{gloria2015corrector}. Without loss of generality, we may assume that $x=0$.\newline
\newline
\textbf{Step 1. Proof of \eqref{LemE2}. }We set for all $t\in [0,1]$, $w(t,\cdot):=\int_{0}^t v(s,\cdot)\dd s$ and we note that $w$ is a weak solution of
$$
\left\{
    \begin{array}{ll}
        \partial_{\tau}w-\nabla\cdot a\nabla w=\nabla\cdot \int_{0}^{\tau} f(s,\cdot)\dd s &\text{ on $(0,1]\times \mathbb{R}^d$}, \\
        w(0)=0. & 
    \end{array}
\right.
$$
The idea of the proof is to use the estimate \eqref{LemE1} by expressing $w$ with help of the Duhamel formula. We denote by $\text{S}$ the semigroup associated to the operator $-\nabla\cdot a\nabla$, namely $(\text{S}(t))_{t\in\mathbb{R}^+}$ is a family of operators such that for all Schwartz distributions $\zeta$ on $\mathbb{R}^d$, $z:=\text{S}(\cdot)\zeta$ is the unique weak solution of
$$
\left\{
    \begin{array}{ll}
        \partial_{\tau}z-\nabla\cdot a\nabla z=0 &\text{ on $(0,1]\times \mathbb{R}^d$}, \\
        z(0)=\zeta.& 
    \end{array}
\right.
$$
We express $\nabla w(1,\cdot)$ with help of $\text{S}$ in form of the Duhamel formula, that is
$$\nabla w(1,\cdot)=\int_{0}^1 \nabla \left(\text{S}(1-t)\nabla\cdot \int_{0}^t f(s,\cdot)\dd s\right)\dd t.$$
Thus, we write for all $R\geq 1$, using the triangle inequality and $\int_{0}^t f(s,\cdot)\dd s=\int_{0}^1 f(s,\cdot)\dd s -\int_{t}^1 f(s,\cdot)\dd s$ in the last line
\begin{align}
\left(\int_{\mathbb{R}^d}\eta_{R}(y)\vert\nabla w(1,y)\vert^2\dd y\right)^{\frac{1}{2}}=&\left(\int_{\mathbb{R}^d}\eta_{R}(y)\left\vert \int_{0}^1 \nabla \left(\text{S}(1-t)\nabla\cdot \int_{0}^t f(s,y)\dd s\right)\dd t\right\vert^2\dd y\right)^{\frac{1}{2}}\nonumber\\
\leq &\left(\int_{\mathbb{R}^d}\eta_{R}(y)\left\vert \int_{0}^1 \nabla \left(\text{S}(1-t)\nabla\cdot \int_{0}^1 f(s,y)\dd s\right)\dd t\right\vert^2\dd y\right)^{\frac{1}{2}}\nonumber\\
&+\left(\int_{\mathbb{R}^d}\eta_{R}(y)\left\vert \int_{0}^1 \nabla \left(\text{S}(1-t)\nabla\cdot \int_{t}^1 f(s,y)\dd s\right)\dd t\right\vert^2\dd y\right)^{\frac{1}{2}}.\label{lemE10}
\end{align}
For the first r.h.s term of \eqref{lemE10}, we use \eqref{LemE1} for $T=1$ in form of 
$$\left(\int_{\mathbb{R}^d}\eta_{R}(y)\left\vert \int_{0}^1 \nabla \left(\text{S}(1-t)\nabla\cdot \int_{0}^1 f(s,y)\dd s\right)\dd t\right\vert^2\dd y\right)^{\frac{1}{2}}\lesssim\left(\int_{\mathbb{R}^d}\eta_{R}(y)\left\vert \int_{0}^1 f(s,y)\dd s\right\vert^2\dd y\right)^{\frac{1}{2}},$$
which gives the first r.h.s term of \eqref{LemE2}. For the second r.h.s term of \eqref{lemE10}, we use \eqref{LemE1} for $T=1-t$ with $t\in (0,1)$, this time in the pointwise way, combined with the Minkowski inequality in $\LL^2(\mathbb{R}^d,\eta_R\dd x)$ (exchanging the order of integration in the $y$ and $s$ variables)  to get
\begin{align*}
\bigg(\int_{\mathbb{R}^d}\eta_{R}(y)\left\vert \int_{0}^1 \nabla \left(\text{S}(1-t)\nabla\cdot \int_{t}^1 f(s,y)\dd s\right)\dd t\right\vert^2\dd y\bigg)^{\frac{1}{2}}&\leq \int_{0}^1\left(\int_{\mathbb{R}^d}\eta_{R}(y)\left\vert \nabla \left(\text{S}(1-t)\nabla\cdot \int_{t}^1 f(s,y)\dd s\right)\right\vert^2\dd y\right)^{\frac{1}{2}}\dd t\\
&\stackrel{\eqref{LemE1}}{\lesssim}\int_{0}^1\frac{1}{1-t}\left(\int_{\mathbb{R}^d}\eta_R(y)\left\vert \int_{t}^{1} f(s,y)\dd s\right\vert^2\right)^{\frac{1}{2}}\dd t\\
&\leq \int_{0}^{1}\frac{1}{1-t}\int_{t}^1 \left(\int_{\mathbb{R}^d}\eta_R(y)\vert f(s,y)\vert^2\dd y\right)^{\frac{1}{2}}\dd s\,\dd t,
\end{align*}
which gives the second r.h.s term of \eqref{LemE2}.\newline
\newline
\textbf{Step 2. Poof of \eqref{LemE3}}. Since, for all $(y,z)\in \bb_{\ell}\times\mathbb{R}^d\backslash \bb_{\ell}$, we have $e^{-\frac{\vert y-z\vert}{c\ell}}\leq e^{\frac{1}{c}-\frac{\vert z\vert}{c\ell}}$, we deduce from \eqref{LemE1}, \eqref{LemE2} (applied with $R=\ell$) and the fact that $f$ and $q$ are compactly supported in $\bb_{\ell}$:
\begin{align*}
e^{\frac{\vert z\vert}{2c\ell}}\int_{\bb_{\ell}(z)}\left\vert \int_{0}^1\nabla v(s,y)\dd s\right\vert^2\dd y\lesssim & \,e^{-\frac{\vert z\vert}{2c\ell}}\left(\int_{\bb_{\ell}}\vert q(y)\vert^2\dd y+\int_{\bb_{\ell}}\left\vert \int_{0}^1f(s,y)\dd s\right\vert^2\dd y\right)\nonumber\\
&+e^{-\frac{\vert z\vert}{2c\ell}}\left(\int_{0}^1\frac{1}{1-t}\int_{t}^1\left(\int_{\bb_{\ell}}\vert f(s,y)\vert^2\dd y\right)^{\frac{1}{2}}\dd s\,\dd t\right)^2,
\end{align*}
which yields \eqref{LemE3} by integrating over $\mathbb{R}^d\backslash \bb_{\ell}$.

\subsection{Proof of the large scale regularity results}
We provide the proofs of Lemmas \ref{ctrlav2} and \ref{ctrlav}. Our main tool here in the large-scale regularity theory for parabolic system recalled in Appendix \ref{reggech}.
\subsubsection{Proof of Lemma \ref{ctrlav2}: Pointwise estimates on the dual problem}
We prove Lemma \ref{ctrlav2} in two steps. The first step is devoted to prove \eqref{lem2} and we do it in two substeps. First, we treat the particular case where $f_r$ is compactly supported in the ball $\bb_r$ for some $r\geq 1$. We prove that \eqref{lem2} holds without the logarithmic correction. Second, we treat the general case by decomposing $\mathbb{R}^d$ into dyadic annuli $(\mathcal{B}_k)_{k\in\mathbb{N}}$, defined by $\mathcal{B}_k:=\bb_{2^{k}r}\backslash \bb_{2^{k-1} r}$ for $k\geq 1$ and $\mathcal{B}_0:= \bb_r$, and writing $f_r=\sum_{k=0}^{+\infty}f_r\chi_k$, where $(\chi_k)_{k\in\mathbb{N}}$ is a partition of unity according to the decomposition $(\mathcal{B}_k)_{k\in\mathbb{N}}$. We then apply the result of the compactly supported case for each $k\in\mathbb{N}$. The second step is devoted to prove \eqref{dualboundd2} and this is done by using the results of the first step. This extends Lemmas $2$, $3$ and $4$ of \cite{gloria2019quantitative} from the elliptic to the parabolic setting.\newline
\newline
\textbf{Step 1. Proof of \eqref{lem2}. }We split the proof into two substeps.\newline
\newline
\textbf{Substep 1.1.}
 We prove that under the assumptions 
\begin{equation}
\text{$f_r$ is supported in $\bb_r$$\quad$ and $\quad$$r^{d}\sup_{x\in\mathbb{R}^{d}}\vert f_r(x)\vert+r^{d+1}\sup_{x\in\mathbb{R}^d}\vert\nabla f_r(x)\vert\lesssim 1$},
\label{assumelem4step1}
\end{equation}
we have for all $x\in\mathbb{R}^{d}$ and $\sqrt{-t}\geq 2r_*(x)$
\begin{equation}
\left(\fint_{\bb_{r_*(x)}(x)}\vert \nabla v_r(s,y)\vert^2\dd y\right)^{\frac{1}{2}}\lesssim \frac{\left(\frac{r_*(0)}{r}\vee 1\right)^{\frac{d}{2}}}{(\vert x\vert+r)^d}.
\label{boundlem22}
\end{equation}
The estimate \eqref{boundlem22} will come from the following four relations and estimates:
\begin{enumerate}
\item For all $(t,x)\in\mathbb{R}_-\times\mathbb{R}^d$
\begin{equation}
\nabla v_r(t,x)=\int_{0}^{-t}\nabla w_r(s,x)\dd s,
\label{proofpointwise1}
\end{equation}
with $w_r$ is the weak solution of
\begin{equation}
\left\{
    \begin{array}{ll}
        \partial_{\tau} w_r-\nabla\cdot a\nabla w_r=0 & \text{ in $(0,+\infty)\times \mathbb{R}^d$},\\
        w_r(0)=-\nabla\cdot af_r\overline{e}. &
    \end{array}
\right.
\label{proofpointwise2}
\end{equation}
\item The plain energy estimate: for all $t\in\mathbb{R}_-$
\begin{equation}
\int_{\mathbb{R}^{d}}\vert\nabla v_r(t,x)\vert^2 \dd x\lesssim r^{-d}.
\label{fullenergylem2}
\end{equation}
\item The large-scale regularity estimate: for all $(t,x)\in\mathbb{R}_-\times\mathbb{R}^{d}$ such that $\vert x\vert \geq 4(r_*(x)\vee r)$
\begin{equation}
\fint_{\bb_{r_*(x)}(x)}\vert \nabla v_r(t,y)\vert^2\dd y\lesssim \frac{\left(\frac{r_*(0)}{r}\vee 1\right)^{d}}{(\vert x\vert+r)^{2d}}.
\label{boundlem23}
\end{equation}
\item The large-scale $\text{C}^{0,1}$ estimates: for all $x\in\mathbb{R}^d$, $\sqrt{-t}\geq 2r_*(x)$ and $r\geq r_*(x)$:
\begin{itemize}
\item For $\sqrt{-t}\leq r$,
\begin{equation}
\fint_{\bb_{r_*(x)}(x)}\vert\nabla v_r(t,y)\vert^2\dd y\lesssim \fint_{t}^{0}\fint_{\bb_{\sqrt{-t}}(x)}\vert\nabla v_r(s,y)\vert^2\dd y\, \dd s+r^{-2d}.
\label{proofpointwise3}
\end{equation}
\item For $\sqrt{-t}\geq r$,
\begin{equation}
\fint_{\bb_{r_*(x)}(x)}\vert\nabla v_r(t,y)\vert^2\dd y\lesssim \fint_{t}^{t+r^2}\fint_{\bb_{r}(x)}\vert\nabla v_r(s,y)\vert^2\dd y\, \dd s+r^{-2d}.
\label{proofpointwise4}
\end{equation}
\end{itemize}
\end{enumerate}
\textbf{Argument for \eqref{proofpointwise1}. }A direct computation shows that $(t,x)\in \mathbb{R}_-\times\mathbb{R}^d\mapsto \int_{0}^{-t} w_r(s,x)\dd s$ is a weak solution of \eqref{equationvlem2}.
Indeed, for every $\psi\in \cc^{\infty}_c((-\infty,0)\times \mathbb{R}^d)$, we first have by applying Fubini's theorem
$$-\int_{\mathbb{R}^d}\int_{-\infty}^{0}\partial_t\psi(t,x)\int_{0}^{-t} w_r(s,x)\dd s\, \dd t\, \dd x=-\int_{\mathbb{R}^d}\int_{0}^{+\infty}w_r(s,x)\int_{-\infty}^{-s}\partial_t\psi(t,x)\dd t\, \dd s\, \dd x.$$
Then, noticing that $\int_{-\infty}^{-s}\partial_t\psi(t,x)\dd t=-\partial_s\int_{-\infty}^{-s}\psi(t,x)\dd t$, we obtain from \eqref{proofpointwise2}
\begin{align*}
-\int_{\mathbb{R}^d}\int_{0}^{+\infty}w_r(s,x)\int_{-\infty}^{-s}\partial_t\psi(t,x)\dd t\, \dd s\, \dd x
=&\int_{\mathbb{R}^d}\int_{0}^{+\infty} w_r(s,x)\partial_s\int_{-\infty}^{-s}\psi(t,x)\dd t\, \dd s\, \dd x\\
\stackrel{\eqref{proofpointwise2}}{=}&- \int_{\mathbb{R}^d}\int_{-\infty}^{0} a(x)f_r(x)\overline{e}\cdot \nabla\psi(t,x)\dd t\,\dd x\\
&+\int_{\mathbb{R}^d}\int_{0}^{+\infty}\nabla w_r(s,x)\cdot a(x)\nabla\int_{-\infty}^{-s}\psi(t,x)\dd t\, \dd s\, \dd x.
\end{align*}
Finally, applying once more Fubini's theorem in the last integral yields
\begin{align*}
-\int_{\mathbb{R}^d}\int_{-\infty}^{0}\partial_t\psi(t,x)\int_{0}^{-t} w_r(s,x)\dd s\, \dd t\, \dd x=&-\int_{\mathbb{R}^d}\int_{-\infty}^{0} a(x)f_r(x)\overline{e}\cdot \nabla\psi(t,x)\dd t\,\dd x\\
&+\int_{\mathbb{R}^d}\int_{-\infty}^{0}\nabla \int_{0}^{-t}w_r(s,x)\dd s\cdot a(x)\nabla\psi(t,x)\dd t\, \dd x,
\end{align*}
which shows that $(t,x)\in \mathbb{R}_-\times\mathbb{R}^d\mapsto \int_{0}^{-t} w_r(s,x)\dd s$ is a weak solution of \eqref{equationvlem2}. Thus, by uniqueness, for all $t\in \mathbb{R}_-$, $v_r(t,\cdot)=\int_{0}^{-t} w_r(s,\cdot)\dd s$ and \eqref{proofpointwise1} follows.\newline
\newline
\textbf{Argument for \eqref{fullenergylem2}. }We have, by using the formula \eqref{proofpointwise1} combined with the localized energy estimate \eqref{LemE1} applied to the equation \eqref{proofpointwise2} and the support condition \eqref{assumelem4step1} of $f_r$, for all $(t,x)\in\mathbb{R}_-\times\mathbb{R}^d$ 
\begin{align}
\int_{\mathbb{R}^d}\eta_{\sqrt{-t}}(y-x)\vert \nabla v_r(t,y)\vert^2\dd y&\stackrel{\eqref{proofpointwise1}}{=}\int_{\mathbb{R}^d}\eta_{\sqrt{-t}}(y-x)\left\vert \int_{0}^{-t}\nabla w_r(t,y)\right\vert^2\dd y\nonumber\\
&\stackrel{\eqref{LemE1}}{\lesssim} \int_{\mathbb{R}^d}\eta_{\sqrt{-t}}(y-x)\vert f_r(y)\vert^2\dd y\nonumber\\
&\stackrel{\eqref{assumelem4step1}}{\lesssim} r^{-2d}\int_{\bb_r}\eta_{\sqrt{-t}}(y-x)\dd y,\label{proofpointwise12}
\end{align}
which gives \eqref{fullenergylem2} by integrating the estimate over $x\in\mathbb{R}^d$.\newline
\newline
\textbf{Argument for \eqref{boundlem23}. }We first prove by a duality argument that for all $R\geq 2(r_*(0)\vee r)$ and $t\in\mathbb{R}_-$
\begin{equation}
\int_{\mathbb{R}^{d}\backslash \bb_R}\vert \nabla v_r(t,y)\vert^2 \dd y\lesssim\left(\frac{r_*(0)}{r}\vee 1\right)^{d}R^{-d}.
\label{dualitylem2}
\end{equation}
Let $h\in \cc^{\infty}_c(\mathbb{R}^{d})$ supported in $\mathbb{R}^{d}\backslash \bb_R$. Let $t\in\mathbb{R}_-$, $s\in [0,-t]$ and $k^{s}$ be the weak solution of the backward parabolic system (corresponding to the dual system of \eqref{equationvlem2} with final time $s$),
\begin{equation}
\left\{
    \begin{array}{ll}
        \partial_{\tau}k^{s}+\nabla\cdot a^{*}\nabla k^{s}=0 & \text{ on $(-\infty,s)\times \mathbb{R}^d$},\\
        k^s(s)=\nabla\cdot h. &
    \end{array}
\right.
\label{equationwlem2}
\end{equation}
For all $\tau\in (-\infty,s)$, we have, by testing \eqref{equationwlem2} with $w_r$
$$\int_{\mathbb{R}^{d}}w_r(\tau,y)\partial_{\tau}k^{s}(\tau,y) \dd y-\int_{\mathbb{R}^{d}}\nabla k^s(\tau,y)\cdot a(y)\nabla w_r(\tau,y)\dd y=0,$$
and by testing \eqref{proofpointwise2} with $k^{s}$
$$\int_{\mathbb{R}^{d}}k^{s}(\tau,y)\partial_{\tau }w_r(\tau,y) \dd y+\int_{\mathbb{R}^{d}}\nabla k^s(\tau,y)\cdot a(y)\nabla w_r(\tau,y) \dd y=0. $$
By summing the two identities above, integrating in time over $\tau\in [0,s)$ and noticing that from the initial conditions of $w_r$ and $k^{s}$ we have
\begin{align*}
&\int_{\mathbb{R}^{d}}\int_{0}^{s} w_r(\tau,y)\partial_{\tau}k^{s}(\tau,y) \dd \tau\, \dd y+\int_{\mathbb{R}^{d}}\int_{0}^{s} k^{s}(\tau,y)\partial_{\tau }w_r(\tau,y)\dd \tau\, \dd y\\
&=\left[\int_{\mathbb{R}^d}w_r(\tau,y)k^{s}(\tau,y)\dd y\right]^{s}_{0}=-\int_{\mathbb{R}^d}\nabla k^{s}(0,y)\cdot a(y)f_r(y)\overline{e}\dd y-\int_{\mathbb{R}^d}\nabla w_r(s,y)\cdot h(y)\dd y,
\end{align*}
we get 
$$\int_{\mathbb{R}^d}\nabla w_r(t,y)\cdot h(y)\dd y=-\int_{\mathbb{R}^d}\nabla k^{s}(0,y)\cdot a(y)f_r(y)\overline{e}\,\dd y.$$ 
It follows by integrating over $s\in [0,-t]$, using the formula \eqref{proofpointwise1} combined with the Cauchy-Schwarz inequality and \eqref{assumelem4step1} that 
\begin{equation}
\left\vert \int_{\mathbb{R}^{d}}h(y)\cdot\nabla v_r(t,y)\dd y\right\vert \leq \left(\fint_{\bb_r}\left\vert\int_{0}^{-t}\nabla k^s(0,y)\dd s\right\vert^2\dd y\right)^{\frac{1}{2}}.
\label{lem2bound5}
\end{equation}
The r.h.s of \eqref{lem2bound5} is then dominated as follows. First, we set $\tilde{v}(t,\cdot):=\int_{0}^{-t}k^{s}(0,y)\dd s$ and by noticing that for all $s\in (0,-t)$, $k^s(0,\cdot)=k^{0}(-s,\cdot)$, we have 
\begin{equation*}
\left\{
    \begin{array}{ll}
        \partial_{\tau}\tilde{v}+\nabla\cdot a^*\nabla \tilde{v}=\nabla\cdot h & \text{ in $(-\infty,0)\times\mathbb{R}^d$}, \\
        \tilde{v}(0)=0. & 
    \end{array}
		\right.
\end{equation*}
Second, we denote by $\overline{v}$ the weak solution of 
\begin{equation}
\left\{
    \begin{array}{ll}
        \partial_{\tau}\overline{v}+\nabla\cdot a^*\nabla \overline{v}=\mathds{1}_{(-\infty,0)}\nabla\cdot h & \text{ in $\mathbb{R}^{d+1}$}, \\
        \overline{v}=0 & \text{ in $\mathbb{R}^+\times\mathbb{R}^d$}.
    \end{array}
		\right.
		\label{proofpointwise5}
\end{equation}
$\overline{v}$ is an extension of $\tilde{v}$ in the sense that $\nabla\tilde{v}(s,\cdot)=\nabla\overline{v}(s,\cdot)$ as long as $s\leq 0$. Now, since $h=0$ in $ \bb_R$, we have by using the estimate $\fint_{\bb_r}\lesssim (\frac{r_*(0)}{r}\vee 1)^{d}\fint_{\bb_{r_*(0)\vee r}}$, Lemma \ref{scott} and the large-scale $\text{C}^{0,1}$ estimate \eqref{meanvalue} (recalling that $R\geq 2(r_*(0)\vee r)$):
\begin{align}
\left(\fint_{\bb_r}\left\vert\int_{0}^{-t}\nabla k^s(0,y)\dd s\right\vert^2\dd y\right)^{\frac{1}{2}}&=\left(\fint_{\bb_r}\left\vert \nabla\overline{v}(t,y)\right\vert^2\dd y\right)^{\frac{1}{2}}\nonumber\\
&\lesssim\left(\frac{r_*(0)}{r}\vee 1\right)^{\frac{d}{2}}\left(\fint_{\bb_{r_*(0)\vee r}}\left\vert\nabla\overline{v}(t,y)\right\vert^2\dd y\right)^{\frac{1}{2}}\nonumber\\
&\lesssim \left(\frac{r_*(0)}{r}\vee 1\right)^{\frac{d}{2}}\left(\fint_{t}^{t+4(r_*(0)\vee r)^2}\fint_{\bb_{2(r_*(0)\vee r)}}\left\vert\nabla\overline{v}(s,y)\right\vert^2\dd y\, \dd s\right)^{\frac{1}{2}}\nonumber\\
&\stackrel{\eqref{meanvalue}}{\lesssim}\left(\frac{r_*(0)}{r}\vee 1\right)^{\frac{d}{2}}\left(\fint_{t}^{t+R^2}\fint_{\bb_R}\vert\nabla\overline{v}(s,y)\vert^2\dd y\, \dd s\right)^{\frac{1}{2}}.\label{proofpointwise6}
\end{align}
Now, using the plain energy estimate (for which the proof is identical as \eqref{fullenergylem2}): for all $s\in \mathbb{R}_-$
$$\int_{\mathbb{R}^d}\vert\nabla\tilde{v}(s,y)\vert^2\dd y\lesssim \int_{\mathbb{R}^d}\vert h(y)\vert^2\dd y,$$
and since $\overline{v}\equiv 0$ in $\mathbb{R}^+\times\mathbb{R}^d$, we get
\begin{itemize}
\item[(i)] for $\sqrt{-t}\geq R$
\begin{equation}
\fint_{t}^{t+R^2}\fint_{\bb_R}\vert\nabla\overline{v}(s,y)\vert^2\dd y\, \dd s=\fint_{t}^{t+R^2}\fint_{\bb_R}\vert\nabla\tilde{v}(s,y)\vert^2\dd y\, \dd s\lesssim R^{-d}\int_{\mathbb{R}^d}\vert h(y)\vert^2\dd y.
\label{proofpointwise7}
\end{equation}
\item[(ii)]for $R\geq \sqrt{-t}$
\begin{equation}
\fint_{t}^{t+R^2}\fint_{\bb_R}\vert\nabla\overline{v}(s,y)\vert^2\dd y\, \dd s=R^{-2}\int_{t}^0\fint_{\bb_R}\vert\nabla\tilde{v}(s,y)\vert^2\dd y\, \dd s\lesssim \left(\frac{\sqrt{-t}}{R}\right)^2R^{-d}\int_{\mathbb{R}^d}\vert h(y)\vert^2\dd y\lesssim R^{-d}\int_{\mathbb{R}^d}\vert h(y)\vert^2\dd y.
\label{proofpointwise8}
\end{equation}
\end{itemize}
The combination of \eqref{lem2bound5}, \eqref{proofpointwise6}, \eqref{proofpointwise7} and \eqref{proofpointwise8} yields 
\begin{equation*}
\left\vert \int_{\mathbb{R}^{d}}h(y)\cdot\nabla v_r(t,y)\dd y\right\vert\lesssim \left(\frac{r_*(0)}{r}\vee 1\right)^{\frac{d}{2}}R^{-\frac{d}{2}}\left(\int_{\mathbb{R}^d}\vert h(y)\vert^2\dd y\right)^{\frac{1}{2}},
\end{equation*}
which gives \eqref{dualitylem2} by the arbitrariness of $h$.
\newline
\newline
We now prove \eqref{boundlem23}. Let $R:=\frac{1}{2}\vert x\vert$ and assume that $R\geq 2(r_*(x)\vee r)$. Without loss of generality, we may assume that $R\geq 2r_*(0)$. Indeed, otherwise, we deduce from the $\frac{1}{8}$-Lipschitz property of $r_*(x)$ in form of 
$$r_*(0)\leq r_*(x)+\frac{\vert x\vert}{8}\quad \Rightarrow\quad r_*(x)\geq \frac{\vert x\vert}{4},$$
and 
$$r_*(0)\geq r_*(x)-\frac{\vert x\vert}{8}\quad \Rightarrow\quad \frac{3}{2}r_*(0)\geq r_*(x),$$
as well as \eqref{fullenergylem2} that
$$\fint_{\bb_{r_*(x)}(x)}\vert\nabla v_r(t,y)\vert^2\dd y\lesssim r^{-d}_*(x)\int_{\mathbb{R}^d}\vert\nabla v_r(t,y)\vert^2\dd y\lesssim \frac{r^d_*(0)}{(\vert x\vert+r)^d}.$$
Now, we observe that $\bb_R(x)\subset \mathbb{R}^{d}\backslash \bb_R$.
Indeed for all $y\in \bb_R(x)$, the triangle inequality yields
$$\vert y\vert\geq \left\vert \vert y-x\vert-\vert x\vert\right\vert\geq 2R-R=R,$$
so that $y\notin \bb_R$. We then argue once again by extension and we consider $\overline{v}_r$ the weak solution of
\begin{equation}
\left\{
    \begin{array}{ll}
        \partial_{\tau}\overline{v}_r+\nabla\cdot a\nabla\overline{v}_r=\mathds{1}_{(-\infty,0)}\nabla\cdot af_r\overline{e} & \text{in $\mathbb{R}^{d+1}$}, \\
        \overline{v}_r=0 & \text{ in $\mathbb{R}^+\times\mathbb{R}^d$},
    \end{array}
\right.
\label{proofpointwise9}
\end{equation}
for which $\overline{v}_r(s,\cdot)=v_r(s,\cdot)$ as long as $s\leq 0$. It then follows from Lemma \ref{scott} applied to the equation \eqref{proofpointwise9} and the large-scale $\text{C}^{0,1}$ estimate \eqref{meanvalue} (noticing that $\bb_R(x)\subset \mathbb{R}^{d}\backslash \bb_R$ and \eqref{assumelem4step1} implies $f_r\equiv 0$ on $\bb_R(x)$): for all $t\in \mathbb{R}_-$
\begin{align}
\fint_{\bb_{r_*(x)}(x)}\vert \nabla v_r(t,y)\vert^2\dd y=\fint_{\bb_{r_*(x)}(x)}\vert \nabla \overline{v}_r(t,y)\vert^2 \dd y
&\lesssim\fint_{t}^{t+4r^{2}_*(x)}\fint_{\bb_{2r_*(x)}(x)}\vert\nabla \overline{v}_r(s,y)\vert^2\dd s\, \dd y\nonumber\\
&\stackrel{\eqref{meanvalue}}{\lesssim} \fint_{t}^{t+R^2}\fint_{\bb_R(x)}\vert \nabla \overline{v}_r(s,y)\vert^2\dd s\, \dd y.\label{lem4esti1duality}
\end{align}
Now, since $\overline{v}_r(s,\cdot)=v_r(s,\cdot)$ as long as $s\leq 0$ and $\overline{v}_r(s,\cdot)\equiv 0$ for $s\geq 0$, we have from \eqref{dualitylem2}:
\begin{itemize}
\item[(i)]For $\sqrt{-t}\geq R$,
\begin{align}
\fint_{t}^{t+R^2}\fint_{\bb_R(x)}\vert \nabla \overline{v}_r(s,y)\vert^2\dd s\, \dd y =\fint_{t}^{t+R^2}\fint_{\bb_R(x)}\vert \nabla v_r(s,y)\vert^2\dd s\, \dd y
&\leq R^{-d} \fint_{t}^{t+R^2}\int_{\mathbb{R}^d\backslash\bb_R}\vert \nabla v_r(s,y)\vert^2\dd s\, \dd y\nonumber\\
&\stackrel{\eqref{dualitylem2}}{\lesssim}\left(\frac{r_*(0)}{r}\vee 1\right)^{d}R^{-2d}. 
\label{proofpointwise10}
\end{align}
\item[(ii)]For $R\geq \sqrt{-t}$,
\begin{align}
\fint_{t}^{t+R^2}\fint_{\bb_R(x)}\vert \nabla \overline{v}_r(s,y)\vert^2\dd s\, \dd y =R^{-2}\int_{t}^{0}\fint_{\bb_R(x)}\vert \nabla v_r(s,y)\vert^2\dd s\, \dd y
&\leq R^{-d}\left(\frac{\sqrt{-t}}{R}\right)^2 \fint_{t}^{0}\int_{\mathbb{R}^d\backslash\bb_R}\vert \nabla v_r(s,y)\vert^2\dd s\, \dd y\nonumber\\
&\stackrel{\eqref{dualitylem2}}{\lesssim}\left(\frac{r_*(0)}{r}\vee 1\right)^{d}R^{-2d}. 
\label{proofpointwise11}
\end{align}
\end{itemize}
The combination of \eqref{lem4esti1duality}, \eqref{proofpointwise10} and \eqref{proofpointwise11} concludes the argument for \eqref{boundlem23} since $\vert x\vert\geq 2r$ implies that $R=\frac{1}{2}\vert x\vert\geq \frac{1}{2}( \frac{1}{2}\vert x\vert+r)$. \newline
\newline
\textbf{Argument for \eqref{proofpointwise3} and \eqref{proofpointwise4}. }It follows directly from the  combination of Lemma \ref{scott}, the large-scale $\cc^{0,1}$ estimate \eqref{meanvalue2}, the Poincaré inequality in $\bb_{\rho}(x)$ and the assumption \eqref{assumelem4step1}: for $r\geq r_*(x)$
\begin{itemize}
\item[(i)]if $\sqrt{-t}\leq r$, we use \eqref{meanvalue2} up to the scale $\sqrt{-t}$ in form of
\begin{align*}
\fint_{\bb_{r_*(x)}(x)}\vert\nabla v_r(t,y)\vert^2 \dd y\lesssim&\fint_{t}^{t+4r^2_*(x)}\fint_{\bb_{2r_*(x)}(x)}\vert\nabla v_r(s,y)\vert^2 \dd s\,\dd y+\fint_{\bb_{2r_{*}(x)}(x)}\vert f_r(y)\vert^2\dd y\\
\stackrel{\eqref{meanvalue2},\eqref{assumelem4step1}}{\lesssim}& \fint_{t}^{0}\fint_{\bb_{\sqrt{-t}}(x)}\vert\nabla v_r(s,y)\vert^2\dd s\, \dd y\\
&+\sup_{r_*(x)\leq \rho\leq \sqrt{-t}}\left(\frac{\sqrt{-t}}{\rho}\right)^2\fint_{\bb_{\rho}(x)}\left\vert f_r(y)\overline{e}-\fint_{\bb_{\rho}(x)}f_r(z)\overline{e}\,\dd z\right\vert^2 \dd y+r^{-2d}\\
\stackrel{\eqref{assumelem4step1}}{\lesssim}&\fint_{t}^{0}\fint_{\bb_{\sqrt{-t}}(x)}\vert\nabla v_r(s,y)\vert^2\dd s\, \dd y+r^{-2d}\left(\frac{\sqrt{-t}}{r}\right)^2+r^{-2d}\\
\lesssim& \fint_{t}^{0}\fint_{\bb_{\sqrt{-t}}(x)}\vert\nabla v_r(s,y)\vert^2\dd s\, \dd y+r^{-2d},
\end{align*}
which gives \eqref{proofpointwise3}.
\item[(ii)]If $\sqrt{-t}\geq r$, we use \eqref{meanvalue2} up to the scale $r$ and we obtain the same way \eqref{proofpointwise4}.
\end{itemize}
\textbf{Argument for \eqref{boundlem22} from \eqref{fullenergylem2}, \eqref{boundlem23}, \eqref{proofpointwise3} and \eqref{proofpointwise4}. }The case $\vert x\vert\geq 4(r_*(x)\vee r)$ is done via \eqref{boundlem23}. It remains to treat the case $\vert x\vert\leq 4(r_*(x)\vee r)$ and we distinguish two sub-cases:
\begin{itemize}
\item[(1)] Assume that $r_*(x)\geq r$, which means that $\vert x\vert \leq 4r_*(x)$. We have by the $\frac{1}{8}$-Lipschitz continuity property of $r_*$
\begin{equation}
r_*(x)\gtrsim r_*(0),
\label{lipschitzlem22}
\end{equation}
and
\begin{equation}
r_*(0)\geq r_*(x)-\frac{1}{8}\vert x\vert\geq \frac{r}{3}+\frac{2r_*(x)}{3}-\frac{1}{8}\vert x\vert\gtrsim r+\vert x\vert.
\label{lipschitzlem2}
\end{equation}
Thus, from \eqref{fullenergylem2} we deduce that for all $t\in \mathbb{R}_-$
\begin{align*}
\fint_{\bb_{r_*(x)}(x)}\vert \nabla v_r(t,y)\vert^2 \dd y\lesssim r^{-d}_*(x)\int_{\mathbb{R}^{d}}\vert \nabla v_r(t,y)\vert^2\dd y&\stackrel{\eqref{lipschitzlem22}}{\lesssim}r^{-d}_*(0)\int_{\mathbb{R}^{d}}\vert \nabla v_r(t,y)\vert^2 \dd y\\
&\stackrel{\eqref{fullenergylem2},\eqref{lipschitzlem2}}{\lesssim}\frac{r^{d}_*(0)}{(\vert x\vert+r)^{2d}}.
\end{align*}
\item[(2)]Assume that $r_*(x)\leq r$ which means that $\vert x\vert\leq 4r$. For all $\sqrt{-t}\geq 2r_*(x)$, we have
\begin{itemize}
\item[(i)] If $\sqrt{-t}\leq r$, we use \eqref{proofpointwise3} and the plain energy estimate \eqref{proofpointwise12} to obtain
\begin{align*}
\fint_{\bb_{r_*(x)}(x)}\vert \nabla v_r(t,y)\vert^2 \dd y\stackrel{\eqref{proofpointwise3}}{\lesssim}\fint_{t}^{0}\fint_{\bb_{\sqrt{-t}}(x)} \vert v_r(s,y)\vert^2\dd y\, \dd s+r^{-2d}&\stackrel{\eqref{proofpointwise12}}{\lesssim} r^{-2d}\int_{\bb_{r}}\eta_{\sqrt{-t}}(y-x)\dd y+r^{-2d}\\
&\leq r^{-2d}\lesssim \frac{r^d_*(0)}{(\vert x\vert+r)^{2d}}.
\end{align*}
\item[(ii)] If $\sqrt{-t}\geq r$, we use \eqref{proofpointwise4} and the plain energy estimate \eqref{fullenergylem2} to get
$$\fint_{\bb_{r_*(x)}(x)}\vert \nabla v_r(t,y)\vert^2 \dd y\stackrel{\eqref{proofpointwise4}}{\lesssim}
\fint_{t}^{t+r^2}\fint_{\bb_r(x)}\vert\nabla v_r(s,y)\vert^2\dd y\, \dd s+r^{-2d}\stackrel{\eqref{fullenergylem2}}{\lesssim} r^{-2d}\lesssim \frac{r^d_*(0)}{(\vert x\vert+r)^{2d}}.$$
\end{itemize}
\end{itemize}
This concludes the proof of \eqref{boundlem22}.\newline
\newline
\textbf{Substep 1.2. }We prove \eqref{lem2} without the support condition \eqref{assumelem4step1} on $f_r$. We decompose the r.h.s of \eqref{equationvlem2} according to a family of dyadic annuli $(\mathcal{B}_k)_{k\in\mathbb{N}}$, defined by $\mathcal{B}_k:=\bb_{2^{k}r}\backslash \bb_{2^{k-1} r}$ for all $k\geq 1$ and $\mathcal{B}_0:=\bb_r$. Namely, we set for all $k\in\mathbb{N}$, $f_{r,k}:=f_r\chi_{k}$, where $(\chi_k)_{k\in\mathbb{N}}$ is a partition of unity according to the decomposition $(\mathcal{B}_k)_{k\in\mathbb{N}}$, and we denote by $v_{r,k}$ the weak solution of \eqref{equationvlem2} with r.h.s $\nabla\cdot a f_{r,k}\overline{e}$. By uniqueness, we have $\nabla v_r=\sum_{k=0}^{+\infty}\nabla v_{r,k}$. Hence, we get by the triangle inequality
$$\left(\fint_{\bb_{r_*(x)}(x)}\vert \nabla v_r(t,y)\vert^2 \dd y\right)^{\frac{1}{2}}\leq\sum_{k=0}^{+\infty}\left(\fint_{\bb_{r_*(x)}(x)}\vert \nabla v_{r,k}(t,y)\vert^2 \dd y\right)^{\frac{1}{2}}.$$
Thanks to \eqref{assumerhsdual}, $f_{r,k}$ satisfies \eqref{assumelem4step1} with radius $2^{k-1} r$, thus, by \eqref{boundlem22}, we have for all $k\geq 0$ and $\sqrt{-t}\geq 2r_*(x)$
$$\left(\fint_{\bb_{r_*(x)}(x)}\vert \nabla v_{r,k}(t,y)\vert^2 \dd y\right)^{\frac{1}{2}}\lesssim \frac{\left(\frac{r_*(0)}{2^kr}\vee 1\right)^{\frac{d}{2}}}{(\vert x\vert+2^{k-1}r)^{d}}.$$
We deduce, setting $\text{N}_r:=\left\lceil \log_2\left(1+\frac{\vert x\vert}{r}\right)\right\rceil$
\begin{align*}
\left(\fint_{\bb_{r_*(x)}(x)}\vert \nabla v_r(t,y)\vert^2 \dd y\right)^{\frac{1}{2}}\lesssim\sum_{k=0}^{+\infty}\frac{\left(\frac{r_*(0)}{2^{k-1}r}\vee 1\right)^{\frac{d}{2}}}{(\vert x\vert+2^kr)^{d}}\lesssim \left(\frac{r_*(0)}{r}\vee 1\right)^{\frac{d}{2}}\left((\vert x\vert+r)^{-d}\text{N}_r+r^{-d}\sum_{k=\text{N}_r+1}^{+\infty}2^{-kd}\right),
\end{align*}
which gives \eqref{lem2}.\newline
\newline
\textbf{Step 2. Proof of \eqref{dualboundd2}. }We use the same type of decomposition as in Substep $1.2$ : we have $\nabla v_r=\sum_{k=0}^{+\infty}\nabla v_{r,k}$ where this time, for all $k\geq 1$, $\mathcal{B}_k:=\bb_{2^{k}}\backslash \bb_{2^{k-1}}$ and $\mathcal{B}_0:=\bb_1$. We then split the proof into two steps.\newline
\newline
\textbf{Substep 2.1. }We argue in favor of the first alternative in \eqref{dualboundd2}, that is when the r.h.s is equal to $\frac{r^{\frac{d}{2}}_*(0)r\log(1+\vert x\vert)}{(\vert x\vert+1)^d}$. From the assumption \eqref{assumerhsdual2} used in form of $\vert f_r(x)\vert +(\vert x\vert+1)\vert\nabla f_r(x)\vert\lesssim \frac{r}{(\vert x\vert+1)^d}$, we note that $\frac{1}{r}f_{r,k}$ satisfies \eqref{assumelem4step1} with radius $2^k$. Thus by \eqref{boundlem22}, we have for all $k\geq 0$ and $\sqrt{-t}\geq 2r_*(x)$
$$\left(\fint_{\bb_{r_*(x)}(x)}\vert\nabla v_{r,k}(t,y)\vert^2\dd y\right)^{\frac{1}{2}}\lesssim \frac{r\left(\frac{r_*(0)}{2^k}\vee 1\right)^{\frac{d}{2}}}{(\vert x\vert+2^k)^d}.$$
We then conclude exactly as in Substep $1.2$.\newline
\newline
\textbf{Substep 2.2. }We argue in favor of the second alternative in \eqref{dualboundd2}, that is when the r.h.s is equal to $\frac{r^{\frac{d}{2}}_*(0)}{(\vert x\vert+1)^{d-1}}$. From the assumption \eqref{assumerhsdual2} used in form of $\vert f_r(x)\vert+(\vert x\vert+1)\vert\nabla f_r(x)\vert\lesssim \frac{1}{(\vert x\vert+1)^{d-1}}$, we note that $\frac{1}{2^{k}}f_{r,k}$ satisfies \eqref{assumelem4step1} with radius $2^k$. Thus by \eqref{boundlem22}, we have for all $k\geq 0$ and $\sqrt{-t}\geq 2r_*(x)$
$$\left(\fint_{\bb_{r_*(x)}(x)}\vert\nabla v_{r,k}(t,y)\vert^2\dd y\right)^{\frac{1}{2}}\lesssim \frac{2^k\left(\frac{r_*(0)}{2^k}\vee 1\right)^{\frac{d}{2}}}{(\vert x\vert+2^k)^d}.$$
We then conclude by the same decomposition as in Substep $1.2$ : setting $\text{N}:=\left\lceil \log_2(1+\vert x\vert)\right\rceil$
\begin{align*}
\left(\fint_{\bb_{r_*(x)}(x)}\vert\nabla v_{r}(t,y)\vert^2\dd y\right)^{\frac{1}{2}} \lesssim \sum_{k=0}^{+\infty}
\frac{2^k\left(\frac{r_*(0)}{2^k}\vee 1\right)^{\frac{d}{2}}}{(\vert x\vert+2^k)^d}\lesssim  r^{\frac{d}{2}}_*(0)\left((\vert x\vert+1)^{-d}\sum_{k=0}^{\text{N}}2^k+\sum_{k=\text{N}+1}^{+\infty} 2^{(1-d)k}\right),
\end{align*}
which concludes the proof since 
$$\sum_{k=0}^{\text{N}}2^k\lesssim \vert x\vert+1 \quad\text{and}\quad \sum_{k=\text{N}+1}^{+\infty} 2^{(1-d)k}\lesssim \frac{1}{(\vert x\vert+1)^{d-1}}.$$
\subsubsection{Proof of Lemma \ref{ctrlav}: Control of averages } 
We treat separately the two regimes: the non-generic case $R\leq r_*(x)$ and the generic case $R\geq r_*(x)$.
\begin{enumerate} 
\item We start with the non-generic case $R\leq r_*(x)$. We distinguish two sub-cases.
\begin{itemize}
\item[(i)] In the case where $\sqrt{\frac{T}{2}}\leq 2r_*(x)$, we have, using \eqref{assumelem5} and $R<\sqrt{T}$
$$\fint_{\bb_R(x)}\vert\nabla u(T,y)\vert^2\dd y\lesssim\bigg(\frac{\sqrt{T}}{r}\bigg)^{d}\fint_{\bb_{\sqrt{T}}(x)}\vert \nabla u(T,y)\vert^2\dd y\stackrel{\eqref{assumelem5}}{\lesssim} \left(\frac{r_*(x)}{R}\right)^dC(x,T)f(T)g(T).$$
\item[(ii)] In the case where $\sqrt{\frac{T}{2}}\geq 2r_*(x)$, we combine, by noticing that $\partial_{\tau} u-\nabla\cdot a\nabla u=0$ on $(\frac{T}{2},T)\times \bb_{\sqrt{\frac{T}{2}}}(x)$, Lemma \ref{scott} and the large-scale $\text{C}^{0,1}$ estimate \eqref{meanvalue} to obtain
\begin{align*}
\fint_{\bb_R(x)}\vert \nabla u(T,y)\vert^2\dd y \lesssim \left(\frac{r_*(x)}{R}\right)^d\fint_{\bb_{r_*(x)}(x)}\vert \nabla u(T,y)\vert^2\dd y&\lesssim\left(\frac{r_*(x)}{R}\right)^d\fint_{\cc_{2r_*(x)}(T,x)}\vert \nabla u(s,y)\vert^2\dd s\, \dd y\\
&\stackrel{\eqref{meanvalue}}{\lesssim}\left(\frac{r_*(x)}{R}\right)^d\fint_{\frac{T}{2}}^{T}\fint_{\bb_{\sqrt{\frac{T}{2}}}(x)}\vert \nabla u(s,y)\vert^2\dd s\, \dd y\\
&\lesssim \left(\frac{r_*(x)}{R}\right)^d\fint_{\frac{T}{2}}^{T}\fint_{\bb_{\sqrt{s}}(x)}\vert \nabla u(s,y)\vert^2\dd s\, \dd y\\
&\stackrel{\eqref{assumelem5}}{\lesssim} \left(\frac{r_*(x)}{R}\right)^d f(T)g(\tfrac{T}{2})\fint_{\frac{T}{2}}^{T}C(x,s)\dd s .
\end{align*}
\end{itemize}
\item We now consider the generic case $R\geq r_*(x)$ and, without loss of generality, we may assume that $\sqrt{T}>2\sqrt{2}R$ since otherwise $\fint_{\bb_{R}}\lesssim \fint_{\bb_{\sqrt{T}}}$ and the conclusion follows from \eqref{assumelem5}. We have $\partial_{\tau} u-\nabla\cdot a\nabla u=0$ on $(0,T)\times \bb_{\sqrt{T}}$ and $r_*(x)\leq 2R<\sqrt{\frac{T}{2}}$ so that from Lemma \ref{scott} and the large-scale $\text{C}^{0,1}$estimate \eqref{meanvalue} we deduce 
$$\fint_{\bb_R(x)} \vert\nabla u(T,y)\vert^2\dd y\lesssim\fint_{T-4R^2}^{T}\fint_{\bb_{2R}}\vert \nabla u(s,y)\vert^2\dd y\, \dd s\lesssim \fint_{\frac{T}{2}}^{T}\fint_{\bb_{\sqrt{T}}(x)}\vert \nabla u(s,y)\vert^2\dd y\, \dd s\lesssim C(x,T)f(T)g(\tfrac{T}{2}).$$
\end{enumerate}
\subsection{Proof of the suboptimal control of fluctuations of the time dependent flux }
We provide in this section the proofs of Lemmas \ref{functioderiv}, \ref{functioderiv2}, \ref{near0estiq} and Proposition \ref{Sensitivitysubop} of Section \ref{lemsec2}.
\subsubsection{Proof of Lemmas \ref{functioderiv} and \ref{functioderiv2}: Control of the functional derivatives}
We prove Lemma \ref{functioderiv} and Lemma \ref{functioderiv2} independently. For heuristic arguments, we refer to I) in Section \ref{stratsec}.
\begin{proof}[Proof of lemma \ref{functioderiv}.]
It is enough to prove \eqref{lem6} for the quantities $q_r(T)\cdot e_k$, for all $r>0$, $T\geq 1$ and $k\in \llbracket 1,d\rrbracket$. In particular, we only treat the case $k=1$, since the other contributions are controlled the same way. For notational convenience, we simply write $v^T$ for $v^T_1$ .\newline
\newline
Let $x\in\mathbb{R}^d$, $h\in (0,1)$, $T\geq 1$, $r>0$, $\ell\in [1,\infty)$ and $\delta a$ be compactly supported in $\bb_{\ell}(x)$ such that $\sup_{y\in \bb_{\ell}(x)}\vert \delta a(y)\vert\leq 1$. We compute the finite difference
\begin{align}
\delta^h q_r(T)\cdot e_1:=&\frac{q_r(a+h\delta a, T)\cdot e_1-q_r(a,T)\cdot e_1}{h}\nonumber\\
=&\int_{\mathbb{R}^d}  g_r(y)e_1\cdot \delta a(y)e\,\dd y+\int_{\mathbb{R}^d} g_r(y)e_1\cdot \delta a(y)\bigg(\int_{0}^T\nabla u(a+h\delta a,t,y)\dd t\bigg)\dd y\nonumber\\
&+\int_{0}^T\int_{\mathbb{R}^d} g_r(y)e_1\cdot a(y)\nabla \delta^h u(t,y)\dd y\,\dd t,\label{lem62}
\end{align}
where $\delta^h u(t,\cdot):=\frac{u(a+h\delta a,t,\cdot)-u(a,t,\cdot)}{h}$ is the weak solution of
\begin{equation}
\left\{
    \begin{array}{ll}
        \partial_{\tau}\delta^h u-\nabla\cdot a\nabla \delta^h u=\nabla\cdot \delta a\nabla u(a+h\delta a,\cdot,\cdot) & \text{on $(0,+\infty)\times \mathbb{R}^d$}, \\
       \delta^h u(0,\cdot)=\nabla\cdot \delta a(\cdot)e. & 
    \end{array}
\right.
\label{equationdeltaulem6}
\end{equation}
The first r.h.s term of \eqref{lem62} gives directly the first r.h.s terms of \eqref{lem6}. For the second r.h.s term of \eqref{lem62}, we easily derive from the localized energy estimates \eqref{LemE2} and \eqref{LemE3} applied to the equation \eqref{equationdeltaulem6} combined with \eqref{LemE1} with $a$ replaced by $a+h\delta a$ (which control the norm of the r.h.s of \eqref{equationdeltaulem6}) that 
$$\int_{\mathbb{R}^d} g_r(y)e_1\cdot \delta a(y)\bigg(\int_{0}^T\nabla u(a+h\delta a,t,y)\dd t\bigg)\dd y\underset{h\downarrow 0}{\rightarrow} \int_{\mathbb{R}^d} g_r(y)e_1\cdot \delta a(y)\bigg(\int_{0}^T\nabla u(t,y)\dd t\bigg)\dd y.$$
We thus obtain the second second r.h.s term of \eqref{lem6}. We now focus on the third r.h.s term of \eqref{lem62}.
 We split the time integral into the singular part $t\leq 1$ and the regular part $t\geq 1$
\begin{equation}
\int_{0}^T\int_{\mathbb{R}^d} g_r(y)e_1\cdot a(y)\nabla \delta^h u(t,y)\dd y\,\dd t=\int_{\mathbb{R}^d} g_r(y)e_1\cdot a(y)\bigg(\int_{0}^1\nabla \delta^h u(t,y)\dd t\bigg)\dd y+\int_{1}^T\int_{\mathbb{R}^d} g_r(y)e_1\cdot a(y)\nabla \delta^h u(t,y)\dd y\,\dd t.
\label{splitlem6}
\end{equation}
We now split the rest of the proof into two steps, treating the two r.h.s terms of \eqref{splitlem6} separately.\newline
\newline
\textbf{Step 1. First r.h.s term of \eqref{splitlem6}. }We prove that
\begin{equation}
\limsup_{h\downarrow 0}\bigg\vert\int_{\mathbb{R}^d} g_r(y)e_1\cdot a(y)\bigg(\int_{0}^1\nabla\delta^h u(t,y)\dd t\bigg)\dd y\bigg\vert\lesssim \mathds{1}_{\ell<\sqrt{T}}\,\ell^{\frac{d}{2}}\bigg(\int_{\mathbb{R}^d} e^{-\frac{\vert x-z\vert}{2c\ell}}\fint_{\bb_{\ell}(z)}\vert g_r(y)\vert^2\dd y\, \dd z\bigg)^{\frac{1}{2}}+\mathds{1}_{\ell\geq \sqrt{T}}\mathcal{T}_{x,\ell}(\eta_r)(0),
\label{near0lem71}
\end{equation}
where $\mathcal{T}_{x,\ell}(\eta_r)$ is defined in \eqref{functionallem6}. We argue differently between the two regimes $\ell<\sqrt{T}$ and $\ell\geq \sqrt{T}$. \newline
\newline
\textbf{Regime $\ell<\sqrt{T}$. }Using that $\int_{\mathbb{R}^d}=\int\fint_{\bb_{\ell}(z)}$ and by splitting $\mathbb{R}^d$ into $\bb_{\ell}(x)$ and $\mathbb{R}^d\backslash \bb_{\ell}(x)$, we have
\begin{align}
\bigg\vert\int_{\mathbb{R}^d} g_r(y)e_1\cdot a(y)\bigg(\int_{0}^1\nabla \delta^h u(t,y)\dd t\bigg) \dd y\bigg\vert\lesssim & \int_{\mathbb{R}^d}\fint_{\bb_{\ell}(z)}\bigg\vert\int_{0}^1\nabla \delta^h u(t,y)\dd t\bigg\vert g_r(y)\dd y\,\dd z\nonumber\\
=& \int_{\bb_{\ell}(x)}\fint_{B_{\ell}(z)}\bigg\vert\int_{0}^1\nabla \delta^h u(t,y)\dd t\bigg\vert g_r(y)\dd y\,\dd z& (=:I^{\ell}_1)\nonumber\\
&+\int_{\mathbb{R}^d\backslash \bb_{\ell}(x)}\fint_{\bb_{\ell}(z)}\bigg\vert\int_{0}^1\nabla \delta^h u(t,y)\dd t\bigg\vert g_r(y)\dd y\,\dd z.& (=:I^{\ell}_2)\label{lem64}
\end{align}
We now show that 
\begin{equation}
I^{\ell}_1+I^{\ell}_2\lesssim\ell^{\frac{d}{2}}\bigg(\int_{\mathbb{R}^d} e^{-\frac{\vert x-z\vert}{2c\ell}}\fint_{\bb_{\ell}(z)}\vert g_r(y)\vert^2\dd y\,\dd z\bigg)^{\frac{1}{2}}.
\label{I1I2}
\end{equation}
Since the arguments are similar we only give the details for $I^{\ell}_2$. By Cauchy-Schwarz's inequality, we have
\begin{equation}
I^{\ell}_2\leq \bigg(\int_{\mathbb{R}^d\backslash \bb_{\ell}(x)}e^{\frac{\vert x-z\vert}{2c\ell}}\fint_{\bb_{\ell}(z)}\bigg\vert\int_{0}^1\nabla\delta^h u(s,y)\dd s\bigg\vert^2\dd y\,\dd z\bigg)^{\frac{1}{2}}\bigg(\int_{\mathbb{R}^d\backslash \bb_{\ell}(x)}e^{-\frac{\vert x-z\vert}{2c\ell}}\fint_{\bb_{\ell}(z)}\vert g_r(y)\vert^2\dd y\,\dd z\bigg)^{\frac{1}{2}}.\label{arg1lem7}
\end{equation}
It remains to estimate the first r.h.s factor of \eqref{arg1lem7}. First, by the localized energy estimate \eqref{LemE3} applied to the equation \eqref{equationdeltaulem6} and $\bigg(\int_{\bb_{\ell}(x)}\vert\delta a(y) e\vert^2\dd y\bigg)^{\frac{1}{2}}\lesssim \ell^{\frac{d}{2}}$, we obtain
\begin{align}
\bigg(\int_{\mathbb{R}^d\backslash \bb_{\ell}(x)}e^{\frac{\vert x-z\vert}{2c\ell}}\fint_{\bb_{\ell}(z)}\bigg\vert\int_{0}^1\nabla\delta^h u(s,y)\dd s\bigg\vert^2\dd y\,\dd z\bigg)^{\frac{1}{2}}\lesssim &\, \ell^{\frac{d}{2}}+\bigg(\int_{\bb_{\ell}(x)}\bigg\vert\int_{0}^1\delta a(y)\nabla u(a+h\delta a,s,y)\dd s\bigg\vert^2\dd y\bigg)^{\frac{1}{2}}\nonumber\\
&+\int_{0}^1\frac{1}{1-t}\int_{t}^1\bigg(\int_{\bb_{\ell}(x)}\vert \delta a(y)\nabla u(a+h\delta a,s,y)\vert^2\dd y\bigg)^{\frac{1}{2}}\dd s\, \dd t.\label{locapertubu2}
\end{align}
Second, by the localized energy estimate \eqref{LemE1} applied to the equation \eqref{equationu} with $a$ replaced by $a+h\delta a$ and for $R=\ell\geq 1$, we obtain (since $\mathds{1}_{\bb_{\ell}(x)}(y)\lesssim e^{-\frac{\vert y-x\vert}{\ell}}$)
\begin{equation}
\bigg(\int_{\bb_{\ell}(x)}\bigg\vert\int_{0}^1\nabla u(a+h\delta a,t,y)\dd t\bigg\vert^2\dd y\bigg)^{\frac{1}{2}}\lesssim \ell^{\frac{d}{2}},
\label{energylem668}
\end{equation}
and 
\begin{equation}
\int_{0}^1\frac{1}{1-t}\int_{t}^1\bigg(\int_{\bb_{\ell}(x)}\vert\nabla u(a+h\delta a,s,y)\vert^2\dd y\bigg)^{\frac{1}{2}}\dd s\,\dd t\lesssim \ell^{\frac{d}{2}}\int_{0}^1\frac{1}{1-t}\int_{t}^1s^{-1}\dd s\, \dd t=\ell^{\frac{d}{2}}\int_{0}^1\frac{-\log(t)}{1-t}\dd t\lesssim \ell^{\frac{d}{2}}.
\label{locaperturbu}
\end{equation}
Finally the combination of \eqref{arg1lem7}, \eqref{locapertubu2}, \eqref{energylem668} and \eqref{locaperturbu} yields \eqref{I1I2}. It then follows from \eqref{lem64} that
\begin{equation}
\bigg\vert\int_{\mathbb{R}^d} g_r(y)e_1\cdot a(y)\bigg(\int_{0}^1\nabla \delta^h u(t,y)\dd t\bigg) \dd y\bigg\vert\lesssim\ell^{\frac{d}{2}}\bigg(\int_{\mathbb{R}^d} e^{-\frac{\vert x-z\vert}{2c\ell}}\fint_{\bb_{\ell}(z)}\vert g_r(y)\vert^2\dd y\,\dd z\bigg)^{\frac{1}{2}}.
\label{near0rhslem72}
\end{equation}
\textbf{Regime $\ell\geq\sqrt{T}$. }Using Cauchy-Schwarz's inequality and by dominating the Gaussian kernel $g_r$ by the exponential kernel $\eta_r$, we have
\begin{align}
\bigg\vert\int_{\mathbb{R}^d} g_r(y)e_1\cdot a(y)\bigg(\int_{0}^1\nabla \delta^h u(t,y)\dd t\bigg)\dd y\bigg\vert\leq&  \bigg(\int_{\mathbb{R}^d}\eta_r(y)\bigg\vert\int_{0}^1\nabla \delta^h u(t,y)\dd t\bigg\vert^2\dd y\bigg)^{\frac{1}{2}},\label{lem652prime}
\end{align}
with, by applying the localized energy estimates \eqref{LemE1} and \eqref{LemE2} to \eqref{equationdeltaulem6} for $R=r$
\begin{equation}\label{lem652}
\begin{aligned}
\bigg(\int_{\mathbb{R}^d}\eta_r(y)\bigg\vert\int_{0}^1\nabla \delta^h u(t,y)\dd t\bigg\vert^2\dd y\bigg)^{\frac{1}{2}}\lesssim&
\bigg(\int_{\mathbb{R}^d}\eta_r(y)\vert \delta a(y)\vert^2\dd y\bigg)^{\frac{1}{2}}+\bigg(\int_{\mathbb{R}^d}\eta_r(y)\vert \delta a(y)\vert^2\bigg\vert\int_{0}^1\nabla u(a+h\delta a,t,y)\dd t\bigg\vert^2\dd y\bigg)^{\frac{1}{2}}\\
&+\int_{0}^1\frac{1}{1-t}\int_{t}^1\bigg(\int_{\mathbb{R}^d}\eta_r(y)\vert \delta a(y)\vert^2\vert \nabla u(a+h\delta a,s,y)\vert^2\dd y\bigg)^{\frac{1}{2}}\dd s\, \dd t\bigg).
\end{aligned}
\end{equation}
Now, since $\delta a$ is supported in $\bb_{\ell}(x)$, the localized energy estimates \eqref{LemE2} and \eqref{LemE3} applied to the equation \eqref{equationdeltaulem6} with $R=\ell$ combined with \eqref{energylem668} and \eqref{locaperturbu} yield
\begin{equation}
\int_{\bb_{\ell}(x)}\bigg\vert\int_{0}^1\nabla u(a+h\delta a,t,y)-\nabla u(a,t,y)\dd t\bigg\vert^2\dd y+t^2\Big(\frac{\sqrt{t}}{\ell}\vee 1\Big)^{-d}\int_{\bb_{\ell}(x)}\vert\nabla u(a+h\delta a,t,y)-\nabla u(a,t,y)\vert^2\dd y\lesssim h^2\ell^{d},
\label{lem651}
\end{equation}
which allow us to pass to the limit when $h\downarrow 0$ in \eqref{lem652prime} and \eqref{lem652}, and obtain
$$\limsup_{h\downarrow 0}\bigg\vert\int_{\mathbb{R}^d} g_r(y)e_1\cdot a(y)\bigg(\int_{0}^1\nabla \delta^h u(t,y)\dd t\bigg)\dd y\bigg\vert\lesssim\mathcal{T}_{x,\ell}(\eta_r)(0).$$
This concludes the argument for \eqref{near0lem71}.\newline
\newline
\textbf{Step 2. Second r.h.s term of \eqref{splitlem6}. }We prove that
\begin{equation}\label{far0rhslem7}
\begin{aligned}
&\limsup_{h\downarrow 0}\bigg\vert \int_{1}^T\int_{\mathbb{R}^d} g_r(y)e_1\cdot a(y)\nabla\delta^h u(t,y)\dd y\, \dd t\bigg\vert\\
\lesssim &
\int_{\bb_{\ell}(x)}\int_{1}^T\vert\nabla u(t,y)\vert\vert\nabla v^T(t,y)\vert\dd t\,\dd y+\int_{\bb_{\ell}(x)}\vert \nabla v^T(1,y)\vert\bigg(1+\bigg\vert\int_{0}^{1}\nabla u(s,y)\dd s\bigg\vert\bigg)\dd y\\
&+\mathcal{G}_{r,\ell}(x)\mathds{1}_{\ell\geq \sqrt{T}}+\mathds{1}_{\ell<\sqrt{T}}\,\ell^{\frac{d}{2}}\bigg(\int_{\mathbb{R}^d} e^{-\frac{\vert x-z\vert}{2c\ell}}\fint_{\bb_{\ell}(z)}\vert\nabla v^T(1,y)\vert^2\dd y\, \dd z\bigg)^{\frac{1}{2}}.
\end{aligned}
\end{equation}
Recall that $v^T$ denotes the weak solution of the dual system associated with \eqref{equationdeltaulem6}, which reads
\begin{equation}
\left\{
    \begin{array}{ll}
        \partial_{\tau}v^T+\nabla\cdot a^*\nabla v^T=\nabla\cdot ag_re_1 & \text{ on $(-\infty,T)\times \mathbb{R}^d$}, \\
        v^T(T)= 0. & 
    \end{array}
\right.
\label{equationvTlem6}
\end{equation}
We reformulate the l.h.s of \eqref{far0rhslem7} with help of the dual system \eqref{equationvTlem6}. We have by testing the equation \eqref{equationvTlem6} with $\delta^h u$ and integrating in time 
\begin{equation}
\int_{1}^T\int_{\mathbb{R}^d}\delta^h u(t,y)\partial_{t}v^T(t,y)\dd y\,\dd t-\int_{1}^T\int_{\mathbb{R}^d}\nabla \delta^h u(t,y)\cdot a^*(y) \nabla v^T(t,y)\dd y\, \dd t=-\int_{1}^T\int_{\mathbb{R}^d} g_r(y)e_1\cdot a(y)\nabla\delta^h u(t,y)\dd y\, \dd t,
\label{testlem61}
\end{equation}
and also, by testing \eqref{equationdeltaulem6} with $v^T$
\begin{equation}
\int_{1}^T\int_{\mathbb{R}^d} v^T(t,y)\partial_{t}\delta^h u(t,y)\dd y\,\dd t+\int_{1}^T\int_{\mathbb{R}^d}\nabla v^T(t,y) \cdot a(y) \nabla \delta^h u(t,y) \dd y\, \dd t=-\int_{1}^T\int_{\mathbb{R}^d}\nabla v^T(t,y)\cdot \delta a(y)\nabla u(a+h\delta a,t,y)\dd y\, \dd t.
\label{testlem62}
\end{equation}
Consequently, by summing \eqref{testlem61} and \eqref{testlem62}, using an integration by part in time and the fact that $v^T(T,\cdot)\equiv 0$, we get 
\begin{equation}
\int_{1}^T\int_{\mathbb{R}^d} g_r(y)e_1\cdot a(y)\nabla\delta^h u(t,y)\dd y\, \dd t=\int_{\mathbb{R}^d}\delta^h u(1,y)v^T(1,y)\dd y-\int_{1}^T\int_{\mathbb{R}^d}\nabla v^T(t,y)\cdot \delta a(y)\nabla u(a+h\delta a,t,y)\dd y\, \dd t.
\label{lem63}
\end{equation}
Moreover, from \eqref{lem651} we can pass to the limit when $h\downarrow 0$ in the second r.h.s term of \eqref{lem63}, namely
$$\int_{1}^T\int_{\mathbb{R}^d}\nabla v^T(t,y)\cdot \delta a(y)\nabla u(a+h\delta a,t,y)\dd y\, \dd t\underset{h\downarrow 0}{\rightarrow}\int_{1}^T\int_{\mathbb{R}^d}\nabla v^T(t,y)\cdot \delta a(y)\nabla u(t,y)\dd y\, \dd t,$$
and obtain the first r.h.s term of \eqref{far0rhslem7}. It remains to control the first r.h.s term of \eqref{lem63}. To this aim, we integrate in time the equation \eqref{equationdeltaulem6} between $0$ and $1$:
$$\delta^hu(1,\cdot)-\nabla\cdot a\int_{0}^1\nabla\delta^h u(t,\cdot)\dd t=\nabla\cdot \delta a\int_{0}^1\nabla u(a+h\delta,t,\cdot)\dd t+\nabla\cdot \delta ae,$$
which provides by testing with $v^T(1,\cdot)$ 
\begin{equation}\label{int01dual}
\begin{aligned}
\int_{\mathbb{R}^d}\delta^h u(1,y)v^T(1,y)\dd y=&\int_{\mathbb{R}^d}\nabla v^T(1,y)\cdot\delta a(y)\bigg(\int_{0}^1\nabla u(a+h\delta a,t,y)\dd t\bigg)\dd y\\
&-\int_{\mathbb{R}^d}\nabla v^T(1,y)\cdot a(y)\bigg(\int_{0}^1\nabla \delta^h u(t,y)\dd t\bigg)\dd y-\int_{\mathbb{R}^d}\nabla v^T(1,y)\cdot \delta a(y)\,e\,\dd y.
\end{aligned}
\end{equation}
The first and the third r.h.s terms of \eqref{int01dual} combined with \eqref{lem651} give the second r.h.s term of \eqref{far0rhslem7}. The second r.h.s term of \eqref{int01dual} is then dominated in two ways, depending on the regime in $\ell$.\newline
\newline
\textbf{Regime $\ell\geq\sqrt{T}$. }For the generic case $\ell\geq r_*(0)$, we use the identity $\int=\int_{\mathbb{R}^d}\fint_{\bb_{\ell}(y)}$ and we split the integral into the two contributions $\int_{\mathbb{R}^d\backslash \bb_{4\ell}}$ and $\int_{\bb_{4\ell}}$ in form of
\begin{align*}
\bigg\vert \int_{\mathbb{R}^d}\nabla v^T(1,y)\cdot a(y)\bigg(\int_{0}^1\nabla \delta^h u(t,y)\dd t\bigg)\dd y\bigg\vert&\leq \int_{\mathbb{R}^d}\fint_{\bb_{\ell}(y)}\vert \nabla v^T(1,z)\vert \bigg\vert\int_{0}^1 \nabla\delta^h u(t,z)\dd t\bigg\vert\dd z\, \dd y \\
&=\bigg(\int_{\mathbb{R}^d\backslash \bb_{4\ell}}+\int_{\bb_{4\ell}}\bigg)\fint_{\bb_{\ell}(y)}\vert \nabla v^T(1,z)\vert \bigg\vert\int_{0}^1 \nabla\delta^h u(t,z)\dd t\bigg\vert\dd z\, \dd y.
\end{align*}
For the far-field contribution $\vert y\vert\geq 4\ell$, we use of Cauchy-Schwarz's inequality and the computations done in \eqref{lem652} as well as \eqref{lem651} to get
\begin{equation*}
\limsup_{h\downarrow 0} \int_{\mathbb{R}^d\backslash \bb_{4\ell}}\fint_{\bb_{\ell}(y)}\vert \nabla v^T(1,z)\vert \bigg\vert\int_{0}^1 \nabla\delta^h u(t,z)\dd t\bigg\vert\dd z\, \dd y\lesssim \int_{\mathbb{R}^d\backslash \bb_{4\ell}}\bigg(\fint_{\bb_{\ell}(y)}\vert\nabla v^T(1,z)\vert^2\dd z\bigg)^{\frac{1}{2}}\mathcal{T}_{x,\ell}(\eta_{\ell})(y)\dd y.
\label{derivothertestesti50}
\end{equation*}
For the near-field contribution $\vert y\vert<4\ell$, we first note that from the assumption $\ell>r_*(0)$ we have for all $y\in \mathbb{R}^d$
\begin{equation}
\bb_{r_*(y)}(y)\cap \bb_{5\ell}\neq \emptyset \quad\Rightarrow\quad y\in \bb_{7\ell}.
\label{proofsubopesti7}
\end{equation}
Indeed, from the $\frac{1}{8}$-Lipschitz regularity property of $r_*$, if there exists $z\in \bb_{r_*(y)}(y)\cap \bb_{5\ell}$ then $\vert y\vert\leq \vert y-z\vert+\vert z\vert\leq r_*(y)+5\ell\leq r_*(0)+\frac{\vert y\vert}{8}+5\ell\leq 6\ell+\frac{\vert y\vert}{8}$ and thus $\vert y\vert\leq \frac{16}{3}\ell\leq 7\ell$. 
Therefore, using the property \eqref{averager*ctrl} combined with Cauchy-Schwarz's inequality, we get
\begin{align*}
\int_{\bb_{4\ell}}\fint_{\bb_{\ell}(y)}\vert \nabla v^T(1,z)\vert \bigg\vert\int_{0}^1 \nabla\delta^h u(t,z)\dd t\bigg\vert\dd z\, \dd y&\leq \int_{\bb_{5\ell}}\vert \nabla v^T(1,z)\vert \bigg\vert\int_{0}^1 \nabla\delta^h u(t,z)\dd t\bigg\vert\dd z\\
&\stackrel{\eqref{averager*ctrl}}{\lesssim}\int_{\mathbb{R}^d}\fint_{\bb_{r_*(y)}(y)}\vert \nabla v^T(1,z)\vert \bigg\vert\int_{0}^1 \nabla\delta^h u(t,z)\dd t\bigg\vert\mathds{1}_{\bb_{5\ell}}(z)\dd z\, \dd y\\
&\stackrel{\eqref{proofsubopesti7}}{\leq}\int_{\bb_{7\ell}}\bigg(\fint_{\bb_{r_*(y)}(y)}\vert \nabla v^T(1,z)\vert^2\dd z\bigg)^{\frac{1}{2}} \bigg(\fint_{\bb_{r_*(y)}(y)}\bigg\vert\int_{0}^1 \nabla\delta^h u(t,z)\dd t\bigg\vert^2\dd z\bigg)^{\frac{1}{2}}\dd y,
\end{align*}
and we finally end up with, using \eqref{lem652} as well as \eqref{lem651}
$$\limsup_{h\downarrow 0}\int_{\bb_{4\ell}}\fint_{\bb_{\ell}(y)}\vert \nabla v^T(1,z)\vert \bigg\vert\int_{0}^1 \nabla\delta^h u(t,z)\dd t\bigg\vert\dd z\, \dd y\lesssim \int_{\bb_{7\ell}}\bigg(\fint_{\bb_{r_*(y)}(y)}\vert \nabla v^T(1,z)\vert^2\dd z\bigg)^{\frac{1}{2}} \mathcal{T}_{x,\ell}(\eta_{r_*(y)})(y)\dd y.$$
For the non-generic regime $\ell< r_*(0)$, we use the estimate \eqref{near0rhslem72} which clearly holds by replacing $g_r$ with $v^T(1,\cdot)$ and we bound $\ell^{\frac{d}{2}}$ by $r^{\frac{d}{2}}_*(0)$.\newline
\newline
\textbf{Regime $\ell<\sqrt{T}$. }We use the estimate \eqref{near0rhslem72} which clearly holds by replacing $g_r$ with $v^T(1,\cdot)$.
\end{proof}
\noindent We now turn to the proof of Lemma \ref{functioderiv2}.
\begin{proof}[Proof of Lemma \ref{functioderiv2}.]We keep the notations of the previous proof and we give only the argument for $q(r^2)\star f_r\cdot e_1$. First, equality \eqref{lem62} holds and take the form
\begin{equation}
\delta^h q(r^2)\star f_r\cdot e_1=\int_{\mathbb{R}^d}f_r(y)e_1\cdot \delta a(y) e\dd y+\int_{\mathbb{R}^d} f_r(y)e_1\cdot \delta a(y)\left(\int_{0}^{r^2}\nabla u(t,y)\dd t\right)\dd y+\int_{0}^{r^2}\int_{\mathbb{R}^d}f_r(y)e_1\cdot a(y)\nabla\delta^h u(t,y)\dd y\, \dd t.
\label{derivothertestesti1}
\end{equation}
The first two r.h.s terms of \eqref{derivothertestesti1} give directly the first two r.h.s terms of \eqref{functioderivothertest}, respectively. As in \eqref{splitlem6}, we make the decomposition
\begin{equation}
\int_{0}^{r^2}\int_{\mathbb{R}^d}f_r(y)e_1\cdot a(y)\nabla\delta^h u(t,y)\dd y\, \dd t=\int_{\mathbb{R}^d} f_r(y)e_1\cdot a(y)\left(\int_{0}^1\nabla\delta^h u(t,y)\dd t\right)\dd y+\int_{1}^{r^2}\int_{\mathbb{R}^d}f_r(y)e_1\cdot a(y)\nabla\delta^h u(t,y)\dd y\, \dd t.
\label{derivothertestesti2}
\end{equation}
\textbf{Control of the second r.h.s term of \eqref{derivothertestesti2}. }This term is controlled the same way as we did in Step $2$ of the proof of Lemma \ref{functioderiv} and provides the third, the fifth and the sixth r.h.s term of \eqref{functioderivothertest}.\newline
\newline
\textbf{Control of the first r.h.s term of \eqref{derivothertestesti2}. }It remains to argue that the first r.h.s term of \eqref{derivothertestesti2} is dominated by the fourth r.h.s term of \eqref{functioderivothertest}. We distinguish the two regimes $\ell\geq r$ and $\ell\leq r$.\newline
\newline
\textbf{Regime $\ell\geq r$. }We use the assumption \eqref{assumeothertest} in form of, for all $y\in\mathbb{R}^d$, $\vert f_r(y)\vert\lesssim \int_{1}^{r^2}s^{-\frac{1}{2}}g_{\sqrt{2s}}(y)\dd s$ and \eqref{lem652} applied with $r=\sqrt{s}$ to get
\begin{align*}
\limsup_{h\downarrow 0}\left\vert\int_{\mathbb{R}^d} f_r(y)e_1\cdot a(y)\left(\int_{0}^1\nabla\delta^h u(t,y)\dd t\right)\dd y\right\vert&\lesssim \limsup_{h\downarrow 0}\int_{1}^{r^2}s^{-\frac{1}{2}}\int_{\mathbb{R}^d}g_{\sqrt{2s}}(y)\left\vert\int_{0}^1\nabla\delta^h u(t,y)\dd t\right\vert\dd y\\
&\lesssim  \int_{1}^{r^2}s^{-\frac{1}{2}}\mathcal{T}_{x,\ell}(\eta_{\sqrt{s}})(0)\dd s.
\end{align*}
\textbf{Regime $\ell < r$. }We first note that from the assumption \eqref{assumeothertest} we have, for all $y\in\mathbb{R}^d$
\begin{equation}
\vert f_r(y)\vert\lesssim \vert y\vert \int_{1}^{r^2} s^{-1-\frac{d}{2}}(1+\frac{\vert y\vert^2}{s})^{-\frac{d}{2}-\frac{3}{2}} \dd s\lesssim \frac{r}{(\vert y\vert+1)^{d}}\wedge \frac{1}{(\vert y\vert+1)^{d-1}},
\label{derivothertestesti1prime}
\end{equation}
We now make use of the identity $\int_{\mathbb{R}^d}=\int_{\mathbb{R}^d}\fint_{\bb_{\ell}(y)}$ to get
$$\left\vert\int_{\mathbb{R}^d} f_r(y)e_1\cdot a(y)\left(\int_{0}^1\nabla\delta^h u(t,y)\dd t\right)\dd y\right\vert\lesssim \int_{\mathbb{R}^d}\fint_{\bb_{\ell}(y)}\vert f_r(z)\vert\left\vert\int_{0}^1\nabla\delta^h u(t,z)\dd t\right\vert\dd z\,\dd y.$$
We then split the integral into the far-field contribution $\vert y\vert\geq 2\ell$ and the near-field contribution $\vert y\vert<2\ell$. For the near-field contribution, we make use of a dyadic decomposition and \eqref{derivothertestesti1prime} in form of $\vert f_r(y)\vert\lesssim (\vert y\vert+1)^{-d+1}$ to get
\begin{align*}
\int_{\bb_{2\ell}}\fint_{\bb_{\ell}(y)}\vert f_r(z)\vert\left\vert\int_{0}^1\nabla\delta^h u(t,z)\dd t\right\vert\dd z\,\dd y&\lesssim \int_{\bb_{3\ell}}\vert f_r(y)\vert\left\vert\int_{0}^1\nabla\delta^h u(t,y)\dd t\right\vert\dd y\\
&\lesssim \int_{\bb_1}\left\vert\int_{0}^1\nabla\delta^h u(t,y)\dd t\right\vert\dd y+\sum_{n=0}^{\left\lceil \log_2(3\ell)\right\rceil}\int_{\bb_{2^{n+1}}\backslash \bb_{2^n}}\vert f_r(y)\vert\left\vert\int_{0}^1\nabla\delta^h u(t,y)\dd t\right\vert\dd y\\
&\stackrel{\eqref{derivothertestesti1prime}}{\lesssim} \int_{\bb_1}\left\vert\int_{0}^1\nabla\delta^h u(t,y)\dd t\right\vert\dd y+\sum_{n=0}^{\left\lceil \log_2(3\ell)\right\rceil}2^n\left(\fint_{\bb_{2^{n+1}}}\left\vert\int_{0}^1\nabla\delta^h u(t,y)\dd t\right\vert^2\dd y\right)^{\frac{1}{2}}
\end{align*}
which gives the first term in the second r.h.s term of \eqref{Kothertest} by passing to the $\limsup$ and combining \eqref{lem652}, applied both with $R=1$ and $R=2^{n+1}$, with \eqref{lem651}. For the far-field contribution, we make use of Cauchy-Schwarz's inequality in form of 
$$\int_{\mathbb{R}^d\backslash \bb_{2\ell}} \fint_{\bb_{\ell}(y)}\vert f_r(z)\vert\left\vert\int_{0}^1\nabla\delta^h u(t,z)\dd t\right\vert\dd z\,\dd y\lesssim \int_{\mathbb{R}^d\backslash \bb_{2\ell}} \left(\fint_{\bb_{\ell}(y)}\vert f_r(z)\vert^2\dd z\right)^{\frac{1}{2}}\left(\fint_{\bb_{\ell}(y)}\left\vert\int_{0}^1\nabla\delta^h u(t,z)\dd t\right\vert^2\dd z\right)^{\frac{1}{2}}\dd y,$$
and we get the second term in the second r.h.s term of \eqref{Kothertest} by passing to the $\limsup$ and combining \eqref{lem652}, applied with $R=\ell$, with \eqref{lem651}.

\end{proof}
\subsubsection{Proof of Lemma \ref{near0estiq}}
We split the proof into two steps. In the first step, we prove a deterministic bound on $q_r(T,x)$, using the energy estimates of Lemma \ref{nergyestideter} and the control of averages of $\nabla u$ deduced from Lemma \ref{ctrlav}. The deterministic bound will depend on a random variable built from an average of $r_*$. In the second step, we prove that the random constant has stretched exponential moments, using the moment bound \eqref{momentr*} on $r_*$.\newline
\newline
\textbf{Step 1. } Proof that for all $r\leq \min\{1,\sqrt{T}\}$ and $x\in\mathbb{R}^d$
\begin{equation}
\vert q_r(T,x)\vert \lesssim 1+r^{-\frac{d}{2}}\log(\tfrac{\sqrt{T}}{r})\left(1+r_*^{\frac{d}{2}}(0)+\int_{\mathbb{R}^d\backslash \bb_1(x)}\eta_1(y)r^{\frac{d}{2}}_*(ry+x)\dd y\right).
\label{near0esti1}
\end{equation}
Without loss of generality, we may assume that $x=0$. We use the definition \eqref{defq} of $q_r(T)$ followed by the triangle inequality, the domination of the Gaussian kernel $g_r$ by the exponential kernel $\eta_r$ and the continuous embedding $\LL^2(\mathbb{R}^d,\eta_r\dd x)\hookrightarrow \LL^1(\mathbb{R}^d,\eta_r \dd x)$ to obtain
\begin{align}
\vert q_r(T)\vert&\stackrel{\eqref{defq}}{\lesssim} \int_{\mathbb{R}^d}g_r(y)\dd y+\int_{\mathbb{R}^d}g_r(y)\left\vert \int_{0}^T\nabla u(s,y)\dd s\right\vert\dd y\nonumber\\
&\lesssim 1+\int_{\mathbb{R}^d}\eta_r(y)\left\vert \int_{0}^{r^2}\nabla u(s,y)\dd s\right\vert\dd y+\int_{\mathbb{R}^d}\eta_r(y)\left\vert \int_{r^2}^T\nabla u(s,y)\dd s\right\vert\dd y\nonumber\\
&\lesssim 1+\left(\int_{\mathbb{R}^d}\eta_r(y)\left\vert \int_{0}^{r^2}\nabla u(s,y)\dd s\right\vert^2\dd y\right)^{\frac{1}{2}}+\int_{\mathbb{R}^d}\eta_r(y)\left\vert \int_{r^2}^T\nabla u(s,y)\dd s\right\vert\dd y.\label{near0esti2}
\end{align}
For the first r.h.s term of \eqref{near0esti2}, we use the energy estimate \eqref{LemE1} in form of
$$\left(\int_{\mathbb{R}^d}\eta_r(y)\left\vert \int_{0}^{r^2}\nabla u(s,y)\dd s\right\vert^2\dd y\right)^{\frac{1}{2}}\lesssim 1.$$
For the second r.h.s term of \eqref{near0esti2}, we write with $\int_{\mathbb{R}^d}=\int_{\mathbb{R}^d}\fint_{\bb_r(x)}$ and Cauchy-Schwarz's inequality,
\begin{equation}
\int_{\mathbb{R}^d}\eta_r(y)\left\vert\int_{r^2}^T\nabla u(s,y)\dd s\right\vert\dd y\lesssim \int_{\mathbb{R}^d}\left(\fint_{\bb_r(x)}\eta^2_r(y)\dd y\right)^{\frac{1}{2}}\left(\fint_{\bb_r(x)}\left\vert \int_{r^2}^T\nabla u(s,y)\dd s\right\vert^2\dd y\right)^{\frac{1}{2}}\dd x.
\label{near0esti5}
\end{equation}
Using the energy estimate \eqref{LemE1} applied to the equation \eqref{equationu} in form of, for all $T>0$
$$\left(\fint_{\bb_{\sqrt{T}}(x)}\vert\nabla u(T,y)\vert^2\dd y\right)^{\frac{1}{2}}\lesssim T^{-1},$$
we deduce from Minkowski's inequality in $\LL^2(\bb_r(x))$ and Lemma \ref{ctrlav} applied with $R=1$, $f\equiv 1$ and $g:T\in\mathbb{R}^+\mapsto T^{-1}$, that, since $r\leq 1$
$$\left(\fint_{\bb_r(x)}\left\vert \int_{r^2}^T\nabla u(s,y)\dd s\right\vert^2\dd y\right)^{\frac{1}{2}}\lesssim r^{-\frac{d}{2}}\int_{r^2}^T\left(\int_{\bb_1(x)}\vert\nabla u(s,y)\vert^2\dd y\right)^{\frac{1}{2}}\dd s\lesssim r^{-\frac{d}{2}}r^{\frac{d}{2}}_*(x)\log(\tfrac{\sqrt{T}}{r}).$$
Consequently \eqref{near0esti5} turns into
\begin{equation*}
\int_{\mathbb{R}^d}\eta_r(y)\left\vert\int_{r^2}^T\nabla u(s,y)\dd s\right\vert\dd y\lesssim r^{-\frac{d}{2}}\log(\tfrac{\sqrt{T}}{r})\int_{\mathbb{R}^d}r^{\frac{d}{2}}_*(x)\left(\fint_{\bb_r(x)}\eta^2_r(y)\dd y\right)^{\frac{1}{2}}\dd x.
\end{equation*}
It remains to show that
\begin{equation}
\int_{\mathbb{R}^d}r^{\frac{d}{2}}_*(x)\left(\fint_{\bb_r(x)}\eta^2_r(y)\dd y\right)^{\frac{1}{2}}\dd x\lesssim 1+r_*^{\frac{d}{2}}(0)+\int_{\mathbb{R}^d\backslash \bb_1}\eta_1(x)r^{\frac{d}{2}}_*(rx)\dd x.
\label{near0esti3}
\end{equation}
For this, we decompose $\mathbb{R}^d$ into $\bb_r$ and $\mathbb{R}^d\backslash \bb_r$ to the effect of
\begin{equation}
\int_{\mathbb{R}^d}r^{\frac{d}{2}}_*(x)\left(\fint_{\bb_r(x)}\eta^2_r(y)\dd y\right)^{\frac{1}{2}}\dd x=\int_{\bb_r}r^{\frac{d}{2}}_*(x)\left(\fint_{\bb_r(x)}\eta^2_r(y)\dd y\right)^{\frac{1}{2}}\dd x+\int_{\mathbb{R}^d\backslash \bb_r}r^{\frac{d}{2}}_*(x)\left(\fint_{\bb_r(x)}\eta^2_r(y)\dd y\right)^{\frac{1}{2}}\dd x.
\label{near0esti6}
\end{equation}
The first r.h.s term of \eqref{near0esti6} is controlled by, using the $\frac{1}{8}$-Lipschitz regularity of $r_*$ in form of $\sup_{x\in \bb_r} r_*(x)\lesssim r_*(0)+r\lesssim r_*(0)+1$ and Cauchy-Schwarz's inequality
$$\int_{\bb_r}r^{\frac{d}{2}}_*(x)\left(\fint_{\bb_r(x)}\eta^2_r(y)\dd y\right)^{\frac{1}{2}}\dd x\lesssim (r^{\frac{d}{2}}_*(0)+1)r^{\frac{d}{2}}\left(\int_{\mathbb{R}^d}\eta^2_r(y)\dd y\right)^{\frac{1}{2}}\lesssim r^{\frac{d}{2}}_*(0)+1.$$
For the second r.h.s term of \eqref{near0esti6}, we note that for all $x\in \mathbb{R}^d\backslash \bb_r$ and $y\in \bb_r(x)$ we have $\eta_r(y)\lesssim \eta_r(x)$, so that
$$\int_{\mathbb{R}^d\backslash \bb_r}r^{\frac{d}{2}}_*(x)\left(\fint_{\bb_r(x)}\eta^2_r(y)\dd y\right)^{\frac{1}{2}}\dd x \lesssim \int_{\mathbb{R}^d\backslash \bb_r}r^{\frac{d}{2}}_*(x)\eta_r(x)\dd x=\int_{\mathbb{R}^d\backslash \bb_1}r^{\frac{d}{2}}_*(rx)\eta_1(x)\dd x.$$
This concludes the proof of \eqref{near0esti3} and the argument for \eqref{near0esti1}.\newline
\newline
\textbf{Step 2.} We prove \eqref{near0est4}. It remains to show that $\mathcal{C}_\star(r,x):=1+r^{\frac{d}{2}}_*(0)+\int_{\mathbb{R}^d\backslash \bb_1(x)}\eta_1(y)r^{\frac{d}{2}}_*(ry+x)\dd y$ has the desired stretched exponential moment \eqref{rajoutfinale1}. From the moment bound \eqref{momentr*} of $r_*$ and the equivalence between algebraic moments and exponential moments (see Lemma \ref{momentexp}), we have\footnote{In the case $\beta=d$, Lemma \ref{momentexp} gives $\langle(\tfrac{r^d_*(y)}{\log(r_*(y))})^p\rangle^\frac{1}{p}\lesssim p$. This implies the loss in $p$ in \eqref{momentr*lem1} that we choose to write in terms of powers of $p$ for simplicity.} for any $\gamma>0$
\begin{equation}
\left\langle r^{dp}_*(y)\right\rangle^{\frac{1}{p}}\lesssim_\gamma p^{\frac{d}{\beta \wedge d}}\mathds{1}_{\beta\neq d}+p^{1+\gamma}\mathds{1}_{d=\beta} \quad \text{for all $p\in [1,\infty)$\quad and\quad $y\in\mathbb{R}^d$}.
\label{momentr*lem1}
\end{equation}
Therefore,
\begin{align*}
\left\langle \left(1+r^{\frac{d}{2}}_*(0)+\int_{\mathbb{R}^d\backslash \bb_1(x)}\eta_1(y)r^{\frac{d}{2}}_*(ry+x)\dd y\right)^p\right\rangle^{\frac{1}{p}}&\lesssim 1+\left\langle r^{\frac{dp}{2}}_*(0)\right\rangle^{\frac{1}{p}}+\int_{\mathbb{R}^d\backslash \bb_1(x)}\eta_1(y)\left\langle r^{\frac{dp}{2}}_*(ry+x)\right\rangle^{\frac{1}{p}}\dd y\\
&\stackrel{\eqref{momentr*lem1}}{\lesssim_\gamma} 1+(p^{\frac{1}{2}\frac{d}{\beta \wedge d}}\mathds{1}_{\beta\neq d}+p^{\frac{1}{2}(1+\gamma)}\mathds{1}_{d=\beta})\left(1+\int_{\mathbb{R}^d}\eta_1(y)\dd y\right), 
\end{align*}
which gives \eqref{near0est4} by applying once again Lemma \ref{momentexp}.
\subsubsection{Proof of Proposition \ref{Sensitivitysubop}: Suboptimal fluctuation estimates. }
We split the proof into three steps. In the first two steps we control $\int_{\mathbb{R}^d}\vert\partial^{\text{fct}}_{x,\ell}q_r(T)\vert^2\dd x$, using the bound on $\partial^{\text{fct}}_{x,\ell}q_r(T)$ proved in Lemma \ref{functioderiv}, treating separately the two regimes $\ell<\sqrt{T}$ and $\ell\geq \sqrt{T}$. In the last step we deduce the desired moment bound \eqref{subopsensiesti} from the multiscale logarithmic Sobolev inequality, in form of \eqref{SGinegp1}, and the moment bound \eqref{momentr*} on $r_*$. We start with preliminary estimates.
\newline
\newline
\textbf{Step 0. Preliminary. }First, we will use several times the following sub-optimal deterministic decay in time of averages of $\nabla u$: 
\begin{equation}
\int_{\bb_{\ell}(x)}\vert\nabla u(t,y)\vert^2\dd y\lesssim\ell^d\left(\left(\frac{r_*(x)}{\ell}\vee 1\right)^d\mathds{1}_{\ell<\sqrt{t}}+\mathds{1}_{\ell\geq \sqrt{t}}\right)t^{-2}\quad \text{for all $t\in (0,T]$}.
\label{decayuproofsubop}
\end{equation}
This estimate is a direct  consequence of the combination of the localized energy estimate \eqref{LemE1} applied to \eqref{equationu} with $R=\sqrt{t}$ and Lemma \ref{ctrlav} applied with $f\equiv 1$, $g: t\in\mathbb{R}^+\mapsto t^{-2}$ and $R=\ell$. The estimate \eqref{decayuproofsubop} is our starting point, as the role that \eqref{strategyunifboundnablau} played in the heuristic arguments of Section \ref{stratsec}.

\medskip

\noindent Second, we will use several times the following large-scale regularity result: for all $t\in (-\infty,T]$ and $x\in\mathbb{R}^d$
\begin{equation}
\fint_{\bb_{r_*(x)}(x)}\vert\nabla v^T(t,y)\vert^2\dd y\lesssim r^d_*(0)\frac{\log^2(1+\frac{\vert x\vert}{r})}{(\vert x\vert+r)^{2d}}+r^{-d}\eta_{4r_*(x)}(x)+g^2_{2r}(x),
\label{proofsubopesti1}
\end{equation}
where we recall that $v^T$ is defined in \eqref{dualequationlem7}. This estimate is a consequence of the localized energy estimate \eqref{LemE1} and the large-scale regularity estimate \eqref{lem2}. Indeed, we first notice that $v^{T}=(\tilde{v}_k(\cdot-T,\cdot))_{k\in\llbracket 1,d\rrbracket}$ with $\tilde{v}_k$ the weak solution of \eqref{equationvlem2} for $\overline{e}=e_k$ and $f_r=g_r$. In addition, from the identity \eqref{proofpointwise1} we get 
\begin{equation}
\nabla v^T(t,x)=\int_{0}^{T-t}\nabla w(s,x)\dd s\quad \text{for all $(t,x)\in (-\infty,T]\times\mathbb{R}^d$},
\label{proofsubopesti2}
\end{equation}
with $w=(w_k)_{k\in\llbracket 1,d\rrbracket}$ and $w_k$ solves \eqref{proofpointwise2} with $f_r=g_r$ and $\overline{e}=e_k$.
Therefore, in the regime $\sqrt{T-t}\geq 2r_*(x)$, we use \eqref{lem2} (where we bound $(\frac{r_*(0)}{r}\vee 1)^{d}\leq 2r_*(0)$ since $r\geq 1$) in form of 
\begin{equation*}
\fint_{\bb_{r_*(x)}(x)}\vert\nabla v^T(t,y)\vert^2\dd y=\sum_{k=1}^{d}\fint_{\bb_{r_*(x)}(x)}\vert\nabla\tilde{v}_k(t-T,y)\vert^2\dd y\lesssim r^{d}_*(0)\frac{\log^2(1+\frac{\vert x\vert}{r})}{(\vert x\vert+r)^d}.
\end{equation*}
In the regime $2r_*(x)\geq \sqrt{T-t}$, we use the formula \eqref{proofsubopesti2} combined with the localized energy estimate \eqref{LemE1} applied with $R=2r_*(x)$ in form of 
$$\fint_{\bb_{r_*(x)}(x)}\vert\nabla v^T(t,y)\vert^2\dd y\lesssim \fint_{\bb_{2r_*(x)}(x)}\left\vert\int_{0}^{T-t}\nabla w(s,y)\dd s\right\vert^2\dd y\lesssim \int_{\mathbb{R}^d}\eta_{2r_*(x)}(x-y)g^2_r(y)\dd y.$$
Finally the r.h.s is dominated via
$$\int_{\mathbb{R}^d}\eta_{2r_*(x)}(x-y)g^2_r(y)\dd y=\int_{\bb_{\frac{\vert x\vert}{2}}(x)}\eta_{2r_*(x)}(x-y)g^2_r(y)\dd y+\int_{\mathbb{R}^d\backslash \bb_{\frac{\vert x\vert}{2}}(x)}\eta_{2r_*(x)}(x-y)g^2_r(y)\dd y,$$
with
$$\int_{\bb_{\frac{\vert x\vert}{2}}(x)}\eta_{2r_*(x)}(x-y)g^2_r(y)\dd y\lesssim g^2_{2r}(x)\int_{\mathbb{R}^d}\eta_{2r_*(x)}(y)\dd y\lesssim g^2_{2r}(x),$$
and 
$$\int_{\mathbb{R}^d\backslash \bb_{\frac{\vert x\vert}{2}}(x)}\eta_{2r_*(x)}(x-y)g^2_r(y)\dd y\lesssim r^{-d}\eta_{4r_*(x)}(x)\int_{\mathbb{R}^d} g_{\frac{1}{\sqrt{2}}r}(y)\dd y\lesssim r^{-d}\eta_{4r_*(x)}(x).$$
We now turn to the proof of \eqref{subopsensiesti}. In the following, $p\in [1,\infty)$ denotes an arbitrary stochastic integrability exponent.
\newline
\newline
\textbf{Step 1. Regime $\ell< \sqrt{T}$. }
Proof that for all $\ell\in [1,\sqrt{T})$ and $\gamma<1$
\begin{equation}
\int_{\mathbb{R}^d}\vert\partial_{x,\ell}^{\text{fct}}q_r(T)\vert^2\dd x\lesssim\ell^{2d}r^{-d}(1+\log^2(T)+\log^2(\tfrac{\sqrt{T}}{r}))\mathcal{C}_\star(r,\ell)\quad \text{with}\quad \sup_{r,\ell\geq 1}\langle \mathcal{C}^p_\star(r,\ell)\rangle^{\frac{1}{p}}\lesssim_\gamma p^{2\frac{d}{\beta\wedge d}}\mathds{1}_{\beta\neq d}+p^{2(1+\gamma)}\mathds{1}_{\beta=d}.
\label{step1subopsensi}
\end{equation}
%
%
This estimate corresponds to \eqref{stratllesst} in the heuristic arguments of Section \ref{stratsec}. We split this step into two parts. The first part is devoted to the control of the first four r.h.s terms of \eqref{lem6} and the second part is devoted to the control of the last term. \newline
\newline
\textbf{Substep 1.1. }Proof that for all $\ell\in [1,\sqrt{T})$
\begin{equation}
\int_{\mathbb{R}^d}\left(\int_{\bb_{\ell}(x)}\left\vert\int_{0}^T\nabla u(t,y)\dd t\right\vert g_r(y)\dd y\right)^2\dd x\lesssim \ell^{2d}r^{-d}\left(1+\log^2(\tfrac{\sqrt{T}}{r})\left(\int_{\mathbb{R}^d}r^d_*(rx)g^2_1(x)\dd x\right)^{\frac{1}{2}}\right),
\label{lem65}
\end{equation}
and
\begin{equation}\label{lem66}
\begin{aligned}
\ell^{-d}\int_{\mathbb{R}^d}\left(\int_{\bb_{\ell}(x)} g_r(y)\dd y\right)^2\dd x&+\ell^{-d}\int_{\mathbb{R}^d}\left(\int_{\bb_{\ell}(x)}\vert\nabla v^T(1,y)\vert\left(1+\int_{0}^1\nabla u(s,y)\dd s\right)\dd y\right)^2\dd x\\
&+\int_{\mathbb{R}^d}\int_{\mathbb{R}^d}e^{-\frac{\vert x-z\vert}{2c\ell}}\fint_{\bb_{\ell}(z)}\vert g_r(y)\vert^2+\vert \nabla v^T(1,y)\vert^2\dd y\, \dd z\, \dd x\lesssim \ell^{d}r^{-d}.
\end{aligned}
\end{equation}
The estimate of the second l.h.s term of \eqref{lem66} corresponds to \eqref{stratproofesti2} in the heuristic arguments of Section \ref{stratsec}\newline
\newline
\textbf{Argument for \eqref{lem65}.} Since $r\leq \sqrt{T}$, we split $\int_0^T$ into $\int_{0}^{r^2}+\int_{r^2}^T$ and we apply the triangle inequality followed by Jensen's inequality, the identity $\int_{\mathbb{R}^d}\fint_{\bb_{\ell}(x)}=\int_{\mathbb{R}^d}$ as well as Cauchy-Schwarz's inequality to the effect of 
\begin{equation}\label{lem67}
\begin{aligned}
\left(\int_{\mathbb{R}^d}\left(\int_{\bb_{\ell}(x)}\left\vert\int_{0}^T\nabla u(t,y)\dd t\right\vert g_r(y)\dd y\right)^2\dd x\right)^{\frac{1}{2}}\lesssim& \ell^d\left(\int_{\mathbb{R}^d}\left\vert\int_{0}^{r^2}\nabla u(t,y)\dd t\right\vert^2 g^2_r(y)\dd y\right)^{\frac{1}{2}}\\
&+\left(\int_{\mathbb{R}^d}\Big(\int_{\bb_{\ell}(x)}\left\vert\int_{r^2}^T\nabla u(t,y)\dd t\right\vert^2\dd y\Big)\Big(\int_{\bb_{\ell}(x)} g^2_r(y)\,\dd y\Big)\,\dd x\right)^{\frac{1}{2}}.
\end{aligned}
\end{equation}
For the first r.h.s term of \eqref{lem67}, we dominate the Gaussian kernel $g_r$ by the exponential kernel $\eta_r$ and we use the localized energy estimate \eqref{LemE1} applied to \eqref{equationu} in form of
\begin{equation}
\left(\int_{\mathbb{R}^d}\left\vert\int_{0}^{r^2}\nabla u(t,y)\dd t\right\vert^2 g^2_r(y)\dd y\right)^{\frac{1}{2}}\lesssim r^{-\frac{d}{2}}\left(\int_{\mathbb{R}^d} \eta_{r}(y)\left\vert\int_{0}^{r^2}\nabla u(t,y)\dd t\right\vert^2\dd y\right)^{\frac{1}{2}}\stackrel{\eqref{LemE1}}{\lesssim} r^{-\frac{d}{2}}.
\label{lowerenergy1}
\end{equation}
For the second r.h.s term of \eqref{lem67}, since $\ell<\sqrt{T}$, we have by Minkowski's inequality in $\LL^2(\bb_{\ell}(x))$ and the decay estimate \eqref{decayuproofsubop}
$$\int_{\bb_{\ell}(x)}\left\vert\int_{r^2}^T\nabla u(t,y)\dd t\right\vert^2\dd y\leq \left(\int_{r^2}^T\left(\int_{\bb_{\ell}(x)}\vert \nabla u(t,y)\vert^2\dd y\right)^{\frac{1}{2}}\dd t\right)^2\stackrel{\eqref{decayuproofsubop}}{\lesssim} (r_*(x)\vee \ell)^d\log^2(\tfrac{\sqrt{T}}{r}),$$
so that we finally get, using in the last line the Lipschitz property of $r_*$ in form of $\sup_{y\in \bb_{\ell}(x)}\left(\frac{r_*(y)}{\ell}\vee 1\right)^d\lesssim \left(\frac{r_*(x)}{\ell}\vee 1\right)^d$ combined with the identity $\int_{\mathbb{R}^d}\fint_{\bb_{\ell}(x)}\dd x=\int_{\mathbb{R}^d}$ and change of variables $x\mapsto \frac{x}{r}$
\begin{align*}
\left(\int_{\mathbb{R}^d}\int_{\bb_{\ell}(x)}\left\vert\int_{r^2}^T\nabla u(t,y)\dd t\right\vert^2\dd y\int_{\bb_{\ell}(x)} g^2_r(y)\dd y\,\dd x\right)^{\frac{1}{2}}&\lesssim \log(\tfrac{\sqrt{T}}{r})\left(\int_{\mathbb{R}^d}(r_*(x)\vee \ell)^d\int_{\bb_\ell(x)}g^2_r(y)\dd y\, \dd x\right)^{\frac{1}{2}}\\
&\lesssim\ell^{d}r^{-\frac{d}{2}}\log(\tfrac{\sqrt{T}}{r})\left(\int_{\mathbb{R}^d}\left(\frac{r_*(rx)}{\ell}\vee 1\right)^d g^2_1(x)\dd x\right)^{\frac{1}{2}}.
\end{align*}
Estimate \eqref{lem65} then follows in combination with the previous estimate, \eqref{lem67} and \eqref{lowerenergy1} as well as $\frac{r_*(rx)}{\ell}\vee 1\lesssim r_*(rx)$.\newline
\newline
\textbf{Argument for \eqref{lem66}.} On the one hand, we have from Cauchy-Schwarz's inequality, the identity $\int_{\mathbb{R}^d}\fint_{\bb_{\ell}(x)}=\int_{\mathbb{R}^d}$ and $\int_{\mathbb{R}^d} g^2_r(y)\dd y\lesssim r^{-d}$ :
\begin{equation}
\begin{aligned}
&\int_{\mathbb{R}^d}\left(\int_{\bb_{\ell}(x)} g_r(y)\dd y\right)^2+\left(\int_{\bb_{\ell}(x)}\vert\nabla v^T(1,y)\vert\left(1+\int_{0}^1\nabla u(s,y)\dd s\right)\dd y\right)^2\dd x\\
&\lesssim\ell^{2d}r^{-d}+\int_{\mathbb{R}^d}\bigg(\int_{\bb_{\ell}(x)}\vert\nabla v^T(1,y)\vert^2\dd y\bigg)\bigg(\int_{\bb_{\ell}(x)}1+\left\vert\int_{0}^1\nabla u(s,y)\dd s\right\vert^2\dd y\bigg)\dd x\label{proofsubopesti3}.
\end{aligned}
\end{equation}
Then, using the localized energy estimate \eqref{LemE1} applied to \eqref{equationu} with $R=\ell$ and the plain energy estimate $\int_{\mathbb{R}^d}\vert\nabla v^T(1,y)\vert^2\dd y\lesssim r^{-d}$ (the proof is identical as for \eqref{fullenergylem2}) we get
\begin{equation}
\int_{\mathbb{R}^d}\bigg(\int_{\bb_{\ell}(x)}\vert\nabla v^T(1,y)\vert^2\dd y\bigg)\bigg(\int_{\bb_{\ell}(x)}1+\left\vert\int_{0}^1\nabla u(s,y)\dd s\right\vert^2\dd y\bigg)\dd x\stackrel{\eqref{LemE1}}{\lesssim}\ell^d\int_{\mathbb{R}^d}\int_{\bb_{\ell}(x)}\vert\nabla v^T(1,y)\vert^2\dd y\, \dd x\lesssim \ell^{2d}r^{-d}.
\label{proofsubopesti4}
\end{equation}
On the other hand, by noticing that from Fubini-Tonelli's theorem, we have for all measurable functions $f : \mathbb{R}^d\rightarrow \mathbb{R}^+$
\begin{align*}
\int_{\mathbb{R}^d}\int_{\mathbb{R}^d}e^{-\frac{\vert x-z\vert}{2c\ell}}\fint_{\bb_{\ell}(z)} f(y)\dd y\, \dd z\, \dd x &= \int_{\mathbb{R}^d} f(y)\int_{\mathbb{R}^d}\ell^{-d}\mathds{1}_{\bb_{\ell}(y)}(z)\int_{\mathbb{R}^d} e^{-\frac{\vert x-z\vert}{2c\ell}}\dd x\, \dd z\,\dd y\lesssim \ell^{d}\int_{\mathbb{R}^d} f(y)\dd y,
\end{align*}
we get, recalling that $\int_{\mathbb{R}^d}\vert\nabla v^{T}(1,y)\vert^2\dd y\lesssim r^{-d}$,
\begin{equation}
\int_{\mathbb{R}^d}\int_{\mathbb{R}^d}e^{-\frac{\vert x-z\vert}{2c\ell}}\fint_{\bb_{\ell}(z)} g^2_r(y)+\vert \nabla v^T(1,y)\vert^2\dd y\, \dd z\, \dd x\lesssim \ell^{d}\int_{\mathbb{R}^d} g^2_r(y)+\vert \nabla v^T(1,y)\vert^2\dd y\lesssim \ell^{d}r^{-d}.
\label{proofsubopesti5}
\end{equation}
The combination of \eqref{proofsubopesti3}, \eqref{proofsubopesti4} and \eqref{proofsubopesti5} give the desired \eqref{lem66}.\newline
\newline
\textbf{Substep 1.2. }Proof that for all $\ell\in [1,\sqrt{T})$
\begin{align}
&\int_{\mathbb{R}^d}\left(\int_{1}^T\int_{\bb_\ell(x)}\vert \nabla u(t,y)\vert\vert \nabla v^T(t,y)\vert\dd y\, \dd t\right)^2\dd x\nonumber\\
&\lesssim \ell^{2d}r^{-d}\log^2(T)r^{d}_*(0)\int_{\mathbb{R}^d}r^d_*(rx)\frac{\log^2(\vert x\vert+1)}{(\vert x\vert+1)^{2d}}\dd x+\int_{\mathbb{R}^d}\frac{r^{2d}_*(x)}{(\vert x\vert+1)^{2d}}\dd x +\int_{\mathbb{R}^d}r^d_*(rx)g^2_2(x)\dd x.
\label{lem7step2}
\end{align}
For heuristic arguments which lead to \eqref{lem7step2}, we refer to \eqref{heuristic2} in Section \ref{stratsec}.

\medskip

\noindent Let $N:= \left\lceil \log_2(T)\right\rceil$. We start by decomposing the time interval $[1,T]$ into dyadic intervals $[2^j,2^{j+1}]$ for $j\in\llbracket 1,N-1\rrbracket$. By the triangle inequality, Cauchy-Schwarz's inequality and the decay estimate \eqref{decayuproofsubop}
\begin{align}
&\left(\int_{\mathbb{R}^d}\left(\int_{1}^T\int_{\bb_{\ell}(x)}\vert \nabla u(t,y)\vert\vert \nabla v^T(t,y)\vert\dd y\, \dd t\right)^2\dd x\right)^{\frac{1}{2}}\nonumber\\
&\leq \sum_{j=0}^{N-1}\left(\int_{\mathbb{R}^d}\left(\int_{2^j}^{2^{j+1}}\int_{\bb_{\ell}(x)}\vert \nabla u(t,y)\vert\vert \nabla v^T(T,y)\vert\dd y\, \dd t\right)^2\dd x\right)^{\frac{1}{2}}\nonumber\\
&\leq \sum_{j=0}^{N-1}\left(\int_{\mathbb{R}^d}\Big(\int_{2^j}^{2^{j+1}}\int_{\bb_{\ell}(x)}\vert \nabla u(t,y)\vert^2\dd y\, \dd t\Big)\Big(\int_{2^j}^{2^{j+1}}\int_{\bb_{\ell}(x)}\vert\nabla v^T(t,y)\vert^2\dd y\, \dd t\Big) \dd x\right)^{\frac{1}{2}}\label{lem71uselater}\\
&\stackrel{\eqref{decayuproofsubop}}{\lesssim} \sum_{j=0}^{N-1}2^{-\frac{j}{2}}\left(\int_{\mathbb{R}^d}(r_*(x)\vee\ell)^d\int_{2^j}^{2^{j+1}}\int_{\bb_{\ell}(x)}\vert \nabla v^T(t,y)\vert^2\dd y\, \dd t\right)^{\frac{1}{2}}.
\label{lem71}
\end{align}
In addition, by the $\frac{1}{8}$-Lipschitz property of $r_*$ in form of $r_*(x)\vee \ell\lesssim \inf_{y\in \bb_{\ell}(x)}r_*(y)\vee \ell$ and $\sup_{y\in \bb_{r*(x)}(x)}r_*(y)\vee \ell\lesssim r_*(x)\vee\ell$ as well as the identity $\int_{\mathbb{R}^d}\fint_{\bb_{\ell}(x)}\dd x=\int_{\mathbb{R}^d}$ and the property \eqref{averager*ctrl}, we have for all $j\in \llbracket 0,N-1\rrbracket$
\begin{align}
\left(\int_{\mathbb{R}^d}(r_*(x)\vee \ell)^d\int_{2^j}^{2^{j+1}}\int_{\bb_{\ell}(x)}\vert \nabla v^T(t,y)\vert^2\dd y\, \dd t\,\dd x\right)^{\frac{1}{2}}&\lesssim \ell^{\frac{d}{2}}\left(\int_{\mathbb{R}^d}\int_{2^j}^{2^{j+1}}(r_*(x)\vee \ell)^d\vert \nabla v^T(t,x)\vert^2\dd t\, \dd x\right)^{\frac{1}{2}}\nonumber\\
&\stackrel{\eqref{averager*ctrl}}{\lesssim}  \ell^{\frac{d}{2}}\left(\int_{\mathbb{R}^d}\int_{2^j}^{2^{j+1}}\fint_{\bb_{r_*(x)}(x)}(r_*(y)\vee \ell)^d\vert \nabla v^T(t,y)\vert^2\dd y\, \dd t\,\dd x\right)^{\frac{1}{2}}\nonumber\\
&\lesssim\ell^{\frac{d}{2}}\left(\int_{\mathbb{R}^d}(r_*(x)\vee \ell)^d\int_{2^j}^{2^{j+1}}\fint_{\bb_{r_*(x)}(x)}\vert \nabla v^T(t,y)\vert^2\dd y\, \dd t\,\dd x\right)^{\frac{1}{2}}.\label{fintvT}
\end{align}
Then, using the large-scale estimate \eqref{proofsubopesti1} we have
\begin{align}
\int_{2^j}^{2^{j+1}}\fint_{\bb_{r_*(x)}(x)}\vert \nabla v^T(s,y)\vert^2\dd y\,\dd s\lesssim 2^{j}\left(r^d_*(0)\frac{\log^2(1+\frac{\vert x\vert}{r})}{(\vert x\vert+r)^{2d}}+r^{-d}\eta_{4r_*(x)}(x)+g^2_{2r}(x)\right).\label{fintvT2}
\end{align}
Therefore, from \eqref{fintvT}, \eqref{fintvT2} and the change of variables $x\mapsto \frac{x}{r}$, we get
\begin{equation}\label{Lem789}
\begin{aligned}
&\left(\int_{\mathbb{R}^d}(r_*(x)\vee \ell)^d\int_{2^j}^{2^{j+1}}\int_{\bb_{\ell}(x)}\vert \nabla v^T(t,y)\vert^2\dd y\, \dd t\, \dd x\right)^{\frac{1}{2}}\\
&\lesssim 2^{\frac{j}{2}}\ell^{\frac{d}{2}}r^{-\frac{d}{2}}\left(r^{d}_*(0)\int_{\mathbb{R}^d}(r_*(rx)\vee \ell)^d\frac{\log^2(\vert x\vert+1)}{(\vert x\vert+1)^{2d}}\dd x+\int_{\mathbb{R}^d}(r_*(x)\vee \ell)^d\eta_{4r_*(x)}(x)\dd x+\int_{\mathbb{R}^d}(r_*(rx)\vee \ell)^d g^2_2(x)\dd x\right)^{\frac{1}{2}},
\end{aligned} 
\end{equation}
which gives the estimate \eqref{lem7step2} by plugging \eqref{Lem789} into \eqref{lem71} and using $\eta_{4r_*(x)}(x)\lesssim r^d_*(x)(\vert x\vert+1)^{-2d}$ as well as $\frac{r_*(x)}{\ell}\vee 1\lesssim r_*(x)$. This concludes the argument for the first item of \eqref{step1subopsensi}. Finally, we recall \eqref{momentr*lem1}:
\begin{equation}\label{equivmomentr*}
\langle r^{dp}_*(0)\rangle^{\frac{1}{p}}\lesssim_\gamma p^{\frac{d}{\beta \wedge d}}\mathds{1}_{\beta\neq d}+p^{1+\gamma}\mathds{1}_{\beta=d} \quad\text{for any $\gamma>0$.}
\end{equation}
Applying \eqref{equivmomentr*} directly on the random variables involved in \eqref{lem65} and \eqref{lem7step2} yields the second item of \eqref{step1subopsensi}.
\newline
\newline
\textbf{Step 2. Regime $\ell\geq \sqrt{T}$. } Proof that for all $\ell\in [\sqrt{T},\infty)$
\begin{equation}
\int_{\mathbb{R}^d}\vert\partial^{\text{fct}}_{x,\ell}q_r(T)\vert^2\dd x\lesssim\ell^d((r^{2d}_*(0)+\mathcal{D}_{\star,1}(r,\ell))\log^2(T)+\mathcal{D}_{\star,2}(r)\log^2(\tfrac{\sqrt{T}}{r})),
\label{step2subopsensi}
\end{equation}
with for any $\gamma>0$
\begin{equation}\label{Step2MomentBoundOfTheRC}
\langle\mathcal{D}^p_{\star,1}(r,\ell)\rangle^{\frac{1}{p}}\lesssim_\gamma (p^{2\frac{d+1}{\beta\wedge d}}\mathds{1}_{\beta\neq d}+p^{2(1+\frac{1}{d})(1+\gamma)}\mathds{1}_{\beta=d})(1+\log^4(\tfrac{\ell}{r}))\quad \text{and}\quad \sup_{r\geq 1}\langle \mathcal{D}^{p}_{\star,2}(r)\rangle^{\frac{1}{p}}\lesssim_\gamma p^{2\frac{d}{\beta\wedge d}}\mathds{1}_{\beta\neq d}+p^{2(1+\gamma)}\mathds{1}_{\beta=d}.
\end{equation}
This estimate corresponds to \eqref{UseRefHeuristicLater} in the heuristic arguments of Section \ref{stratsec}. We split this step into two parts. The first one is devoted to the control of the first four r.h.s terms of \eqref{lem6} and the second one is devoted to the control of the last r.h.s term.
\newline
\newline
\textbf{Substep 2.1.} Proof that for all $\ell\in [\sqrt{T},\infty)$
\begin{equation}
\int_{\mathbb{R}^d}\left(\int_{\bb_{\ell}(x)} g_r(y)\dd y\right)^2\dd x\lesssim \ell^d\quad\text{and}\quad\int_{\mathbb{R}^d}\left(\int_{\bb_{\ell}(x)}\left\vert\int_{0}^T\nabla u(t,y)\dd t\right\vert g_r(y)\dd y\right)^2\dd x\lesssim \ell^{d}\mathcal{D}_{\star,2}(r)\log^2(\tfrac{\sqrt{T}}{r}),
\label{substep211}
\end{equation}
as well as
\begin{equation}
\int_{\mathbb{R}^d}\left(\int_{\bb_{\ell}(x)}\vert\nabla v^T(1,y)\vert\left(1+\left\vert\int_{0}^1\nabla u(s,y)\dd s\right\vert\right)\dd y\right)^2\dd x+\int_{\mathbb{R}^d}\mathcal{G}^2_{r,\ell}(x)\dd x\lesssim \ell^d\mathcal{D}_{\star,1}(r,\ell),
\label{substep22}
\end{equation}
where we recall that $\mathcal{G}_{r,\ell}$ is defined in \eqref{Grllem6}. The estimate of the first l.h.s of \eqref{substep22} corresponds to \eqref{UseRefHeuristicLater:Eq2} in the heuristic arguments of Section \ref{stratsec}.\newline
\newline
\textbf{Argument for \eqref{substep211}. }For the first item, we use Young's inequality in form of
$$\int_{\mathbb{R}^d}\left(\int_{\bb_{\ell}(x)} g_r(y)\dd y\right)^2\dd x=\|g_r\star\mathds{1}_{\bb_{\ell}}\|^2_{\text{L}^2(\mathbb{R}^d)}\leq \|g_r\|^2_{\text{L}^1(\mathbb{R}^d)}\|\mathds{1}_{\bb_{\ell}}\|^2_{\text{L}^2(\mathbb{R}^d)}\lesssim \ell^d.$$
For the second item, we use for all $x\in\mathbb{R}^d$, $\int_{\bb_{\ell}(x)}g_r(y)\dd y\leq \int_{\mathbb{R}^d}g_r(y)\dd y\lesssim 1$, Cauchy-Schwarz's inequality applied in form of $(\int_{\bb_{\ell}(x)}f(y)g_r(y)\dd y)^2\leq\int_{\bb_{\ell}(x)}g_r(y)\dd y\int_{\bb_{\ell}(x)}\vert f(y)\vert^2 g_r(y)\dd y$ and the identity $\int_{\mathbb{R}^d}\fint_{\bb_{\ell}(x)}=\int_{\mathbb{R}^d}$, to get
$$\int_{\mathbb{R}^d}\left(\int_{\bb_{\ell}(x)}\left\vert\int_{0}^T\nabla u(t,y)\dd t\right\vert g_r(y)\dd y\right)^2\dd x\lesssim \int_{\mathbb{R}^d}\int_{\bb_{\ell}(x)}\left\vert\int_{0}^T\nabla u(t,y)\dd t\right\vert^2 g_r(y)\dd y\, \dd x\lesssim \ell^d\int_{\mathbb{R}^d}\left\vert\int_{0}^T\nabla u(t,y)\dd t\right\vert^2 g_r(y)\dd y.$$
Then, since $r\leq \sqrt{T}$, we split the integral $\int_{0}^T$ into $\int_{0}^{r^2}+\int_{r^2}^{T}$ and we use the localized energy estimate \eqref{LemE1} applied to \eqref{equationu} in form of $\int_{\mathbb{R}^d}\left\vert\int_{0}^{r^2}\nabla u(t,y)\dd t\right\vert^2g_r(y)\dd y\lesssim 1$, to obtain
\begin{align}
\int_{\mathbb{R}^d}\left\vert\int_{0}^T\nabla u(t,y)\dd t\right\vert^2 g_r(y)\dd y&\lesssim \int_{\mathbb{R}^d}\left\vert\int_{0}^{r^2}\nabla u(t,y)\dd t\right\vert^2 g_r(y)\dd y+\int_{\mathbb{R}^d}\left\vert\int_{r^2}^T\nabla u(t,y)\dd t\right\vert^2 g_r(y)\dd y\nonumber\\
&\stackrel{\eqref{LemE1}}{\lesssim}1+\int_{\mathbb{R}^d}\left\vert\int_{r^2}^T\nabla u(t,y)\dd t\right\vert^2 g_r(y)\dd y.
\label{othertestesti7}
\end{align}
Finally, using that $\int_{\mathbb{R}^d}=\int_{\mathbb{R}^d}\fint_{\bb_r(x)}\dd x$  and for all $x\in\mathbb{R}^d\backslash \bb_r$, $\sup_{y\in\bb_r(x)}g_r(y)\lesssim g_r(x)$, in form of
$$\int_{\mathbb{R}^d}\left\vert\int_{r^2}^T\nabla u(t,y)\dd t\right\vert^2 g_r(y)\dd y\lesssim \fint_{\bb_r}\fint_{\bb_r(x)}\left\vert\int_{r^2}^T\nabla u(t,y)\dd t\right\vert^2\dd y\, \dd x+\int_{\mathbb{R}^d\backslash \bb_r}g_r(x)\fint_{\bb_r(x)}\left\vert\int_{r^2}^T\nabla u(t,y)\dd t\right\vert^2\dd y\, \dd x,$$
and Minkowski's inequality in $\LL^2(\bb_r(x))$ as well as the decay estimate \eqref{decayuproofsubop} (where we use that, since $r\geq 1$, $\frac{r_*(x)}{r}\vee 1\leq r_*(x)$), we deduce
\begin{align}
\int_{\mathbb{R}^d}\left\vert\int_{r^2}^T\nabla u(t,y)\dd t\right\vert^2 g_r(y)\dd y\lesssim & \fint_{\bb_r}\left(\int_{r^2}^T\left(\fint_{\bb_r(x)}\vert\nabla u(t,y)\vert^2\dd y\right)^{\frac{1}{2}}\dd t\right)^2\dd x\nonumber\\
&+\int_{\mathbb{R}^d\backslash \bb_r}g_r(x)\left(\int_{r^2}^T\left(\fint_{\bb_r(x)}\vert\nabla u(t,y)\vert^2\dd y\right)^{\frac{1}{2}}\dd t\right)^2 \dd x\nonumber\\
\stackrel{\eqref{decayuproofsubop}}{\lesssim} &\left(\fint_{\bb_r} r^{d}_*(x)\dd x+\int_{\mathbb{R}^d\backslash \bb_1}r^d_*(rx)g_1(x)\dd x\right)\log^2(\tfrac{\sqrt{T}}{r}).
\label{othertestesti8}
\end{align}
Finally, applying \eqref{equivmomentr*} on the random variable involved in \eqref{othertestesti8} yields \eqref{substep211} with the second item in \eqref{Step2MomentBoundOfTheRC}.
\newline
\newline 
\textbf{Argument for \eqref{substep22}. } 
We start with the first l.h.s term. We distinguish between the generic case $\ell\geq r_*(0)$ and the non-generic case $\ell<r_*(0)$. \newline
\newline
\textbf{Regime $\ell\geq r_*(0)$. }We split the integral into the far-field contribution $\vert x\vert\geq 4\ell$ and the near-field contribution $\vert x\vert<4\ell$. For the far-field contribution, we write using Cauchy-Schwarz's inequality, the localized energy estimate \eqref{LemE1} applied to \eqref{equationu} with $R=\ell$ and $T=1$, the identity $\int_{\mathbb{R}^d\backslash \bb_{4\ell}}\int_{\bb_{\ell}(x)}\lesssim\ell^d \int_{\mathbb{R}^d\backslash \bb_{3\ell}}$ as well as the property \eqref{averager*ctrl}
\begin{align}
\int_{\mathbb{R}^d\backslash \bb_{4\ell}}\left(\int_{\bb_{\ell}(x)}\vert \nabla v^T(1,y)\vert\left\vert 1+\int_{0}^1\nabla u(s,y)\dd s\right\vert\dd y\right)^2\dd x&\leq \int_{\mathbb{R}^d\backslash \bb_{4\ell}}\Big(\int_{\bb_{\ell}(x)}\vert\nabla v^T(1,y)\vert^2\dd y\Big)\Big(\int_{\bb_{\ell}(x)}1+\left\vert\int_{0}^1\nabla u(s,y)\dd y\right\vert^2\Big)\dd x\nonumber\\
&\stackrel{\eqref{LemE1}}{\lesssim} \ell^{d}\int_{\mathbb{R}^d\backslash \bb_{4\ell}}\int_{\bb_{\ell}(x)}\vert\nabla v^T(1,y)\vert^2\dd y\label{proofsubopesti10}\\
&\lesssim \ell^{2d}\int_{\mathbb{R}^d\backslash \bb_{3\ell}}\vert\nabla v^T(1,y)\vert^2\dd y\nonumber\\
&\stackrel{\eqref{averager*ctrl}}{\lesssim} \ell^{2d}\int_{\mathbb{R}^d}\fint_{\bb_{r_*(x)}(x)}\vert\nabla v^T(1,y)\vert^2\mathds{1}_{\mathbb{R}^d\backslash \bb_{3\ell}}(y)\dd y\, \dd x.\label{proofsubopesti13}
\end{align}
By the $\frac{1}{8}$-Lipschitz continuity property of $r_*$ and the assumption $r_*(0)\leq \ell$ one has $\bb_{r_*(x)}(x)\cap \mathbb{R}^d\backslash \bb_{3\ell}\neq \emptyset$ $\Rightarrow$ $x\in \mathbb{R}^d\backslash \bb_{\ell}$. Indeed, by contradiction and recalling that $\ell\geq r_*(0)$, if $x\in\bb_{\ell}$ then $r_*(x)\leq r_*(0)+\frac{\vert x\vert}{8}\leq \frac{9}{8}\ell$ so that $\bb_{r_*(x)}(x)\subset \bb_{\ell+\frac{9}{8}\ell}\subset \bb_{\frac{17}{8}\ell}$ and consequently $\bb_{r_*(x)}(x)\cap \mathbb{R}^d\backslash \bb_{3\ell}=\emptyset$. Hence,
\begin{equation}
\int_{\mathbb{R}^d}\fint_{\bb_{r_*(x)}(x)}\vert\nabla v^T(1,y)\vert^2\mathds{1}_{\mathbb{R}^d\backslash \bb_{3\ell}}(y)\dd y\, \dd x\leq \int_{\mathbb{R}^d\backslash \bb_{\ell}}\fint_{\bb_{r_*(x)}(x)}\vert\nabla v^T(1,y)\vert^2\dd y\, \dd x,
\label{proofsubopesti12}
\end{equation}
and the above inequality, with help of \eqref{proofsubopesti1} reduces to 
\begin{align}
&\int_{\mathbb{R}^d\backslash \bb_{4\ell}}\left(\int_{\bb_{\ell}(x)}\vert \nabla v^T(1,y)\vert\left\vert 1+\int_{0}^1\nabla u(s,y)\dd s\right\vert\dd y\right)^2\dd x\nonumber\\
&\lesssim \ell^{2d}\int_{\mathbb{R}^d\backslash \bb_{\ell}}\fint_{\bb_{r_*(x)}(x)}\vert\nabla v^T(1,y)\vert^2\dd y\, \dd x\label{proofsubopesti14}\\
&\stackrel{\eqref{proofsubopesti1}}{\lesssim} \ell^{2d}\left(r^d_*(0)\int_{\mathbb{R}^d\backslash \bb_{\ell}}\frac{\log^2(1+\frac{\vert x\vert}{r})}{(\vert x\vert+r)^{2d}}\dd x +\int_{\mathbb{R}^d\backslash \bb_{\ell}}\eta_{4r_*(x)}(x)\dd x+\int_{\mathbb{R}^d\backslash \bb_{\ell}} g^2_{2r}(x)\dd x\right),\label{proofsubopesti15}
\end{align}
For the near-field contribution $\vert x\vert<4\ell$, using Minkowski's inequality in $\LL^2(\bb_{4\ell})$, Fubini-Tonnelli's theorem, the property \eqref{averager*ctrl} and the assumption $\ell\geq r_*(0)$ in form of \eqref{proofsubopesti7}, we write
\begin{align}
&\left(\int_{\bb_{4\ell}}\left(\int_{\bb_{\ell}(x)}\vert\nabla v^T(1,y)\left(1+\left\vert\int_{0}^1\nabla u(s,y)\dd s\right\vert\right)\dd y\right)^2\dd x\right)^{\frac{1}{2}}\nonumber\\
&\leq \int_{\mathbb{R}^d}\left(\int_{\bb_{4\ell}}\vert\nabla v^T(1,y)\vert^2\left(1+\left\vert\int_{0}^1\nabla u(s,y)\dd s\right\vert^2\right)\mathds{1}_{\bb_{\ell}(x)}(y)\dd x\right)^{\frac{1}{2}}\dd y\nonumber\\
&\lesssim \ell^{\frac{d}{2}}\int_{\bb_{5\ell}}\vert\nabla v^T(1,y)\vert\left(1+\left\vert\int_{0}^1\nabla u(s,y)\dd s\right\vert\right)\dd y\nonumber\\
&\stackrel{\eqref{averager*ctrl}}{\lesssim} \ell^{\frac{d}{2}}\int_{\mathbb{R}^d}\fint_{\bb_{r_*(x)}(x)}\vert\nabla v^T(1,y)\vert\left(1+\left\vert\int_{0}^1\nabla u(s,y)\dd s\right\vert\right)\mathds{1}_{\bb_{5\ell}}(y)\dd y\, \dd x\nonumber\\
&\stackrel{\eqref{proofsubopesti7}}{\lesssim} \ell^{\frac{d}{2}}\int_{\bb_{7\ell}}\fint_{\bb_{r_*(x)}(x)}\vert\nabla v^T(1,y)\vert\left(1+\left\vert\int_{0}^1\nabla u(s,y)\dd s\right\vert\right)\dd y\, \dd x.\label{proofsubopesti6}
\end{align}
Then, from the localized energy estimate \eqref{LemE1} applied to \eqref{equationu}  with $R=r_*(x)$ and $T=1$, we have 
$$\left(\fint_{\bb_{r_*(x)}(x)}1+\left\vert\int_{0}^{1}\nabla u(s,y)\dd s\right\vert^2\dd y\right)^{\frac{1}{2}}\lesssim 1,$$
consequently, by Cauchy-Schwarz's inequality and \eqref{proofsubopesti1}, \eqref{proofsubopesti6} turns into
\begin{align}
&\left(\int_{\bb_{4\ell}}\left(\int_{\bb_{\ell}(x)}\vert\nabla v^T(1,y)\left(1+\left\vert\int_{0}^1\nabla u(s,y)\dd s\right\vert\right)\dd y\right)^2\dd x\right)^{\frac{1}{2}}\nonumber\\
&\lesssim \ell^{\frac{d}{2}}\int_{\bb_{7\ell}}\left(\fint_{\bb_{r_*(x)}(x)}\vert\nabla v^T(1,y)\vert^2\dd y\right)^{\frac{1}{2}}\dd x\label{proofsubopesti11}\\
&\lesssim \ell^{\frac{d}{2}}\left(r^{\frac{d}{2}}_*(0)\int_{\bb_{7\ell}}\frac{\log(1+\frac{\vert x\vert}{r})}{(\vert x\vert+r)^d}\dd x+\int_{\bb_{7\ell}}\eta^{\frac{1}{2}}_{4r_*(x)}(x)\dd x+\int_{\bb_{7\ell}} g_r(y)\dd y\right).\label{ChangingWithMoment:Eq1}
\end{align}
\textbf{Regime $\ell<r_*(0)$. }We use the estimate \eqref{lem66} and we bound one $\ell^d$ by $r^d_*(0)$ and $r^{-d}$ by $1$.\newline
\newline
We now turn to the second l.h.s term of \eqref{substep22}. The first term in the definition \eqref{Grllem6} of $\mathcal{G}_{r,\ell}$ is bounded as follows. By definition of $\mathcal{T}_{x,\ell}(\eta_r)(0)$, that we recall here
$$\mathcal{T}_{x,\ell}(\eta_r)(0)=\left(\int_{\bb_{\ell}(x)}\eta_r(z)\left(1+\left\vert\int_{0}^1\nabla u(t,z)\dd t\right\vert^2\right)\dd z\right)^{\frac{1}{2}}+\int_{0}^1\frac{1}{1-t}\int_{t}^1\left(\int_{\bb_{\ell}(x)}\eta_r(z)\vert\nabla u(s,z)\vert^2\dd z\right)^{\frac{1}{2}}\dd s\, \dd t,$$
and the localized energy estimates \eqref{LemE1} as well as \eqref{LemE2} applied to \eqref{equationu}, Minkowski's inequality in $\LL^2(\mathbb{R}^d)$ (exchanging the order of integration in the $s$ and $x$ variables) and the estimate $\int_{\mathbb{R}^d}\int_{\bb_{\ell}(x)}\lesssim \ell^d\int_{\mathbb{R}^d}$, we have 
\begin{align}
\int_{\mathbb{R}^d} \mathcal{T}^2_{x,\ell}(\eta_r)(0)\dd x&\lesssim \ell^d \left(1+\int_{\mathbb{R}^d}\eta_r(y)\left\vert \int_{0}^1\nabla u(t,y)\dd t\right\vert^2\dd y+\left(\int_{0}^1\frac{1}{1-t}\int_{t}^1\left(\int_{\mathbb{R}^d}\eta_r(y)\vert\nabla u(s,y)\vert^2\dd y\right)^{\frac{1}{2}}\dd s\, \dd t\right)^2\right)\nonumber\\
&\stackrel{\eqref{LemE1},\eqref{LemE2}}{\lesssim}\ell^d\left(1+\left(\int_{0}^1\frac{-\log(t)}{1-t}\dd t\right)^2\right)\lesssim \ell^d.\label{othertestesti13}
\end{align}
For the other term in \eqref{Grllem6}, we can use previous estimates. To this aim, we need preliminary inequalities and we distinguish between the two regimes $\ell\geq r_*(0)$ and $\ell<r_*(0)$.\newline
\newline
\textbf{Regime $\ell\geq r_*(0)$. }For the far-field contribution $\vert x\vert\geq 4\ell$, we make use of Cauchy-Schwarz's inequality in form of, for all $x\in\mathbb{R}^d$
\begin{equation}
\left(\int_{\mathbb{R}^d\backslash \bb_{4\ell}}\left(\fint_{\bb_{\ell}(y)}\vert \nabla v^T(1,z)\vert^2\dd z\right)^{\frac{1}{2}}\mathcal{T}_{x,\ell}(\eta_{\ell})(y)\dd y\right)^2\leq \int_{\mathbb{R}^d\backslash \bb_{4\ell}}\mathcal{T}_{x,\ell}(\eta_{\ell})(y)\dd y\int_{\mathbb{R}^d\backslash \bb_{4\ell}}\fint_{\bb_{\ell}(y)}\vert \nabla v^T(1,z)\vert^2\dd z\,\mathcal{T}_{x,\ell}(\eta_{\ell})(y)\dd y.
\label{proofsubopesti8}
\end{equation}
Next, we have 
\begin{equation}
 \int_{\mathbb{R}^d}\mathcal{T}_{x,\ell}(\eta_{\ell})(y)\dd y\lesssim \ell^d.
\label{proofsubopesti9}
\end{equation}
Indeed, we first split the integral into two contributions
$$ \int_{\mathbb{R}^d}\mathcal{T}_{x,\ell}(\eta_{\ell})(y)\dd y= \int_{\bb_{2\ell}(x)}\mathcal{T}_{x,\ell}(\eta_{\ell})(y)\dd y+ \int_{\mathbb{R}^d\backslash \bb_{2\ell}(x)}\mathcal{T}_{x,\ell}(\eta_{\ell})(y)\dd y.$$
On the one hand, since from the localized energy estimate \eqref{LemE1} applied to \eqref{equationu} with $R=\ell$ we have $\sup_{y\in\mathbb{R}^d} \mathcal{T}_{x,\ell}(\eta_{\ell})(y)\lesssim 1+\int_{0}^1\frac{-\log(t)}{1-t}\dd t\lesssim 1$ (we bound the integral $\int_{\bb_{\ell}(x)}$ by $\int_{\mathbb{R}^d}$ in the definition \eqref{functionallem6} of $\mathcal{T}_{x,\ell}(\eta_{\ell})(y)$), we get
$$\int_{\bb_{2\ell}(x)}\mathcal{T}_{x,\ell}(\eta_{\ell})(y)\dd y\lesssim \ell^d.$$
On the other hand, for all $y\in \mathbb{R}^d\backslash \bb_{2\ell}(x)$ and $z\in \bb_{\ell}(x)$ we have $\vert z-y\vert\geq \vert y-x\vert-\ell\geq \frac{\vert y-x\vert}{2}$ and thus $\eta_{\ell}(z-y)\lesssim \eta_{2\ell}(y-x)$. Therefore, the localized energy estimate \eqref{LemE1} applied to \eqref{equationu} with $R=\ell$ yields
\begin{align*}
\mathcal{T}_{x,\ell}(\eta_{\ell})(y)&\lesssim \eta^{\frac{1}{2}}_{2\ell}(y-x)\left(\left(\int_{\bb_{\ell}(x)}1+\left\vert\int_{0}^1\nabla u(s,z)\dd s\right\vert^2\dd z\right)^{\frac{1}{2}}+\int_{0}^{1}\frac{1}{1-t}\int_{t}^1\left(\int_{\bb_{\ell}(x)}\vert\nabla u(s,z)\vert^2\dd z\right)^{\frac{1}{2}}\dd t\right)\\
&\lesssim \ell^{\frac{d}{2}}\eta^{\frac{1}{2}}_{2\ell}(y-x).
\end{align*}
Consequently,
$$\int_{\mathbb{R}^d\backslash \bb_{2\ell}(x)}\mathcal{T}_{x,\ell}(\eta_{\ell})(y)\dd y\lesssim \ell^{\frac{d}{2}}\int_{\mathbb{R}^d\backslash \bb_{2\ell}(x)}\eta^{\frac{1}{2}}_{2\ell}(y-x)\dd y\lesssim \ell^{d},$$
and this concludes the argument for \eqref{proofsubopesti9}. The combination of \eqref{proofsubopesti8}, \eqref{proofsubopesti9} and the estimate $\sup_{y\in\mathbb{R}^d}\int_{\mathbb{R}^d}\mathcal{T}_{x,\ell}(\eta_{\ell})(y)\dd x\lesssim \ell^d$ (which can be proved with the same computation as \eqref{proofsubopesti9} by exchanging the role of $x$ and $y$) leads to 
\begin{align*}
\int_{\mathbb{R}^d}\left(\int_{\mathbb{R}^d\backslash \bb_{4\ell}}\left(\fint_{\bb_{\ell}(y)}\vert \nabla v^T(1,z)\vert^2\dd z\right)^{\frac{1}{2}}\mathcal{T}_{x,\ell}(\eta_{\ell})(y)\dd y\right)^2\dd x&\stackrel{\eqref{proofsubopesti8},\eqref{proofsubopesti9}}{\lesssim}\ell^d\int_{\mathbb{R}^d}\int_{\mathbb{R}^d\backslash \bb_{4\ell}}\fint_{\bb_{\ell}(y)}\vert\nabla v^T(1,z)\vert^2\dd z\, \mathcal{T}_{x,\ell}(\eta_{\ell})(y)\dd y\, \dd x\\
&\lesssim \ell^{2d}\int_{\mathbb{R}^d\backslash \bb_{4\ell}}\fint_{\bb_{\ell}(y)}\vert\nabla v^T(1,z)\vert^2\dd z\, \dd y\, \dd x,
\end{align*}
and we then proceed as we did from \eqref{proofsubopesti10} to conclude. For the near-field contribution, we make use of Minkowski's inequality in $\LL^2(\mathbb{R}^d)$ and the estimate $\sup_{y\in\mathbb{R}^d}\int_{\mathbb{R}^d}\mathcal{T}^2_{x,\ell}(\eta_{r_*(y)})(y)\dd x\lesssim \ell^d$ (obtained the same way as \eqref{othertestesti13}) to obtain
\begin{align}
&\int_{\mathbb{R}^d}\bigg(\int_{\bb_{7\ell}}\Big(\fint_{\bb_{r_*(y)}(y)}\vert\nabla v^T(1,z)\vert^2\dd z\Big)^{\frac{1}{2}}\mathcal{T}_{x,\ell}(\eta_{r_*(y)})(y)\dd y\bigg)^2\dd x\nonumber\\
&\leq \bigg(\int_{\bb_{7\ell}}\Big(\fint_{\bb_{r_*(y)}(y)}\vert\nabla v^T(1,z)\vert^2\dd z\Big)^{\frac{1}{2}}\Big(\int_{\mathbb{R}^d}\mathcal{T}^2_{x,\ell}(\eta_{r_*(y)})(y)\dd x\Big)^{\frac{1}{2}}\dd y\bigg)^2\nonumber\\
&\lesssim\ell^d\bigg(\int_{\bb_{7\ell}}\Big(\fint_{\bb_{r_*(y)}(y)}\vert\nabla v^T(1,z)\vert^2\dd z\Big)^{\frac{1}{2}}\dd y\bigg)^2,
\label{proofth1othertest5prime}
\end{align}
and we then proceed as we did from \eqref{proofsubopesti11}.\newline
\newline
\textbf{Regime $\ell<r_*(0)$. }In this regime, we use the estimate \eqref{proofsubopesti5}.\newline
\newline
To conclude, the combination of \eqref{proofsubopesti15} (together with $\int_{\mathbb{R}^d\backslash \bb_{\ell}} g^2_{2r}(x)\dd x\lesssim \ell^{-d}$ and $\eta_{4r_*(x)}(x)\lesssim r^d_*(x)\vert x\vert^{-2d}$), \eqref{ChangingWithMoment:Eq1} (together with $\int_{\bb_{7\ell}} g_r(y)\dd y \lesssim 1$ and $\eta^{\frac{1}{2}}_{4r_*(x)}(x)\lesssim r^{\frac{d}{2}+1}_*(x)(\vert x\vert+1)^{-d-1}$), \eqref{lem66} and \eqref{proofsubopesti5} gives \eqref{substep22} with 
\begin{equation}\label{stosubopsensi2}
\begin{aligned}
\mathcal{D}_{\star,1}(r,\ell)= & r^d_*(0)+r^d_*(0)\ell^d\int_{\mathbb{R}^d\backslash \bb_{\ell}}\frac{\log^2(1+\frac{\vert x\vert}{r})}{(\vert x\vert+r)^{2d}}\dd x +\ell^{d}\int_{\mathbb{R}^d\backslash \bb_{\ell}} r^d_*(x)\vert x\vert^{-2d}\dd x\\
&+r^d_*(0)\left(\int_{\bb_{7\ell}}\frac{\log(1+\frac{\vert x\vert}{r})}{(\vert x\vert+r)^d}\dd x\right)^2+\left(\int_{\bb_{7\ell}}r^{\frac{d}{2}+1}_*(x)(\vert x\vert+1)^{-d-1}\dd x\right)^2+r^d_*(0)\int_{\mathbb{R}^d}\frac{r^d_*(x)}{(\vert x\vert+1)^{2d}}\dd x.
\end{aligned}
\end{equation}
Applying \eqref{equivmomentr*} on \eqref{stosubopsensi2} followed by a polar change of coordinates and the change of variable $\rho\mapsto \frac{\rho}{r}$ yields
\begin{align*}
\langle \mathcal{D}^p_{\star,1}(r,\ell)\rangle^{\frac{1}{p}}&\stackrel{\eqref{equivmomentr*}}{\lesssim_\gamma} (p^{2\frac{d+1}{\beta\wedge d}}\mathds{1}_{\beta\neq d}+p^{2(1+\frac{1}{d})(1+\gamma)}\mathds{1}_{\beta=d})\int_{1}^{+\infty}\ell^{-1-\beta}\left(\int_{0}^{7\ell}\frac{\log(1+\frac{\rho}{r})}{(\rho+r)^{d}}\rho^{d-1}\dd\rho+1\right)^2\dd\ell\nonumber\\
&\leq (p^{2\frac{d+1}{\beta\wedge d}}\mathds{1}_{\beta\neq d}+p^{2(1+\frac{1}{d})(1+\gamma)}\mathds{1}_{\beta=d})\log^2(1+7\,\tfrac{\ell}{r})\int_{1}^{+\infty}\ell^{-1-\beta}\left(\int_{0}^{7 \tfrac{\ell}{r}}(\rho+1)^{-d}\rho^{d-1}\dd\rho+1\right)^2\dd\ell.
\end{align*}
We then estimate the integral in the integrand of the r.h.s by
\begin{align*}
\int_{0}^{7\frac{\ell}{r}}(\rho+1)^{-d}\rho^{d-1}\dd\rho\lesssim\int_{0}^1 \rho^{d-1}\dd \rho+\int_{1}^{7\frac{\ell}{r}}\rho^{-1}\dd \rho\lesssim 1+\log(7\,\tfrac{\ell}{r}),
\end{align*} 
which yields the first item in \eqref{Step2MomentBoundOfTheRC}.
\newline
\newline
\textbf{Substep 2.2. }Proof that for $\ell\in [\sqrt{T}, \infty)$
\begin{equation}\label{substep220}
\int_{\mathbb{R}^d}\left(\int_{1}^T\int_{\bb_{\ell}(x)}\vert \nabla u(t,y)\vert\vert\nabla v^T(t,y)\vert\dd t\, \dd y\right)^2\dd x\lesssim\, \ell^d\log^2(T)(r^{2d}_*(0)+\mathcal{D}_{\star,1}(r,\ell)),
\end{equation}
where we recall that $\mathcal{D}_{\star,1}(r,\ell)$ is defined in \eqref{stosubopsensi2} and satisfies the first item of \eqref{Step2MomentBoundOfTheRC}. For heuristic arguments which lead to \eqref{substep220}, we refer to \eqref{strataway0lgtr} in Section \ref{stratsec}.

\medskip

\noindent For the proof of \eqref{substep220}, we argue as previously and we distinguish between the generic case $\ell\geq r_*(0)$ and the non-generic case $\ell<r_*(0)$. We mainly make use of previous ideas and estimates.\newline
\newline
\textbf{Regime $\ell\geq r_*(0)$. }As in Substep $1.2$, we decompose the time interval $[1,T]$ into dyadic intervals $[2^{j},2^{j+1}]$ for $j\in\llbracket 1,N-1\rrbracket$ and $N:=\left\lceil \log_2(T)\right\rceil$ and we write by the triangle inequality 
\begin{equation}
\int_{\mathbb{R}^d}\left(\int_{1}^T\int_{\bb_{\ell}(x)}\vert \nabla u(t,y)\vert\vert\nabla v^T(t,y)\vert\dd t\, \dd y\right)^2\dd x\leq\left(\sum_{j=0}^{N-1}\left(\int_{\mathbb{R}^d}\left(\int_{2^j}^{2^{j+1}}\int_{\bb_{\ell}(x)}\vert \nabla u(t,y)\vert\vert\nabla v^T(t,y)\vert\dd y\, \dd t\right)^2\dd x\right)^{\frac{1}{2}}\right)^2.
\label{substep221}
\end{equation}
We split the integral in the r.h.s of \eqref{substep221} into the far-field contribution $\vert x\vert\geq 4\ell$ and the near-field contribution $\vert x\vert< 4\ell$. For the far-field contribution, we write for all $j\in \llbracket 1,N-1\rrbracket$, using Cauchy-Schwarz's inequality, the decay estimate \eqref{decayuproofsubop} applied for $\ell\geq \sqrt{T}$, 
\begin{equation*}
\int_{\mathbb{R}^d\backslash \bb_{4\ell}}\left(\int_{2^j}^{2^{j+1}}\int_{\bb_{\ell}(x)}\vert \nabla u(t,y)\vert\vert\nabla v^T(t,y)\vert\dd y\, \dd t\right)^2\dd x\stackrel{\eqref{decayuproofsubop}}{\leq} \ell^d 2^{-j}\int_{2^j}^{2^{j+1}}\int_{\mathbb{R}^d\backslash \bb_{4\ell}}\int_{\bb_{\ell}(x)}\vert \nabla v^T(t,y)\vert^2\dd y\, \dd x\, \dd t.
\end{equation*}
We then argue as from \eqref{proofsubopesti10} to \eqref{proofsubopesti13}, \eqref{proofsubopesti12}, \eqref{proofsubopesti14} and \eqref{proofsubopesti15} (noticing that the evaluation at time $1$ plays no role in the estimates) and finally obtain
\begin{equation*}
\int_{\mathbb{R}^d\backslash \bb_{4\ell}}\left(\int_{2^j}^{2^{j+1}}\int_{\bb_{\ell}(x)}\vert \nabla u(t,y)\vert\vert\nabla v^T(t,y)\vert\dd y\, \dd t\right)^2\dd x\lesssim \ell^{2d}\left(r^d_*(0)\int_{\mathbb{R}^d\backslash\bb_{\ell}}\frac{\log^2(1+\frac{\vert x\vert}{r})}{(\vert x\vert+r)^{2d}}\dd x+\int_{\mathbb{R}^d\backslash\bb_{\ell}}r^d_*(x)\vert x\vert^{-2d}\dd x+\ell^{-d}\right),
\end{equation*}
and we conclude by summing over $j\in \llbracket 1,N-1\rrbracket$, which gives \eqref{substep220} for $\vert x\vert\geq 4\ell$. For the near-field contribution $\vert x\vert\leq 4\ell$ we write, using Minkowski's inequality in $\LL^2(\bb_{4\ell})$, Fubini-Tonnelli's theorem, the property \eqref{averager*ctrl} and the assumption $\ell\geq r_*(0)$ in form of \eqref{proofsubopesti7}
\begin{align}
\left(\int_{\bb_{4\ell}}\left(\int_{2^j}^{2^{j+1}}\int_{\bb_{\ell}(x)}\vert \nabla u(t,y)\vert\vert\nabla v^T(t,y)\vert\dd y\, \dd t\right)^2\dd x\right)^{\frac{1}{2}}&\leq \int_{2^j}^{2^{j+1}}\int_{\mathbb{R}^d}\left(\int_{\bb_{4\ell}}\vert \nabla u(t,y)\vert^2\vert\nabla v^T(t,y)\vert^2\mathds{1}_{\bb_{\ell}(x)}(y)\dd x\right)^{\frac{1}{2}}\dd y\,\dd t\nonumber\\
&\lesssim \ell^{\frac{d}{2}}\int_{\bb_{5\ell}}\int_{2^j}^{2^{j+1}}\vert \nabla u(t,y)\vert\vert\nabla v^T(t,y)\vert\dd t\, \dd y\nonumber\\
&\stackrel{\eqref{averager*ctrl}}{\lesssim} \ell^{\frac{d}{2}}\int_{\mathbb{R}^d}\fint_{\bb_{r_*(x)}(x)}\int_{2^j}^{2^{j+1}}\vert \nabla u(t,y)\vert\vert \nabla v^T(t,y)\vert\mathds{1}_{\bb_{5\ell}}(y)\dd t\, \dd y\, \dd x\nonumber\\
&\stackrel{\eqref{proofsubopesti7}}{\leq} \ell^{\frac{d}{2}}\int_{\bb_{7\ell}}\fint_{\bb_{r_*(x)}(x)}\int_{2^j}^{2^{j+1}}\vert \nabla u(t,y)\vert\vert \nabla v^T(t,y)\vert\dd t\, \dd y\, \dd x.\label{substep222}
\end{align}
Then, from the decay \eqref{decayuproofsubop}, we have
\begin{equation}
\left(\int_{2^j}^{2^{j+1}}\fint_{\bb_{r_*(x)}(x)}\vert \nabla u(t,y)\vert^2\dd y\, \dd t\right)^{\frac{1}{2}}\lesssim 2^{-\frac{j}{2}}.
\label{ctrlr*nablauL2}
\end{equation}
Thus, by Cauchy-Schwarz's inequality and \eqref{fintvT2}, \eqref{substep222} turns into
\begin{align*}
&\left(\int_{\bb_{4\ell}}\left(\int_{2^j}^{2^{j+1}}\int_{\bb_{\ell}(x)}\vert \nabla u(t,y)\vert\vert\nabla v^T(t,y)\vert\dd y\, \dd t\right)^2\dd x\right)^{\frac{1}{2}}\\
&\lesssim 2^{-\frac{j}{2}}\ell^{\frac{d}{2}}\int_{\bb_{7\ell}}\left(\int_{2^j}^{2^{j+1}}\fint_{\bb_{r_*(x)}(x)}\vert \nabla v^T(t,y)\vert^2\dd y\, \dd t\right)^{\frac{1}{2}}\dd x\\
&\stackrel{\eqref{fintvT2}}{\lesssim}\ell^{\frac{d}{2}}\left(r^{\frac{d}{2}}_*(0)\int_{\bb_{7\ell}}\frac{\log(1+\frac{\vert x\vert}{r})}{(\vert x\vert+r)^{d}}\dd x+\int_{\bb_{7\ell}}\eta^{\frac{1}{2}}_{4r_*(x)}(x)\dd x+\int_{\bb_{7\ell}}g_{2r}(x)\dd x\right),
\end{align*}
and we deduce \eqref{substep220} for $\vert x\vert\leq 4\ell$, since $\int_{\bb_{7\ell}} g_{2r}(x)\dd x\lesssim 1$, $\eta^{\frac{1}{2}}_{4r_*(x)}(x)\lesssim r^{\frac{d}{2}+1}_*(x)(\vert x\vert+1)^{-d-1}$  and by summing over $j\in \llbracket 0, N-1\rrbracket$.\newline
\newline
\textbf{Regime $\ell<r_*(0)$. }We use the estimate \eqref{lem7step2} which holds for $\ell\geq \sqrt{T}$ by removing $r^d_*(rx)$ and $r^d_*(x)$ in each integral in the r.h.s (by using \eqref{decayuproofsubop} for $\ell\geq \sqrt{T}$) and we estimate $\ell^d$ by $r^d_*(0)$ as well as $r^{-\frac{d}{2}}$ by $1$.\newline
\newline
\textbf{Step 3. Proof of \eqref{subopsensiesti} and conclusion. }We have from the multiscale logarithmic Sobolev inequality in form of \eqref{SGinegp1}:
\begin{equation}
\left\langle \vert q_r(T)-\left\langle q_r(T)\right\rangle\vert^p\right\rangle^{\frac{1}{p}}\lesssim \sqrt{p}\left\langle \left(\int_{1}^{+\infty}\ell^{-d} \pi(\ell)\int_{\mathbb{R}^d}\vert\partial^{\text{fct}}_{x,\ell} q_r(T)\vert^2\dd x\, \dd \ell\right)^{\frac{p}{2}}\right\rangle^{\frac{1}{p}}.
\label{LSIsubop1}
\end{equation}
We follow the heuristic arguments done in \eqref{RefHeuristicLater:Eq3} and \eqref{RefHeuristicLater:Eq4} of the Section \ref{stratsec}. We split the integral over $\ell$ into two parts.
\begin{itemize}
\item[(i)]In the regime $\ell<\sqrt{T}$ we use \eqref{step1subopsensi}:
\begin{align}
\left\langle \left(\int_{1}^{\sqrt{T}}\ell^{-d} \pi(\ell)\int_{\mathbb{R}^d}\vert\partial^{\text{fct}}_{x,\ell} q_r(T)\vert^2\dd x\, \dd \ell\right)^{\frac{p}{2}}\right\rangle^{\frac{1}{p}}&\stackrel{\eqref{step1subopsensi}}{\lesssim}r^{-\frac{d}{2}}(1+\log(T)+\log(\tfrac{\sqrt{T}}{r}))\left\langle \left(\int_{1}^{\sqrt{T}}\ell^d \pi(\ell)\mathcal{C}_\star(r,\ell)\dd \ell\right)^{\frac{p}{2}}\right\rangle^{\frac{1}{p}}\nonumber\\
&\stackrel{\eqref{assumeMSPC}}{\leq}r^{-\frac{d}{2}}(1+\log(T)+\log(\tfrac{\sqrt{T}}{r}))\left(\int_{1}^{\sqrt{T}}\ell^{d-1-\beta}\left\langle \mathcal{C}^{p}_\star(r,\ell)\right\rangle^{\frac{1}{p}}\dd \ell\right)^{\frac{1}{2}}\nonumber\\
&\stackrel{\eqref{step1subopsensi}}{\lesssim_\gamma} (p^{\frac{d}{\beta\wedge d}}\mathds{1}_{\beta\neq d}+p^{1+\gamma}\mathds{1}_{\beta=d})\,r^{-\frac{d}{2}}(1+\log(T)+\log(\tfrac{\sqrt{T}}{r}))\mu_{\beta}(T)\label{regimeLSI1},
\end{align}
by the definition \eqref{defmubeta} of $\mu_{\beta}(T)$.
\item[(ii)]In the regime $\ell\geq \sqrt{T}$ we use \eqref{equivmomentr*}, \eqref{step2subopsensi}, \eqref{Step2MomentBoundOfTheRC} and the change of variable $\ell\mapsto \frac{\ell}{\sqrt{T}}$ and the fact that $r\leq \sqrt{T}$ in the last line:
\begin{align}
&\left\langle \left(\int_{\sqrt{T}}^{+\infty}\ell^{-d} \pi(\ell)\int_{\mathbb{R}^d}\vert\partial^{\text{fct}}_{x,\ell} q_r(T)\vert^2\dd x\, \dd \ell\right)^{\frac{p}{2}}\right\rangle^{\frac{1}{p}}\nonumber\\
&\stackrel{\eqref{step2subopsensi},\eqref{assumeMSPC}}{\lesssim}\log(T)\left\langle \left(\int_{\sqrt{T}}^{+\infty}\ell^{-1-\beta}(r^{2d}_*(0)+\mathcal{D}_{\star,1}(r,\ell))\dd \ell\right)^{\frac{p}{2}}\right\rangle^{\frac{1}{p}}+\log(\tfrac{\sqrt{T}}{r})\left\langle \mathcal{D}^{\frac{p}{2}}_{\star,2}(r)\right\rangle^{\frac{1}{p}}\left(\int_{\sqrt{T}}^{+\infty} \ell^{-1-\beta}\dd \ell\right)^{\frac{1}{2}}\nonumber\\
&\lesssim T^{-\frac{\beta}{4}}\log(T)\left(\int_{1}^{+\infty}\ell^{-1-\beta}\left\langle (r^d_*(0)+\mathcal{D}_{\star,1}(r,\ell\sqrt{T}))^p\right\rangle^{\frac{1}{p}}\dd \ell\right)^{\frac{1}{2}}+\log(\tfrac{\sqrt{T}}{r})\left\langle \mathcal{D}^{\frac{p}{2}}_{\star,2}(r)\right\rangle^{\frac{1}{p}}\nonumber\\
&\stackrel{\eqref{Step2MomentBoundOfTheRC}, \eqref{equivmomentr*}}{\lesssim_\gamma} r^{-\frac{d}{2}}\mu_{\beta}(T)\Big((p^{\frac{d+1}{\beta\wedge d}}\mathds{1}_{\beta\neq d}+p^{(1+\frac{1}{d})(1+\gamma)}\mathds{1}_{\beta=d})\log(T)(1+\log^2(\tfrac{\sqrt{T}}{r}))+(p^{\frac{d}{\beta\wedge d}}\mathds{1}_{\beta\neq d}+p^{1+\gamma}\mathds{1}_{\beta=d})\log(\tfrac{\sqrt{T}}{r})\Big).\label{regimeLSI2}
\end{align}
\end{itemize}
The combination of \eqref{regimeLSI1} and \eqref{regimeLSI2} gives the desired bound \eqref{subopsensiesti}.
\subsection{Proof of the main results}
\subsubsection{Proof of Theorem \ref{semigroup}: Fluctuations of the time dependent flux }\label{sectionremovelog}
We only give the argument for the flux $q(T)$, the computations for $\phi(T)$ are done by a straightforward adaptation of the argument of this proof and the ones to prove Lemmas \ref{functioderiv} and \ref{functioderiv2}.\newline
\newline
Our first goal is to remove the $\log(T)$ correction in the r.h.s of \eqref{subopsensiesti}, which will lead to \eqref{Sensitilem3}. To this aim, we first use the $\LL^2$-$\LL^1$ type estimate of Lemma \ref{cacciopo}, which allow us to make the link between the $r^{-\frac{d}{2}}$ decay of the fluctuations of $(q_r(T,\cdot))_{r\leq \sqrt{T}}$ proved in Proposition \ref{Sensitivitysubop} and the decay in $T$ of moments of $(\int_{\mathbb{R}^d}\eta_R\vert\nabla u(T,\cdot)\vert^2)_{R\geq\sqrt{T}}$. This yields an improvement on the decay in $T$, see \eqref{decayunewth12}. With this new decay in hand,  we are able to obtain optimal estimates in scaling. The price to pay in this step is a small loss of stochastic integrability. Our second goal is to prove estimate \eqref{sensiothertest}. This does not require new ideas and this is done by dominating carefully the terms in the derivative \eqref{functioderivothertest} and by using some estimates already established in the proof of \eqref{Sensitilem3}. 
\begin{proof}[Proof of \eqref{Sensitilem3}.]
We split the proof into three steps. In the following, $p\in [1,\infty)$ denotes an arbitrary stochastic integrability exponent.\newline
\newline
\textbf{Step 1. Improvement of \eqref{decayuproofsubop}. }Proof that for all $x\in\mathbb{R}^d$, $T\geq 1$ and $\varepsilon\in (0,1)$
\begin{equation}
\int_{\bb_{\ell}(x)}\vert\nabla u(T,y)\vert^2\dd y\leq\ell^d\left(r^d_*(x)\mathds{1}_{\ell<\sqrt{T}}+\mathds{1}_{\ell\geq \sqrt{T}}\right)\mathcal{D}^{2\varepsilon}_{\star}(T,\ell,x)\eta^2_{\varepsilon,\beta}(T),
\label{decayunewth12}
\end{equation}
with
\begin{equation}
\eta_{\varepsilon,\beta}(T) = \left\{
    \begin{array}{ll}
        \log^{\varepsilon}(T)T^{-1-\varepsilon\frac{\beta}{4}} & \text{ if $\beta<d$}, \\
        \log^{\frac{3}{2}\varepsilon}(T)T^{-1-\varepsilon \frac{d}{4}} & \text{ if $\beta=d$},\\
			\log^{\varepsilon}(T)T^{-1-\varepsilon\frac{d}{4}} & \text{ if $\beta>d$},
   \end{array}
\right.
\label{defetaespsibeta}
\end{equation}
and for some stationary random field $\mathcal{D}_\star(T,\ell,\cdot)$ that satisfies for any $\gamma>0$ 
\begin{equation}
\sup_{(T,\ell,x)\in\mathbb{R}^+\times [1,\infty)\times \mathbb{R}^d}\left\langle \mathcal{D}^p_\star(T,\ell,x)\right\rangle^{\frac{1}{p}}\lesssim_\gamma p^{\alpha_\gamma},
\label{momentboundDstar}
\end{equation}
where $\alpha_\gamma$ as in Proposition \ref{Sensitivitysubop}. We have from Lemma \ref{cacciopo}, Minkowski's inequality in $\LL^p_{\left\langle\cdot\right\rangle}(\Omega)$ and the stationarity of $q_r$: for all $T\geq 4$ and $R\geq \sqrt{T}$
\begin{align}
\bigg\langle \Big(\int_{\mathbb{R}^d}\eta_{\sqrt{2}R}(y)\vert\nabla u(T,y)\vert^2\dd y\Big)^{\frac{p}{2}}\bigg\rangle^{\frac{1}{p}}&\lesssim\frac{1}{T}\left\langle \left(\fint_{\frac{T}{4}}^{\frac{T}{2}}\fint_{0}^{\sqrt{t}}\left(\frac{r}{\sqrt{t}}\right)^{\frac{d}{2}}\int_{\mathbb{R}^d}\eta_{2R}(y)\vert q_r(t,y)-\langle q_r(t,y)\rangle\vert\dd y\, \dd r\, \dd t\right)^{p}\right\rangle^{\frac{1}{p}}\nonumber\\
&\lesssim \frac{1}{T}\fint_{\frac{T}{4}}^{\frac{T}{2}}\fint_{0}^{\sqrt{t}}\left(\frac{r}{\sqrt{t}}\right)^{\frac{d}{2}}\left\langle \vert q_r(t)-\langle q_r(t)\rangle\vert^p\right\rangle^{\frac{1}{p}}\dd r\, \dd t.\label{th1estisensi1}
\end{align}
Then, we split the integral over $[0,\sqrt{t}]$ into the two contributions $r\leq 1$ and $1\leq r\leq \sqrt{t}$:
\begin{itemize}
\item[(i)]For $r\leq 1$ we use \eqref{near0est4} and the change of variable $r\mapsto \tfrac{\sqrt{T}}{r}$:
\begin{align}
\frac{1}{T}\fint_{\frac{T}{4}}^{\frac{T}{2}}\frac{1}{\sqrt{t}}\int_{0}^{1}\left(\frac{r}{\sqrt{t}}\right)^{\frac{d}{2}}\left\langle \vert q_r(t)-\langle q_r(t)\rangle\vert^p\right\rangle^{\frac{1}{p}}\dd r\, \dd t&\stackrel{\eqref{near0est4}}{\lesssim}p^{\frac{1}{\eta_\gamma}}T^{-1-\frac{d}{4}}\fint_{\frac{T}{4}}^{\frac{T}{2}}\frac{1}{\sqrt{t}}\int_{0}^{1}\left(r^{\frac{d}{2}}+\log(\tfrac{\sqrt{T}}{r})\right)\dd r\, \dd t\nonumber\\
&\lesssim p^{\frac{1}{\eta_\gamma}}T^{-1-\frac{d}{4}}\left(\int_{\sqrt{T}}^{+\infty}r^{-2}(1+\log(r))\dd r+\frac{1}{\sqrt{T}}\right)\nonumber\\
&\lesssim p^{\frac{1}{\eta_\gamma}}T^{-\frac{3}{2}-\frac{d}{4}}\log(T)\label{th1estisensi5}.
\end{align}
\item[(ii)]For $1\leq r\leq \sqrt{t}$ we use \eqref{subopsensiesti} and the change of variable $r\mapsto \tfrac{\sqrt{T}}{r}$: 
\begin{align}
\frac{1}{T}\fint_{\frac{T}{4}}^{\frac{T}{2}}\frac{1}{\sqrt{t}}\int_{\frac{1}{4}}^{\sqrt{t}}\left(\frac{r}{\sqrt{t}}\right)^{\frac{d}{2}}\left\langle \vert q_r(t)-\langle q_r(t)\rangle\vert^p\right\rangle^{\frac{1}{p}}\dd r\, \dd t&\stackrel{\eqref{subopsensiesti}}{\lesssim}p^{\alpha_\gamma}\log(T)T^{-1-\frac{d}{4}}\mu_{\beta}(T)\fint_{\frac{T}{4}}^{\frac{T}{2}}\frac{1}{\sqrt{t}}\int_{1}^{\sqrt{t}}\log^2(\tfrac{\sqrt{T}}{r})\dd r\, \dd t\nonumber\\
&\leq p^{\alpha_\gamma}\log(T)T^{-\frac{1}{2}}\eta_{\beta}(T)\int_{1}^{\sqrt{\frac{T}{2}}}r^{-2}\log^2(r)\dd r\nonumber\\
&\lesssim p^{\alpha_\gamma}\log(T)T^{-\frac{1}{2}}\eta_{\beta}(T)\label{th1estisensi6},
\end{align}
where $\eta_{\beta}$ is defined in \eqref{defetabeta}.
\end{itemize}
The combination of \eqref{th1estisensi1}, \eqref{th1estisensi5} and \eqref{th1estisensi6} yields that for all $x\in\mathbb{R}^d$ and $R\geq \sqrt{T}$
\begin{equation}
\fint_{\bb_{R}(x)}\vert \nabla u(T,y)\vert^2\dd y\lesssim\tilde{\mathcal{D}}^2_\star(T,R,x)\log^2(T)\eta^2_{\beta}(T),
\label{decaynablaulargescale}
\end{equation}
where $\tilde{\mathcal{D}}_\star(T,R,\cdot):=\frac{\fint_{\bb_R(\cdot)}\vert\nabla u(T,y)\vert^2\dd y}{\log^2(T)T^{-1}\eta^2_{\beta}(T)}$ has the moment bound \eqref{momentboundDstar}. This implies, from Lemma \ref{ctrlav} applied with $f: t\in\mathbb{R}_*^+\mapsto \log^2(t)$ and $g:t\in\mathbb{R}_*^+\mapsto t^{-2-\frac{d}{2}}$ that, for all $\ell<\sqrt{T}$
\begin{equation}
\int_{\bb_{\ell}(x)}\vert \nabla u(T,y)\vert^2\dd y\lesssim (r_*(x)\vee \ell)^d \Big(\tilde{\mathcal{D}}^2_\star(T,\sqrt{T},x)\vee \fint_{\frac{T}{2}}^{T}\tilde{\mathcal{D}}^2_\star(s,\sqrt{s},x)\dd s\Big)\log^2(T)T^{-1}\eta^2_{\beta}(T).
\label{decayunewth1}
\end{equation}
By interpolating between \eqref{decayuproofsubop} and the combination of \eqref{decaynablaulargescale} and \eqref{decayunewth1} as well as using that $\frac{r_*(x)}{\ell}\vee 1\leq 2r_*(x)$ in the last line, we finally obtain for all $\varepsilon\in (0,1)$
\begin{align*}
\int_{\bb_{\ell}(x)}\vert\nabla u(T,y)\vert^2\dd y&= \left(\int_{\bb_{\ell}(x)}\vert\nabla u(T,y)\vert^2\dd y\right)^{1-\varepsilon}\left(\int_{\bb_{\ell}(x)}\vert\nabla u(T,y)\vert^2\dd y\right)^{\varepsilon}\\
&\stackrel{\eqref{decayuproofsubop},\eqref{decayunewth1}}{\lesssim}\ell^d\left(\left(\frac{r_*(x)}{\ell}\vee 1\right)^{d}\mathds{1}_{\ell<\sqrt{T}}+\mathds{1}_{\ell\geq\sqrt{T}}\right)T^{-2(1-\varepsilon)}\mathcal{D}^{2\varepsilon}_{\star}(T,\ell,x)\log^{2\varepsilon} (T)T^{-\varepsilon}\eta^{2\varepsilon}_{\beta}(T)\\
&\leq 2\ell^d\left(r^d_*(x)\mathds{1}_{\ell<\sqrt{T}}+\mathds{1}_{\ell\geq\sqrt{T}}\right)\mathcal{D}^{2\varepsilon}_{\star}(T,\ell,x) \eta^2_{\varepsilon,\beta}(T),
\end{align*}
with $\eta_{\varepsilon, \beta}$ defined in \eqref{defetaespsibeta} and $\mathcal{D}_\star(T,\ell,x):=\max\left\{\tilde{\mathcal{D}}_{\star}(T,\ell,x),\tilde{\mathcal{D}}^2_\star(T,\sqrt{T},x)\vee \fint_{\frac{T}{2}}^{T}\tilde{\mathcal{D}}^2_\star(s,\sqrt{s},x)\dd s\right\}$ which satisfies the moment bound \eqref{momentboundDstar}. \newline
\newline
\textbf{Step 2. }Equipped with \eqref{decayunewth12}, we improve the estimates \eqref{step1subopsensi} and \eqref{step2subopsensi} for the control of $(x,\ell)\mapsto\int_{\mathbb{R}^d}\vert\partial^{\text{fct}}_{x,\ell} q_r(T)\vert^2\dd x$ (corresponding to the substeps $1$ and $2$ of the proof of Lemma \ref{Sensitivitysubop}). We split this step into two parts, one for the improvement of \eqref{step1subopsensi} and an other for \eqref{step2subopsensi}, treating separately the two regimes $\ell<\sqrt{T}$ and $\ell\geq \sqrt{T}$.\newline
\newline
\textbf{Substep 2.1. Improvement of \eqref{step1subopsensi}. }Proof that for all $\varepsilon>0$ and $\ell<\sqrt{T}$
\begin{equation}
\int_{\mathbb{R}^d}\vert\partial^{\text{fct}}_{x,\ell} q_r(T)\vert^2\dd x\leq \ell^{2d} r^{-d}\bigg(1+\mathcal{E}_{\star,\varepsilon}(r,\ell)+\Big(\sum_{j=0}^{+\infty}2^{\frac{j}{2}}\eta_{\varepsilon,\beta}(2^j)\mathcal{F}_{j,\star,\varepsilon}(r,\ell)\Big)^2\bigg),
\label{newsensith1}
\end{equation}
where
\begin{equation}
\mathcal{E}_{\star,\varepsilon}(r,\ell):=\int_{\mathbb{R}^d}r^d_*(rx) \left(\int_{\frac{1}{16}}^{+\infty}\mathcal{D}^{\varepsilon}_{\star}(t,\ell,rx)\eta_{\varepsilon,\beta}(t)\dd t\right)^2 \fint_{\bb_{\frac{\ell}{r}}(x)}g^2_1(y)\dd y\, \dd x,
\label{mathcalEnewestith1}
\end{equation}
and
\begin{equation}\label{mathcalFnewestith1}
\begin{aligned}
&\mathcal{F}_{j,\star,\varepsilon}(r,\ell)\\
&:=\bigg(\int_{\mathbb{R}^d}r^d_*(x)\Big(\int_{2^j}^{2^{j+1}} \mathcal{D}^{2\varepsilon}_\star(t,\ell,x)\dd t\Big)\\
&\times\bigg(\int_{\mathbb{R}^d}\Big(\frac{\mathds{1}_{\bb_\ell(x)}(y)}{\ell^d}+\frac{r^{d+\varepsilon}_*(x)\ell^{\varepsilon}}{\vert y-x\vert^{d+\varepsilon}}\mathds{1}_{\mathbb{R}^d\backslash \bb_{\ell}(x)}\Big)r^d_*(0)r^d\frac{\log^2(1+\frac{\vert y\vert}{r})}{(\vert y\vert +r)^{2d}}\dd y+\frac{r^d_*(y)}{(\vert y\vert+1)^{2d}}+r^d g^2_{2r}(y)\dd y\bigg)\dd x\bigg)^{\frac{1}{2}}.
\end{aligned}
\end{equation}
The estimate \eqref{lem66} is unchanged and gives the first contribution in \eqref{newsensith1}. We improve the estimates \eqref{lem65} and \eqref{lem7step2} (corresponding to the estimate of the second and last r.h.s term of \eqref{lem6}, respectively). On the one hand, noticing that from Minkwoski's inequality in $\LL^2(\mathbb{R}^d)$ and \eqref{decayunewth12} as well as $r\geq \frac{1}{2}$, we have for all $x\in\mathbb{R}^d$
$$\int_{\bb_{\ell}(x)}\left\vert\int_{r^2}^{T}\nabla u(t,y)\dd t\right\vert^2\dd y\leq \left(\int_{r^2}^T\left(\int_{\bb_{\ell}(x)}\vert\nabla u(t,y)\vert^2\dd y\right)^{\frac{1}{2}}\dd t\right)^2\stackrel{\eqref{decayunewth12}}{\lesssim} \ell^d r^d_*(x)\left(\int_{\frac{1}{16}}^{+\infty}\mathcal{D}^{\varepsilon}_\star(t,\ell,x)\eta_{\varepsilon,\beta}(t)\dd t\right)^2,$$
thus we deduce from \eqref{lem67} and \eqref{lowerenergy1} that
\begin{align}
\left(\int_{\mathbb{R}^d}\left(\int_{\bb_{\ell}(x)}\left\vert\int_{0}^T\nabla u(t,y)\dd t\right\vert g_r(y)\dd y\right)^2\dd x\right)^{\frac{1}{2}}&\stackrel{\eqref{lem67},\eqref{lowerenergy1}}\leq \ell^d r^{-\frac{d}{2}}+\left(\int_{\mathbb{R}^d}\Big(\int_{\bb_{\ell}(x)}\left\vert\int_{r^2}^T\nabla u(t,y)\dd t\right\vert^2\dd y\Big)\Big(\int_{\bb_{\ell}(x)} g^2_r(y)\dd y\Big) \dd x\right)^{\frac{1}{2}}\nonumber\\
&\stackrel{\eqref{decayunewth12}}{\lesssim} \ell^d r^{-\frac{d}{2}}+\ell^{\frac{d}{2}}\left(\int_{\mathbb{R}^d} r^d_*(x)\mathcal{J}_{\star,\varepsilon}(x)\int_{\bb_{\ell}(x)} g^2_r(y)\dd y\,\dd x\right)^{\frac{1}{2}},\label{matcalEnewesti1}
\end{align}
where 
$$\mathcal{J}_{\star,\varepsilon}(x):=\left(\int_{\frac{1}{16}}^{+\infty}\mathcal{D}^{\varepsilon}_\star(t,\ell,x)\eta_{\varepsilon,\beta}(t)\dd t\right)^2.$$
Using the change of variables $x\mapsto \frac{x}{r}$ in the last r.h.s term of \eqref{matcalEnewesti1}, this gives the term $\mathcal{E}_{\star,\varepsilon}(r,\ell)$ defined in \eqref{mathcalEnewestith1}.
On the other hand, noticing that, by monotonicity of $t\in\mathbb{R}^+_*\mapsto \eta_{\varepsilon,\beta}(t)$, for all $j\in\mathbb{N}$, we have
\begin{equation}
\int_{2^j}^{2^{j+1}}\mathcal{D}^{2\varepsilon}_{\star}(t,\ell,x)\eta^2_{\varepsilon,\beta}(t)\dd t\lesssim \eta^2_{\varepsilon,\beta}(2^j)\int_{2^j}^{2^{j+1}}\mathcal{D}^{2\varepsilon}_\star(t,\ell,x)\dd t,
\label{monotonocityeta}
\end{equation}
we deduce from \eqref{lem71uselater} and \eqref{decayunewth12}
\begin{equation}\label{th1estisensi2}
\begin{aligned}
&\left(\int_{\mathbb{R}^d}\left(\int_{1}^T\int_{\bb_{\ell}(x)}\vert \nabla u(t,y)\vert\vert \nabla v^T(t,y)\vert\dd y\,\dd t\right)^2\dd x\right)^{\frac{1}{2}}\\
&\lesssim\ell^{\frac{d}{2}}\sum_{j=0}^{+\infty}\eta_{\varepsilon,\beta}(2^j)\left(\int_{\mathbb{R}^d}r^d_*(x)\left(\int_{2^j}^{2^{j+1}}\mathcal{D}^{2\varepsilon}_\star(t,\ell,x)\dd t\right)\int_{2^j}^{2^{j+1}}\int_{\bb_{\ell}(x)}\vert \nabla v^T(t,y)\vert^2\dd y\, \dd t\,\dd x\right)^{\frac{1}{2}}.
\end{aligned}
\end{equation}
It remains to control the r.h.s integral of \eqref{th1estisensi2} by $2^{j}r^{-d}\mathcal{F}_{j,\star,\varepsilon}(r,\ell)$, where $\mathcal{F}_{j,\star,\varepsilon}(r,\ell)$ is defined in \eqref{mathcalFnewestith1}. To this aim, we note that, by the $\frac{1}{8}$-Lipschitz property of $r_*$, we have for all $(x,y)\in\mathbb{R}^{2d}$
\begin{equation}
\bb_{\ell}(x)\cap \bb_{r_*(y)}(y)\neq \emptyset\quad \Rightarrow\quad \vert y-x\vert\leq 3 (r_*(x)\vee \ell).
\label{th1capr*}
\end{equation}
Indeed if there exists $z\in \bb_{\ell}(x)\cap \bb_{r_*(y)}(y)$, we have by the triangle inequality
$$\vert y-x\vert\leq \vert y-z\vert+\vert z-x\vert\leq r_*(y)+\ell\leq r_*(x)+\frac{1}{8}\vert y-x\vert+\ell,$$
and thus \eqref{th1capr*} holds. Then, we use the property \eqref{averager*ctrl} and the decomposition, for all $x\in\mathbb{R}^d$, $\int_{\mathbb{R}^d}=\int_{\bb_{\ell}(x)}+\int_{\mathbb{R}^d\backslash \bb_{\ell}(x)}$, to obtain
\begin{align}
\int_{2^j}^{2^{j+1}}\int_{\bb_{\ell}(x)}\vert \nabla v^T(t,y)\vert^2\dd y\, \dd t\stackrel{\eqref{averager*ctrl}}{\lesssim}&\int_{2^j}^{2^{j+1}}\int_{\mathbb{R}^d}\fint_{\bb_{r_*(y)}(y)}\mathds{1}_{\bb_{\ell}(x)}(z)\vert\nabla v^T(t,z)\vert^2\dd z\, \dd y\, \dd t\nonumber\\
\leq &\int_{2^j}^{2^{j+1}}\int_{\bb_{\ell}(x)}\fint_{\bb_{r_*(y)}(y)}\vert\nabla v^T(t,z)\vert^2\dd z\, \dd y\, \dd t\nonumber\\
&+\int_{2^j}^{2^{j+1}}\int_{\mathbb{R}^d\backslash \bb_{\ell}(x)}\fint_{\bb_{r_*(y)}(y)}\mathds{1}_{\bb_{\ell}(x)}(z)\vert\nabla v^T(t,z)\vert^2\dd z\, \dd y\, \dd t.\label{proofth1esti1}
\end{align}
Next, we make use of \eqref{th1capr*} to bound the second r.h.s term of \eqref{proofth1esti1} with 
\begin{equation}\label{proofth1esti2}
\begin{aligned}
&\int_{2^j}^{2^{j+1}}\int_{\mathbb{R}^d\backslash \bb_{\ell}(x)}\fint_{\bb_{r_*(y)}(y)}\mathds{1}_{\bb_{\ell}(x)}(z)\vert\nabla v^T(t,z)\vert^2\dd z\, \dd y\, \dd t\\
&\lesssim 
(r_*(x)\vee \ell)^{d+\varepsilon}\int_{2^j}^{2^{j+1}}\int_{\mathbb{R}^d\backslash \bb_{\ell}(x)}\vert y-x\vert^{-d-\varepsilon}\fint_{\bb_{r_*(y)}(y)}\vert\nabla v^T(t,z)\vert^2\dd z\, \dd y\, \dd t.
\end{aligned}
\end{equation}
The combination of \eqref{proofth1esti1} and \eqref{proofth1esti2} (where we bound $(r_*(x)\vee \ell)^{d+\varepsilon}\lesssim \ell^{d+\varepsilon}r^{d+\varepsilon}_*(x)$) as well as \eqref{fintvT2} (where we bound $\eta_{4r_*(x)}(x)\lesssim r^d_*(x)(\vert x\vert+1)^{-2d}$) proves that the r.h.s integral of \eqref{th1estisensi2} is indeed bounded by $2^j r^{-d}\mathcal{F}_{j,\star,\varepsilon}(r,\ell)$ and this concludes the argument for \eqref{newsensith1}.\newline
\newline
\textbf{Substep 2.2. Improvement of \eqref{step2subopsensi}. }Proof that for all $\varepsilon>0$ and $\ell\geq \sqrt{T}$
\begin{equation}
\int_{\mathbb{R}^d}\vert \partial^{\text{fct}}_{x,\ell} q_r(T)\vert^2\dd x\lesssim\ell^d\bigg(\mathcal{D}_{\star,1}(r,\ell)+\mathcal{D}_{\star,2}(r)\log^2(\tfrac{\sqrt{T}}{r})+\Big(\sum_{j=0}^{+\infty}2^{\frac{j}{2}}\eta_{\varepsilon,\beta}(2^j)\mathcal{H}_{j,\star,\varepsilon}(r,\ell)\Big)^2\bigg),
\label{newsensith1largel}
\end{equation}
with $\mathcal{D}_{\star,1}(r,\ell)$ and $\mathcal{D}_{\star,2}(r)$ as in \eqref{Step2MomentBoundOfTheRC},
\begin{align}
\mathcal{H}_{j,\star,\varepsilon}(r,\ell)=&\ell^d\mathcal{H}_{j,\star,\varepsilon}(r,\ell,\mathbb{R}^d\backslash \bb_{4\ell})+r^d_*(0)\mathcal{H}_{j,\star,\varepsilon}(r,\ell,\mathbb{R}^d)\nonumber\\
&+\int_{\bb_{7\ell}}\left(\int_{2^j}^{2^{j+1}} \mathcal{D}^{2\varepsilon}_{\star}(t,\ell,x)\dd t\right)^{\frac{1}{2}}\left(r^{\frac{d}{2}}_*(0)\frac{\log(1+\frac{\vert x\vert}{r})}{(\vert x\vert +r)^d}+\frac{r^{\frac{d}{2}+1}_*(x)}{(\vert x\vert+1)^{d+1}}+g_{2r}(x)\right)\dd x,\label{cstmathcalH}
\end{align}
as well as for all open subsets $\mathcal{U}$ of $\mathbb{R}^d$
\begin{align*}
\mathcal{H}_{j,\star,\varepsilon}(r,\ell,\mathcal{U})=&\bigg(\int_{\mathcal{U}}\Big(\int_{2^j}^{2^{j+1}} \mathcal{D}^{2\varepsilon}_\star(t,\ell,x)\dd t\Big)\\
&\times\bigg(\int_{\mathbb{R}^d}\Big(\frac{\mathds{1}_{\bb_\ell(x)}(y)}{\ell^d}+\frac{r^{d+\varepsilon}_*(x)\ell^{\varepsilon}}{\vert y-x\vert^{d+\varepsilon}}\mathds{1}_{\mathbb{R}^d\backslash \bb_{\ell}(x)}\Big)r^d_*(0)r^d\frac{\log^2(1+\frac{\vert y\vert}{r})}{(\vert y\vert +r)^{2d}}\dd y+\frac{r^d_*(y)}{(\vert y\vert+1)^{2d}}+r^d g^2_{2r}(y)\dd y\bigg)\dd x\bigg)^{\frac{1}{2}}.
\end{align*}
The estimates \eqref{substep211} and \eqref{substep22} are unchanged and give respectively the $\ell^d\mathcal{D}_{\star,2}(r)\log^2(\tfrac{\sqrt{T}}{r})$ and $\ell^d\mathcal{D}_{\star,1}(r,\ell)$ contributions in the r.h.s of \eqref{newsensith1largel}. We improve the estimate \eqref{substep220}. We argue differently between the generic case $\ell\geq r_*(0)$ and the non-generic case $\ell<r_*(0)$.\newline
\newline
\textbf{Regime $\ell\geq r_*(0)$. }We have from \eqref{substep221}
$$\int_{\mathbb{R}^d}\left(\int_{1}^T \int_{\bb_{\ell}(x)}\vert\nabla u(t,y)\vert\vert\nabla v^T(t,y)\vert\dd y\, \dd t\right)^2\dd x\leq \left(\sum_{j=0}^{+\infty}\left(\int_{\mathbb{R}^d}\left(\int_{2^{j}}^{2^{j+1}}\int_{\bb_{\ell}(x)}\vert \nabla u(t,y)\vert\vert \nabla v^T(t,y)\vert\dd y\,\dd t\right)^2\dd x\right)^{\frac{1}{2}}\right)^2.$$
We then split the argument between the far-field regime $\vert x\vert\geq 4\ell$ and the near-field regime $\vert x\vert\leq 4\ell$. For the far-field regime, we use Cauchy-Schwarz's inequality combined with \eqref{decayunewth12} (applied for $\ell\geq\sqrt{T}$) and \eqref{monotonocityeta} to the effect of
\begin{align}
&\int_{\mathbb{R}^d\backslash \bb_{4\ell}}\bigg(\int_{2^{j}}^{2^{j+1}}\int_{B_{\ell}(x)}\vert \nabla u(t,y)\vert\vert \nabla v^T(t,y)\vert\dd y\,\dd t\bigg)^2\dd x\nonumber\\
&\leq \int_{\mathbb{R}^d}\bigg(\int_{2^{j}}^{2^{j+1}}\int_{\bb_{\ell}(x)}\vert\nabla u(t,y)\vert^2\dd y\, \dd t\bigg)\bigg(\int_{2^{j}}^{2^{j+1}}\int_{\bb_{\ell}(x)}\vert\nabla v^T(t,y)\vert^2\dd y\, \dd t\bigg)\dd x\nonumber\\
&\stackrel{\eqref{decayunewth12},\eqref{monotonocityeta}}{\lesssim}\ell^{d}\eta^2_{\varepsilon,\beta}(2^j)\int_{\mathbb{R}^d\backslash \bb_{4\ell}}\left(\int_{2^j}^{2^{j+1}}\mathcal{D}^{2\varepsilon}_{\star}(t,\ell,x)\dd t\right)\int_{2^{j}}^{2^{j+1}}\int_{\bb_{\ell}(x)}\vert\nabla v^T(t,y)\vert^2\dd y\, \dd t\, \dd x.\label{proofth1othertest5}
\end{align}
This then gives the first term of \eqref{cstmathcalH} by dominating $\int_{2^{j}}^{2^{j+1}}\int_{\bb_{\ell}(x)}\vert\nabla v^T(t,y)\vert^2\dd y\, \dd t$ using the arguments for \eqref{proofth1esti1} and \eqref{proofth1esti2}. For the near-field regime, we use \eqref{decayunewth12} and \eqref{monotonocityeta} in form of
\begin{equation}
\left(\int_{2^j}^{2^{j+1}}\fint_{\bb_{r_*(x)}(x)}\vert\nabla u(t,y)\vert^2\dd y\, \dd t\right)^{\frac{1}{2}}\stackrel{\eqref{decayunewth12}, \eqref{monotonocityeta}}{\lesssim} \eta_{\varepsilon,\beta}(2^j)\left(\int_{2^j}^{2^{j+1}}\mathcal{D}^{2\varepsilon}_{\star}(t,\ell,x)\dd t\right)^{\frac{1}{2}},
\label{monotoeta2}
\end{equation}
 which has the effect of, combined with \eqref{substep222} and Cauchy-Schwarz's inequality 
\begin{align}
&\left(\int_{\bb_{4\ell}}\left(\int_{2^j}^{2^{j+1}}\int_{\bb_{\ell}(x)}\vert\nabla u(t,y)\vert\vert\nabla v^T(t,y)\vert\dd y\, \dd t\right)^2\dd x\right)^{\frac{1}{2}}\nonumber\\
&\stackrel{\eqref{substep222}}{\lesssim}\ell^{\frac{d}{2}}\int_{\bb_{7\ell}}\fint_{\bb_{r_*(x)}(x)}\int_{2^j}^{2^{j+1}}\vert\nabla u(t,y)\vert\vert\nabla v^T(t,y)\vert\dd t\, \dd y\, \dd x\nonumber\\
&\stackrel{\eqref{monotoeta2}}{\lesssim} \ell^{\frac{d}{2}}\eta_{\varepsilon,\beta}(2^j)\int_{\bb_{7\ell}}\left(\int_{2^j}^{2^{j+1}}\mathcal{D}^{2\varepsilon}_{\star}(t,\ell,x)\dd t\right)^{\frac{1}{2}}\left(\fint_{\bb_{r_*(x)}(x)}\int_{2^j}^{2^{j+1}}\vert\nabla v^T(t,y)\vert\dd t\, \dd y\right)^{\frac{1}{2}}\dd x,\label{proofth1othertest6}
\end{align}
and yields the third term of \eqref{cstmathcalH} by using \eqref{fintvT2} (where we use $\eta_{4r_*(x)}(x)\lesssim r^d_*(x)(\vert x\vert+1)^{-2d}$).\newline
\newline
\textbf{Regime $\ell<r_*(0)$. }For the non-generic case $\ell<r_*(0)$, we use \eqref{th1estisensi2}, \eqref{proofth1esti1}, \eqref{proofth1esti2} and we bound one $\ell^d$ by $r^d_*(0)$ which gives the second term of \eqref{cstmathcalH}.
\newline
\newline
\textbf{Step 3. Proof of \eqref{momentboundlem3}. }We have from the multiscale logarithmic Sobolev inequality in form of \eqref{SGinegp1}, for all $p\in [1,\infty)$
\begin{equation}
\left\langle \vert q_r(T)-\left\langle q_r(T)\right\rangle\vert^p\right\rangle^{\frac{1}{p}}\lesssim \sqrt{p}\left\langle\left(\int_{1}^{+\infty}\ell^{-d}\pi(\ell)\int_{\mathbb{R}^d}\vert\partial^{\text{fct}}_{x,\ell}q_r(T)\vert^2\dd x\, \dd \ell\right)^{\frac{p}{2}}\right\rangle^{\frac{1}{p}}\leq \sqrt{p}(\mathcal{I}^1_{\sqrt{T}}+\mathcal{I}^2_{\sqrt{T}}),
\label{LSIop1}
\end{equation}
with
$$\mathcal{I}^1_{\sqrt{T}}:=\left\langle\left(\int_{1}^{\sqrt{T}}\ell^{-d}\pi(\ell)\int_{\mathbb{R}^d}\vert\partial^{\text{fct}}_{x,\ell}q_r(T)\vert^2\dd x\, \dd \ell\right)^{\frac{p}{2}}\right\rangle^{\frac{1}{p}}\quad \text{and}\quad \mathcal{I}^2_{\sqrt{T}}:=\left\langle\left(\int_{\sqrt{T}}^{+\infty}\ell^{-d}\pi(\ell)\int_{\mathbb{R}^d}\vert\partial^{\text{fct}}_{x,\ell}q_r(T)\vert^2\dd x\, \dd \ell\right)^{\frac{p}{2}}\right\rangle^{\frac{1}{p}}.$$
We then treat separately the two terms above.
\begin{itemize}
\item[(i)] In the regime $\ell<\sqrt{T}$ we use \eqref{newsensith1} combined with Minkowski's inequality in $\LL^p_{\left\langle\cdot\right\rangle}(\Omega)$:
\begin{align}
\mathcal{I}^1_{\sqrt{T}}&\stackrel{\eqref{newsensith1}}{\lesssim}  r^{-\frac{d}{2}}\left\langle \bigg(\int_{1}^{\sqrt{T}}\ell^d\pi(\ell)\bigg(1+\mathcal{E}_{\star,\varepsilon}(r,\ell)+\Big(\sum_{j=0}^{+\infty}2^{\frac{j}{2}}\eta_{\varepsilon,\beta}(2^j)\mathcal{F}_{j,\star,\varepsilon}(r,\ell)\Big)^2\bigg)\dd \ell\bigg)^{\frac{p}{2}}\right\rangle^{\frac{1}{p}}\nonumber\\
&\stackrel{\eqref{assumeMSPC}}{\leq}r^{-\frac{d}{2}}\bigg(\int_{1}^{\sqrt{T}}\ell^{d-1-\beta}\bigg(1+\left\langle \mathcal{E}^p_{\star,\varepsilon}(r,\ell)\right\rangle^{\frac{1}{p}}+\Big(\sum_{j=0}^{+\infty}2^{\frac{j}{2}}\eta_{\varepsilon,\beta}(2^j)\left\langle \mathcal{F}^{2p}_{j,\star,\varepsilon}(r,\ell)\right\rangle^{\frac{1}{2p}}\Big)^2\bigg)\dd\ell\bigg)^{\frac{1}{2}}\nonumber\\
&\lesssim r^{-\frac{d}{2}}\mu_{\beta}(T)\bigg(1+\sup_{\ell\geq 1}\left\langle \mathcal{E}^p_{\star,\varepsilon}(r,\ell)\right\rangle^{\frac{1}{2p}}+\sup_{\ell\geq 1}\sum_{j=0}^{+\infty}2^{\frac{j}{2}}\eta_{\varepsilon,\beta}(2^j)\left\langle \mathcal{F}^{2p}_{j,\star,\varepsilon}(r,\ell)\right\rangle^{\frac{1}{2p}}\bigg)\label{LSIop12},
\end{align}
by the definition \eqref{defmubeta} of $\mu_{\beta}$.
\item[(ii)]In the regime $\ell\geq \sqrt{T}$ we use \eqref{newsensith1largel}, Minkowski's inequality in $\LL^p_{\left\langle\cdot\right\rangle}(\Omega)$ and the fact that $r\leq \sqrt{T}$ in the last line:
\begin{align}
\mathcal{I}^2_{\sqrt{T}}\lesssim & \left\langle\bigg(\int_{\sqrt{T}}^{+\infty}\pi(\ell)\bigg(\mathcal{D}_{\star,1}(r,\ell)+\mathcal{D}_{\star,2}(r)\log^2(\tfrac{\sqrt{T}}{r})+\Big(\sum_{j=0}^{+\infty} 2^{\frac{j}{2}}\eta_{\varepsilon,\beta}(2^j)\mathcal{H}_{j,\star,\varepsilon}(r,\ell)\Big)^2\bigg)\dd\ell\bigg)^{\frac{p}{2}}\right\rangle^{\frac{1}{p}}\nonumber\\
\stackrel{\eqref{assumeMSPC}}{\lesssim}&\bigg(\int_{\sqrt{T}}^{+\infty}\ell^{-1-\beta}\bigg(\left\langle \mathcal{D}^{p}_{\star,1}(r,\ell)\right\rangle^{\frac{1}{p}}+\left\langle \mathcal{D}^{p}_{\star,2}(r)\right\rangle^{\frac{1}{p}}\log^{2}(\tfrac{\sqrt{T}}{r})+\Big(\sum_{j=0}^{+\infty}2^{\frac{j}{2}}\eta_{\varepsilon,\beta}(2^j)\left\langle \mathcal{H}^{2p}_{j,\star,\varepsilon}(r,\ell)\right\rangle^{\frac{1}{2p}}\Big)^2\bigg)\dd\ell\bigg)^{\frac{1}{2}}\nonumber\\
\lesssim &\, r^{-\frac{d}{2}}\mu_{\beta}(T)\bigg(\Big(\int_{1}^{+\infty}\ell^{-1-\beta}\left\langle \mathcal{D}^{p}_{\star,1}(r,\ell\sqrt{T}))\right\rangle^{\frac{1}{p}}\dd \ell\Big)^{\frac{1}{2}}+\log(\tfrac{\sqrt{T}}{r})\sup_{r\geq 1}\left\langle \mathcal{D}^{p}_{\star,2}(r)\right\rangle^{\frac{1}{2p}}\nonumber\\
&+\sup_{\ell\geq 1}\sum_{j=0}^{+\infty}2^{\frac{j}{2}}\eta_{\varepsilon,\beta}(2^j)\left\langle \mathcal{H}^{2p}_{j,\star,\varepsilon}(r,\ell)\right\rangle^{\frac{1}{2p}}\bigg)\label{LSIop13}.
\end{align}
\end{itemize}
It remains to show that, for all $\varepsilon\in (0,1)$
\begin{equation}
\sup_{\ell\geq 1}\left\langle \mathcal{E}^p_{\star,\varepsilon}(r,\ell)\right\rangle^{\frac{1}{2p}}+\sup_{\ell\geq 1}\sum_{j=0}^{+\infty}2^{\frac{j}{2}}\eta_{\varepsilon,\beta}(2^j)\left\langle \mathcal{F}^{2p}_{j,\star,\varepsilon}(r,\ell)\right\rangle^{\frac{1}{2p}}\lesssim p^{2\frac{d+1}{\beta\wedge d}+\varepsilon}(1+\log(\tfrac{\sqrt{T}}{r})),
\label{LSIop14}
\end{equation}
and 
\begin{equation}\label{LSIop15}
\begin{aligned}
&\left(\int_{1}^{+\infty}\ell^{-1-\beta}\left\langle \mathcal{D}^{p}_{\star,1}(r,\ell\sqrt{T}))\right\rangle^{\frac{1}{p}}\dd \ell\right)^{\frac{1}{2}}+\log(\tfrac{\sqrt{T}}{r})\sup_{r\geq 1}\left\langle \mathcal{D}^{p}_{\star,2}(r)\right\rangle^{\frac{1}{2p}}+\sup_{\ell\geq 1}\sum_{j=0}^{+\infty}2^{\frac{j}{2}}\eta_{\varepsilon,\beta}(2^j)\left\langle \mathcal{H}^{2p}_{j,\star,\varepsilon}(r,\ell)\right\rangle^{\frac{1}{2p}}\\
&\lesssim p^{2\frac{d+1}{\beta\wedge d}+\varepsilon}(1+\log^2(\tfrac{\sqrt{T}}{r})).
\end{aligned}
\end{equation}
This is a direct consequence of the combination of \eqref{momentboundDstar}, \eqref{equivmomentr*} and $\sum_{j=0}^{+\infty}2^j\eta_{\varepsilon,\beta}(2^j)\lesssim_{\varepsilon} 1$.
\end{proof}
We now turn to the proof of \eqref{sensiothertest}. The proof does not require new ideas and essentially uses estimates previously established.
\begin{proof}[Proof of \eqref{sensiothertest}.]We use the same notations as in the previous proof. We split the proof into two parts, treating separately the two regimes $\ell< r$ and $\ell\geq r$. We start with preliminary estimates.\newline
\newline
\textbf{Step 0. Preliminary.
}First, we will use several times the assumption \eqref{assumesensiothertest} in form of \eqref{derivothertestesti1prime}, that we restate here: 
\begin{equation}
\vert f_r(y)\vert\lesssim \frac{r}{(\vert y\vert+1)^{d}}\wedge \frac{1}{(\vert y\vert+1)^{d-1}}\quad \text{for all $y\in\mathbb{R}^d$}.
\label{derivothertestesti1primeprime}
\end{equation}
Second, we note that from \eqref{derivothertestesti1primeprime}, we have 
\begin{equation}
\int_{\mathbb{R}^d} \vert f_r(y)\vert^2\dd y\lesssim \int_{\bb_r}\frac{1}{(\vert y\vert+1)^{2(d-1)}}\dd y+r^2\int_{\mathbb{R}^d\backslash \bb_r}\frac{1}{(\vert y\vert+1)^{2d}}\dd y\lesssim \left\{
    \begin{array}{ll}
        1 & \text{ for $d\geq 3$}, \\
        \log(r+1) & \text{ for $d=2$,}
    \end{array}
\right.
\label{proofth1othertest12}
\end{equation}
and 
\begin{align}
\int_{\bb_{\ell}}\vert f_r(y)\vert\dd y&\lesssim \mathds{1}_{\ell<r}\int_{\bb_{\ell}}\frac{1}{(\vert y\vert+1)^{d-1}}\dd y+\mathds{1}_{\ell\geq r}\left(\int_{\bb_r}\frac{1}{(\vert y\vert+1)^{d-1}}\dd y+r\int_{\bb_{\ell}\backslash \bb_{r}}\frac{1}{(\vert y\vert+1)^d}\dd y\right)\nonumber\\
&\lesssim\ell\,\mathds{1}_{\ell<r}+r(1+\log(\tfrac{\ell}{r}+1))\mathds{1}_{\ell\geq r}.
\label{proofth1othertestesti55}
\end{align}
Therefore, by arguing the same way as for \eqref{proofsubopesti1}, using the system \eqref{dualequationlem7othertest} and the estimate \eqref{dualboundd2} instead of \eqref{lem2} (since from \eqref{derivothertestesti1primeprime}, \eqref{assumerhsdual2} holds) as well as the estimate \eqref{proofth1othertest12}, we have the following large-scale regularity estimate: for all $t\in (-\infty,r^2)$ and $x\in\mathbb{R}^d$
\begin{equation}
\fint_{\bb_{r_*(x)}(x)}\vert\nabla v^{r^2}(t,x)\vert^2\dd y\lesssim r^d_*(0)\frac{r^2\log^2(1+\vert x\vert)}{(\vert x\vert+1)^{2d}}\wedge \frac{1}{(\vert x\vert+1)^{2(d-1)}}+(1+\log(r+1)\mathds{1}_{d=2})\eta_{4r_*(x)}(x)+f^2_{2r}(x).
\label{proofth1othertest1}
\end{equation}
\textbf{Step 1. Regime $\ell<r$. }Proof that for all $\varepsilon>0$ and $\ell<r$
\begin{align}
\int_{\mathbb{R}^d}\vert\partial^{\text{fct}}_{x,\ell}q(r^2)\star f_r\vert^2\dd x\leq&\ell^d\bigg( r^d_*(0)(1+\log(r+1)\mathds{1}_{d=2})+\ell^2(1+\log(\frac{r}{\ell}+1)\mathds{1}_{d=2})+\mathcal{M}_{\star,1}(r,\ell)+\mathcal{M}_{\star,2}(r,\ell)\nonumber\\
&+\Big(\sum_{j=0}^{+\infty} 2^{\frac{j}{2}}\eta_{\varepsilon,\beta}(2^j)\mathcal{K}_{j,\star,\varepsilon}(r,\ell)\Big)^2\bigg),
\label{othertestesti1}
\end{align}
with 
\begin{equation}\label{proofth1othertest3}
\begin{aligned}
\mathcal{M}_{\star,1}(r,\ell):=&1+\mathcal{N}_{\star,\varepsilon}(r,\ell,1)+\bigg(\sum_{n=0}^{\left\lceil \log_2(3\ell)\right\rceil} 2^n\mathcal{N}^{\frac{1}{2}}_{\star,\varepsilon}(r,\ell,2^{n+1})\bigg)^{2}+\ell^d\sum_{n=\left\lfloor \log_2(\ell)\right\rfloor}^{\left\lceil \log_2(r)\right\rceil} 2^{-n(d-2)}\mathcal{N}_{\star,\varepsilon}(r,\ell,2^{n+1})\\
&+r^2\ell^d\sum_{n=\left\lfloor \log_2(r)\right\rfloor}^{+\infty}2^{-nd}\mathcal{N}_{\star,\varepsilon}(r,\ell,2^{n+1}),
\end{aligned}
\end{equation}
where for all $\rho>0$, $\mathcal{N}_{\star,\varepsilon}(r,\ell,\rho)=r^{d}_*(0)\left(\int_{1}^{r^2}\mathcal{D}^{\varepsilon}_{\star}(t,\rho,0)\eta_{\varepsilon,\beta}(t)\dd t\right)^2$ (where $\mathcal{D}_{\star}$ has the moment bound \eqref{momentboundDstar})  and 
\begin{align}
&\mathcal{M}_{\star,2}(r,\ell)\nonumber\\
&:=r^d_*(0)\ell^d\int_{\mathbb{R}^d\backslash \bb_{\ell}}\frac{r^2\log^2(1+\vert x\vert)}{(\vert x\vert+1)^{2d}}\wedge \frac{1}{(\vert x\vert+1)^{2(d-1)}}\dd x+(1+\log(r)\mathds{1}_{d=2})\ell^d\int_{\mathbb{R}^d\backslash \bb_{\ell}}\frac{r^d_*(x)}{(\vert x\vert+1)^{2d}}\dd x+\ell^d\int_{\mathbb{R}^d\backslash \bb_{\ell}} f^2_{2r}(x)\dd x\nonumber\\
&+r^d_*(0)\left(\int_{\bb_{7\ell}}\frac{r\log(1+\vert x\vert)}{(\vert x\vert+1)^{d}}\wedge \frac{1}{(\vert x\vert+1)^{d-1}}\dd x\right)^2+(1+\log(r)\mathds{1}_{d=2})\left(\int_{\bb_{7\ell}}\frac{r^{\frac{d}{2}+1}_*(x)}{(\vert x\vert+1)^{d+1}}\dd x\right)^2+\left(\int_{ \bb_{7\ell}} f_{2r}(x)\dd x\right)^2,
\label{proofth1othertest4}
\end{align}
as well as 
\begin{equation}\label{proofth1othertest7}
\begin{aligned}
&\mathcal{K}_{j,\star,\varepsilon}(r,\ell):=\mathcal{K}_{j,\star,\varepsilon}(r,\ell,\mathbb{R}^d\backslash \bb_{4\ell})+r^{\frac{d}{2}}_*(0)\mathcal{K}_{j,\star,\varepsilon}(r,\ell,\mathbb{R}^d)\\
&+\int_{\bb_{7\ell}}\left(\int_{2^j}^{2^{j+1}}\mathcal{D}^{2\varepsilon}_{\star}(t,\ell,x)\dd t\right)^{\frac{1}{2}}\left(r^{\frac{d}{2}}(0)\frac{r\log(1+\vert x\vert)}{(\vert x\vert+1)^d}\wedge \frac{1}{(\vert x\vert+1)^{d-1}}+(1+\log^{\frac{1}{2}}(r)\mathds{1}_{d=2})\frac{r^{\frac{d}{2}+1}_*(x)}{(\vert x\vert+1)^d}+ f_{2r}(x)\right)\dd x,
\end{aligned}
\end{equation}
with for all open set $\mathcal{U}$ of $\mathbb{R}^d$
\begin{align*}
&\mathcal{K}_{j,\star,\varepsilon}(r,\ell,\mathcal{U})\\
&:=\bigg(\int_{\mathcal{U}}\Big(\int_{2^j}^{2^{j+1}}\mathcal{D}^{2\varepsilon}_{\star}(t,\ell,x)\dd t\Big)^{\frac{1}{2}} \bigg(\int \Big(\frac{\mathds{1}_{\bb_{\ell}(x)}(y)}{\ell^d}+\frac{r_*^{d+\varepsilon}(x)\ell^{\varepsilon}}{\vert y-x\vert^{d+\varepsilon}}\mathds{1}_{\mathbb{R}^d\backslash \bb_{\ell}(x)}\Big)r^d_*(0)\frac{r^2\log^2(1+\vert y\vert)}{(\vert y\vert+1)^{2d}}\wedge \frac{1}{(\vert y\vert+1)^{2(d-1)}}\\
&+(1+\log(r)\mathds{1}_{d=2})\frac{r^d_*(y)}{(\vert y\vert+1)^{2d}}+ f^2_{2r}(y)\dd y\bigg)\dd x\bigg)^{\frac{1}{2}}.
\end{align*}
We split this step into two parts. The first part is devoted to the control of the first five r.h.s terms of \eqref{functioderivothertest}.\newline
\newline
\textbf{Substep 1.1. }Proof that for all $\ell\in [1,r)$
\begin{equation}
\int_{\mathbb{R}^d}\left(\int_{\bb_{\ell}(x)}\vert f_r(y)\vert\dd y\right)^2\dd x+\int_{\mathbb{R}^d}\left(\int_{\bb_{\ell}(x)}\left\vert\int_{0}^{r^2}\nabla u(t,y)\dd t\right\vert\vert f_r(y)\vert\dd y\right)^2\dd x\lesssim \ell^{d}\mathcal{M}_{\star,1}(r,\ell),
\label{othertestesti2}
\end{equation}
and 
\begin{equation}\label{othertestesti14}
\begin{aligned}
&\int_{\mathbb{R}^d}\mathcal{T}^2_{x,\ell}(\eta_1)(0)\dd x+\int_{\mathbb{R}^d}\left(\sum_{n=0}^{\left\lceil \log_2(3\ell)\right\rceil}2^n\mathcal{T}_{x,\ell}(\eta_{2^{n+1}})(0)\right)^2\dd x\\
&+\int_{\mathbb{R}^d}\left(\int_{\mathbb{R}^d\backslash \bb_{2\ell}}\left(\fint_{\bb_{\ell}(y)}\vert f_r(z)\vert^2\dd z\right)^{\frac{1}{2}}\mathcal{T}_{x,\ell}(\eta_\ell)(y)\dd y\right)^2\dd x\lesssim \ell^{d}(1+\ell^2(1+\log(\frac{r}{\ell}+1)\mathds{1}_{d=2})),
\end{aligned}
\end{equation}
as well as 
\begin{equation}
\int_{\mathbb{R}^d}\left(\int_{\bb_{\ell}(x)}\vert\nabla v^{r^2}(1,y)\vert\left(1+\left\vert\int_{0}^1\nabla u(s,y)\dd s\right\vert\right)\dd y\right)^2\dd x+\int_{\mathbb{R}^d}\mathcal{G}^2_{r,\ell}(x)\dd x\lesssim \ell^{d}(r^d_*(0)(1+\log(r+1)\mathds{1}_{d=2})+\mathcal{M}_{\star,2}(r,\ell)).
\label{othertestesti15}
\end{equation}
\textbf{Proof of \eqref{othertestesti2}. }We only give the argument for the second l.h.s term, the first one is dominated the same way. We split the argument into the far-field regime $\vert x\vert\geq 4\ell$ and the near-field regime $\vert x\vert<4\ell$. For the near-field regime, we make use of a dyadic decomposition and the estimate \eqref{derivothertestesti1primeprime} in form of $\vert f_r(y)\vert\lesssim (\vert y\vert+1)^{-d+1}$ to get
\begin{align*}
\int_{\bb_{2\ell}}\left(\int_{\bb_{\ell}(x)}\left\vert\int_{0}^{r^2}\nabla u(t,y)\dd t\right\vert\vert f_r(y)\vert \dd y\right)^{2}\dd x\lesssim&\, \ell^d\left(\int_{\bb_{3\ell}}\left\vert\int_{0}^{r^2}\nabla u(t,y)\dd t\right\vert\vert f_r(y)\vert \dd y\right)^{2}\\
\lesssim&\, \ell^d\bigg(\left(\int_{\bb_{1}}\left\vert\int_{0}^{r^2}\nabla u(t,y)\dd t\right\vert \dd y\right)^{2}\\
&+\left(\sum_{n=0}^{\left\lceil \log_2(3\ell)\right\rceil}\int_{\bb_{2^{n+1}}\backslash \bb_{2^n}}\left\vert\int_{0}^{r^2}\nabla u(t,y)\dd t\right\vert\vert f_r(y)\vert \dd y\right)^{2}\bigg)\\
\stackrel{\eqref{derivothertestesti1primeprime}}{\lesssim} &\ell^d\bigg(\left(\int_{\bb_{1}}\left\vert\int_{0}^{r^2}\nabla u(t,y)\dd t\right\vert \dd y\right)^{2}+\left(\sum_{n=0}^{\left\lceil \log_2(3\ell)\right\rceil}2^n\fint_{\bb_{2^{n+1}}}\left\vert\int_{0}^{r^2}\nabla u(t,y)\dd t\right\vert \dd y\right)^{2}\bigg),
\end{align*}
which gives the second  and third r.h.s terms of \eqref{proofth1othertest3} by applying the localized energy estimate \eqref{LemE1} to the equation \eqref{equationu} and the estimate \eqref{decayunewth12} (applied for both $\ell=1$ and $\ell=2^{n+1}$) in form of 
\begin{align*}
\left(\int_{\bb_{1}}\left\vert\int_{0}^{r^2}\nabla u(t,y)\dd t\right\vert \dd y\right)^{2}&\lesssim \left(\int_{\bb_{1}}\left\vert\int_{0}^{1}\nabla u(t,y)\dd t\right\vert \dd y\right)^{2}+\left(\int_{\bb_{1}}\left\vert\int_{1}^{r^2}\nabla u(t,y)\dd t\right\vert \dd y\right)^{2}\\
&\lesssim 1+r^{d}_*(0)\left(\int_{1}^{r^2}\mathcal{D}^{\varepsilon}_{\star}(t,1,0)\eta_{\varepsilon,\beta}(t)\dd t\right)^2,
\end{align*}
and for all $n\in\mathbb{N}$
\begin{equation}
\fint_{\bb_{2^{n+1}}}\left\vert\int_{0}^{r^2}\nabla u(t,y)\dd t\right\vert^2 \dd y\leq 1+r^{d}(0)\left(\int_{1}^{r^2}\mathcal{D}^{\varepsilon}_{\star}(t,2^{n+1},0)\eta_{\varepsilon,\beta}(t)\dd t\right)^2.
\label{proofth1othertestesti51}
\end{equation}
For the far-field contribution, we first make use of Jensen's inequality combined with the inequality $\int_{\mathbb{R}^d\backslash \bb_{2\ell}}\int_{\bb_{\ell}(x)}\dd x\lesssim \ell^d\int_{\mathbb{R}^d\backslash \bb_{\ell}}$ and the decomposition $\int_{\mathbb{R}^d\backslash \bb_{\ell}}=\int_{\bb_r\backslash \bb_{\ell}}+\int_{\mathbb{R}^d\backslash \bb_r}$ to obtain
\begin{align}
\int_{\mathbb{R}^d\backslash \bb_{2\ell}}\left(\int_{\bb_{\ell}(x)}\left\vert\int_{0}^{r^2}\nabla u(t,y)\dd t\right\vert\vert f_r(y)\vert\dd y\right)^2\dd x\lesssim & \ell^{2d}\int_{\mathbb{R}^d\backslash \bb_{\ell}}\left\vert\int_{0}^{r^2}\nabla u(t,y)\dd t\right\vert^2\vert f_r(y)\vert^2\dd y\nonumber\\
=&\ell^{2d}\bigg(\int_{\bb_r\backslash \bb_{\ell}}\left\vert\int_{0}^{r^2}\nabla u(t,y)\dd t\right\vert^2\vert f_r(y)\vert^2\dd y\nonumber\\
&+\int_{\mathbb{R}^d\backslash \bb_{r}}\left\vert\int_{0}^{r^2}\nabla u(t,y)\dd t\right\vert^2\vert f_r(y)\vert^2\dd y\bigg).\label{proofth1othertest50}
\end{align}
For the first r.h.s term of \eqref{proofth1othertest50}, we make use of a dyadic decomposition and the estimate \eqref{derivothertestesti1primeprime} in form of $\vert f_r(y)\vert\lesssim (\vert y\vert+1)^{-d+1}$ to get
\begin{align*}
\int_{\bb_r\backslash \bb_{\ell}}\left\vert\int_{0}^{r^2}\nabla u(t,y)\dd t\right\vert^2\vert f_r(y)\vert^2\dd y&=\sum_{\left\lfloor \log_2(\ell)\right\rfloor}^{\left\lceil \log_2(r)\right\rceil}\int_{\bb_{2^{n+1}}\backslash \bb_{2^n}}\left\vert\int_{0}^{r^2}\nabla u(t,y)\dd t\right\vert^2\vert f_r(y)\vert^2\dd y\\
&\stackrel{\eqref{derivothertestesti1primeprime}}{\lesssim} \sum_{\left\lfloor \log_2(\ell)\right\rfloor}^{\left\lceil \log_2(r)\right\rceil}2^{-n(d-2)}\fint_{\bb_{2^{n+1}}}\left\vert\int_{0}^{r^2}\nabla u(t,y)\dd t\right\vert^2\dd y,
\end{align*}
which gives the fourth r.h.s term of \eqref{proofth1othertest3} using \eqref{proofth1othertestesti51}. For the second r.h.s term of \eqref{proofth1othertest50}, we make use of a dyadic decomposition and \eqref{derivothertestesti1primeprime} in form of $\vert f_r(y)\vert\lesssim r(\vert y\vert+1)^{-d}$ to get
\begin{align*}
\int_{\mathbb{R}^d\backslash \bb_{r}}\left\vert\int_{0}^{r^2}\nabla u(t,y)\dd t\right\vert^2\vert f_r(y)\vert^2\dd y&=\sum_{n=\left\lfloor \log_2(r)\right\rfloor}^{+\infty} \int_{\bb_{2^{n+1}}\backslash \bb_{2^n}}\left\vert\int_{0}^{r^2}\nabla u(t,y)\dd t\right\vert^2\vert f_r(y)\vert^2\dd y\\
&\stackrel{\eqref{derivothertestesti1primeprime}}{\lesssim} r^2\sum_{n=\left\lfloor \log_2(r)\right\rfloor}^{+\infty} 2^{-nd}\fint_{\bb_{2^{n+1}}}\left\vert\int_{0}^{r^2}\nabla u(t,y)\dd t\right\vert^2\dd y,
\end{align*}
which finally gives fifth r.h.s of \eqref{proofth1othertest3} using once more \eqref{proofth1othertestesti51}.\newline
\newline
\textbf{Proof of \eqref{othertestesti14}. }The estimate of the first two l.h.s terms is an immediate consequence of Minkowski's inequality in $\LL^2(\mathbb{R}^d)$ and the estimate \eqref{othertestesti13} applied with $r=\sqrt{s}$, which provides the first r.h.s contribution in \eqref{othertestesti14}. For the third l.h.s term, we can use previous estimates. To this aim, we use Cauchy-Schwarz's inequality and \eqref{proofsubopesti9} as well as $\int_{\mathbb{R}^d}\mathcal{T}_{x,\ell}(\eta_{\ell})(y)\dd x\lesssim \ell^{d}$ (which may be obtained by changing the role of $x$ and $y$ in the proof of \eqref{proofsubopesti9}) combined with the inequality $\int_{\mathbb{R}^d\backslash \bb_{2\ell}}\fint_{\bb_{\ell}(y)}\dd y\lesssim \int_{\mathbb{R}^d\backslash \bb_{\ell}}$ to obtain
\begin{align*}
\int_{\mathbb{R}^d}\left(\int_{\mathbb{R}^d\backslash \bb_{2\ell}}\left(\fint_{\bb_{\ell}(y)}\vert f_r(z)\vert^2\dd z\right)^{\frac{1}{2}}\mathcal{T}_{x,\ell}(\eta_{\ell})(y)\dd y\right)^2\dd x&\stackrel{\eqref{proofsubopesti9}}{\lesssim}\ell^d\int_{\mathbb{R}^d}\int_{\mathbb{R}^d\backslash \bb_{2\ell}}\fint_{\bb_{\ell}(y)}\vert f_r(z)\vert^2\dd z\,\mathcal{T}_{x,\ell}(\eta_{\ell})(y)\dd y\,\dd x\\
&\lesssim\ell^{2d}\int_{\mathbb{R}^d\backslash\bb_{2\ell}}\fint_{\bb_{\ell}(y)}\vert f_r(z)\vert^2\dd z\, \dd y\\
&\lesssim\ell^{2d}\int_{\mathbb{R}^d\backslash \bb_{\ell}}\vert f_r(y)\vert^2\dd y.
\end{align*}
We then get \eqref{othertestesti14} using \eqref{derivothertestesti1primeprime} in form of 
\begin{align}
\int_{\mathbb{R}^d\backslash \bb_{\ell}}\vert f_r(y)\vert^2\dd y&\lesssim\int_{\bb_r\backslash \bb_{\ell}}(\vert y\vert+1)^{-2(d-1)}\dd y+r^2\int_{\mathbb{R}^d\backslash \bb_{r}}(\vert y\vert+1)^{-2d}\dd y\nonumber\\
&\lesssim \ell^{2-d}(1+\log(\tfrac{r}{\ell}+1)\mathds{1}_{d=2}),
\label{proofth1othertestesti52}
\end{align}
\textbf{Proof of \eqref{othertestesti15}. }We start with the first l.h.s term. We distinguish between the generic case $\ell\geq r_*(0)$ and the non-generic case $\ell<r_*(0)$.\newline
\newline
\textbf{Regime $\ell\geq r_*(0)$. }We split the integral into the far-field contribution $\vert x\vert\geq 4\ell$ and the near-field contribution $\vert x\vert<4\ell$. For the far-field contribution, we make use of the estimates \eqref{proofsubopesti14} (applied with $T=r^2$) combined with \eqref{proofth1othertest1} to obtain 
\begin{align*}
\int_{\mathbb{R}^d\backslash\bb_{4\ell}}\left(\int_{\bb_{\ell}(x)}\vert\nabla v^{r^2}(1,y)\vert\left(1+\left\vert\int_{0}^{1}\nabla u(s,y)\dd s\right\vert\right)\dd y\right)^2\dd x\lesssim &\, \ell^{2d}\bigg(r^d_*(0)\int_{\mathbb{R}^d\backslash\bb_{\ell}}\frac{r^2\log^2(1+\vert x\vert)}{(\vert x\vert+1)^{2d}}\wedge \frac{1}{(\vert x\vert+1)^{2(d-1)}}\dd x\\
&+(1+\log(r+1)\mathds{1}_{d=2})\int_{\mathbb{R}^d\backslash\bb_{\ell}}\eta_{4r_*(x)}(x)\dd x+\int_{\mathbb{R}^d\backslash\bb_{\ell}}f^2_{2r}(x)\dd x\bigg),
\end{align*}
which gives the first term in \eqref{proofth1othertest4} using that $\eta_{4r_*(x)}(x)\lesssim r^d_*(x)(\vert x\vert+1)^{-2d}$.
For the near-field contribution, we make use of the estimate \eqref{proofsubopesti11} combined with \eqref{proofth1othertest1} to obtain 
\begin{align*}
&\left(\int_{\bb_{4\ell}}\left(\int_{\bb_{\ell}(x)}\vert\nabla v^{r^2}(1,y)\vert\left(1+\left\vert\int_{0}^{1}\nabla u(s,y)\dd s\right\vert\right)\dd y\right)^2\dd x\right)^{\frac{1}{2}}\\
&\lesssim \ell^{\frac{d}{2}}\bigg(r^{\frac{d}{2}}_*(0)\int_{\bb_{7\ell}}\frac{r\log(1+\vert x\vert)}{(\vert x\vert+1)^{d}}\wedge \frac{1}{(\vert x\vert+1)^{d-1}}\dd x+(1+\log^{\frac{1}{2}}(r+1)\mathds{1}_{d=2})\int_{\bb_{7\ell}}\eta^{\frac{1}{2}}_{4r_*(x)}(x)\dd x+\int_{\bb_{7\ell}}f_{2r}(x)\dd x\bigg),
\end{align*}
which gives the second term of \eqref{proofth1othertest4} using that $\eta^{\frac{1}{2}}_{4r_*(x)}(x)\lesssim r^{\frac{d}{2}+1}_*(x)(\vert x\vert+1)^{-d-1}$.\newline
\newline
\textbf{Regime $\ell<r_*(0)$. }We use Cauchy-Schwarz's inequality, the identity $\int_{\mathbb{R}^d}\fint_{\bb_{\ell}(x)}=\int_{\mathbb{R}^d}$ and the localized energy estimate \eqref{LemE1} to get
\begin{align*}
\int_{\mathbb{R}^d}\left(\int_{\bb_{\ell}(x)}\vert\nabla v^{r^2}(1,y)\vert\left(1+\left\vert\int_{0}^{1}\nabla u(s,y)\dd s\right\vert\right)\dd y\right)^2\dd x&\lesssim \int_{\mathbb{R}^d}\int_{\bb_{\ell}(x)}\vert\nabla v^{r^2}(1,y)\vert^2\dd y\int_{\bb_{\ell}(x)}\left(1+\left\vert\int_{0}^{1}\nabla u(s,y)\dd s\right\vert^2\right)\dd y\,\dd x\\
&\lesssim \ell^{2d}\int_{\mathbb{R}^d}\vert\nabla v^{r^2}(1,y)\vert^2\dd y\lesssim \ell^d r^d_*(0)\int_{\mathbb{R}^d}\vert\nabla v^{r^2}(1,y)\vert^2\dd y,
\end{align*}
and we conclude with the plain energy estimate in the equation \eqref{dualequationlem7othertest} (for which a proof is identical as the one for \eqref{fullenergylem2}) combined with \eqref{proofth1othertest3} that
$$\int_{\mathbb{R}^d}\vert\nabla v^{r^2}(1,y)\vert^2\dd y\lesssim \int_{\mathbb{R}^d} \vert f_r(y)\vert^2\dd y\lesssim 1+\log(r+1)\mathds{1}_{d=2},$$
which gives the first contribution in \eqref{othertestesti1}. For the second l.h.s term of \eqref{othertestesti15}, we argue as in \eqref{proofsubopesti8} and \eqref{proofth1othertest5prime}.\newline
\newline
\textbf{Substep 1.2. }Proof that for all $\ell\in [1,r)$
\begin{equation}
\int_{\mathbb{R}^d}\left(\int_{\bb_{\ell}(x)}\int_{1}^{r^2}\vert\nabla u(t,y)\vert\vert\nabla v^{r^2}(t,y)\vert\dd t\, \dd y\right)^2\dd x\lesssim \ell^d\left(\sum_{j=0}^{+\infty} 2^{\frac{j}{2}}\eta_{\varepsilon,\beta}(2^j)\mathcal{K}_{j,\star,\varepsilon}(r,\ell)\right)^2.
\label{proofth1othertest9}
\end{equation}
We argue differently with the generic case $\ell\geq r_*(0)$ and the non-generic case $\ell<r_*(0)$ and we use several previous estimates.\newline
\newline
\textbf{Regime $\ell\geq r_*(0)$. }We split the argument between the far-field regime $\vert x\vert\geq 4\ell$ and the near-field regime $\vert x\vert<4\ell$. For the far-field contribution, we make use of \eqref{substep221}, \eqref{proofth1othertest5} and \eqref{proofth1esti1} as well as \eqref{proofth1esti2} combined with \eqref{proofth1othertest1} to obtain 
\begin{equation*}
\int_{\mathbb{R}^d\backslash \bb_{4\ell}}\left(\int_{\bb_{\ell}(x)}\int_{1}^{r^2}\vert\nabla u(t,y)\vert\vert\nabla v^{r^2}(t,y)\vert\dd t\, \dd y\right)^2\dd x\lesssim \left(\sum_{j=0}^{+\infty}2^{\frac{j}{2}}\eta_{\varepsilon,\beta}(2^j)\mathcal{K}_{j,\star,\varepsilon}(r,\ell,\mathbb{R}^d\backslash \bb_{4\ell})\right)^2.
\end{equation*}
For the near-field contribution, we make use of \eqref{proofth1othertest6} combined with \eqref{proofth1othertest1} which leads to
\begin{align*}
&\int_{\bb_{4\ell}}\left(\int_{\bb_{\ell}(x)}\int_{1}^{r^2}\vert\nabla u(t,y)\vert\vert\nabla v^{r^2}(t,y)\vert\dd t\, \dd y\right)^2\dd x\\
&\lesssim \left(\int_{\bb_{7\ell}}\left(\int_{2^j}^{2^{j+1}}\mathcal{D}^{2\varepsilon}_{\star}(t,\ell,x)\dd t\right)^{\frac{1}{2}}\left(r^{\frac{d}{2}}_*(0)\frac{r\log(1+\vert x\vert)}{(\vert x\vert+1)^d}\wedge \frac{1}{(\vert x\vert+1)^{d-1}}+(1+\log^{\frac{1}{2}}(r+1)\mathds{1}_{d=2})\frac{r^{\frac{d}{2}+1}_*(x)}{(\vert x\vert+1)^d}+ f_{2r}(x)\right)\dd x\right)^2,
\end{align*}
where we used that $\sum_{j=0}^{+\infty}2^{\frac{j}{2}}\eta_{\varepsilon,\beta}(2^j)<+\infty$.\newline
\newline
\textbf{Regime $\ell<r_*(0)$. }We make use of \eqref{th1estisensi2}, \eqref{proofth1esti1} and \eqref{proofth1esti2} combined with \eqref{proofth1othertest1} and we bound one $\ell^d$ by $r^d_*(0)$ to obtain 
$$\int_{\mathbb{R}^d}\left(\int_{\bb_{\ell}(x)}\int_{1}^{r^2}\vert\nabla u(t,y)\vert\vert\nabla v^{r^2}(t,y)\vert\dd t\, \dd y\right)^2\dd x\lesssim  r^d_*(0)\left(\sum_{j=0}^{+\infty}2^{\frac{j}{2}}\eta_{\varepsilon,\beta}(2^j)\mathcal{K}_{j,\star,\varepsilon}(r,\ell,\mathbb{R}^d)\right)^2.$$
\textbf{Step 2. Regime $\ell\geq r$. }Proof that for all $\varepsilon>0$ and $\ell\geq r$
\begin{equation}
\int_{\mathbb{R}^d}\vert\partial^{\text{fct}}_{x,\ell} q(r^2)\star f_r\vert^2\dd x\leq \ell^d\bigg(r^2+\mathcal{M}_{\star,3}(r)+\mathcal{M}_{\star,2}(r,\ell)+\Big(\sum_{j=0}^{+\infty} 2^{\frac{j}{2}}\eta_{\varepsilon,\beta}(2^j)\mathcal{K}_{j,\star,\varepsilon}(r,\ell)\Big)^2\bigg),
\label{proofth1othertest8}
\end{equation}
with 
$$\mathcal{M}_{\star,3}(r)=\bigg(\int_{1}^{r^2}s^{-\frac{1}{2}}(1+\log(\tfrac{r}{\sqrt{s}}))\Big(\fint_{\bb_{\sqrt{s}}}r^{d}_*(y)\dd y+\int_{\mathbb{R}^d\backslash \bb_1}r^d_*(y)(\sqrt{s}y)g_{\sqrt{2}}(y)\dd y\Big)^{\frac{1}{2}}\dd s\bigg)^2.$$
The estimates \eqref{othertestesti15} and \eqref{proofth1othertest9} are unchanged. We provide the arguments for the first two and the fourth r.h.s terms of \eqref{functioderivothertest} and we prove that 
\begin{equation}
\int_{\mathbb{R}^d}\left(\int_{1}^{r^2}s^{-\frac{1}{2}}\mathcal{T}_{x,\ell}(\eta_{\sqrt{s}})(0)\dd s\right)^2\dd x\lesssim \ell^d r^2.
\label{proofth1othertest11}
\end{equation}
and
\begin{equation}
\int_{\mathbb{R}^d}\left(\int_{\bb_{\ell}(x)}\vert f_r(y)\vert\dd y\right)^2\dd x+\int_{\mathbb{R}^d}\left(\int_{\bb_{\ell}(x)}\left\vert\int_{0}^{r^2}\nabla u(s,y)\dd s\right\vert\vert f_r(y)\vert\dd y\right)^2\dd x\lesssim \mathcal{M}_{\star,3}(r).
\label{proofth1othertest10}
\end{equation}
First, \eqref{proofth1othertest11} follows from Minkowski's inequality in $\LL^2(\mathbb{R}^d)$ and \eqref{othertestesti13} applied with $r=\sqrt{s}$ (noticing that the evaluation at $0$ plays no role in the estimate). Secondly, using Minkowski's inequality in $\LL^2(\mathbb{R}^d)$ and the assumption \eqref{assumesensiothertest}, we have 
\begin{align*}
&\int_{\mathbb{R}^d}\left(\int_{\bb_{\ell}(x)}\vert f_r(y)\vert\dd y\right)^2\dd x+\int_{\mathbb{R}^d}\left(\int_{\bb_{\ell}(x)}\left\vert\int_{0}^{r^2}\nabla u(s,y)\dd s\right\vert\vert f_r(y)\vert\dd y\right)^2\dd x\\
&\lesssim \ell^{d}\left(\int_{1}^{r^2} s^{-1}\int_{\mathbb{R}^d}\vert y\vert g_{\sqrt{s}}(y)\left(1+\left\vert\int_{0}^{r^2}\nabla u(t,y)\dd t\right\vert\right)\dd y\right)^2,
\end{align*}
which gives \eqref{proofth1othertest10} using that $\vert y\vert g_{\sqrt{s}}(y)\lesssim s^{\frac{1}{2}}g_{\sqrt{2s}}(y)$, the estimates \eqref{othertestesti7} and \eqref{othertestesti8} with $(T,r)$ replaced by $(r^2,s)$ 
\begin{equation}\label{othertestesti9}
\begin{aligned}
\int_{1}^{r^2}s^{-1}\int_{\mathbb{R}^d} \vert y\vert g_{\sqrt{s}}(y)\left\vert\int_{0}^{r^2}\nabla u(t,y)\dd t\right\vert\dd y&\lesssim \int_{1}^{r^2}s^{-\frac{1}{2}}\int_{\mathbb{R}^d}g_{\sqrt{2s}}(y)\left\vert\int_{0}^{r^2}\nabla u(t,y)\dd t\right\vert\dd y\\
&\lesssim\int_{1}^{r^2}s^{-\frac{1}{2}}(1+\log(\tfrac{r}{\sqrt{s}}))\left(\fint_{\bb_{\sqrt{s}}}r^{d}_*(y)\dd y+\int_{\mathbb{R}^d\backslash \bb_1}r^d_*(y)(\sqrt{s}y)g_{\sqrt{2}}(y)\dd y\right)^{\frac{1}{2}}\dd s.
\end{aligned}
\end{equation}
\newline
\newline
\textbf{Step 3. Proof of \eqref{sensiothertest}. }We have from the logarithm Sobolev inequality in form of 
\eqref{SGinegp1}, for all $p\in [1,\infty)$
\begin{equation*}
\left\langle \vert q(r^2)\star f_r-\left\langle q(r^2)\star f_r\right\rangle\vert^p\right\rangle^{\frac{1}{p}}\lesssim \sqrt{p}\left\langle \left(\int_{1}^{+\infty}\ell^{-d}\pi(\ell)\int_{\mathbb{R}^d}\vert\partial^{\text{fct}}_{x,\ell} q_r(r^2)\star f_r\vert^2\dd x\, \dd \ell\right)^{\frac{p}{2}}\right\rangle^{\frac{1}{p}}\leq
 \sqrt{p}(\mathcal{I}^1_r+\mathcal{I}^{2}_r),
\end{equation*}
with 
$$\mathcal{I}^1_r:=\left\langle \left(\int_{1}^{r}\ell^{-d}\pi(\ell)\int_{\mathbb{R}^d}\vert\partial^{\text{fct}}_{x,\ell} q_r(r^2)\star f_r\vert^2\dd x\, \dd \ell\right)^{\frac{p}{2}}\right\rangle^{\frac{1}{p}}\quad \text{and}\quad\mathcal{I}^2_r:=\left\langle \left(\int_{r}^{+\infty}\ell^{-d}\pi(\ell)\int_{\mathbb{R}^d}\vert\partial^{\text{fct}}_{x,\ell} q_r(r^2)\star f_r\vert^2\dd x\, \dd \ell\right)^{\frac{p}{2}}\right\rangle^{\frac{1}{p}}.$$
We then treat separately the two terms above. 
\begin{itemize}
\item[(i)]In the regime $\ell<r$ we use \eqref{othertestesti1} combined with Minkowski's inequality in $\LL^p_{\left\langle\cdot\right\rangle}(\Omega)$ and the moment bound \eqref{equivmomentr*} of $r_*$ (for $\gamma=\varepsilon$) as well as $\left(\int_{1}^{r}\ell^{1-\beta}(1+\log(\frac{r}{\ell}+1)\mathds{1}_{d=2})\dd \ell\right)^{\frac{1}{2}}\lesssim \chi_{d,\beta}(r)$ and $1+\log(r+1)\mathds{1}_{d=2}\leq \chi_{d,\beta}(r)$ (where we recall that $\chi_{d,\beta}(r)$ is defined in \eqref{defchibeta}):
\begin{align*}
\mathcal{I}^1_r\leq &\bigg\langle \bigg(\int_{1}^r \ell^{-1-\beta}\bigg(r^d_*(0)(1+\log^2(r+1)\mathds{1}_{d=2})+\ell^2(1+\log(\frac{r}{\ell}+1)\mathds{1}_{d=2})\\
&+\mathcal{M}_{\star,1}(r,\ell)+\mathcal{M}_{\star,2}(r,\ell)+\Big(\sum_{j=0}^{+\infty}2^{\frac{j}{2}}\eta_{\varepsilon,\beta}(2^j)\mathcal{K}_{j,\star,\varepsilon}(r,\ell)\Big)^2\bigg)\bigg)^{\frac{p}{2}}\bigg\rangle^{\frac{1}{p}}\\
&\lesssim p^{\frac{d}{\beta\wedge d}+\varepsilon}(1+\chi_{d,\beta}(r))+\bigg(\int_{1}^r\ell^{-1-\beta}\bigg(\left\langle \mathcal{M}^p_{\star,1}(r,\ell)\right\rangle^{\frac{1}{p}}+\left\langle \mathcal{M}^p_{\star,2}(r,\ell)\right\rangle^{\frac{1}{p}}+\sum_{j=0}^{+\infty} 2^{\frac{j}{2}}\eta_{\varepsilon,\beta}(2^j)\left\langle \mathcal{K}^p_{j,\star,\varepsilon}(r,\ell)\right\rangle^{\frac{1}{p}}\bigg)\dd \ell\bigg)^{\frac{1}{2}}.
\end{align*}
\item[(ii)]In the regime $\ell\geq r$ we use \eqref{proofth1othertest8} combined with the Minkowski inequality $\LL^p_{\left\langle\cdot\right\rangle}(\Omega)$:
\begin{align*}
\mathcal{I}^2_r&\leq \left\langle\bigg(\int_{r}^{+\infty}\ell^{-1-\beta}\bigg(r^2+\mathcal{M}_{\star,3}(r)+\mathcal{M}_{\star,2}(r,\ell)+\Big(\sum_{j=0}^{+\infty} 2^{\frac{j}{2}}\eta_{\varepsilon,\beta}(2^j)\mathcal{K}_{j,\star,\varepsilon}(r,\ell)\Big)^2\bigg)\dd \ell\bigg)^{\frac{p}{2}}\right\rangle^{\frac{1}{p}}\\
&\lesssim r^{1-\frac{\beta}{2}}+\bigg(\int_{r}^{+\infty}\ell^{-1-\beta}\bigg(\left\langle\mathcal{M}^p_{\star,3}(r,\ell)\right\rangle^{\frac{1}{p}}+\left\langle\mathcal{M}^p_{\star,2}(r,\ell)\right\rangle^{\frac{1}{p}}+\sum_{j=0}^{+\infty}2^{\frac{j}{2}}\eta_{\varepsilon,\beta}(2^{j})\left\langle \mathcal{K}^p_{j,\star,\varepsilon}(r,\ell)\right\rangle^{\frac{1}{p}}\bigg)\dd \ell\bigg)^{\frac{1}{2}}.
\end{align*}
\end{itemize}
It remains to prove that for all $\varepsilon\in (0,1)$
\begin{equation}
\int_{1}^r\ell^{-1-\beta}\bigg(\left\langle \mathcal{M}^p_{\star,1}(r,\ell)\right\rangle^{\frac{1}{p}}+\left\langle \mathcal{M}^p_{\star,2}(r,\ell)\right\rangle^{\frac{1}{p}}+\sum_{j=0}^{+\infty} 2^{\frac{j}{2}}\eta_{\varepsilon,\beta}(2^j)\left\langle \mathcal{K}^p_{j,\star,\varepsilon}(r,\ell)\right\rangle^{\frac{1}{p}}\bigg)\dd \ell\lesssim p^{\frac{d+2}{\beta\wedge d}+\varepsilon}\chi^2_{d,\beta}(r),
\label{proofth1othertestesti53}
\end{equation}
and 
\begin{equation}
\int_{r}^{+\infty}\ell^{-1-\beta}\bigg(\left\langle\mathcal{M}^p_{\star,3}(r,\ell)\right\rangle^{\frac{1}{p}}+\left\langle\mathcal{M}^p_{\star,2}(r,\ell)\right\rangle^{\frac{1}{p}}+\sum_{j=0}^{+\infty}2^{\frac{j}{2}}\eta_{\varepsilon,\beta}(2^{j})\left\langle \mathcal{K}^p_{j,\star,\varepsilon}(r,\ell)\right\rangle^{\frac{1}{p}}\bigg)\dd \ell\lesssim  p^{\frac{d+2}{\beta\wedge d}+\varepsilon}\chi^2_{d,\beta}(r).
\label{proofth1othertestesti54}
\end{equation}
We start with \eqref{proofth1othertestesti53}. First, using the moment bound \eqref{momentboundDstar} of $\mathcal{D}_{\star}$ and the definition of $\eta_{\varepsilon,\beta}$ in \eqref{defetaespsibeta} as well as \eqref{equivmomentr*} (choosing $\gamma=\varepsilon$), we have for all $\rho>0$, $\left\langle \mathcal{N}^p_{\star,\varepsilon}(r,\ell,\rho)\right\rangle^{\frac{1}{p}}\lesssim p^{\frac{d}{\beta\wedge d}+\varepsilon(\alpha_1+\frac{d}{\beta\wedge d})}$. Therefore, by making use of the triangle inequality, we get 
$$\int_{1}^{r}\ell^{-1-\beta}\left\langle \mathcal{M}^{p}_{\star,1}(r,\ell)\right\rangle^{\frac{1}{p}}\dd \ell \lesssim 1+p^{\frac{d}{\beta\wedge d}+\varepsilon(\alpha_1+\frac{d}{\beta\wedge d})}\int_{1}^{r}\ell^{-1-\beta}(1+\ell^2(1+\log(\frac{r}{\ell}+1)\mathds{1}_{d=2})+\ell^d r^{2-d})\dd \ell\lesssim p^{\frac{d}{\beta\wedge d}+\varepsilon(\alpha_1+\frac{d}{\beta\wedge d})}\chi^2_{d,\beta}(r).$$
Secondly, from the triangle inequality, the moment bounds \eqref{equivmomentr*} (again for $\gamma=\varepsilon$) on $r_*$ and by splitting the first integral in the r.h.s of \eqref{proofth1othertest4} in form of $\int_{\mathbb{R}^d\backslash \bb_{\ell}}=\int_{\bb_{r}\backslash \bb_{\ell}}+\int_{\mathbb{R}^d\backslash \bb_{r}}$ as well as \eqref{proofth1othertestesti52} and \eqref{proofth1othertestesti55}, we have for all $\ell<r$
\begin{align*}
\left\langle\mathcal{M}^p_{\star,2}(r,\ell)\right\rangle^{\frac{1}{p}}\lesssim & p^{\frac{d}{\beta\wedge d}(1+\varepsilon)}\bigg(\ell^d\int_{\bb_{r}\backslash \bb_{\ell}}\frac{1}{(\vert x\vert+1)^{2(d-1)}}\dd x+\ell^d\int_{\mathbb{R}^d\backslash \bb_r}\frac{r^2\log^2(1+\vert x\vert)}{(\vert x\vert+1)^{2d}}\wedge \frac{1}{(\vert x\vert+1)^{2(d-1)}}\dd x\bigg)\\
&+(1+\log^2(r+1)\mathds{1}_{d=2})\ell^d\int_{\mathbb{R}^d\backslash \bb_{\ell}}\frac{1}{(\vert x\vert+1)^{2d}}\dd x+\ell^d\int_{\mathbb{R}^d\backslash \bb_{\ell}}\vert f_{2r}(y)\vert^2\dd x+\left(\int_{\bb_{7\ell}}\frac{1}{(\vert x\vert+1)^{d-1}}\dd x\right)^2\\
&+(1+\log^2(r+1)\mathds{1}_{d=2})p^{\frac{2}{\beta\wedge d}(1+\varepsilon)}\left(\int_{\bb_{7\ell}}\frac{1}{(\vert x\vert+1)^{d+1}}\dd x\right)^2+\left(\int_{\bb_{7\ell}} f_{2r}(x)\dd x\right)^2\\
\stackrel{\eqref{proofth1othertestesti55},\eqref{proofth1othertestesti52}}{\lesssim}& p^{\frac{d}{\beta\wedge d}(1+\varepsilon)}\left(\ell^{2}(1+\log(\frac{r}{\ell}+1)\mathds{1}_{d=2})+(1+\ell^{d}r^{2-d}+p^{\frac{2}{\beta\wedge d}(1+\varepsilon)})(1+\log^2(r+1)\mathds{1}_{d=2})\right),
\end{align*}
which provides 
$$\int_{1}^{r}\left\langle\mathcal{M}^p_{\star,2}(r,\ell)\right\rangle^{\frac{1}{p}}\dd \ell\lesssim p^{\frac{d+2}{\beta\wedge d}(1+\varepsilon)}\chi^2_{d,\beta}(r).$$
Finally, using the same decomposition as before and in addition the moment bound \eqref{momentboundDstar} of $\mathcal{D}_{\star}$ we get (up to adjusting $\varepsilon$)
$$\int_{1}^{r}\sum_{j=0}^{+\infty} 2^{\frac{j}{2}}\eta_{\varepsilon,\beta}(2^j)\left\langle \mathcal{K}^p_{j,\star,\varepsilon}(r,\ell)\right\rangle^{\frac{1}{p}}\dd \ell\lesssim p^{\frac{d+2}{\beta\wedge d}+\varepsilon}\chi^2_{d,\beta}(r),$$
which concludes the proof of \eqref{proofth1othertestesti53}. We now turn to the proof of \eqref{proofth1othertestesti54}. First, using the moment bound \eqref{equivmomentr*} (for $\gamma=\varepsilon$) on $r_*$ we have 
$$\int_{r}^{+\infty}\ell^{-1-\beta}\left\langle \mathcal{M}^p_{\star,3}(r,\ell)\right\rangle^{\frac{1}{p}}\dd \ell\lesssim p^{
\frac{d}{\beta\wedge d}}r^{2-\beta}\lesssim  p^{
\frac{d}{\beta\wedge d}(1+\varepsilon)}\chi^2_{d,\beta}(r).$$
Secondly, using the triangle inequality
\begin{align}
&\left\langle \mathcal{M}^p_{\star,2}(r,\ell)\right\rangle^{\frac{1}{p}}\nonumber\\
\lesssim& p^{\frac{d}{\beta\wedge d}(1+\varepsilon)}\bigg(\ell^d\int_{\mathbb{R}^d\backslash \bb_{\ell}}\frac{r^2\log^2(1+\vert x\vert)}{(\vert x\vert+1)^{2d}}\wedge \frac{1}{(\vert x\vert+1)^{2(d-1)}}\dd x+(1+\log(r+1)\mathds{1}_{d=2})\ell^d\int_{\mathbb{R}^d\backslash \bb_{\ell}}\frac{1}{(\vert x\vert+1)^{2d}}\dd x+\ell^{d}\int_{\mathbb{R}^d\backslash \bb_{\ell}} \vert f_{2r}(x)\vert^2\dd x\nonumber\\
&+\left(\int_{\bb_{7\ell}}\frac{r\log(1+\vert x\vert)}{(\vert x\vert+1)^d}\wedge \frac{1}{(\vert x\vert+1)^{d-1}}\dd x\right)^2+(1+\log(r+1)\mathds{1}_{d=2})p^{\frac{2}{\beta\wedge d}(1+\varepsilon)}\left(\int_{\bb_{7\ell}}\frac{1}{(\vert x\vert+1)^{d+1}}\dd x\right)^2+\left(\int_{\bb_{7\ell}} \vert f_{2r}(x)\vert\dd x\right)^2\bigg)\nonumber\\
&\lesssim p^{\frac{d+2}{\beta\wedge d}(1+\varepsilon)}\left(\int_{\mathbb{R}^d\backslash \bb_{\ell}}\frac{r^2\log^2(1+\vert x\vert)}{(\vert x\vert+1)^{2d}}\wedge \frac{1}{(\vert x\vert+1)^{2(d-1)}}\dd x+\left(\int_{\bb_{7\ell}}\frac{r\log(1+\vert x\vert)}{(\vert x\vert+1)^d}\wedge \frac{1}{(\vert x\vert+1)^{d-1}}\dd x\right)^2+1+\log(r+1)\mathds{1}_{d=2}+r^2\right).\label{proofth1othertestesti56}
\end{align}
We then argue differently, depending on the regime of $\beta$ and $d$: 
\begin{itemize}
\item[(i)]For $\beta>2$, we use
\begin{align}
&\int_{\mathbb{R}^d\backslash \bb_{\ell}}\frac{r^2\log^2(1+\vert x\vert)}{(\vert x\vert+1)^{2d}}\wedge \frac{1}{(\vert x\vert+1)^{2(d-1)}}\dd x+\left(\int_{\bb_{7\ell}}\frac{r\log(1+\vert x\vert)}{(\vert x\vert+1)^d}\wedge \frac{1}{(\vert x\vert+1)^{d-1}}\dd x\right)^2\nonumber\\
&\leq r^2\int_{\mathbb{R}^d\backslash \bb_{\ell}}\frac{\log^2(1+\vert x\vert)}{(\vert x\vert+1)^{2d}}\dd x+r^2\left(\int_{\bb_{7\ell}}\frac{\log(1+\vert x\vert)}{(\vert x\vert+1)^d}\dd x\right)^2\nonumber\\
&\lesssim r^2(1+\log^2(\ell)),
\label{proofth1othertestesti57}
\end{align}
to deduce, combined with \eqref{proofth1othertestesti56}
\begin{align}
\int_{r}^{+\infty}\ell^{-1-\beta}\left\langle \mathcal{M}^p_{\star,2}(r,\ell)\right\rangle^{\frac{1}{p}}\dd \ell
&\lesssim p^{\frac{d+2}{\beta\wedge d}(1+\varepsilon)}\int_{r}^{+\infty}\ell^{-1-\beta}(r^2\log^2(\ell)+1+r^2+\log(r+1)\mathds{1}_{d=2})\dd \ell\nonumber\\
&\lesssim p^{\frac{d+2}{\beta\wedge d}(1+\varepsilon)}r^{-\beta}(r^2+1+\log(r+1)\mathds{1}_{d=2}+\log^2(r+1))\label{proofth1othertestesti58}\\
&\lesssim p^{\frac{d+2}{\beta\wedge d}(1+\varepsilon)},\nonumber
\end{align}
where we used in the last line that $\beta>2$.
\item[(ii)]For $\beta\leq 2$ and $d>2$, we use 
$$\int_{\mathbb{R}^d\backslash \bb_{\ell}}\frac{r^2\log^2(1+\vert x\vert)}{(\vert x\vert+1)^{2d}}\wedge \frac{1}{(\vert x\vert+1)^{2(d-1)}}\dd x+\left(\int_{\bb_{7\ell}}\frac{r\log(1+\vert x\vert)}{(\vert x\vert+1)^d}\wedge \frac{1}{(\vert x\vert+1)^{d-1}}\dd x\right)^2
\lesssim r^2\log^2(\ell)\wedge \ell^2,$$
which yields combined with \eqref{proofth1othertestesti56} and a dyadic decomposition
\begin{align*}
p^{-\frac{d+2}{\beta\wedge d}(1+\varepsilon)}\int_{r}^{+\infty}\ell^{-1-\beta}\left\langle \mathcal{M}^p_{\star,2}(r,\ell)\right\rangle^{\frac{1}{p}}\dd \ell\lesssim &\sum_{n=0}^{\left\lceil \log_2(r)\right\rceil}\int_{2^{n}r}^{2^{n+1}r}\ell^{-1-\beta}(\ell^2+1+r^2+\log(r+1)\mathds{1}_{d=2})\dd \ell\\
&+\sum_{n=\left\lceil \log_2(r)\right\rceil}^{+\infty}\int_{2^n r}^{2^{n+1}r}\ell^{-1-\beta}(r^2\log^2(\ell)+1+r^2+\log(r+1)\mathds{1}_{d=2})\dd \ell\\
\lesssim &\sum_{n=0}^{\left\lceil \log_2(r)\right\rceil}\left((2^n r)^{2-\beta}\mathds{1}_{\beta<2}+\mathds{1}_{\beta=2}+r^{-\beta}2^{-n\beta}(r^2+\log(r+1)\mathds{1}_{d=2})\right) \\
&+\sum_{n=\left\lceil \log_2(r)\right\rceil}^{+\infty} 2^{-n\beta}(r^{2-\beta}\log^2(2^{n+1}r)+1+r^{2-\beta}+\log(r+1)\mathds{1}_{d=2})\\
\lesssim& \chi_{\beta,d}(r).
\end{align*}
\item[(iii)]Finally, for $\beta\leq 2$ and $d=2$, we combine \eqref{proofth1othertestesti56}, \eqref{proofth1othertestesti57} and \eqref{proofth1othertestesti58} to obtain
\begin{align*}
\int_{r}^{+\infty}\ell^{-1-\beta}\left\langle \mathcal{M}^p_{\star,2}(r,\ell)\right\rangle^{\frac{1}{p}}\dd \ell\lesssim p^{\frac{d+2}{\beta\wedge d}(1+\varepsilon)}r^{2-\beta}(1+\log^2(r+1))+1+\log(r+1).
\end{align*}
\end{itemize}
To conclude, using the same decomposition as before and in addition the moment bound \eqref{momentboundDstar} of $\mathcal{D}_{\star}$ we get (up to adjusting $\varepsilon$)
$$\int_{r}^{+\infty}\sum_{j=0}^{+\infty}2^{\frac{j}{2}}\eta_{\varepsilon,\beta}(2^j)\left\langle \mathcal{K}^p_{j,\star,\varepsilon}(r,\ell)\right\rangle^{\frac{1}{p}}\lesssim p^{\frac{d+2}{\beta\wedge d}+\varepsilon}\chi^2_{d,\beta}(r).$$
\end{proof}

\subsubsection{Proof of Corollary \ref{decayu}: Decay of the semigroup. }
We apply Lemma \ref{cacciopo} and we make use of Minkowski's inequality in $\LL^p_{\left\langle\cdot\right\rangle}(\Omega)$ and the stationarity of $q_r(T,\cdot)$ to the effect of: for all $p\in [1,\infty)$, $T\geq 4$ and $R\geq \sqrt{T}$
\begin{align*}
\left\langle \left(\int_{\mathbb{R}^d}\eta_{\sqrt{2}R}(\tfrac{y}{c})\vert(u(T,y),\sqrt{T}\nabla u(T,y))\vert^2\dd y\right)^{\frac{p}{2}}\right\rangle^{\frac{1}{p}}&\lesssim\frac{1}{\sqrt{T}}\left\langle \left(\fint_{\frac{T}{4}}^{\frac{T}{2}}\fint_{0}^{\sqrt{t}}\left(\frac{r}{\sqrt{t}}\right)^{\frac{d}{2}}\int_{\mathbb{R}^d}\eta_{2R}(\tfrac{y}{c})\vert q_r(t,y)-\langle q_r(t,y)\rangle\vert\dd y\, \dd r\, \dd t\right)^{p}\right\rangle^{\frac{1}{p}}\\
&\lesssim \frac{1}{\sqrt{T}}\fint_{\frac{T}{4}}^{\frac{T}{2}}\fint_{0}^{\sqrt{t}}\left(\frac{r}{\sqrt{t}}\right)^{\frac{d}{2}}\left\langle \vert q_r(t)-\langle q_r(t)\rangle\vert^p\right\rangle^{\frac{1}{p}}\dd r\, \dd t.
\end{align*}
Then, we split the integral over $[0,\sqrt{t}]$ into the contributions $r\leq 1$ and $1\leq r\leq \sqrt{t}$:
\begin{itemize}
\item[(i)]For $r\leq 1$ we use \eqref{near0est4}: for any $\gamma>0$
\begin{align*}
\frac{1}{\sqrt{T}}\fint_{\frac{T}{4}}^{\frac{T}{2}}\frac{1}{\sqrt{t}}\int_{0}^{1}\left(\frac{r}{\sqrt{t}}\right)^{\frac{d}{2}}\left\langle \vert q_r(t)-\langle q_r(t)\rangle\vert^p\right\rangle^{\frac{1}{p}}\dd r\, \dd t&\stackrel{\eqref{near0est4}}{\lesssim}p^{\frac{1}{\eta_\gamma}}T^{-\frac{1}{2}-\frac{d}{4}}\fint_{\frac{T}{4}}^{\frac{T}{2}}\frac{1}{\sqrt{t}}\int_{0}^{1}\left(r^{\frac{d}{2}}+\log(\tfrac{\sqrt{T}}{r})\right)\dd r\, \dd t\nonumber\\
&\lesssim p^{\frac{1}{\eta_\gamma}}T^{-1-\frac{d}{4}}\log(T).
\end{align*}
\item[(ii)]For $1\leq r\leq \sqrt{t}$ we use Theorem \ref{semigroup} and the change of variable $r\mapsto \tfrac{\sqrt{T}}{r}$: for all $\alpha<\frac{1}{\frac{1}{2}+2\frac{d+1}{\beta\vee d}}$
\begin{align*}
\frac{1}{\sqrt{T}}\fint_{\frac{T}{4}}^{\frac{T}{2}}\frac{1}{\sqrt{t}}\int_{1}^{\sqrt{t}}\left(\frac{r}{\sqrt{t}}\right)^{\frac{d}{2}}\left\langle \vert q_r(t)-\langle q_r(t)\rangle\vert^p\right\rangle^{\frac{1}{p}}\dd r\, \dd t&\stackrel{\eqref{Sensitilem3}}{\lesssim}p^{\frac{1}{\alpha}}T^{-\frac{1}{2}-\frac{d}{4}}\mu_{\beta}(T)\fint_{\frac{T}{4}}^{\frac{T}{2}}\frac{1}{\sqrt{t}}\int_{1}^{\sqrt{t}}\left(1+\log^2(\tfrac{\sqrt{T}}{r})\right)\dd r\, \dd t\nonumber\\
&\leq p^{\frac{1}{\alpha}}\eta_{\beta}(T)\int_{1}^{\sqrt{\frac{T}{2}}}r^{-2}(1+\log^2(r))\dd r\nonumber\\
&\lesssim p^{\frac{1}{\alpha}}\eta_{\beta}(T),
\end{align*}
where $\eta_{\beta}$ is defined in \eqref{defetabeta}. This concludes the argument for \eqref{estisemigroup}.
\end{itemize}
The estimate \eqref{estip2} is a direct consequence of \eqref{estisemigroup} and the stationarity of $\nabla u$: for all $x\in\mathbb{R}^d$ and $T\geq 1$
$$\left\langle \vert\nabla u(T,x)\vert^2\right\rangle=\left\langle\int_{\mathbb{R}^d}\eta_{\sqrt{2}R}(\tfrac{y}{c})\vert\nabla u(T,y)\vert^2\dd y\right\rangle\stackrel{\eqref{estisemigroup}}{\lesssim}T^{-1}\eta^2_{\beta}(T)\left\langle \mathcal{C}^2_{\star,d,\lambda,\beta}\right\rangle\lesssim T^{-1}\eta^2_{\beta}(T).$$

\subsubsection{Proof of Corollary \ref{boundphi} : Bounds on the flux and gradient of correctors. }
We split the proof into two steps. The first one gives a rigorous proof of the formula \eqref{expliequationcor}. The second step prove \eqref{correctorbound2}.\newline
\newline
\textbf{Step 1. } We prove the two following integral formulas
\begin{equation}
\nabla \phi=\int_{0}^{+\infty}\nabla u(t,\cdot)\,\dd t,
\label{linkphiu}
\end{equation}
and for all $T\geq 1$
\begin{equation}
\nabla \phi_T=\int_{0}^{+\infty}e^{-\frac{t}{T}}\nabla u(t,\cdot)\,\dd t.
\label{linkphiumassive}
\end{equation}
We first note that the r.h.s of \eqref{linkphiu} is well defined as a random variable in $\LL^2_{\left\langle\cdot\right\rangle}(\Omega)$, since from \eqref{estip2} and $\eqref{LemE1}$ we have for all $x\in\mathbb{R}^d$
$$\left\langle\left\vert\int_{0}^{+\infty}\nabla u(t,x)\dd t\right\vert^2\right\rangle^{\frac{1}{2}}\leq \left\langle\left\vert\int_{0}^{1}\nabla u(t,x)\dd t\right\vert^2\right\rangle^{\frac{1}{2}}+\int_{1}^{+\infty}\left\langle\vert\nabla u(t,x)\vert^2\right\rangle^{\frac{1}{2}}\dd t\lesssim 1.$$
We only provide the argument for \eqref{linkphiu}, \eqref{linkphiumassive} follows the same way. To this aim, we prove that there exists a potential $\zeta\in \LL^2(\Omega\times\mathbb{R}^d)$, sub-linear at infinity, such that $\int_{0}^{+\infty} \nabla u(t,\cdot)\dd t=\nabla\zeta$ and solving $-\nabla\cdot a(\nabla\zeta+e)=0$ in the distributional sense on $\mathbb{R}^d$. By uniqueness of $\nabla\phi$ defined by \eqref{correctorequation2}, it will imply \eqref{linkphiu}.\newline
\newline
Let $\psi\in \cc^{\infty}_c(\mathbb{R}^d)$ be supported in $\bb_R$ for some $R>0$ and let $\sqrt{T}> R$. We have by testing \eqref{equationu} with $\psi$ and integrating in time from $0$ to $T$
\begin{equation}
\int_{\mathbb{R}^d} u(T,y)\psi(y)\dd y+\int_{\mathbb{R}^d}\nabla\psi(y)\cdot a(y)e\,\dd y +\int_{\mathbb{R}^d}\nabla\psi(y)\cdot a(y)\int_{0}^T \nabla u(s,y)\dd s\,\dd y=0.
\label{testweakcor3}
\end{equation}
We now check that each term of \eqref{testweakcor3} pass to the limit, almost surely, as $T\uparrow \infty$. For the first l.h.s term of \eqref{testweakcor3}, we use the triangle inequality followed by Poincaré's inequality, \eqref{estisemigroup} and \eqref{estip2}: 
\begin{align}
\left\langle\fint_{\bb_{R}}\vert u(T,y)\vert^2\dd y\right\rangle^{\frac{1}{2}}\leq & \left\langle\fint_{\bb_{R}}\left\vert u(T,y)-\fint_{\bb_R} u(T,z)\dd z\right\vert^2\dd y\right\rangle^{\frac{1}{2}}+\left\langle \left\vert \fint_{\bb_{\sqrt{T}}} u(T,y)\dd y\right\vert^2\right\rangle^{\frac{1}{2}}\nonumber\\
&+\left\langle \left\vert \fint_{\bb_R} u(T,y)\dd y-\fint_{\bb_{\sqrt{T}}} u(T,y)\dd y\right\vert^2\right\rangle^{\frac{1}{2}}\nonumber\\
\lesssim R&\left\langle \fint_{\bb_R}\vert\nabla u(T,y)\vert^2\dd y\right\rangle^{\frac{1}{2}}+\left\langle \fint_{\bb_{\sqrt{T}}} \vert u(T,y)\vert^2\dd y\right\rangle^{\frac{1}{2}}\nonumber\\
&+\left\langle \left\vert \fint_{\bb_R} u(T,y)\dd y-\fint_{\bb_{\sqrt{T}}} u(T,y)\dd y\right\vert^2\right\rangle^{\frac{1}{2}}\nonumber\\
\stackrel{\eqref{estisemigroup},\eqref{estip2}}{\lesssim}& (\frac{R}{\sqrt{T}}+1)\eta_{\beta}(T)+\left\langle \left\vert \fint_{\bb_R} u(T,y)\dd y-\fint_{\bb_{\sqrt{T}}} u(T,y)\dd y\right\vert^2\right\rangle^{\frac{1}{2}}.\label{quantiucor31}
\end{align}
From the fundamental calculus theorem, the stationarity of $\nabla u$ and the application of \eqref{estip2}, we also have 
\begin{align*}
\left\langle \left\vert \fint_{\bb_R} u(t,y)\dd y-\fint_{\bb_{\sqrt{T}}} u(t,y)\dd y\right\vert^2\right\rangle^{\frac{1}{2}}&=\left\langle \left\vert \int_{R}^{\sqrt{T}}\fint_{\bb_1}\nabla u(t,\tau z)\cdot z\dd z\, \dd \tau\right\vert^2\right\rangle^{\frac{1}{2}}\nonumber\\
&\stackrel{\eqref{estip2}}{\lesssim}(1-\frac{R}{\sqrt{T}})\eta_{\beta}(T)\leq \eta_{\beta}(T).
\end{align*}
Hence, since $R<\sqrt{T}$, \eqref{quantiucor31} turns into
$$\left\langle\fint_{\bb_R}\vert u(T,y)\vert^2\dd y\right\rangle^{\frac{1}{2}}\lesssim \eta_{\beta}(T),$$
and yields
\begin{align*}
\left\langle\left\vert\int_{\mathbb{R}^d}u(T,y)\psi(y)\dd y\right\vert^2\right\rangle^{\frac{1}{2}}&\leq R^{\frac{d}{2}}\|\psi\|_{\LL^{\infty}(\mathbb{R}^d)}\left\langle\fint_{\bb_R}\vert u(T,y)\vert^2\dd y\right\rangle^{\frac{1}{2}}\\
&\stackrel{\eqref{quantiucor31}}{\lesssim}  R^{\frac{d}{2}}\|\psi\|_{\LL^{\infty}(\mathbb{R}^d)}\eta_{\beta}(T).
\end{align*}
We have in particular, 
$$\int_{\mathbb{R}^d}u(T,y)\psi(y)\dd y\underset{T\uparrow \infty}{\rightarrow} 0\quad \text{almost surely}.$$
For the second l.h.s term of \eqref{testweakcor3}, we have directly using \eqref{estip2}: 
\begin{align*}
\left\langle\left\vert\int_{\mathbb{R}^d}\nabla\psi(y)\cdot a(y)\int_{T}^{+\infty} \nabla u(s,y)\dd s\,\dd y\right\vert^2\right\rangle^{\frac{1}{2}}&\leq \|\nabla\psi\|_{\LL^{\infty}(\mathbb{R}^d)}\int_{\bb_R}\int_{T}^{+\infty}\left\langle\vert\nabla u(s,y)\vert^2\right\rangle^{\frac{1}{2}}\dd s\, \dd y\\
&\stackrel{\eqref{estip2}}{\lesssim}R^{d} \|\nabla\psi\|_{\LL^{\infty}(\mathbb{R}^d)}T^{\frac{1}{2}}\eta_{\beta}(T),
\end{align*}
with $T^{\frac{1}{2}}\eta_{\beta}(T)\underset{T\uparrow \infty}{\rightarrow}0$, which yields 
$$\lim_{T\rightarrow +\infty}\int_{\mathbb{R}^d}\nabla\psi(y)\cdot a\int_{0}^T \nabla u(s,y)\dd s\,\dd y=\int_{\mathbb{R}^d}\nabla\psi(y)\cdot a\int_{0}^{+\infty} \nabla u(s,y)\dd s\,\dd y\quad \text{almost surely}.$$
To conclude, we can pass to the limit as $T\uparrow \infty$ in \eqref{testweakcor3} and obtain 
\begin{equation}
\int_{\mathbb{R}^d}\nabla\psi(y)\cdot a(y)e\, \dd y+\int_{\mathbb{R}^d}\nabla\psi(y)\cdot a(y)\int_{0}^{+\infty}\nabla u(s,y)\dd s\, \dd y=0\quad \text{almost surely}.
\label{coreceqcor2}
\end{equation}
Now, since $\int_{0}^{+\infty}\nabla u(t,\cdot)\dd t$ is curl free and belongs to $\LL^2(\Omega\times \mathbb{R}^d)$, there exists a potential $\zeta\in \LL^2(\Omega\times\mathbb{R}^d)$ such that $\int_{0}^{+\infty}\nabla u(t,\cdot)\dd t=\nabla \zeta$ and \eqref{coreceqcor2} takes the form
$$\int_{\mathbb{R}^d}\nabla\psi(y)\cdot a(y)\,e\dd y+\int_{\mathbb{R}^d}\nabla \psi(y)\cdot a(y)\nabla \zeta=0,$$
which means that $-\nabla\cdot a(\nabla\zeta+e)=0$ in the distributional sense on $\mathbb{R}^d$. Since $\int_{0}^{+\infty}\nabla u(t,\cdot)\dd t$ has finite second moment, it is well known that $\zeta$ own the sub-linear property. By the uniqueness of $\nabla\phi$ defined by \eqref{correctorequation2}, this concludes the argument for \eqref{linkphiu}.\newline
\newline
\textbf{Step 2. }We prove \eqref{correctorbound2} and we split the proof into three steps. For the rest of the proof, we fix $\alpha<\frac{1}{\frac{1}{2}+2\frac{d+1}{\beta\vee d}}$ and we let $p\in [1,\infty)$ be arbitrary.\newline
\newline
\textbf{Substep 2.1. }We start with the control of the flux and we only treat $\vert q_r-\left\langle q_r\right\rangle\vert$, the control of $\vert (q_T)_r-\left\langle (q_T)_r\right\rangle\vert$ is obtained the same way, using \eqref{linkphiumassive} instead of \eqref{linkphiu}. We use the triangle inequality combined with Theorem \ref{semigroup} with $T=r^2$ to get
\begin{align*}
\left\langle \vert q_r-\left\langle q_r\right\rangle\vert^p\right\rangle^{\frac{1}{p}}&\leq \left\langle \vert q_r-(q(r^2))_r\vert^p\right\rangle^{\frac{1}{p}}+\left\langle \vert (q(r^2))_r-\left\langle (q(r^2))_r\right\rangle \vert^p\right\rangle^{\frac{1}{p}}+\vert \left\langle q_r-(q(r^2))_r\right\rangle\vert\\
&\stackrel{\eqref{Sensitilem3}}{\lesssim}\left\langle \vert q_r-(q(r^2))_r\vert^p\right\rangle^{\frac{1}{p}}+\pi^{-\frac{1}{2}}_*(r)p^{\frac{1}{\alpha}}.
\end{align*}
It remains to control the first r.h.s term of the above inequality. To this aim, we write by dominating the Gaussian kernel $g_r$ by the exponential kernel $\eta_r$ and using Minkowski's inequality in $\LL^p_{\left\langle\cdot\right\rangle}(\Omega)$ as well as Jensen's inequality
\begin{align}
\left\langle \vert q_r-(q(r^2))_r\vert^p\right\rangle^{\frac{1}{p}}&\leq \left\langle \left(\int_{\mathbb{R}^d} g_r(x)\left\vert\int_{r^2}^{+\infty}\nabla u(t,x)\dd t\right\vert\right)^p\right\rangle^{\frac{1}{p}}\nonumber\\
&\leq \int_{r^2}^{+\infty}\left\langle \left(\int_{\mathbb{R}^d}g_r(x)\vert\nabla u(t,y)\vert\dd y\right)^p\right\rangle^{\frac{1}{p}}\dd t\nonumber\\
&\leq \int_{r^2}^{+\infty}\left\langle\left(\int_{\mathbb{R}^d}\eta_r(x)\vert\nabla u(t,x)\vert^2\dd x\right)^{\frac{p}{2}}\right\rangle^{\frac{1}{p}}\dd t.\label{qminusqr1}
\end{align}
It remains to control the space integral of the above inequality. We apply Lemma \ref{ctrlav} using \eqref{estisemigroup}, $f\equiv 1$ and $g: t\in\mathbb{R}^+\mapsto \eta_{\beta}(t)$  as well as the moment bound \eqref{momentr*} of $r_*$ to obtain, for all $r\geq 1$
\begin{equation}
\left\langle\left(\fint_{\bb_r}\vert \nabla u(t,x)\vert^2\dd x\right)^{\frac{p}{2}}\right\rangle^{\frac{1}{p}}\lesssim p^{\frac{1}{2}\frac{d}{\beta\wedge d}+\frac{1}{\alpha}}t^{-\frac{1}{2}}\eta_{\beta}(t).
\label{Lem2semigroup}
\end{equation}
Consequently, for all $t\geq r^2$
\begin{align}
\left\langle\left(\int_{\mathbb{R}^d}\eta_r(x)\vert\nabla u(t,x)\vert^2\dd x\right)^{\frac{p}{2}}\right\rangle^{\frac{1}{p}}&\leq \left\langle\left(\fint_{\bb_r}\vert\nabla u(t,x)\vert^2\dd x\right)^{\frac{p}{2}}\right\rangle^{\frac{1}{p}}+\left\langle\left(\int_{\mathbb{R}^d\backslash \bb_r}\eta_r(x)\vert\nabla u(t,x)\vert^2\dd x\right)^{\frac{p}{2}}\right\rangle^{\frac{1}{p}}\nonumber\\
&\stackrel{\eqref{Lem2semigroup}}{\lesssim}p^{\frac{1}{2}\frac{d}{\beta\wedge d}+\frac{1}{\alpha}}t^{-\frac{1}{2}}\eta_{\beta}(t)+\left\langle\left(\int_{\mathbb{R}^d\backslash \bb_r}\eta_r(x)\vert\nabla u(t,x)\vert^2\dd x\right)^{\frac{p}{2}}\right\rangle^{\frac{1}{p}}.\label{qminusqr2}
\end{align}
For the second r.h.s term of the previous estimate, we decompose $\mathbb{R}^d\backslash \bb_r$ into the family of annuli $(\bb_{(n+1)r}\backslash \bb_{nr})_{n\geq 1}$ to obtain, with \eqref{Lem2semigroup}
\begin{align}
\left\langle\left(\int_{\mathbb{R}^d\backslash \bb_r}\eta_r(x)\vert\nabla u(t,x)\vert^2\dd x\right)^{\frac{p}{2}}\right\rangle^{\frac{1}{p}}&\leq \left(\sum_{n=1}^{+\infty}\left\langle\left(\int_{\bb_{(n+1)r}\backslash \bb_{nr}}\eta_r(x)\vert\nabla u(t,x)\vert^2\dd x\right)^{\frac{p}{2}}\right\rangle^{\frac{2}{p}}\right)^{\frac{1}{2}}\nonumber\\
&\leq \left(\sum_{n=1}^{+\infty}e^{-n}n^{d}\left\langle\left(\fint_{\bb_{(n+1)r}}\vert \nabla u(t,x)\vert^2\dd x\right)^\frac{p}{2}\right\rangle^{\frac{2}{p}}\right)^{\frac{1}{2}}\nonumber\\
&\stackrel{\eqref{Lem2semigroup}}{\lesssim}p^{\frac{1}{2}\frac{d}{\beta\wedge d}+\frac{1}{\alpha}}t^{-\frac{1}{2}}\eta_{\beta}(t).\label{qminusqr3}
\end{align}
We conclude by plugging the two above inequalities into \eqref{qminusqr1} with the fact $\int_{r^2}^{+\infty}t^{-\frac{1}{2}}\eta_{\beta}(t)\dd t\lesssim \pi^{-\frac{1}{2}}_{\star}(r)$. The bound on $(q_{T})_r$ is obtained the same way since from \eqref{linkphiumassive} we have 
$$(q_{T})_r=\int_{0}^{+\infty}e^{-\frac{t}{T}}q_r(t,\cdot)\dd t.$$
\textbf{Substep 2.2. }We prove the control of $\vert\nabla\phi_r\vert$. We first notice that by integrating the equation \eqref{equationu} in time and using that, by stationarity, $\nabla\cdot \left\langle q\right\rangle=0$, we have 
$$u(t,\cdot)=\nabla\cdot (q(t,\cdot)-\left\langle q(t,\cdot)\right\rangle)\quad \text{for all $t\geq 0$}.$$
From the definition \eqref{defphitime}, we thus deduce that
\begin{equation}
(\nabla\phi(r^2))_r=\left(\nabla\left(\nabla\cdot\int_{0}^{r^2}(q(s,\cdot)-\left\langle q(s,\cdot)\right\rangle)\dd s\right)\right)_r.
\label{correctortimedepend}
\end{equation}
By noticing that, from the semigroup property $g_r=g_{\frac{1}{\sqrt{2}}r}\star g_{\frac{1}{\sqrt{2}}r}$ we have for all $f\in \text{H}^2_{\text{loc}}(\mathbb{R}^d)$
\begin{equation}
\vert(\nabla^2 f)_r\vert=\vert(\nabla^2 f_{\frac{1}{\sqrt{2}}r})_{\frac{1}{\sqrt{2}}r}\vert\lesssim \frac{1}{r^2}\vert f_{\frac{1}{\sqrt{2}}r}\vert_{\frac{1}{\sqrt{2}}r},
\label{boundconvolH2}
\end{equation}
we deduce, from \eqref{correctortimedepend}, the stationarity of $q_r$, \eqref{Sensitilem3} and \eqref{algeexpmoment} that
\begin{align}
\left\langle \vert(\nabla\phi(r^2))_r\vert^p\right\rangle^{\frac{1}{p}}&=\left\langle\left\vert \left(\nabla\left(\nabla\cdot\int_{0}^{r^2}(q(s,\cdot)-\left\langle q(s,\cdot)\right\rangle)\dd s\right)\right)_r\right\vert^p\right\rangle^{\frac{1}{p}}\nonumber\\
&\stackrel{\eqref{boundconvolH2}}{\lesssim}\frac{1}{r^2}\int_{0}^{r^2}\left\langle\left\vert q_{\frac{1}{\sqrt{2}}r}(s)-\left\langle q_{\frac{1}{\sqrt{2}}r}(s)\right\rangle\right\vert^p\right\rangle^{\frac{1}{p}}\dd s\nonumber\\
&\stackrel{\eqref{Sensitilem3},\eqref{algeexpmoment}}{\lesssim}p^{\frac{1}{\alpha}}r^{-2-\frac{d}{2}}\int_{0}^{r^2}\mu_{\beta}(s)(1+\log^2(\tfrac{\sqrt{s}}{r}))\dd s\nonumber\\
&\lesssim \pi^{-\frac{1}{2}}_{*}(r)p^{\frac{1}{\alpha}}.\label{boundtimedependphi12}
\end{align}
We finally deduce, from \eqref{linkphiu}, \eqref{qminusqr1}, \eqref{qminusqr2}, \eqref{qminusqr3} and \eqref{boundtimedependphi12}, that 
\begin{align*}
\left\langle\vert\nabla \phi_r\vert^p\right\rangle^{\frac{1}{p}}&\leq \left\langle \vert \nabla\phi_r-(\nabla\phi(r^2))_r\vert^p\right\rangle^{\frac{1}{p}}+\left\langle\vert (\nabla\phi(r^2))_r\vert^p\right\rangle^{\frac{1}{p}}\\
&\stackrel{\eqref{linkphiu},\eqref{boundtimedependphi12}}{\leq} \int_{r^2}^{+\infty}\left\langle\left(\int_{\mathbb{R}^d}\eta_r(x)\vert\nabla u(t,x)\vert^2\dd x\right)^{\frac{p}{2}}\right\rangle^{\frac{1}{p}}\dd t+\pi^{-\frac{1}{2}}_*(r)p^{\frac{1}{\alpha}}\\
&\stackrel{ \eqref{qminusqr1},\eqref{qminusqr2},\eqref{qminusqr3}}{\lesssim} \pi^{-\frac{1}{2}}_*(r)(p^{\frac{1}{2}\frac{d}{\beta\vee d}+\frac{1}{\alpha}}+p^{\frac{1}{\alpha}}).
\end{align*}
\textbf{Substep 2.3. }We prove the control on $\vert\nabla(\sigma_T)_r\vert$ and $\vert\nabla\sigma_r\vert$. Let $i,j,k\in\llbracket 1,d\rrbracket$. Using the equation \eqref{correctorequationsigmamassive}, we note that $(\sigma_{T,i,j,k})_r$ solves
$$\frac{1}{T}(\sigma_{T,i,j,k})_r-\Delta(\sigma_{T,i,j,k})_r=(e_k\cdot q_{e_i,T}-e_j\cdot q_{e_i,T})\star (\partial_j g_r-\partial_k g_r).$$
Therefore, we may express $\nabla(\sigma_{T,i,j,k})_r$ with help of the Green function $G_T$ of the massive Laplace operator $\frac{1}{T}-\Delta$ on $\mathbb{R}^d$
$$\nabla(\sigma_{T,i,j,k})_r=\nabla G_T\star ((e_k\cdot q_{e_i,T}-e_j\cdot q_{e_i,T})\star (\partial_j g_r-\partial_k g_r)).$$
Then, using that there exists a constant $C$ which depends on $d$ such that
$$G_T=C\int_{0}^{+\infty} e^{-\frac{s}{T}}g_{\sqrt{s}}\,\dd s,$$
and by noticing that from the stationarity of $(e_k\cdot q_{e_i,T}-e_j\cdot q_{e_i,T})$ we have
\begin{align*}
&\left\langle \int_{0}^{+\infty}\int_{\mathbb{R}^d}\int_{\mathbb{R}^d} e^{-\frac{s}{T}}\vert\nabla g_{\sqrt{s}}(x)\vert\vert e_k\cdot q_{e_i,T}(x-y)-e_j\cdot q_{e_i,T}(x-y)\vert\vert \partial_j g_r(y)-\partial_k g_r(y)\vert\dd y\, \dd x\, \dd s\right\rangle\\
&\lesssim\left\langle \vert e_k\cdot q_{e_i,T}-e_j\cdot q_{e_i,T}\vert\right\rangle \|\partial_j g_r-\partial_k g_r\|_{\LL^{1}(\mathbb{R}^d)}\int_{0}^{+\infty} e^{-\frac{s}{T}}\|\nabla g_{\sqrt{s}}\|_{L^{1}(\mathbb{R}^d)}\dd s\\
&\lesssim \left\langle \vert e_k\cdot q_{e_i,T}-e_j\cdot q_{e_i,T}\vert\right\rangle \|\partial_j g_r-\partial_k g_r\|_{\LL^{1}(\mathbb{R}^d)}\int_{0}^{+\infty} s^{-\frac{1}{2}}e^{-\frac{s}{T}}\dd s <\infty,
\end{align*}
we deduce from Fubini's theorem combined with the semigroup property $g_r\star g_{\sqrt{s}}=g_{\sqrt{s+r^{2}}}$ that, almost surely
\begin{align*}
\nabla(\sigma_{T,i,j,k})_r=&\int_{0}^{+\infty}e^{-\frac{s}{T}} (e_k\cdot q_{e_i,T}-e_j\cdot q_{e_i,T})\star (\nabla(\partial_j g_{\sqrt{s+r^2}}-\partial_k g_{\sqrt{s+r^2}}))\dd s\\
=&\int_{0}^{+\infty}e^{-\frac{s}{T}}(\nabla(\partial_j(e_k\cdot q_{e_i,T})-\partial_k(e_j\cdot q_{e_i,T})))_{\sqrt{s+r^2}}\, \dd s.
\end{align*}
Consequently, by making once again use of the stationarity of $q_{e_i,T}$ as well as \eqref{boundconvolH2} and \eqref{correctorbound2} proved for $(q_{e_i,T})_r$ in Substep $2.1$, we obtain
\begin{align}
\left\langle\vert\nabla(\sigma_{T,i,j,k})_r\vert^p\right\rangle^{\frac{1}{p}}&=
C\left\langle\left\vert\int_{0}^{+\infty} e^{-\frac{s}{T}}(\nabla(\partial_j(e_k\cdot q_{e_i,T})-\partial_k(e_j\cdot q_{e_i,T})))_{\sqrt{s+r^2}}\, \dd s\right\vert^p\right\rangle^{\frac{1}{p}}\nonumber\\
&\stackrel{\eqref{boundconvolH2}}{\lesssim}\int_{0}^{+\infty}\frac{1}{s+r^2}\left\langle\left\vert(q_{e_i,T})_{\frac{1}{\sqrt{2}}\sqrt{s+r^2}}-\left\langle (q_{e_i,T})_{\frac{1}{\sqrt{2}}\sqrt{s+r^2}}\right\rangle\right\vert^p\right\rangle^{\frac{1}{p}}\dd s\nonumber\\
&\lesssim p^{\frac{1}{2}+\frac{1}{\alpha}}\int_{0}^{+\infty}\frac{1}{s+r^2}\pi^{-\frac{1}{2}}_*(\sqrt{s+r^2})\dd s\nonumber\\
&\lesssim p^{\frac{1}{2}+\frac{1}{\alpha}}\pi^{-\frac{1}{2}}_*(r).\label{boudsigmafinalstep}
\end{align}
which concludes \eqref{correctorbound2} for $\vert\nabla(\sigma_T)_r\vert$ using Lemma \ref{momentexp}. The bound on $\vert\nabla\sigma_r\vert$ then follows from the fact that $\nabla\sigma_T$ tends to $\nabla\sigma$ in $\LL^2_{\left\langle\cdot\right\rangle}(\Omega)$ as $T\uparrow \infty$ (see for instance \cite[Theorem 1]{gloria2016reduction}), and thus also almost surely up to a subsequence, combined with \eqref{boudsigmafinalstep} and Fatou's lemma.  
\subsubsection{Proof of Corollary \ref{boundextendedcorrector} : Growth of the extended corrector $(\phi,\sigma)$}
We only give the arguments for $\phi$. For the bound on $\sigma$, we may rewrite averages $\int_{\mathbb{R}^d}\nabla\sigma_{i,j,k}(x)\cdot g(x)\dd x$ where $g$ is assumed to be a gradient field; i.e., $g=\nabla\theta$ for some potential $\theta$, using the second line of \eqref{correctorequationsigma2} to obtain
$$\int_{\mathbb{R}^d}\nabla\sigma_{i,j,k}(x)\cdot g(x)\dd x=\int_{\mathbb{R}^d}q(x)\cdot \text{S}g(x)\dd x,$$
with $\text{S}=:e_j\otimes e_k-e_k\otimes e_j$ and $q=((q_{e_i})_j)_{i,j}$. Since the averaging field $\text{S}g$ inherits the decay properties of $g$, we then conclude using Theorem \ref{semigroup} for $q$ and the arguments for $\phi$.\newline
\newline 
Let $\alpha<\frac{1}{\frac{1}{2}+\frac{\frac{5}{2}d+2}{\beta\vee d}}$ and $p\in [1,\infty)$. On the one hand, by the triangle inequality combined with Poincar\'e's inequality in $\LL^2(\mathbb{R}^d,g_1\dd x)$ and the stationarity property of $\nabla\phi$, we have for all $x\in\mathbb{R}^d$
\begin{equation}
\left\langle(\vert\phi-\phi_1(0)\vert^2)^{\frac{p}{2}}_1(x)\right\rangle^{\frac{1}{p}}\lesssim \left\langle\vert\phi_1(x)-\phi_1(0)\vert^p\right\rangle^{\frac{1}{p}}+\left\langle(\vert\nabla\phi\vert^2)^{\frac{p}{2}}_1\right\rangle^{\frac{1}{p}}.
\label{growthcorrectoresti1}
\end{equation}
Then, using the formula \eqref{linkphiu}, the energy estimate \eqref{LemE1} applied to the equation \eqref{equationu} in form of $\int_{\mathbb{R}^d}\left\vert\int_{0}^1\nabla u(t,x)\dd t\right\vert^2 g_1(x)\dd x\lesssim 1$, Minkowski's inequality in $\LL^p(\Omega,\LL^2(\mathbb{R}^d,g_1\dd x))$ and the estimates \eqref{qminusqr2} as well as \eqref{qminusqr3} applied with $r=1$ (after dominating the Gaussian kernel $g_1$ by the exponential kernel $\eta_1$), we have
\begin{align*}
\left\langle(\vert\nabla\phi\vert^2)^{\frac{p}{2}}_1\right\rangle^{\frac{1}{p}}&\stackrel{\eqref{linkphiu}}{=}\left\langle \left(\int_{\mathbb{R}^d}\left\vert\int_{0}^{+\infty}\nabla u(t,x)\dd t\right\vert^2 g_1(x)\dd x\right)^{\frac{p}{2}}\right\rangle^{\frac{1}{p}}\\
&\leq \left\langle\left(\int_{\mathbb{R}^d}\left\vert\int_{0}^1\nabla u(t,x)\dd t\right\vert^2 g_1(x)\dd x\right)^{\frac{p}{2}}\right\rangle^{\frac{1}{p}}+\int_{1}^{+\infty}\left\langle\left(\int_{\mathbb{R}^d}\vert\nabla u(t,x)\vert^2 g_1(x)\dd x\right)^{\frac{p}{2}}\right\rangle^{\frac{1}{p}}\dd t\\
&\stackrel{\eqref{qminusqr2},\eqref{qminusqr3}}{\lesssim} 1+p^{\frac{1}{\alpha}}.
\end{align*}
Therefore, the estimate \eqref{growthcorrectoresti1} turns into
\begin{equation}
\left\langle(\vert\phi-\phi_1(0)\vert^2)^{\frac{p}{2}}_1(x)\right\rangle^{\frac{1}{p}}\lesssim \left\langle\vert\phi_1(x)-\phi_1(0)\vert^p\right\rangle^{\frac{1}{p}}+p^{\frac{1}{\alpha}}+1.
\label{growthcorrectoresti2}
\end{equation}
On the other hand, setting $R=\vert x\vert\geq 1$, we have by the triangle inequality 
\begin{equation}
\left\langle\vert\phi_1(x)-\phi_1(0)\vert^p\right\rangle^{\frac{1}{p}}\leq \left\langle\vert\phi_R(x)-\phi_1(x)\vert^p\right\rangle^{\frac{1}{p}}+\left\langle\vert\phi_R(x)-\phi_R(0)\vert^p\right\rangle^{\frac{1}{p}}+\left\langle\vert\phi_R(0)-\phi_1(0)\vert^p\right\rangle^{\frac{1}{p}}.
\label{growthcorrectoresti3}
\end{equation}
The second r.h.s term of \eqref{growthcorrectoresti3} is estimated via the fundamental calculus theorem combined with Minkowski's inequality in $\LL^p_{\left\langle\cdot\right\rangle}(\Omega)$, the stationary property of $\nabla\phi$ and \eqref{correctorbound2} 
\begin{equation}
\left\langle\vert\phi_R(x)-\phi_R(0)\vert^p\right\rangle^{\frac{1}{p}}=\vert x\vert\left\langle\left(\int_{0}^1\frac{x}{\vert x\vert}\cdot \left(\int_{\mathbb{R}^d}\nabla\phi(y+\tau x)g_R(y)\dd y\right)\dd \tau\right)^p\right\rangle^{\frac{1}{p}}\stackrel{\eqref{correctorbound2}}{\lesssim} p^{\frac{1}{\alpha}}\vert x\vert\pi^{-\frac{1}{2}}_\star(R).
\label{growthcorrectoresti4}
\end{equation}
By stationarity, the first and the third r.h.s term of \eqref{growthcorrectoresti3} are estimated the same way and we bound the third term in two different ways, depending on the regimes in $\beta$ and $d$: 
\begin{itemize}
\item[(i)]We consider the regimes $\beta<2$, $\beta=d=2$ and $\beta>2,\, d>2$. Our main tool here are the moment bounds on the gradients of correctors \eqref{correctorbound2}. We write by $\phi_\tau(0)=\int_{\mathbb{R}^d}\phi(\tau x)g_1(x)\dd x$ and the fundamental calculus theorem
\begin{equation}
\phi_{R}(0)-\phi_{1}(0)=\int_{1}^R \partial_\tau \phi_{\tau}(0)\dd \tau=\int_{1}^R\int_{\mathbb{R}^d}\nabla\phi(\tau x)\cdot g_{1}(x)x\,\dd x\, \dd \tau=\int_{1}^R\int_{\mathbb{R}^d}\nabla\phi(x)\cdot g_{\tau}(x)\frac{x}{\tau}\dd x\, \dd \tau.
\label{growthcorrectoresti5}
\end{equation}
Then, by noticing that from the semigroup property of Gaussian field in form of $g_{\tau}=g_{\frac{\tau}{\sqrt{2}}}\star g_{\frac{\tau}{\sqrt{2}}}$, writing $\frac{x}{\tau}=\frac{y}{\tau}+\frac{x-y}{\tau}$ and applying Fubini's theorem, we have for all $\tau\in [1,R]$
\begin{align*}
\int_{\mathbb{R}^d}\nabla\phi(x)\cdot g_{\tau}(x)\frac{x}{\tau}\dd x&=\int_{\mathbb{R}^d}\int_{\mathbb{R}^d}\nabla\phi(x)\cdot \frac{x}{\tau} g_{\frac{\tau}{\sqrt{2}}}(y)g_{\frac{\tau}{\sqrt{2}}}(x-y)\dd y\, \dd x\\
&=\int_{\mathbb{R}^d}\int_{\mathbb{R}^d}\nabla\phi(x)\cdot\frac{y}{\tau}g_{\frac{\tau}{\sqrt{2}}}(y)g_{\frac{\tau}{\sqrt{2}}}(x-y)\dd y\, \dd x+\int_{\mathbb{R}^d}\int_{\mathbb{R}^d}\nabla\phi(x)\cdot\frac{x-y}{\tau}g_{\frac{\tau}{\sqrt{2}}}(y)g_{\frac{\tau}{\sqrt{2}}}(x-y)\dd y\, \dd x\\
&=2\int_{\mathbb{R}^d}\frac{y}{\tau}g_{\frac{\tau}{\sqrt{2}}}(y)\cdot\int_{\mathbb{R}^d}\nabla\phi(x) g_{\frac{\tau}{\sqrt{2}}}(x-y)\dd x\, \dd y\\
&=2\int_{\mathbb{R}^d}\frac{y}{\tau}g_{\frac{\tau}{\sqrt{2}}}(y)\cdot\nabla\phi_{\frac{\tau}{\sqrt{2}}}(-y)\dd y.
\end{align*}
We deduce from Minkowski's inequality in $\LL^p_{\left\langle\cdot\right\rangle}(\Omega)$, the stationarity property of $\nabla\phi$, \eqref{growthcorrectoresti5} and \eqref{correctorbound2}
\begin{equation}
\left\langle\vert \phi_{R}(0)-\phi_1(0)\vert^p\right\rangle^{\frac{1}{p}}\lesssim \int_{1}^R \left\langle\vert\nabla\phi_{\frac{\tau}{\sqrt{2}}}\vert^p\right\rangle^{\frac{1}{p}}\dd \tau\stackrel{\eqref{correctorbound2}}{\lesssim}p^{\frac{1}{\alpha}}\int_{1}^R \pi^{-\frac{1}{2}}_\star(\tfrac{\tau}{\sqrt{2}})\dd \tau\lesssim p^{\frac{1}{\alpha}}\xi_{d,\beta}(R),
\label{growthcorrectoresti6}
\end{equation}
where we recall that $\xi_{d,\beta}$ is defined in \eqref{defxidbeta}. The combination of \eqref{growthcorrectoresti2}, \eqref{growthcorrectoresti3}, \eqref{growthcorrectoresti4}, \eqref{growthcorrectoresti6} and Lemma \ref{momentexp} gives the desired estimate \eqref{growthcorrectormainresult}.
\item[(ii)]We consider the regimes $\beta=2,\, d>2$ and $\beta>2,\, d=2$. Our main tools here are the fluctuation estimate \eqref{sensiothertest} and the decay \eqref{estisemigroup} of $\nabla u$. We claim that
\begin{equation}
\phi_{R}(0)-\phi_{1}(0)=-\int_{\mathbb{R}^d}\nabla H(x)\cdot \nabla\phi(x)\dd x\quad \text{with}\quad H:= \int_{1}^{R^2}g_{\sqrt{\tau}}(\cdot)\dd \tau.
\label{growthcorrectoresti7}
\end{equation}
Indeed, using that for all $\tau>0$, $\partial_{\tau} g_{\sqrt{\tau}}=\Delta g_{\sqrt{\tau}}$, we have 
\begin{equation*}
\phi_{R}(0)-\phi_{1}(0)=\int_{\mathbb{R}^d}(g_R(x)-g_1(x))\phi(x)\dd x=\int_{\mathbb{R}^d}\int_{1}^{R^2}\partial_{\tau} g_{\sqrt{\tau}}(x)\dd \tau\, \phi(x)\dd x=\int_{\mathbb{R}^d}\int_1^{R^2}\Delta g_{\sqrt{\tau}}(x)\dd \tau\, \phi(x)\dd x.
\end{equation*}
Thus \eqref{growthcorrectoresti7} follows from an integration by parts (which is justified by the sub-linearity property of the corrector $\phi$). Now, using the formula \eqref{linkphiu}, we get
\begin{align}
\phi_{R}(0)-\phi_1(0)&=-\int_{\mathbb{R}^d}\nabla H(x)\cdot \Big(\int_{0}^{+\infty}\nabla u(s,x)\dd s\Big)\, \dd x\nonumber\\
&=-\int_{\mathbb{R}^d}\nabla H(x)\cdot \nabla\phi(R^2,x)\dd x-\int_{\mathbb{R}^d}\nabla H(x)\cdot \int_{R^2}^{+\infty}\nabla u(s,x)\dd s\, \dd x,\label{growthcorrectoresti8}
\end{align}
where we recall that the time dependant corrector $\phi(\cdot,\cdot)$ is defined in \eqref{defphitime}. For the first r.h.s term of \eqref{growthcorrectoresti8}, we note that $\nabla H$ satisfies the assumption \eqref{assumesensiothertest}, therefore from Theorem \ref{semigroup} we have for all $p\in [1,\infty)$
\begin{equation*}
\left\langle\left\vert\int_{\mathbb{R}^d}\nabla H(x)\cdot \nabla\phi(R^2,x)\dd x\right\vert^p\right\rangle^{\frac{1}{p}}\lesssim p^{\frac{1}{\alpha}}\log^{\frac{1}{2}}(R+2).
\label{growthcorrectoresti9}
\end{equation*}
For the second r.h.s term of \eqref{growthcorrectoresti8}, we make use of the combination of \eqref{Lem2semigroup}, \eqref{qminusqr2} and \eqref{qminusqr3} as well as the following bound on $\nabla H$: for all $x\in\mathbb{R}^d$
\begin{align*}
\vert \nabla H(x)\vert\lesssim \int_{1}^{R^2}\vert x\vert\tau^{-\frac{d}{2}-1}e^{-\frac{\vert x\vert}{\tau}}\dd \tau &\leq \vert x\vert e^{-\frac{\vert x\vert^2}{2R^2}}\int_{1}^{R^2}\tau^{-\frac{d}{2}-1}e^{-\frac{\vert x\vert^2}{2\tau}}\dd \tau\\
&\lesssim (\vert x\vert+1)^{1-d} e^{-\frac{\vert x\vert^2}{2R^2}}\lesssim R\,g_{2R}(x),
\end{align*}
to obtain for all $p\in [1,\infty)$
\begin{align*}
\left\langle\left\vert \int_{\mathbb{R}^d}\nabla H(x)\cdot \int_{R^2}^{+\infty}\nabla u(s,x)\dd s\, \dd x\right\vert^p\right\rangle^{\frac{1}{p}}\leq & R\int_{R^2}^{+\infty}\left\langle\left\vert\int_{\mathbb{R}^d}g_{2R}(x) \vert \nabla u(s,y)\vert^2\right\vert^{\frac{p}{2}}\right\rangle^{\frac{1}{p}}\\
&\lesssim p^{\frac{1}{\alpha}} R\int_{R^2}^{+\infty}s^{-\frac{1}{2}}\eta_{\beta}(s)\dd s\lesssim p^{\frac{1}{\alpha}}.
\end{align*}
\end{itemize}
\subsubsection{Proof of Corollary \ref{approxcorrector}: Sub-systematic error} We split the proof into two steps.\newline
\newline
\textbf{Step 1. } Proof of \eqref{subsysteT}. Using the two representation formulas \eqref{linkphiu} and \eqref{linkphiumassive}, we have for all $n>\frac{\beta\wedge d}{4}$
\begin{equation}
\left\langle\vert\nabla\phi^n_{e_i,T}-\nabla\phi_{e_i}\vert^2\right\rangle^{\frac{1}{2}}=\left\langle\left\vert\int_{0}^{+\infty}(1-\text{exp}_n(\tau,T))\nabla u(\tau)\dd \tau\right\vert^2\right\rangle^{\frac{1}{2}},
\label{intsubsyste}
\end{equation}
where $(\text{exp}_n(\cdot,T))_{n\in\mathbb{N}}$ is the Richardson extrapolation of $\text{exp}_1(\cdot,T):=e^{-\frac{\cdot}{T}}$. Note that the extrapolation has the effect that for all $\tau\geq 0$
\begin{equation}
\vert1-\text{exp}_n(\tau,T)\vert\lesssim_n \left(\frac{\tau}{T}\right)^n\wedge 1 \text{ and } \left\vert\frac{\partial}{\partial\tau}\text{exp}_n(\tau,T)\right\vert\lesssim_n \frac{1}{T}\left(\frac{\tau}{T}\right)^{n-1}.
\label{richardson}
\end{equation}
We then split the integral \eqref{intsubsyste} into three contributions. We start by the contribution on the interval $(0,1)$. We write by an integration by parts
$$\int_{0}^{1}(1-\text{exp}_n(\tau,T))\nabla u(\tau)\dd \tau=\int_{0}^1\frac{\partial}{\partial \tau}\text{exp}_n(\tau,T)\int_{0}^{\tau}\nabla u(t)\dd t\, \dd\tau+(1-\text{exp}_n(1,T))\int_{0}^1\nabla u(\tau)\dd \tau.$$
Thus, by Minkowski's inequality in $\LL^2_{\left\langle\cdot\right\rangle}(\Omega)$ and the stationarity of $\nabla u$, we get
\begin{align*}
\left\langle\left\vert\int_0^1 (1-\text{exp}_n(\tau,T))\nabla u(\tau)\dd \tau\right\vert^2\right\rangle^{\frac{1}{2}}\lesssim & \int_{0}^1\left\vert\frac{\partial}{\partial\tau} \text{exp}_n(\tau,T)\right\vert^2\dd \tau\int_{0}^1\left\langle\int_{\mathbb{R}^d}\eta_{\sqrt{\tau}}(x)\left\vert\int_{0}^{\tau}\nabla u(t,x)\dd t\right\vert^2\dd x\right\rangle^{\frac{1}{2}}\dd \tau\\
&+\vert 1-\text{exp}_n(1,T)\vert^2\left\langle\int_{\mathbb{R}^d}\eta_1(x)\left\vert\int_{0}^1\nabla u(t,x)\dd t\right\vert^2\dd x\right\rangle^{\frac{1}{2}}.
\end{align*}
Hence, using the localized energy estimate \eqref{LemE1} combined with \eqref{richardson}, we arrive at
$$\left\langle\left\vert\int_0^1 (1-\text{exp}_n(\tau,T))\nabla u(\tau)\dd \tau\right\vert^2\right\rangle^{\frac{1}{2}}\lesssim_n T^{-n},$$
which is of higher order than the r.h.s of \eqref{subsysteT}. We now turn to the contributions on the intervals $(1,T)$ and $(T,\infty)$, for which the estimate of the decay of the semigroup \eqref{estip2} combined with \eqref{richardson} yield
\begin{align*}
\left\langle\left\vert\int_{1}^{+\infty}(1-\text{exp}_n(\tau,T))\nabla u(\tau)\dd\tau\right\vert^2\right\rangle^{\frac{1}{2}}&\lesssim\int_{1}^T\left(\frac{\tau}{T}\right)^n\tau^{-1-\frac{d}{4}}\dd \tau+\int_{T}^{+\infty}\tau^{-1-\frac{d}{4}}\dd \tau\\
&\lesssim T^{\frac{1}{2}}\eta_{\beta}(T).
\end{align*}
This concludes the proof of \eqref{subsysteT}.\newline
\newline
\textbf{Step 2. }Proof of \eqref{subsystehom}. This estimate is a direct consequence of \eqref{subsysteT}. Indeed by the definition \eqref{defabarnT} of $\overline{a}^n_T$, we have
$$e_j\cdot(\overline{a}^n_T-a_{\text{hom}})e_i=\left\langle(\nabla\phi^{*,n}_{e_j,T}-\nabla\phi^{*}_{e_j})\cdot a(\nabla\phi^{n}_{e_i,T}+e_i)\right\rangle-\left\langle(\nabla\phi^*_{e_i}+e_j)\cdot a(\nabla\phi_{e_i}-\nabla\phi^{n}_{e_i,T})\right\rangle.$$
Since we have
$$\left\langle(\nabla\phi^{*}_{e,j}+e_j)\cdot a(\nabla\phi_{e_i}-\nabla\phi^{n}_{e_i,T})\right\rangle=\left\langle(\nabla\phi_{e_i}-\nabla\phi^n_{e_i,T})\cdot a^*(\nabla\phi^{*}_{e_j}+e_j)\right\rangle,$$
the weak formulation of the corrector equation \eqref{correctorequation2} for both $\phi^*_{e_i}$ and $\phi_{e_i}$ yields
$$\left\langle(\nabla\phi_{e_i}-\nabla\phi^n_{e_i,T})\cdot a^*(\nabla\phi^*_{e_j}+e_j)\right\rangle=\left\langle(\nabla\phi^*_{e_j}-\nabla\phi^{*,n}_{e_j,T})\cdot a(\nabla\phi_{e_i}+e_i)\right\rangle,$$
and we conclude that
$$\vert e_j\cdot(\overline{a}^n_{T}-a_{\text{hom}})e_i\vert=\left\vert\left\langle(\nabla\phi^{*,n}_{e_j,T}-\nabla\phi^*_{e_j})\cdot a(\nabla \phi^n_{e_i,T}-\nabla\phi_{e_i})\right\rangle\right\vert\leq \left\langle\vert\nabla\phi^{*,n}_{e_j,T}-\nabla\phi^*_{e_j}\vert^2\right\rangle^{\frac{1}{2}}\left\langle\vert\nabla \phi^n_{e_i,T}-\nabla\phi_{e_i}\vert^2\right\rangle^{\frac{1}{2}},$$
so that the claim follows from \eqref{subsysteT}, used for both $a^*$ and $a$.

\subsubsection{Proof of Corollary \ref{spectralreso}: Spectral resolution }
Let $0<\mu\leq 1$. The starting point is the use of the spectral theorem which allow us to rewrite the definition of $(\phi^{n}_{e,\mu^{-1}})_{n\in\mathbb{N}}$, given in Corollary \ref{approxcorrector}, in the form, for all $n\in\mathbb{N}$
$$\phi^n_{e,T}=g_n(\mathcal{L},\mu^{-1})\Theta,$$
where $g_0 : \zeta\in (0,+\infty)\mapsto \frac{1}{\zeta}$, $g_1 : \zeta\in (0,+\infty)\mapsto (\zeta+\mu)^{-1}$, and $(g_n)_{n\in\mathbb{N}}$ is the Richardson extrapolation of $g_1$ with respect to $\mu^{-1}$. Then, by the spectral theorem, we have for all $n\in\mathbb{N}$
\begin{align}
\left\langle\nabla (\phi^n_{e,\mu^{-1}}-\phi_e)\cdot a\nabla (\phi^n_{e,\mu^{-1}}-\phi_e)\right\rangle &=\left\langle(\phi^n_{e,\mu^{-1}}-\phi_e)\mathcal{L}(\phi^n_{e,\mu^{-1}}-\phi_e)\right\rangle\nonumber\\
&=\left\langle\int_{0}^{+\infty}\zeta(g_n(\zeta,\mu^{-1})-g_0(\zeta))^2\dd\nu_{\Theta}(\zeta)\right\rangle.\label{decompospectralth}
\end{align}
On the one hand, for $n>\frac{\beta\wedge d}{4}$, Corollary \ref{approxcorrector} yields
\begin{equation}
\left\langle\nabla (\phi^n_{e,\mu^{-1}}-\phi_e)\cdot a\nabla (\phi^n_{e,\mu^{-1}}-\phi_e)\right\rangle\lesssim \mu^{-1}\eta^2_{\beta}(\mu^{-1}).
\label{applycor3}
\end{equation}
On the other hand, by induction on $n$ (see for instance \cite[Proof of Lemma 2.5]{gloria2016quantitative}) we have for all $n\in \mathbb{N}$ and $\zeta\in (0,+\infty)$
$$\vert g_n(\zeta,\mu^{-1})-g_0(\zeta)\vert\gtrsim \frac{\mu^{n}}{\zeta(\zeta+\mu)^n},$$
which we use in the form of, for all $\zeta\leq \mu$
\begin{equation}
\zeta(g_n(\zeta,\mu^{-1})-g_0(\zeta))^2\gtrsim \frac{\mu^{2n}}{\zeta(\zeta+\mu)^{2n}}.
\label{induction}
\end{equation}
The combination of \eqref{decompospectralth}, \eqref{applycor3} and \eqref{induction} applied for some $n>\frac{\beta\wedge d}{4}$ gives
\begin{align*}
\left\langle \int_{0}^{\mu}\dd \nu_{\Theta}(\zeta)\right\rangle&\lesssim \mu\left\langle\int_{0}^{\mu}\frac{\mu^{2n}}{\zeta(\zeta+\mu)^{2n}} \dd \nu_{\Theta}(\zeta)\right\rangle\\
&\stackrel{\eqref{induction}}{\lesssim} \mu\left\langle\int_{0}^{\mu}\zeta(g_n(\zeta,\mu^{-1})-g_0(\zeta))^2 \dd \nu_{\Theta}(\zeta)\right\rangle\\
&\stackrel{\eqref{decompospectralth},\eqref{applycor3}}{\lesssim}\eta^2_{\beta}(\mu^{-1}).
\end{align*}

\appendix

\section{Probabilistic tools}
\label{LSIapp}
The following proposition shows that the multiscale logarithmic Sobolev inequality \eqref{SGinegp} gives a control of moments. For a reference, see \cite[Proposition 1.10]{duerinckx2017weighted2}.
\begin{proposition}\label{SGp} Assume that the ensemble $\left\langle\cdot\right\rangle$ satisfies the multi-scale logarithm Sobolev inequality \eqref{SGinegp}. For all $p\in [1,\infty)$ and $F\in L^p_{\left\langle\cdot\right\rangle}(\Omega)$
\begin{equation}
\left\langle \vert F-\left\langle F\right\rangle\vert^p\right\rangle^{\frac{1}{p}}\lesssim_d \sqrt{p}\left\langle\left(\int_{1}^{+\infty}\ell^{-d}\pi(\ell)\int_{\mathbb{R}^d}\vert \partial^{\text{fct}}_{x,\ell} F\vert^2\dd x\,\dd \ell\right)^{\frac{p}{2}}\right\rangle^{\frac{1}{p}}.
\label{SGinegp1}
\end{equation}
\end{proposition}
The following standard lemma gives the link between algebraic moment and exponential moment for non-negative random variables. The short proof is included for completeness.
\begin{lemma}\label{momentexp}Let $X : \Omega\rightarrow \mathbb{R}^+$ a non-negative random variable. We have the following equivalence:
\begin{equation}
\exists C_1>0 \text{ such that } \left\langle\exp\left(\frac{1}{C_1}X\right)\right\rangle\leq 2 \Leftrightarrow \exists C_2>0 \text{ such that } \forall q\geq 1, \left\langle X^q\right\rangle^{\frac{1}{q}}\leq q C_2.
\label{algeexpmoment}
\end{equation}
\end{lemma}
\begin{proof}Let us suppose that there exists $C_2>0$ such that for all $q\geq 1$, $\left\langle X^q\right\rangle^{\frac{1}{q}}\leq q C_2$. We have, for all $C_1>0$
$$\left\langle\exp\left(\frac{1}{C_1}X\right)\right\rangle=\left\langle\sum_{n=0}^{+\infty}\frac{X^n}{n!C_1^n}\right\rangle\leq \sum_{n=0}^{+\infty}\frac{\left(\frac{C_2}{C_1}n\right)^n}{n!},$$
we then choose $C_1$ such that $\sum_{n=0}^{+\infty}\frac{\left(\frac{C_2}{C_1}n\right)^n}{n!}\leq 2$. Let us now suppose that there exists $C_1>0$ such that $\left\langle\exp\left(\frac{1}{C_1}X\right)\right\rangle\leq 2$. This implies that for all $q\geq 1$, $\left\langle X^q\right\rangle\leq  C^q_1 q!$. Since, from the Stirling formula, $q!\leq C q^q$ for some $C>0$, we have for all $q\in\mathbb{N}$, $\left\langle X^q\right\rangle^{\frac{1}{q}}\leq C C_1 q$.
\end{proof}

\section{Large-scale regularity theory for parabolic system}\label{reggech}
In this section we recall the regularity theory for random parabolic operator of the form $\partial_{\tau}-\nabla\cdot a\nabla$ developed in the papers \cite{bella2017liouville,armstrong2018quantitative} and draw some useful consequences. Here, we assume that $a$ does not depend on time. However, the theory also holds with time dependent coefficients, using a time dependent minimal radius $r_*$ different from the one defined in Theorem \ref{regpara} but this is not needed in this paper. 

\medskip

\noindent The general idea of large-scale regularity is to make use of the nice regularity theory that enjoy the homogenized operator $\partial_{\tau}-\nabla\cdot a_{\text{hom}}\nabla$.  Indeed, the proximity of the two resolvent of the operators $\partial_{\tau}-\nabla\cdot a\nabla$ and $\partial_{\tau}-\nabla\cdot a_{\text{hom}}\nabla$ provided by homogenization allows to infer an improvement of regularity for $\partial_{\tau}-\nabla \cdot a\nabla$ on large-scales, say, scale much larger than the correlation length (quantitatively characterized by the random variable $r_*$ in Theorem \ref{regpara}).  In other words, on large-scales, the heterogeneous linear parabolic operator $\partial_{\tau}-\nabla\cdot a\nabla$ "inherits" a suitable version of the regularity theory for the homogenized linear parabolic operator $\partial_{\tau}-\nabla\cdot a_{\text{hom}}\nabla$.

\medskip

\noindent We start by recalling the excess decay property, which can be found  in \cite[Proposition 2]{bella2017liouville} and the moment bound on $r_*$ which can be found in \cite{gloria2014regularity}. We then prove large-scale $\cc^{0,1}$ estimates, following the arguments of \cite{gloria2014regularity}.
\begin{theorem}[Excess decay]\label{regpara}There exists a $\frac{1}{8}$-Lipschitz stationary random field $r_*: \Omega\times\mathbb{R}^d\rightarrow \mathbb{R}^+$ for which there exists a constant $C<+\infty$ such that for all $x\in\mathbb{R}^d$
\begin{equation}
\left\langle\exp\left(\frac{1}{C}\pi_*(r_*(x))\right)\right\rangle\leq 2,
\label{momentr*}
\end{equation}
with for all $r\geq 1$
$$
\pi_*(r)=\left\{
    \begin{array}{ll}
        r^{\beta} & \text{ if $\beta<d$}, \\
        r^d\log^{-1}(r) & \text{ if $\beta=d$},\\
				r^d & \text{ if $\beta>d$}. 
    \end{array}
\right.
$$
In addition, for all distributional solution of 
$$\partial_{\tau}u-\nabla\cdot a\nabla u=0 \text{\quad in $C_R$ for $R\geq r_*$},$$
we have for all $ r\in [r_*,R]$ and $\alpha\in (0,1)$
\begin{equation}
\text{Exc}(\nabla u,r)\lesssim_{d,\lambda,\alpha}\left(\frac{r}{R}\right)^{2\alpha}\text{Exc}(\nabla u, R),
\label{regpara2}
\end{equation}
with $\text{Exc}(\nabla u, r):=\inf_{\xi\in\mathbb{R}^d}\fint_{C_r}\vert \nabla u(t,y)-\xi-\nabla\phi_{\xi}(y)\vert^2\dd y$.
\end{theorem}
A direct consequence of the excess decay property of Theorem \ref{regpara} is the following large-scale $\cc^{0,1}$ estimates, in the spirit of \cite{gloria2014regularity}, stated in the parabolic setting.
\begin{cor}[Large-scale $\cc^{0,1}$ estimates]\label{schauder}Consider the random field $r_*$ defined in Theorem \ref{regpara} and for all $(s,x)\in\mathbb{R}^{d+1}$, $u$ be the weak solution of 
$$\partial_{\tau} u-\nabla\cdot (a(\nabla u+g))=\nabla\cdot h\text{\quad in $\cc_{R}(s,x)$ for $R\geq r_*$},$$
with $(g,h)\in L^2_{\text{loc}}(\mathbb{R}^d)$. We have for all $r\in [r_*(x),R]$ and $\alpha>0$
\begin{equation}
\fint_{\cc_r(s,x)}\vert\nabla u(t,y)\vert^2\dd t\, \dd y\lesssim_{d,\lambda,\alpha}\fint_{\cc_R(s,x)}\vert \nabla u(t,y)\vert^2\dd t\, \dd y+\sup_{\rho\in [r_*,R]}\left(\frac{R}{\rho}\right)^{2\alpha}\fint_{\cc_\rho(s,x)}\left(\left\vert h-\fint_{\cc_\rho(s,x)} h\right\vert^2+\left\vert g-\fint_{\cc_\rho(s,x)}g\right\vert^2\right).
\label{meanvalue2}
\end{equation}
In particular, if $g\equiv h\equiv 0$, we have the following mean value property for $a$-caloric functions: for all $ r\in [r_*(x),R]$
\begin{equation}
\fint_{\cc_r(s,x)}\vert \nabla u(t,y)\vert^2\dd t\, \dd y\lesssim_{d,\lambda}\fint_{\cc_R(s,x)}\vert \nabla u(t,y)\vert^2\dd t\, \dd y. 
\label{meanvalue}
\end{equation}
\end{cor}
\begin{proof}[Proof]
Without loss of generality, we may assume that $(s,x)=(0,0)$. We split the proof into two steps.\newline
\newline
\textbf{Step 1.} Proof of
\begin{equation}
\sup_{r\in[r_*,R]} \frac{1}{r^{2\alpha}}\text{Exc}(\nabla u+g,r)\lesssim_{d,\lambda,\alpha} \frac{1}{R^{2\alpha}}\text{Exc}(\nabla u+g,R)+\sup_{r\in[r_*,R]}\frac{1}{r^{2\alpha}}\fint_{\cc_r}\bigg(\vert h-\fint_{\cc_r} h\vert^2+\vert g-\fint_{\cc_{r}}g\vert^2\bigg),
\label{ScaseRfinite}
\end{equation}
and if $R=+\infty$
\begin{equation}
\sup_{r\geq r_*} \frac{1}{r^{2\alpha}}Exc(\nabla u+g,r)\lesssim_{d,\lambda,\alpha} \sup_{r\geq r_*}\frac{1}{r^{2\alpha}}\fint_{\cc_{\rho}}\bigg(\vert h-\fint_{\cc_{\rho}} h\vert^2+\vert g-\fint_{\cc_{\rho}}g\vert^2\bigg).
\label{ScaseRinfinite}
\end{equation}
Let $\alpha'=\frac{1+\alpha}{2}$ and $r_*\leq r\leq \rho\leq R$. We prove that
\begin{equation}
\text{Exc}(\nabla u+g,r)\leq C_1\bigg(\Big(\frac{r}{\rho}\Big)^{2\alpha'}\text{Exc}(\nabla u+g,\rho)+\Big(\frac{\rho}{r}\Big)^{d+2}\fint_{\cc_{\rho}}\Big(\vert h-\fint_{\cc_{\rho}} h\vert^2+\vert g-\fint_{\cc_{\rho}}g\vert^2\Big)\bigg),
\label{exc1}
\end{equation}
with some constant $C_1$ depending on $\lambda$ and $d$.\newline
\newline
Set $\xi:=\fint_{\cc_{\rho}}g$ and let $w\in \LL^2((-\rho^2,\rho^2),\text{H}_0^1(B_{\rho}))\cap \text{H}^1((-\rho^2,\rho^2),\text{H}^{-1}(B_{\rho}))$ be the weak  solution of 
\begin{equation}
 \left\{
    \begin{array}{ll}
       \partial_t w+\nabla \cdot a\nabla w=\nabla\cdot(a(g-\xi) + h) & \text{ in } \cc_{\rho}, \\
        w=0 \text{ on } \partial_{p} \cc_{\rho},& 
    \end{array}
\right.
\label{edpw}
\end{equation}
where $\partial_p \cc_{\rho}=(\partial \bb_{\rho}\times (-\rho^2,0))\cup \bb_{\rho}\times \{ 0\}$. Then, because $(t,x)\in\mathbb{R}^{d+1}\mapsto u(t,x)-w(t,x)+\xi\cdot x$ is a $a$-caloric function in $\cc_{\rho}$, we have by Theorem \ref{regpara} for the exponent $\alpha'$
\begin{equation}
\text{Exc}(\nabla u-\nabla w+\xi,r)\lesssim \Big(\frac{r}{\rho}\Big)^{2\alpha'}\text{Exc}(\nabla u-\nabla w+\xi,\rho).
\label{excesdapp}
\end{equation}
In addition, we have the following energy estimate 
\begin{equation}
\int_{\cc_{\rho}}\vert \nabla w\vert^2\lesssim \int_{\cc_{\rho}}\left\vert h-\fint_{\cc_{\rho}}h\right\vert^2+\int_{\cc_{\rho}}\left\vert g-\fint_{\cc_{\rho}}g\right\vert^2.
\label{energyest}
\end{equation}
Indeed, by testing \eqref{edpw} by $w$ itself, we get
$$-\int_{\cc_{\rho}}w\partial_\tau w+\int_{\cc_{\rho}}\nabla w\cdot a\nabla w = \int_{\cc_{\rho}}\left(h-\fint_{\cc_{\rho}}h+\nabla w\cdot a(g-\xi)\right).$$
Since
$$-\int_{\cc_{\rho}}w\partial_\tau w=-\int_{-\rho^2}^{0} \frac{d}{dt}\|w(t,\cdot)\|^2_{\LL^2(B_{\rho})}dt=\|w(-\rho^2,\cdot)\|^2_{\LL^2(B_{\rho})}\geq 0,$$
this yields
$$\int_{\cc_{\rho}}\nabla w\cdot a\nabla w \leq \int_{\cc_{\rho}}\nabla w\cdot \left(h-\fint_{\cc_{\rho}}h+a(g-\xi)\right).$$
By uniform ellipticity assumption \eqref{elliptic} on $a$, \eqref{energyest} follows. The combination of \eqref{excesdapp}, \eqref{energyest} and the triangle inequality yields \eqref{exc1}. Now, we conclude by a Campanato iteration. Setting $0<\theta=\frac{r}{\rho}\leq 1$, we rewrite \eqref{exc1} as
$$\text{Exc}(\nabla u+g,\theta \rho)\leq C_1\left(\theta^{2\alpha'}\text{Exc}(\nabla u+g,\rho)+\theta^{-d-2}\fint_{\cc_{\rho}}\left(\vert h-\fint_{\cc_{\rho}} h\vert^2+\vert g-\fint_{\cc_{\rho}}g\vert^2\right)\right).$$
We divide by $(\theta\rho)^{2\alpha}$ and take the supremum over $\rho\in [\frac{r_*}{\theta},R]$ :
\begin{align}
\sup_{r\in [r_*,\theta R]}\frac{1}{r^{2\alpha}}\text{Exc}(\nabla u+g,r)\leq & C_1\Big(\theta^{2(\alpha'-\alpha)}\sup_{r\in [r_*, R]}\frac{1}{r^{2\alpha}}\text{Exc}(\nabla u+g,r)&\nonumber\\
&+\theta^{-d-2-2\alpha}\sup_{r\in [r_*, R]}\frac{1}{r^{2\alpha}}\fint_{\cc_{r}}\left(\vert h-\fint_{\cc_r} h\vert^2+\vert g-\fint_{\cc_{r}}g\vert^2\right)\Big).&\label{appendixA1}
\end{align}
We now choose $\theta = \theta(d,\lambda,\alpha)\leq 1$ so small that $C_1\theta^{2(\alpha'-\alpha)}\leq \frac{1}{2}$. By using
$$\sup_{r\in [r_*, R]}\frac{1}{r^{2\alpha}}\text{Exc}(\nabla u+g,r)\leq \sup_{r\in [\theta R, R]}\frac{1}{r^{2\alpha}}\text{Exc}(\nabla u+g,r)+\sup_{r\in [r_*, \theta R]}\frac{1}{r^{2\alpha}}\text{Exc}(\nabla u+g,r), $$
we may absorb the second r.h.s term of the previous inequality into the l.h.s of \eqref{appendixA1}, which yields
$$
\sup_{r\in [r_*,\theta R]}\frac{1}{r^{2\alpha}}\text{Exc}(\nabla u+g,r)\lesssim  \sup_{r\in [\theta R, R]}\frac{1}{r^{2\alpha}}\text{Exc}(\nabla u+g,r)+\sup_{r\in [r_*, R]}\frac{1}{r^{2\alpha}}\fint_{\cc_{r}}\left(\vert h-\fint_{\cc_r} h\vert^2+\vert g-\fint_{\cc_{r}}g\vert^2\right).$$
Since
$$\sup_{r\in [\theta R, R]}\frac{1}{r^{2\alpha}}\text{Exc}(\nabla u+g,r)\lesssim \frac{1}{R^{2\alpha}}\sup_{r\in[\theta R, R]}\frac{R^{d+2}}{r^{d+2}}\text{Exc}(\nabla u+g,R)\lesssim \frac{1}{R^{2\alpha}}\text{Exc}(\nabla u+g,R),$$
this yields \eqref{ScaseRfinite} in the case $R<+\infty$. In the case $R=\infty$ we obtain \eqref{ScaseRinfinite} in the limit $R\rightarrow +\infty$ by the square integrability of $\nabla u+g$ on $\mathbb{R}^{d+1}$, in form of
 $$\limsup_{R\rightarrow +\infty}\text{Exc}(\nabla u+g, R)\leq \limsup_{R\rightarrow +\infty}\fint_{\cc_{R}}\vert \nabla u+g\vert^2=0.$$
\textbf{ Step 2. }Proof of \eqref{meanvalue2}. We split this step into two parts.\newline
\newline
\textbf{ Substep 2.1. } Proof that for all $\rho>0$, there exists a unique $\xi_{\rho}\in\mathbb{R}^d$ such that 
\begin{equation}
\text{Exc}(\nabla u+g,\rho)=\fint_{\cc_{\rho}}\vert \nabla u+g-(\xi_{\rho}+\nabla\phi_{\xi_{\rho}})\vert^2,
\label{excatteint}
\end{equation}
and for all $r_*\leq r\leq R$ 
\begin{equation}
\vert \xi_r-\xi_R\vert^2\lesssim \sup_{\rho\in [r,R]}\bigg(\frac{R}{\rho}\bigg)^{2\alpha}\text{Exc}(\nabla u+g,\rho)+\sup_{\rho\in[r_*,R]}\left(\frac{R}{\rho}\right)^{2\alpha}\fint_{\cc_{\rho}}\left(\left\vert h-\fint_{\cc_{\rho}}h\right\vert^2+\left\vert g-\fint_{\cc_{\rho}}g\right\vert^2\right).
\label{dependanceenrho}
\end{equation}
We start by proving \eqref{excatteint}. Fix $\rho>0$ and define
$$f : \xi\in\mathbb{R}^d \longmapsto \fint_{\cc_{\rho}} \vert \nabla u+g-(\xi+\nabla \phi_{\xi})\vert^2.$$
$f$ is a continuous function and the mean value property of $\phi$, namely for all $R\geq r_*$: 
\begin{equation}
\fint_{\bb_{R}}\vert\nabla\phi_{\xi}+\xi\vert^2\geq \frac{1}{2}\vert\xi\vert^2,
\label{meanvaluephi}
\end{equation}
shows that $f$ is coercive. Consequently, $\xi_{\rho}$ in \eqref{excatteint} exists. On the other hand, $\xi_{\rho}$ is unique. Indeed, suppose that \eqref{excatteint} is satisfied for two vectors $\xi_1$ and $\xi_2$. We have 
$$\text{Exc}(\nabla u+g,\rho)=\fint_{\cc_{\rho}}\vert \nabla u+g-(\xi_1+\nabla\phi_{\xi_1})\vert^2=\fint_{\cc_{\rho}}\vert \nabla u+g-(\xi_2+\nabla\phi_{\xi_2})\vert^2,$$
and in particular 
$$2\text{Exc}(\nabla u+g,\rho)=\fint_{\cc_{\rho}}\bigg(\vert \nabla u+g-(\xi_1+\nabla\phi_{\xi_1})\vert^2+\vert \nabla u+g-(\xi_2+\nabla\phi_{\xi_2})\vert^2\bigg).$$
The parallelogram identity yields
$$2\text{Exc}(\nabla u+g,\rho)=\fint_{\cc_{\rho}}\frac{1}{2}\vert \xi_1-\xi_2 +\nabla\phi_{\xi_1-\xi_2}\vert^2+2\vert \nabla u+g-(\frac{\xi_1+\xi_2}{2}+\nabla\phi_{\frac{\xi_1+\xi_2}{2}})\vert^2.$$
We infer that
$$\text{Exc}(\nabla u+g,\rho)\geq\fint_{\cc_{\rho}}\frac{1}{4}\vert \xi_1-\xi_2+\nabla \phi_{\xi_1-\xi_2}\vert^2+\text{Exc}(\nabla u+g,\rho),$$
and so 
$$\fint_{\cc_{\rho}}\vert \xi_1-\xi_2+\nabla\phi_{\xi_1-\xi_2}\vert^2=0,$$
which gives $\xi_1=\xi_2$ using the estimate \eqref{meanvaluephi}.\newline
\newline
We turn to the proof of \eqref{dependanceenrho}. It is enough to prove that 
\begin{equation}
\forall r\leq R\leq 2r,\,\vert \xi_{r}-\xi_{R}\vert^2\lesssim \text{Exc}(\nabla u+g,R).
\label{dependanceenrho2}
\end{equation}
Indeed, we argue by a dyadic decomposition. Let $N\in\mathbb{N}$ be such that $2^{-(N+1)}R<r<2^{-N}R$. By \eqref{dependanceenrho2}, we have for all $n\in\{0...,N-1\}$
$$\vert \xi_r-\xi_{2^{-N}R}\vert^2\lesssim \text{Exc}(\nabla u+g,2^{-N}R)\quad \text{and}\quad \vert \xi_{2^{-(n+1)}R}-\xi_{2^{-n}R}\vert^2\lesssim \text{Exc}(\nabla u+g, 2^{-n}R).$$
Thus, by the triangle inequality followed by the excess decay \eqref{regpara2} and the fact that $\sum_{n=0}^{+\infty}2^{-n\alpha}<+\infty$, we have
\begin{align}
\vert \xi_{r}-\xi_{R}\vert^2 &\lesssim \bigg(\sum_{n=0}^{N}\sqrt{\text{Exc}(\nabla u+g,2^{-n}R)}\bigg)^2&\nonumber\\
&\stackrel{\eqref{ScaseRfinite}}{\lesssim} \left(\sum_{n=0}^{N}2^{-n\alpha}\left(\sqrt{\text{Exc}(\nabla u+g,R)}+\left(\sup_{\rho\in[r_*,R]}\left(\frac{R}{\rho}\right)^{2\alpha}\fint_{\cc_{\rho}}\left\vert h-\fint_{\cc_{\rho}}h\right\vert^2+\left\vert g-\fint_{\cc_{\rho}}g\right\vert^2\right)^{\frac{1}{2}}\right)\right)^2 &\nonumber\\
&\lesssim \text{Exc}(\nabla u+g,R)+\sup_{\rho\in[r_*,R]}\left(\frac{R}{\rho}\right)^{2\alpha}\fint_{\cc_{\rho}}\left(\left\vert h-\fint_{\cc_{\rho}}h\right\vert^2+\left\vert g-\fint_{\cc_{\rho}}g\right\vert^2\right)&\nonumber\\
&\lesssim \sup_{\rho\in [r,R]}(\frac{R}{\rho})^{2\alpha}\text{Exc}(\nabla u+g,\rho)+\sup_{\rho\in[r_*,R]}\left(\frac{R}{\rho}\right)^{2\alpha}\fint_{\cc_{\rho}}\left(\left\vert h-\fint_{\cc_{\rho}}h\right\vert^2+\left\vert g-\fint_{\cc_{\rho}}g\right\vert^2\right). &\nonumber
\end{align}
We now turn to the argument for \eqref{dependanceenrho2}. By \eqref{meanvaluephi} we have
$$\vert \xi_{r}-\xi_R\vert^2 \lesssim \fint_{\cc_r}\vert (\xi_{\rho}-\xi_{R})+\nabla \phi_{\xi_{r}-\xi_{R}}\vert^2,$$
which, by linearity of $\xi\mapsto \phi_{\xi}$, we may rewrite as 
$$\vert \xi_r-\xi_R\vert^2\lesssim\fint_{\cc_{r}}\vert (\xi_{\rho}+\nabla \phi_{\xi_{r}})-(\xi_{R}+\nabla\phi_{\xi_{R}})\vert^2,$$
so that, by the triangle inequality in $\LL^2(\cc_{\rho})$ and using that $r\leq R\leq 2r$, we obtain
$$\vert \xi_r-\xi_R\vert^2\lesssim\fint_{\cc_{r}}\vert \nabla u-(\xi_{r}+\nabla\phi_{\xi_{r}})\vert^2+\fint_{\cc_{R}}\vert \nabla u-(\xi_{R}+\nabla\phi_{R})\vert^2.$$
By definition of $\text{Exc}$ and using once more that $r\leq R\leq 2r$, this turns as desired into 
$$\vert \xi_r-\xi_R\vert^2\lesssim \text{Exc}(\nabla u+g,r)+\text{Exc}(\nabla u+g,R)\lesssim \text{Exc}(\nabla u+g,R).$$
\textbf{ Substep 2.2. }We prove \eqref{meanvalue2}. The starting point is \eqref{ScaseRfinite} in the more general form : for all $r\geq r_*$
\begin{align}
\sup_{\rho\in [r,R]}\bigg(\frac{R}{\rho}\bigg)^{2\alpha}\text{Exc}(\nabla u+g,\rho)\lesssim & \text{Exc}(\nabla u+g,R)+\sup_{\rho\in [r,R]}\bigg(\frac{R}{\rho}\bigg)^{2\alpha}\fint_{\cc_{\rho}}\left(\vert g-\fint_{\cc_{\rho}}g\vert^2+\vert h-\fint_{\cc_{\rho}}h\vert^2\right).&\label{ScaseRfinite2}
\end{align}
The estimates \eqref{dependanceenrho} and \eqref{ScaseRfinite2} combined with the triangle inequality yield
\begin{align}
\vert \xi_r\vert^2+\text{Exc}(\nabla u+g,r)\lesssim& \vert \xi_R\vert^2+\text{Exc}(\nabla u+g,R)+\sup_{\rho\in [r_*,R]}(\frac{R}{\rho})^{2\alpha}\fint_{\cc_{\rho}}\left(\vert g-\fint_{\cc_{\rho}}g\vert^2+\vert h-\fint_{\cc_{\rho}}h\vert^2\right).&\label{fstep2.2}
\end{align}
Using the triangle inequality and the definition of the excess in the form of
$$\fint_{\cc_r}\vert \nabla u+g\vert^2\lesssim \vert\xi_r\vert^2+\text{Exc}(\nabla u+g, r),$$
and
$$\vert \xi_{R}\vert^2+\text{Exc}(\nabla u+g,R)\lesssim \fint_{\cc_{R}}\vert \nabla u+g\vert^2,$$
we may finally pass from \eqref{fstep2.2} to \eqref{meanvalue2}.
\end{proof}
We finally recall the following property of average of $r_*$. The proof can be found in \cite[Estimation (139)]{gloria2014regularity}.
\begin{lemma} For all measurable function $f: \mathbb{R}^d\rightarrow \mathbb{R}^+$ there exists two constants $c$ and $C$ which depends only on the dimension $d$ such that 
\begin{equation}
c\int_{\mathbb{R}^d}\fint_{B_{r_*(x)}(x)} f(y)\dd y\, \dd x\leq \int_{\mathbb{R}^d} f(x)\dd x\leq C\int_{\mathbb{R}^d}\fint_{B_{r_*(x)}(x)} f(y)\dd y\, \dd x.
\label{averager*ctrl}
\end{equation}
\end{lemma}
\section{Caccioppoli's inequality}
We state here Caccioppoli's inequality for parabolic system. For a proof, see for instance \cite[Lemma 2]{bella2017liouville}.
\begin{lemma}[Caccioppoli estimate]\label{caccioppotech}There exists a constant $C$ depending on $\lambda$ such that for every $\rho\leq R$ and weak solution $u$ of 
$$\partial_{\tau}u-\nabla\cdot a\nabla u=0 \text{ in $\cc_R$,}$$
we have
$$\int_{\cc_{\rho}}\vert\nabla u(t,x)\vert^2\dd t\,\dd x\leq \frac{C}{(R-\rho)^2}\int_{\cc_{R}\backslash \cc_{\rho}}\left\vert u(t,x)-\fint_{\cc_{R}} u(s,y)\dd s\,\dd y\right\vert^2\dd t\,\dd x,$$
recalling that $C_R=(-R^2,0)\times \bb_R$.
\end{lemma}
\section{Proof of Theorem \ref{semigroup} under a functional inequality with oscillation}\label{proofoscillation}
We fix $T\geq 1$, $1\leq r\leq \sqrt{T}$ and the unit vector $e\in\mathbb{R}^d$. We only give the argument for \eqref{Sensitilem3}, \eqref{sensiothertest} is obtained combining the ideas of this section and the proof of Section \ref{sectionremovelog}. We make for simplicity the two additional assumptions:
\begin{itemize}
\item[(i)] $u$ is real-valued and $a$ is symmetric. We recall that this implies 
\begin{equation}
\|\nabla u(t,\cdot)\|_{\LL^{\infty}(\mathbb{R}^d)}\lesssim t^{-1},
\label{osc1}
\end{equation}
see \cite[Lemma 9.2]{armstrong2019quantitative}.
\item[(ii)]The coefficient field $a$ takes the form, for some $\chi\in \cc^{\infty}_c(\mathbb{R}^d)$ supported in $\bb_1$,
$$a:=\chi\star \tilde{a},$$
with a field $\tilde{a} : \mathbb{R}^d\rightarrow \mathbb{R}^{d\times d}$ which takes value into the set of uniformly elliptic and bounded matrices and with a probability law which satisfies the logarithm Sobolev inequality with oscillation \eqref{LSIoscillation}. In this setting, $\nabla\cdot ae\in \LL^{\infty}(\mathbb{R}^d)$ and 
\begin{equation*}
\|\nabla\cdot ae\|_{\LL^{\infty}(\mathbb{R}^d)}\lesssim_{\chi} 1.
\label{osc2}
\end{equation*}
We recall that this implies the following energy estimate: for all $R\geq 1$ and $z\in\mathbb{R}^d$
\begin{equation}
\int_{0}^1\fint_{\bb_R(z)}\vert\nabla u(t,x)\vert^2\dd x\, \dd t\lesssim_{\chi} 1,
\label{osc3}
\end{equation}
\end{itemize}
where a proof can be found in \cite[Lemma 2]{gloria2015corrector}. The first step is to estimate the derivative $\partial^{\text{osc}}_{x,\ell}q_r(T)$. We claim that for all $(x,\ell)\in\mathbb{R}^d\times [1,\infty)$
\begin{equation}
\vert \partial^{\text{osc}}_{x,\ell}q_r(T)\vert^2\lesssim_{\chi}(\ell+1)^{2d}\left((1+\log^2(T))\fint_{\bb_{\ell+1}(x)}g^2_r(y)\dd y+\fint_{\bb_{\ell+1}(x)}\vert\nabla v^T(0,y)\vert^2\dd y+\log(T)\int_{1}^T t^{-1}\fint_{\bb_{\ell+1}(x)}\vert\nabla v^T(t,y)\vert^2\dd y\, \dd t\right),
\label{osc4}
\end{equation}
with $v^T$ defined in \eqref{strategydualpb}.\newline
\newline
We fix $(x,\ell)\in\mathbb{R}^d\times [1,\infty)$ and we consider $\tilde{a}'$ and $\tilde{a}''$ such that $\tilde{a}'=\tilde{a}''=\tilde{a}$ on $\mathbb{R}^d\backslash \bb_{\ell}(x)$. We then set $a':=\chi\star \tilde{a}'$, $a'':=\chi\star \tilde{a}''$ and note that since $\chi$ is supported in $\bb_1$,
\begin{equation}
a'=a''=a \text{ on $\mathbb{R}^d\backslash \bb_{\ell+1}(x)$}.
\label{osc5}
\end{equation}
Using the notation $\delta u:=u(a',\cdot)-u(a'',\cdot)$, we have 
\begin{align}
q_r(a',T)-q_r(a'',T)=&\int_{\mathbb{R}^d}g_r(y)(a'(y)-a''(y))e\,\dd y+\int_{\mathbb{R}^d}g_r(y)a'(y)\int_{0}^T\nabla u(a',t,y)\dd t\,\dd y-\int_{\mathbb{R}^d}g_r(y)a''(y)\int_{0}^T\nabla u(a'',t,y)\dd t\, \dd y\nonumber\\
=&\int_{\mathbb{R}^d}g_r(y)(a'(y)-a''(y))e\,\dd y+\int_{\mathbb{R}^d}g_r(y)(a'(y)-a(y))\int_{0}^T\nabla u(a',t,y)\dd t\, \dd y\nonumber\\
&+\int_{\mathbb{R}^d}g_r(y)(a''(y)-a(y))\int_{0}^T\nabla u(a'',t,y)\dd t\, \dd y+\int_{\mathbb{R}^d}g_r(y)a(y)\int_{0}^T\nabla\delta u(t,y)\dd t\, \dd y.\label{osc6}
\end{align}
On the one hand, using \eqref{osc5} and Jensen's inequality, the first r.h.s term of \eqref{osc6} is dominated by 
$$\left\vert\int_{\mathbb{R}^d}g_r(y)(a'(y)-a''(y))e\,\dd y\right\vert^2\stackrel{\eqref{osc6}}{=}\left\vert\int_{\bb_{\ell+1}(x)}g_r(y)(a'(y)-a''(y))e\,\dd y\right\vert^2\lesssim_{\chi}(\ell+1)^{2d}\fint_{\bb_{\ell+1}(x)} g^2_r(y)\dd y,$$
which contributes to the first r.h.s term of
\eqref{osc4}. On the other hand, the second and third r.h.s term of \eqref{osc6} are treated the same way (we estimate below the term with $a'$) using \eqref{osc5}, \eqref{osc3} (with $a$ replaced by $a'$), \eqref{osc1}, Cauchy-Schwarz's and Jensen's inequality
\begin{align*}
\left\vert\int_{\mathbb{R}^d}g_r(y)(a'(y)-a(y))\int_{0}^T\nabla u(a',t,y)\dd t\, \dd y\right\vert^2
\stackrel{\eqref{osc5}}{\lesssim_{\chi}}&\left(\int_{\bb_{\ell+1}(x)}g_r(y)\left\vert\int_{0}^T\nabla u(a',t,y)\dd t\right\vert\dd y\right)^2\\
\lesssim& \left(\int_{\bb_{\ell+1}(x)}g_r(y)\left\vert\int_{0}^1\nabla u(a',t,y)\dd t\right\vert\dd y\right)^2\\
&+\left(\int_{\bb_{\ell+1}(x)}g_r(y)\left\vert\int_{1}^T\nabla u(a',t,y)\dd t\right\vert\dd y\right)^2\\
\stackrel{\eqref{osc1}}{\lesssim}& (\ell+1)^{2d}\int_{0}^1\fint_{\bb_{\ell+1}(x)}\vert\nabla u(a',t,y)\vert^2\dd y\, \dd t\fint_{\bb_{\ell+1}(x)} g^2_r(y)\dd y\\
&+\log^2(T)(\ell+1)^d\fint_{\bb_{\ell+1}(x)}g^2_r(y)\dd y\\
\stackrel{\eqref{osc3}}{\lesssim_{\chi}}&(\ell+1)^{2d}(1+\log^2(T))\fint_{\bb_{\ell+1}(x)}g^2_r(y)\dd y,
\end{align*}
which contributes to the first r.h.s term of \eqref{osc4}. It remains to control the fourth r.h.s term of \eqref{osc6}. To this aim, we first write the equation solved by $\delta u$, which we deduce from \eqref{equationu}
\begin{equation}
\left\{
    \begin{array}{ll}
        \partial_{\tau}\delta u-\nabla\cdot a\nabla\delta  u=\nabla\cdot (a-a'')\nabla u(a'',\cdot)-\nabla\cdot (a-a')\nabla u(a',\cdot) & \text{in $(0,+\infty)\times\mathbb{R}^d$}, \\
   \delta u(0)=\nabla\cdot (a'-a'')e. & 
    \end{array}
\right.
\label{osc7}
\end{equation}
Thus, by testing \eqref{strategydualpb} with $\delta u$ and \eqref{osc7} with $v^T$, we deduce that
\begin{align}
\int_{\mathbb{R}^d}g_r(y)a(y)\int_{0}^T\nabla\delta u(t,y)\dd t\, \dd y=&\int_{\mathbb{R}^d}\nabla v^T(0,y)\cdot (a'(y)-a''(y))e\,\dd y+\int_{\mathbb{R}^d}\int_{0}^T(a(y)-a''(y))\nabla u(a'',t,y)\cdot \nabla v^T(t,y)\dd t\, \dd y\nonumber\\
&-\int_{\mathbb{R}^d}\int_{0}^T(a(y)-a'(y))\nabla u(a',t,y)\cdot \nabla v^T(t,y)\dd t\, \dd y.\label{osc8}
\end{align}
The first r.h.s term of \eqref{osc8} is dominated with \eqref{osc5} and gives the second r.h.s term of \eqref{osc4}. The second and third r.h.s term of \eqref{osc8} are dominated the same way (we estimate below the term with $a'$) using \eqref{osc1}, \eqref{osc3}, \eqref{osc5}, \eqref{strat_energyvT} (applied with $r=\ell$ and $z=x$), Cauchy-Schwarz's and Jensen's inequality
\begin{align*}
\left\vert\int_{\mathbb{R}^d}\int_{0}^T(a(y)-a'(y))\nabla u(a',t,y)\cdot \nabla v^T(t,y)\dd t\, \dd y\right\vert^2\stackrel{\eqref{osc5}}{\lesssim_{\chi}}&\left(\int_{\bb_{\ell+1}(x)}\int_{0}^T\vert\nabla u(a',t,y)\vert\vert \nabla v^T(t,y)\vert\dd t\, \dd y\right)^2\\
\stackrel{\eqref{osc1}}{\lesssim}& (\ell+1)^{2d}\int_{0}^1\fint_{\bb_{\ell+1}(x)}\vert \nabla u(a',t,y)\vert^2\dd y\, \dd t\int_{0}^1\fint_{\bb_{\ell+1}(x)}\vert \nabla v^T(t,y)\vert^2\dd y\, \dd t\\
&+\left(\int_{1}^T t^{-1}\int_{B_{\ell+1}(x)}\vert\nabla v^T(t,y)\vert\dd y\,\dd t\right)^2\\
\stackrel{\eqref{osc3}, \eqref{strat_energyvT}}{\lesssim_{\chi}}&(\ell+1)^{2d}\left(\fint_{\bb_{\ell+1}(x)}g^2_r(y)\dd y+\log(T)\int_{1}^{T}t^{-1}\fint_{\bb_{\ell+1}(x)}\vert \nabla v^T(t,y)\vert^2\dd y\, \dd t\right),
\end{align*}
which contributes to the first and third r.h.s term of \eqref{osc4} and concludes the proof.\newline
\newline
We now control the entropy of $q_r(T)$ by applying \eqref{LSIoscillation}, using \eqref{osc4}, the identity $\int_{\mathbb{R}^d}\fint_{\bb_{\ell+1}(x)} \dd x=\int_{\mathbb{R}^d}$, $\int_{\mathbb{R}^d} g^2_r(y)\dd y\lesssim r^{-d}$ and the plain energy estimate $\int_{\mathbb{R}^d}\vert\nabla v^T(t,y)\vert^2\dd y\lesssim \int_{\mathbb{R}^d}g^2_r(y)\dd y$,
\begin{align*}
\text{Ent}(q_r(T))\lesssim_{\chi}&\int_{1}^{+\infty}\ell^{-d}e^{-\frac{1}{C}\ell^{\beta}}(\ell+1)^d(1+\log^2(T))\int_{\mathbb{R}^d}\fint_{\bb_{\ell+1}(x)}g^2_r(y)\dd y\, \dd x\, \dd \ell\\
&+\int_{1}^{+\infty}\ell^{-d}e^{-\frac{1}{C}\ell^{\beta}}(\ell+1)^d\int_{\mathbb{R}^d}\fint_{\bb_{\ell+1}(x)}\vert\nabla v^T(0,y)\vert^2\dd y\, \dd x\, \dd \ell\\
&+\log(T)\int_{1}^{+\infty}\ell^{-d}e^{-\frac{1}{C}\ell^{\beta}}(\ell+1)^d\int_{1}^T t^{-1}\int_{\mathbb{R}^d}\fint_{\bb_{\ell+1}(x)}\vert \nabla v^T(t,y)\vert^2\dd y\, \dd x\, \dd t\, \dd \ell\\
\lesssim &\, r^{-d}(1+\log^2(T))\int_{1}^{+\infty}\ell^d e^{-\frac{1}{C}\ell^{\beta}}\dd \ell\lesssim_{\beta,C} r^{-d}(1+\log^2(T)).
\end{align*}
To conclude, the $\log^2(T)$ correction may be removed following the argument of Subsection \ref{sectionremovelog}, and the control of the entropy yields control of higher moments and provide stretched exponential moments. 
\newline
\newline
\textbf{Acknowledgements.} I would like to warmly thank my PhD advisor Antoine Gloria for suggesting this problem to me and for many
suggestions he made for improving this paper.

\bibliographystyle{plain}
\bibliography{references.bib}

\begin{thebibliography}{10}

\bibitem{abdulle2020analytical}
Assyr Abdulle, Doghonay Arjmand, and Edoardo Paganoni.
\newblock Analytical and numerical study of a modified cell problem for the
  numerical homogenization of multiscale random fields.
\newblock {\em arXiv preprint arXiv:2007.10828}, 2020.

\bibitem{armstrong2018quantitative}
Scott Armstrong, Alexandre Bordas, and Jean-Christophe Mourrat.
\newblock Quantitative stochastic homogenization and regularity theory of
  parabolic equations.
\newblock {\em Analysis \& PDE}, 11(8):1945--2014, 2018.

\bibitem{armstrong2016mesoscopic}
Scott Armstrong, Tuomo Kuusi, and Jean-Christophe Mourrat.
\newblock Mesoscopic higher regularity and subadditivity in elliptic
  homogenization.
\newblock {\em Communications in Mathematical Physics}, 347(2):315--361, 2016.

\bibitem{armstrong2017additive}
Scott Armstrong, Tuomo Kuusi, and Jean-Christophe Mourrat.
\newblock The additive structure of elliptic homogenization.
\newblock {\em Inventiones mathematicae}, 208(3):999--1154, 2017.

\bibitem{armstrong2019quantitative}
Scott Armstrong, Tuomo Kuusi, and Jean-Christophe Mourrat.
\newblock {\em Quantitative stochastic homogenization and large-scale
  regularity}, volume 352.
\newblock Springer, 2019.

\bibitem{armstrong2016quantitative}
Scott Armstrong and Charles~K Smart.
\newblock Quantitative stochastic homogenization of convex integral
  functionals.
\newblock In {\em Annales scientifiques de l'Ecole normale sup{\'e}rieure},
  volume~49, pages 423--481. Societe Mathematique de France, 2016.

\bibitem{armstrong2016lipschitz}
Scott~N Armstrong and Jean-Christophe Mourrat.
\newblock Lipschitz regularity for elliptic equations with random coefficients.
\newblock {\em Archive for Rational Mechanics and Analysis}, 219(1):255--348,
  2016.

\bibitem{avellanedaLp}
Marco Avellaneda and Fang~Hua Lin.
\newblock Lp bounds on singular integrals in homogenization.
\newblock {\em Communications on pure and applied mathematics},
  44(8-9):897--910, 1991.

\bibitem{bella2017liouville}
Peter Bella, Alberto Chiarini, and Benjamin Fehrman.
\newblock A liouville theorem for stationary and ergodic ensembles of parabolic
  systems.
\newblock {\em Probability Theory and Related Fields}, 173(3-4):759--812, 2019.

\bibitem{duerinckx2017weighted2}
Mitia Duerinckx and Antoine Gloria.
\newblock Multiscale functional inequalities in probability: concentration
  properties.
\newblock {\em ALEA, Lat. Am. J. Probab. Math. Stat., in press}, 2019.

\bibitem{duerinckx2017weighted}
Mitia Duerinckx and Antoine Gloria.
\newblock Multiscale functional inequalities in probability: Constructive
  approach.
\newblock {\em Annales Henri Lebesgue}, 3:825--872, 2020.

\bibitem{fischer2017sublinear}
Julian Fischer and Felix Otto.
\newblock Sublinear growth of the corrector in stochastic homogenization:
  optimal stochastic estimates for slowly decaying correlations.
\newblock {\em Stochastics and Partial Differential Equations: Analysis and
  Computations}, 5(2):220--255, 2017.

\bibitem{friedman2008partial}
Avner Friedman.
\newblock {\em Partial differential equations of parabolic type}.
\newblock Courier Dover Publications, 2008.

\bibitem{gloria2016reduction}
Antoine Gloria and Zakaria Habibi.
\newblock Reduction in the resonance error in numerical homogenization ii:
  Correctors and extrapolation.
\newblock {\em Foundations of computational mathematics}, 16(1):217--296, 2016.

\bibitem{gloria2014optimal}
Antoine Gloria, Stefan Neukamm, and Felix Otto.
\newblock An optimal quantitative two-scale expansion in stochastic
  homogenization of discrete elliptic equations.
\newblock {\em ESAIM: Mathematical Modelling and Numerical Analysis},
  48(2):325--346, 2014.

\bibitem{gloria2015quantification}
Antoine Gloria, Stefan Neukamm, and Felix Otto.
\newblock Quantification of ergodicity in stochastic homogenization: optimal
  bounds via spectral gap on glauber dynamics.
\newblock {\em Inventiones mathematicae}, 199(2):455--515, 2015.

\bibitem{gloria2019quantitative}
Antoine Gloria, Stefan Neukamm, and Felix Otto.
\newblock Quantitative estimates in stochastic homogenization for correlated
  coefficient fields.
\newblock {\em Analysis \& PDE, in press}, 2020.

\bibitem{gloria2014regularity}
Antoine Gloria, Stefan Neukamm, and Felix Otto.
\newblock A regularity theory for random elliptic operators.
\newblock {\em Milan Journal of Mathematics}, 88(1):99--170, 2020.

\bibitem{gloria2016quantitative}
Antoine Gloria and James Nolen.
\newblock A quantitative central limit theorem for the effective conductance on
  the discrete torus.
\newblock {\em Communications on Pure and Applied Mathematics},
  69(12):2304--2348, 2016.

\bibitem{gloria2012optimal}
Antoine Gloria and Felix Otto.
\newblock An optimal error estimate in stochastic homogenization of discrete
  elliptic equations.
\newblock {\em The annals of applied probability}, volume 22:1--28, 2012.

\bibitem{gloria2015corrector}
Antoine Gloria and Felix Otto.
\newblock The corrector in stochastic homogenization: optimal rates, stochastic
  integrability, and fluctuations.
\newblock {\em arXiv preprint arXiv:1510.08290}, 2015.

\bibitem{gloria2017quantitative}
Antoine Gloria and Felix Otto.
\newblock Quantitative results on the corrector equation in stochastic
  homogenization.
\newblock {\em Journal of the European Mathematical Society},
  19(11):3489--3548, 2017.

\bibitem{gloria2011optimal}
Antoine Gloria, Felix Otto, et~al.
\newblock An optimal variance estimate in stochastic homogenization of discrete
  elliptic equations.
\newblock {\em The annals of probability}, 39(3):779--856, 2011.

\bibitem{josien2020annealed}
Marc Josien and Felix Otto.
\newblock The annealed calderon-zygmund estimate as convenient tool in
  quantitative stochastic homogenization.
\newblock {\em arXiv preprint arXiv:2005.08811}, 2020.

\bibitem{kozlov1979averaging}
Sergei~Mikhailovich Kozlov.
\newblock Averaging of random operators.
\newblock {\em Matematicheskii Sbornik}, 151(2):188--202, 1979.

\bibitem{lu2021optimal}
Jianfeng Lu, Felix Otto, and Lihan Wang.
\newblock Optimal artificial boundary conditions based on second-order
  correctors for three dimensional random elliptic media.
\newblock {\em arXiv preprint arXiv:2109.01616}, 2021.

\bibitem{papanicolaou1979boundary}
G.C Papanicolaou and S.R.S Varadhan.
\newblock Boundary value problems with rapidly oscillating random coefficients.
\newblock In {\em Colloquia Math. Soc., Janos Bolyai}, volume~27, pages
  853--873, 1979.

\end{thebibliography}


\begin{thebibliography}{000}

\bibitem{abdulle2020analytical}
Abdulle, Assyr and Arjmand, Doghonay and Paganoni, Edoardo.
\newblock Analytical and numerical study of a modified cell problem for the numerical homogenization of multiscale random fields.
\newblock In {\em arXiv preprint arXiv:2007.10828, (2020)}.

\bibitem{avellanedaLp}
Avellaneda, Marco and Lin, Fang Hua.
\newblock $L^p$ bounds on singular integrals in homogenization.
\newblock In {\em Communications on pure and applied mathematics, (1991)},
volume~44 ,
  pages 897--910. Wiley Online Library.

\bibitem{fischer2017sublinear}
Fischer, Julian and Otto, Felix.
\newblock Sublinear growth of the corrector in stochastic homogenization: optimal stochastic estimates for slowly decaying correlations.
\newblock In {\em Stochastics and Partial Differential Equations: Analysis and Computations, (2017)},
volume~5 ,
  pages 220--255. Springer.

\bibitem{armstrong2016quantitative}
Armstrong, Scott and Smart, Charles K.
\newblock Quantitative stochastic homogenization of convex integral functionals.
\newblock In {\em Annales scientifiques de l'Ecole normale sup{\'e}rieure, (2016)},
volume~49 ,
  pages 423--481. Societe Mathematique de France.
	
	\bibitem{armstrong2016lipschitz}
Armstrong, Scott N and Mourrat, Jean-Christophe.
\newblock Lipschitz regularity for elliptic equations with random coefficients.
\newblock In {\em Archive for Rational Mechanics and Analysis, (2016)},
volume~219 ,
  pages 255--348. Springer.
	
		\bibitem{gloria2014optimal}
Gloria, Antoine and Neukamm, Stefan and Otto, Felix.
\newblock An optimal quantitative two-scale expansion in stochastic homogenization of discrete elliptic equations.
\newblock In {\em ESAIM: Mathematical Modelling and Numerical Analysis, (2014)},
volume~48 ,
  pages 325--346. EDP Sciences.
	
	\bibitem{armstrong2018quantitative}
Armstrong, Scott and Bordas, Alexandre and Mourrat, Jean-Christophe.
\newblock Quantitative stochastic homogenization and regularity theory of parabolic equations.
\newblock In {\em Analysis \& PDE, (2018)},
volume~11 ,
  pages 1945--2014. Mathematical Sciences Publishers.

\bibitem{gloria2016quantitative}
Gloria, Antoine and Nolen, James.
\newblock A quantitative central limit theorem for the effective conductance on the discrete torus.
\newblock In {\em Communications on Pure and Applied Mathematics, (2016)},
volume~69 ,
  pages 2304--2348. Wiley Online Library.

\bibitem{armstrong2019quantitative}
Armstrong, Scott and Kuusi, Tuomo and Mourrat, Jean-Christophe.
\newblock Quantitative stochastic homogenization and large-scale regularity.
\newblock volume~352 (2019), Springer.

\bibitem{gloria2015corrector}
Gloria, Antoine and Otto, Felix.
\newblock The corrector in stochastic homogenization: optimal rates, stochastic integrability, and fluctuations.
\newblock In {\em arXiv preprint arXiv:1510.08290, (2016)}.

\bibitem{bella2017liouville}
Bella, Peter and Chiarini, Alberto and Fehrman, Benjamin.
\newblock A Liouville theorem for stationary and ergodic ensembles of parabolic systems.
\newblock In {\em Probability Theory and Related Fields, (2019)},
volume~173 ,
  pages 759--812. Springer.
	
	\bibitem{gloria2017quantitative}
Gloria, Antoine and Otto, Felix.
\newblock Quantitative results on the corrector equation in stochastic homogenization.
\newblock In {\em Journal of the European Mathematical Society, (2017)},
volume~19 ,
  pages 3489--3548. 

	\bibitem{gloria2014regularity}
Gloria, Antoine and Neukamm, Stefan and Otto, Felix.\newblock A regularity theory for random elliptic operators.
\newblock In {\em Milan Journal of Mathematics, (2020)},
volume~88 ,
  pages 99--170. Springer. 
	
		\bibitem{duerinckx2017weighted2}
Duerinckx, Mitia and Gloria, Antoine.
\newblock Multiscale functional inequalities in probability: concentration properties.
\newblock In {\em ALEA, Lat. Am. J. Probab. Math. Stat., in press, (2019)}.

	\bibitem{gloria2019quantitative}
Gloria, Antoine and Neukamm, Stefan and Otto, Felix.\newblock Quantitative estimates in stochastic homogenization for correlated coefficient fields.
\newblock In {\em Analysis \& PDE, in press, (2020)}.

\bibitem{kozlov1979averaging}
Kozlov, Sergei Mikhailovich.
\newblock Averaging of random operators.
\newblock In {\em Matematicheskii Sbornik, (1979)},
volume~151 ,
  pages 188--202. Russian Academy of Sciences, Steklov Mathematical Institute of Russian.  

\bibitem{papanicolaou1979boundary}
Papanicolaou, G.C and Varadhan, S.R.S.
\newblock Boundary value problems with rapidly oscillating random coefficients.
\newblock In {\em Colloquia Math. Soc., Janos Bolyai, (1979)},
volume~27 ,
  pages 853--873. 
	
	\bibitem{duerinckx2017weighted}
Duerinckx, Mitia and Gloria, Antoine.
\newblock Multiscale functional inequalities in probability: Constructive approach.
\newblock In {\em Annales Henri Lebesgue, (2020)},
volume~3 ,
  pages 825--872. ENS Rennes.  
	
	\bibitem{bella2015quantitative}
Bella, Peter and Giunti, Arianna and Otto, Felix.
\newblock Quantitative stochastic homogenization: local control of homogenization error through corrector.
\newblock In {\em arXiv preprint arXiv:1504.02487, (2015)}.

\bibitem{ben2017moment}
Ben-Artzi, Jonathan and Marahrens, Daniel and Neukamm, Stefan.
\newblock Moment bounds on the corrector of stochastic homogenization of non-symmetric elliptic finite difference equations.
\newblock In {\em Communications in Partial Differential Equations, (2017)},
volume~42 ,
  pages 179--234. Taylor \& Francis.

\bibitem{conlon2000homogenization}
Conlon, Joseph and Naddaf, Ali and others.
\newblock On homogenization of elliptic equations with random coefficients.
\newblock In {\em Electronic Journal of Probability, (2000)},
volume~5.
  The Institute of Mathematical Statistics and the Bernoulli Society.
	
	\bibitem{gloria2012spectral}
Gloria, Antoine and Mourrat, Jean-Christophe.
\newblock Spectral measure and approximation of homogenized coefficients.
\newblock In {\em Probability theory and related fields, (2012)},
volume~154 ,
  pages 287--326.Springer.

\bibitem{gloria2011optimal}
Gloria, Antoine and Otto, Felix.
\newblock An optimal variance estimate in stochastic homogenization of discrete elliptic equations.
\newblock In {\em The annals of probability, (2011)},
volume~39 ,
  pages 779--856. Institute of Mathematical Statistics.

\bibitem{gloria2012optimal}
Gloria, Antoine and Otto, Felix.
\newblock An optimal error estimate in stochastic homogenization of discrete elliptic equations.
\newblock In {\em The annals of applied probability, (2012)},
volume~22 ,
  pages 1--28. JSTOR.
	
	\bibitem{armstrong2016mesoscopic}
Armstrong, Scott and Kuusi, Tuomo and Mourrat, Jean-Christophe.
\newblock Mesoscopic higher regularity and subadditivity in elliptic homogenization.
\newblock In {\em Communications in Mathematical Physics, (2016)},
volume~22 ,
  pages 315--361. Springer.
	
	\bibitem{armstrong2017additive}
Armstrong, Scott and Kuusi, Tuomo and Mourrat, Jean-Christophe.
\newblock The additive structure of elliptic homogenization.
\newblock In {\em Inventiones mathematicae, (2017)},
volume~208 ,
  pages 999--1154. Springer.
	
	\bibitem{gloria2015quantification}
Gloria, Antoine and Neukamm, Stefan and Otto, Felix.
\newblock Quantification of ergodicity in stochastic homogenization: optimal bounds via spectral gap on Glauber dynamics.
\newblock In {\em Inventiones mathematicae, (2015)},
volume~199 ,
  pages 455--515. Springer.
  
  \end{thebibliography}

\end{document}